\documentclass[11pt,reqno]{amsart}

 \allowdisplaybreaks
 
\usepackage{amsthm,amssymb,amsmath,epsfig,graphics,appendix,bm,mathrsfs,bbm,float}
\usepackage{graphicx}
\usepackage{hyperref}
\usepackage{bm}
\usepackage[dvipsnames, svgnames, x11names]{xcolor}
\usepackage{stmaryrd}  
\usepackage[left=1in, right=1in, top=1.1in,bottom=1.1in]{geometry}
\usepackage{enumitem}

\def\R{\mathbb R}
\def\E{\mathbb E}

\newtheorem{theorem}{Theorem}
\newtheorem{lemma}{Lemma}[section]
\newtheorem{definition}{Definition}[section]
\newtheorem{corollary}{Corollary}[section]
\newtheorem{remark}{Remark}[section]
\newtheorem{proposition}{Proposition}[section]
\newtheorem{example}{Example}[section]
\newtheorem{assumption}{Assumption}
\newtheorem{assumptionalt}{Assumption}[assumption]
\newenvironment{assumptionp}[1]{
\renewcommand\theassumptionalt{#1}
\assumptionalt
}{\endassumptionalt}

\DeclareMathOperator*{\esssup}{ess\,sup}

\usepackage{geometry}
\usepackage{xcolor}
\usepackage{amsfonts}
\usepackage{lipsum}                     
\usepackage{xargs}         
\usepackage[colorinlistoftodos,prependcaption,textsize=tiny, textwidth=2cm]{todonotes} 

\numberwithin{equation}{section}

\begin{document}
\author{}
\date{}
\title[BSDEs with nonlinear Young drivers]{Backward stochastic differential equations with nonlinear Young drivers II}

\author[J. Song]{Jian Song}\address{Research Center for Mathematics and Interdisciplinary Sciences; Frontiers Science Center for Nonlinear Expectations, Ministry of Education, Shandong University, Qingdao 266237, China}
\email{txjsong@sdu.edu.cn}

\author[H. Zhang]{Huilin Zhang}\address{Research Center for Mathematics and Interdisciplinary Sciences; Frontiers Science Center for Nonlinear Expectations (Ministry of Education), Shandong University, Qingdao 266237, China}
\email{huilinzhang@sdu.edu.cn}

\author[K. Zhang]{Kuan Zhang}\address{Research Center for Mathematics and Interdisciplinary Sciences, Shandong University, Qingdao 266237, China}
\email{201911819@mail.sdu.edu.cn}

\begin{abstract}
This paper continues our previous work (Part I, \cite{BSDEYoung-I}) on the well-posedness of backward stochastic differential equations (BSDEs) involving a nonlinear Young integral of the form $\int_{t}^{T}g(Y_r) \eta(dr,X_r)$, with particular focus on the case where the driver $\eta(t,x)$ is unbounded. To address this setting, we develop a new localization method that extends solvability from BSDEs with bounded drivers to those with unbounded ones. As a direct application, we derive a nonlinear Feynman-Kac formula for a class of partial differential equations driven by Young signals (Young PDEs). Moreover, employing the proposed localization method, we obtain error estimates that compare Cauchy-Dirichlet problems on bounded domains with their whole-space Cauchy counterparts, with special attention to non-Lipschitz PDEs.
\end{abstract}

\maketitle
\tableofcontents

\section{Introduction}

Let $(\Omega, \mathcal F, (\mathcal F_t)_{t\in[0,T]}, \mathbb{P})$ be a probability space
supporting a  $d$-dimensional Brownian motion $W$. Given functions $f$, $g$, $\eta$, and a diffusion process $X$, we consider the following backward stochastic differential equation (BSDE) with the solution pair $(Y_t,Z_t)$:
\begin{equation}\label{e:ourBSDE}
Y_{t} = \xi + \int_{t}^{T}f(r,X_r,Y_{r},Z_{r})dr + \sum_{i=1}^{M}\int_{t}^{T}g_{i}(Y_{r})\eta_{i}(dr,X_{r}) - \int_{t}^{T}Z_{r}dW_{r},\ t\in[0,T],
\end{equation}
where  $\int g(Y_r) \eta(dr,X_{r})$ is a {\it nonlinear Young integral} (we refer to \cite{BSDEYoung-I} and the references therein for more details on nonlinear Young integrals). In Part~I~\cite{BSDEYoung-I}, we established the well-posedness of Equation~\eqref{e:ourBSDE}, under the assumption that the driver $\eta(t,x)$ is uniformly bounded, that is, $\sup_{(t,x)\in[0,T]\times\R^d}|\eta(t,x)| < \infty$. In the present work, we aim to go beyond this restriction by allowing $\eta$ to grow polynomially in the spatial variable. This broader framework in particular  encompasses the important case where $\eta$ is realized from a fractional Brownian sheet, thereby capturing a wider class of stochastic signals of interest.

\subsection{Motivations and difficulties}

Our motivations stem from the following two problems.

(a) The first motivation is to establish the Feynman-Kac formula for partial differential equations driven by rough signals (rough PDEs) of the following form:
\begin{equation}\label{e:rough PDE}
\left\{
\begin{aligned}
- du(t,x) &= \mathcal{L} u(t,x) dt + g(u, \nabla u) \mathbf{M}(dt,x), \quad (t,x)\in (0,T)\times\R^d,\\
u(T,x) &= h(x),\quad x\in\R^d,
\end{aligned}
\right.
\end{equation}
where $\mathcal{L}$ is a second-order differential operator and $\mathbf{M} = (\mathrm{M},\mathbb{M})$ is a rough path with spatial parameter (see Nualart-Xia \cite{Nualart-Xia20} for the definition). 

When $\mathrm{M}$ is independent of $x$ and $\mathrm{M}(t)$ is $\alpha$-H\"older continuous for some  $\alpha \in (1/2,1)$, 
Equation~\eqref{e:rough PDE} is referred to as a Young PDE  since the integral $\int \cdots d\mathbf{M}_t$ reduces to a Young integral. Gubinelli et al.~\cite{gubinelli2006young} first proved that Young evolution equations of the form \eqref{e:rough PDE} admit unique mild solutions. Subsequently, Addona et al.~\cite{Addona2022} investigated the regularity of such solutions. The nonlinear Feynman–Kac representation for Young PDEs was derived by Diehl–Zhang~\cite{DiehlZhang} in the case where $\mathrm{M}$ is independent of $x$, while Hu–Lê~\cite{HuLe} proved a (linear) Feynman–Kac formula when $\mathrm{M}$ depends on $x$.

When $\mathbf{M}$ is independent of the spatial parameter, the theory of rough PDEs has primarily been developed in two frameworks. The first treats \eqref{e:rough PDE} as an evolution equation solved in the mild sense (cf. \cite{Gubinelli2010,deya2012non,Deya2019,Hesse2022,Hocquet2024}). Recently, Liang–Tang \cite{Liang2024a} studied the well-posedness of rough stochastic evolution equations driven by mixed signals. The second approach defines the solution $u$ as the limit of a sequence $\{u^n\}_{n\ge 1}$, where each $u^n$ solves \eqref{e:rough PDE} with $\mathbf{M}$ replaced by $\mathbf{M}^n$. Here $\mathbf{M}^n$ is the canonical lift of a smooth path (see, e.g., \cite[Definition 7.2]{FrizVictoir}) such that $\mathbf{M}^n \to \mathbf{M}$ in rough path space. This method allows for more flexible assumptions on $\mathcal{L}$, often relying on tools such as rough flow transformations and the Feynman–Kac formula (see, e.g., \cite{Caruana2009,caruana2011rough,Friz2014a,Diehl2017b}). More recently, Bugini et al. \cite{Bugini2025a} solved \eqref{e:rough PDE} in the strong sense via the Feynman–Kac formula, proving that $u \in C^{\alpha\text{-H\"ol}}([0,T]; C^{2}(\mathbb{R}^d;\mathbb{R}))$ for $\alpha < 1/2$ under mild regularity conditions. Apart from these two approaches, Buckdahn et al.~\cite{BuckdahnZhang} developed a viscosity framework for fully nonlinear rough PDEs via a transformation along the characteristic equations.

There are comparatively few results on Equation~\eqref{e:rough PDE} with a general nonlinear rough driver $\mathbf{M}(t,x)$. Nualart–Xia~\cite{Nualart-Xia20} investigated linear transport-type equations driven by nonlinear rough integrals, employing inverse transformations derived from nonlinear rough flows. A related contribution is due to Coghi–Nilssen~\cite{Coghi-Nilssen21}, who showed that the law of the solution to a McKean–Vlasov SDE with a nonlinear rough driver 
uniquely solves a (linear) Fokker–Planck equation with an unbounded rough driver. For the general case where both $g$ and $\mathbf{M}$ are nonlinear, Equation~\eqref{e:rough PDE} has been rarely studied in the literature.

In this paper, as a continuation of \cite{BSDEYoung-I}, we investigate the well-posedness of Equation~\eqref{e:rough PDE} with $\int \cdots \mathbf{M}(dt,x)$ interpreted as a \emph{nonlinear Young integral}, with particular focus on the case where $g(u,\nabla u)=g(u)$ is nonlinear and $\mathrm{M}$ is unbounded. A natural strategy is to derive a nonlinear Feynman–Kac formula via the associated BSDE, which motivates our study of Equation~\eqref{e:ourBSDE}. We further note that our forthcoming work will be devoted to BSDEs  with nonlinear \emph{rough} integrals, which is related to the \emph{rough} PDE \eqref{e:rough PDE}.

\

(b) The second motivation is to establish error estimates for localized PDEs, which concern the convergence rates of solutions under varying boundary conditions (see \cite{Barles1995,Kangro2000,Matache2004,Cont2005,Lamberton2008,Siska2012,Hilber2013}). Such estimates are crucial for assessing the accuracy of numerical solutions of parabolic PDEs, as finite-difference methods typically require bounded domains. Consequently, they have numerous applications in mathematical finance, particularly in the numerical computation of option prices (cf. Hilber et al.~\cite[Section~4.3]{Hilber2013} and Lamberton and Lapeyre~\cite[Section~5.3.2]{Lamberton2008}). We further note that in the context of stochastic partial differential equations (SPDEs), especially those driven by noise which is white in time, localization errors have been investigated in works such as \cite{Gerencser2017,candil2022localization}.

To the best of our knowledge, no such estimates have been established for non-Lipschitz PDEs. More precisely, assume that the Lipschitz constant of the mapping $(y,z)\mapsto f(t,x,y,z)$ exhibits polynomial growth in $x$. For a bounded domain $D\subseteq \R^d$,  consider the truncated Cauchy-Dirichlet problem:
\begin{equation}\label{e:localized PDEs with D}
\left\{\begin{aligned}
- \partial_{t} u_{D}(t,x) &= \mathcal{L} u_{D}(t,x) + f(t,x,u_{D},\nabla u_{D}), \ (t,x)\in(0,T)\times D,\\
u_{D}(t,x)&=h(x), \text{ for } (t,x)\in \{T\}\times \bar{D}\cup (0,T)\times \partial D.
\end{aligned}\right.
\end{equation}
Our objective is to prove  that
$u_{D}$ converges to a solution $u$ of the associated Cauchy problem as the domain $D$ expands to the whole space $\R^d$, and to derive an estimate for the convergence rate. This, in turn, naturally leads us to the study of BSDEs with unbounded drivers. The non-Lipschitz nature of Equation~\eqref{e:localized PDEs with D} originates from the growth of the coefficient function $f(t,x,y,z)$ in the spatial variable $x$, which directly yields an unbounded driver (see Remark~\ref{rem:local-error}). Furthermore, by the nonlinear Feynman–Kac representation for PDEs with Dirichlet boundary conditions, the BSDE corresponding to Equation~\eqref{e:localized PDEs with D} is terminated at the first exit time of the diffusion process $X$ from 
$D$. Hence, estimating localization errors for non-Lipschitz PDEs reduces to analyzing the convergence of the associated localized BSDEs with unbounded drivers.

\ \

The distinction between our BSDE~\eqref{e:ourBSDE} and related results in the literature—such as rough BSDEs studied by Diehl–Friz~\cite{DiehlFriz} and Young BSDEs considered in Diehl–Zhang~\cite{DiehlZhang}—was already discussed in Part I~\cite{BSDEYoung-I}. In addition, compared with the bounded case studied in~\cite{BSDEYoung-I}, the unboundedness of $\eta$ gives rise to two major challenges: the loss of Lipschitz continuity of the coefficient function and the potential unboundedness of the solution. We elaborate on these two issues below.

First, since the mapping $y \mapsto g(y)\eta(dr,X_r)$ is no longer Lipschitz, the fixed-point (Picard iteration) method used in \cite{BSDEYoung-I} is not applicable. Moreover, existing methods for non-Lipschitz BSDEs (see, e.g., \cite{bender2000bsdes,briand2008bsdes,li2025random}) cannot be applied to our setting either, as $\eta$ possesses only H\"older continuity in time. In fact, when applying the classical linearization method to establish an \emph{a priori} estimate for the solution (as in \cite[Lemma~5]{bender2000bsdes}, \cite[Lemma~7]{briand2008bsdes}, and \cite[Proposition~2.2]{li2025random}), a crucial step is to show the integrability of $\exp\{\int_{0}^{t} a_r \eta(dr,X_r)\}$, where $a_r := (g(Y_r)-g(0))/Y_r$. However, the upper bound of the nonlinear Young integral $\int_{0}^{t} a_r \eta(dr,X_r)$ depends on the $p$-variation of $a_r$, and thus on the $p$-variation of $Y_r$, which remains unknown in this \emph{a priori} estimate.

Second, the \emph{a priori} estimates developed in \cite{DiehlZhang, BSDEYoung-I} require the boundedness of the solution
$Y$. However, this assumption becomes invalid when $\eta$ is unbounded, as demonstrated by the explicit solution $Y_t =\sum_{i=1}^M\E_t\left[
\int_t^T\eta_i(ds, X_s)\right]$ obtained in the special case $\xi = 0$, $f(\cdot)\equiv0$ and $g(\cdot)\equiv 1$, which generally exhibits unboundedness. We stress that this work removes the boundedness requirement on the terminal condition $\xi$, thereby addressing an open problem posed in \cite[Remark~3]{DiehlZhang} regarding the extension of BSDE theory to unbounded terminal values.

\ \ \ \ 

To overcome the aforementioned difficulties, in this paper we  employ a new localization method to solve BSDEs. More precisely, we study the convergence of a sequence of localized BSDEs truncated at time $T_D$, with $T_D$ being the first exit time of the diffusion process $X$ from  $D$. For $t \le T_D$, since the integrator $\eta(dt,X_t)$ depends only on the mapping $x \mapsto \eta(\cdot,x)$ for $x \in D$, the solvability of each localized BSDE reduces to the case of BSDEs with bounded drivers, which was studied in Part~I~\cite{BSDEYoung-I}. The terminal value of a localized BSDE, denoted by $\xi^D$, is then chosen so that $\xi^D \to \xi$ as $D \to \mathbb{R}^d$.

This localization analysis relies primarily on the regularity properties of BSDE solutions. Indeed, for $D \subset D'$ and $t \le T_D$, estimating the difference $Y^D_t - Y^{D'}_t$ reduces to controlling its terminal value
$|\xi^D-Y^{D'}_{T_{D}}|\le |\xi^{D}-\xi^{D'}| + |\xi^{D'}-Y^{D'}_{T_{D}}|$. The convergence of $|\xi^{D}-\xi^{D'}|$ is by assumption. 
Moreover, regularity results show that the $k$-th moment of $\xi^{D'} - Y^{D'}_{T_D}$ can be bounded in terms of moments of $T_D - T_{D'}$, which decays at the rate $\exp\{-C\text{dist}(\partial D,0)^2\}$, assuming $\langle X \rangle_T$ is bounded. By establishing quantitative estimates for BSDE solutions, particularly for the $p$-variation norm of $Y^D$, this exponential decay ensures that $|Y^D - Y^{D'}|$ becomes arbitrarily small as $D \to \mathbb{R}^d$, thereby proving the convergence of $Y^D$. We  remark that this localization method, which appears to be new even in the deterministic BSDE case, also plays a key role in the study of error estimation for localized PDEs (see  Section~\ref{subsec:BSDEs with non-Lip}).

We would like to point out that  this  localization method still relies on the assumption that  $g(\cdot)$ is bounded  (see Remark~\ref{rem:why g is bounded}). 
Indeed, the unboundedness of the diffusion process $X$ would otherwise compromise the integrability of the integral $\int g(Y_r)\eta(dr,X_r)$. Nevertheless, in the typical case $g(y) = y$, the linear BSDE can be solved directly by constructing its explicit solution.


\subsection{A summary of  applications} We conclude this introduction by summarizing two applications.

$(i)$  In Section~\ref{subsec:Nonlinear-FK}, we derive the nonlinear Feynman-Kac formula by establishing the connection between the solution to the nonlinear Young BSDE~\eqref{e:ourBSDE} and the solution to the Young PDE~\eqref{e:rough PDE}, where $\mathbf{M}(t,x)$ is replaced by $\eta(t,x)$. 
The solution to Equation~\eqref{e:rough PDE} is defined as the unique limit of solutions to the PDEs with Dirichlet boundary conditions and smooth drivers. As a specific example of Equation~\eqref{e:rough PDE}, we take $\eta$ to be a realization of a fractional Brownian sheet $B$ with Hurst parameters $H_0$ in time and $(H,\dots, H)$ in space. In this case, Equation~\eqref{e:rough PDE} can be interpreted as an SPDE driven by $\partial_{t}B(t,x)$ in the Stratonovich sense. Consequently, we establish the well-posedness of the equation and derive the nonlinear Feynman-Kac formula for SPDEs under the conditions
$H_0 + H/2>1$ and $dH<2H_0+1$.

$(ii)$ In Section~\ref{subsec:BSDEs with non-Lip}, by using the localization method developed in  the study of BSDEs, we prove that the unique viscosity solution of PDE~\eqref{e:localized PDEs with D} (with $D$ replaced by $D_n = \{x\in\R^d;|x| < n\}$) converges to a solution $u$ of the PDE on the entire space $\R^d$. In particular, the coefficients of the limiting PDE can be non-Lipschitz. Let $u^n$ denote the solution to the PDE on the domain $D_n$. Then, we have the following  error estimates for the localized PDEs
\[
|u^n(t,x) - u(t,x)|\le C \exp\left\{-n^2/C + C|x|^2\right\},\  \text{ for some constant }\ C>0.
\]

\

This paper is organized as follows. For the case when the driver $\eta(t,x)$ is unbounded, the existence and the uniqueness of the solution to nonlinear BSDEs are studied in Section~\ref{subsec:existence} and Section~\ref{subsec:uniqueness}, respectively. The linear BSDE is studied in Section~\ref{subsec:linear case}. Finally, applications are presented in Section~\ref{sec:applications}, and some miscellaneous results are summarized in Appendix~\ref{miscellaneous results}, \ref{append:B}, and \ref{append:C}.

\subsection{Notations}\label{sec:preliminaries} In this subsection, we gather the notations that will be used in this article. 
\ \ 

\begin{itemize}
\item {\it $\lesssim$:} The notation $f\lesssim g$ means $f\le Cg$ for some generic constant $C>0$, and $f \lesssim_{\zeta} g$ represents $f \le C_{\zeta} g$ for some constant $C_{\zeta}>0$ depending on $\zeta$.\\

\item {\it Probability space:} 
Let $T$ be a fixed positive real number. Let $\{W_t\}_{t\in[0,T]}$ be a $d$-dimensional standard
Brownian motion on the probability space $(\Omega,\mathcal{F}, \mathbb{P})$, and let $(\mathcal{F}_{t})_{t\in[0,T]}$ be the augmented filtration generated by $W$. We use the notation $\mathbb{E}_{t}[\cdot]:=\mathbb{E}\left[\cdot|\mathcal{F}_t\right]$.\\

\item{\it $L^k$-norms:}   For a  random vector $\xi$, we write $\|\xi\|_{L^{k}(\Omega)}$ for the $L^{k}$-norm of $\xi$ with $k\ge 1$, and the essential supremum norm is  denoted by 	\begin{equation*}		
\|\xi\|_{L^{\infty}(\Omega)}:=\esssup_{\omega\in\Omega}|\xi(\omega)|. 
\end{equation*}
Here $|\cdot|$ denotes the Euclidean norm for both vectors and matrices.

\item{\it Spaces of functions:} For  functions $g:[s,t]\rightarrow \mathbb{R}^{m}$ and $f:\mathbb{R}^{n}\rightarrow\mathbb{R}^{m}$, the uniform norms are denoted by
\[\|g\|_{\infty;[s,t]} := \esssup_{r\in[s,t]}|g_{r}|\text{ and }\|f\|_{\infty;\mathbb{R}^{n}} := \esssup_{y\in\mathbb{R}^{n}}|f(y)|.\]
For $k\ge 1$, denote by $L^{k}([s,t];\R^m)$ the space of functions on $[s,t]$ equipped with the norm
\begin{equation*}
\|g\|_{L^{k}(s,t)} := \Big\{\int_{s}^{t} |g_{r}|^k dr\Big\}^{\frac{1}{k}}.
\end{equation*}
We write 
$g\in C([s,t];\mathbb{R}^{m})$ if $g$ is continuous; $f\in C^{k}(\mathbb{R}^n;\mathbb{R}^{m})$ if $f$ is $k$-th continuously differentiable. We denote by $f\in C^{k}_{b}(\mathbb{R}^n;\mathbb{R}^{m})$ if $f\in C^{k}(\mathbb{R}^n;\mathbb{R}^{m})$ with bounded partial derivatives up to the $k$-th order; $f\in C^{\text{Lip}}(\mathbb{R}^n;\mathbb{R}^{m})$ if it is Lipschitz continuous. We denote
\[C^{p\text{-var}}([s,t];\mathbb{R}^{m}) := \left\{g\in C([s,t];\mathbb{R}^{m})\text{ such that }\|g\|_{p\text{-var};[s,t]}<\infty \right\},\]
where the $p$-variation is denoted by  
\[\|g\|_{p\text{-var};[s,t]}:= \Big(\sup_{\pi} \sum_{[t_{i},t_{i+1}]\in \pi}|g_{t_{i+1}}-g_{t_{i}}|^p \Big)^{1/p}.\]
Here, the supremum is taken over all partitions $\pi:=\{[t_i,t_{i+1}];s=t_1<t_2<...< t_n=t\}$.

For $\gamma\in(0,1]$, We denote 
$$ C^{\gamma\text{-H\"ol}}([s,t];\R^m):=\left\{g\in C([s,t];\mathbb{R}^{m})\text{ such that }\|g\|_{\gamma\text{-H\"ol};[s,t]}<\infty \right\},$$ 
where the $\gamma$-H\"older norm is defined as
\begin{equation*}
\|g\|_{\gamma\text{-H\"ol};[s,t]} := \sup_{a,b\in[s,t];  a\neq b} \frac{|g_{a} - g_{b}|}{|a-b|^{\gamma}}\,.
\end{equation*}

When $m=1$, we may omit the range space $\R^m=\R$ for the above-mentioned spaces of functions. For instance, we  write $C^{\text{Lip}}([s,t])$  for $C^{\text{Lip}}([s,t];\mathbb{R})$. 
\\

\item {\it Spaces of the driver $\eta
$:} We introduce some seminorms for the driver $\eta:[0,T]\times\mathbb{R}^{n}\rightarrow\mathbb{R}^{m}$ (see also Hu-Lê~\cite{HuLe}).   Given parameters $\tau, \lambda\in(0,1],\  \beta\in[0,\infty)$,  denote
\begin{equation*}
\begin{aligned}
\|\eta\|_{\tau,\lambda; \beta}&:=\sup _{\substack{0 \leq s<t \leq T \\ x, y \in \mathbb{R}^{n} ; x \neq y}} \frac{|\eta(s, x)-\eta(t, x)-\eta(s, y)+\eta(t, y)|}{|t-s|^{\tau}|x-y|^{\lambda}(1+|x|^{ \beta}+|y|^{ \beta})} \\ 
&\quad +\sup _{\substack{0 \leq s<t \leq T\\ x \in \mathbb{R}^{n}}} \frac{|\eta(s, x)-\eta(t, x)|}{|t-s|^{\tau}(1+|x|^{ \beta + \lambda})}+\sup _{\substack{0 \leq t \leq T \\ x, y \in \mathbb{R}^{n} ; x \neq y}} \frac{|\eta(t, y)-\eta(t, x)|}{|x-y|^{\lambda}(1+|x|^{ \beta}+|y|^{ \beta})} ,
\end{aligned}
\end{equation*}
\begin{equation*}
\begin{aligned}
\|\eta\|_{\tau,\lambda}&:=\sup _{\substack{0 \leq s<t \leq T \\ x, y \in \mathbb{R}^{n} ; x \neq y}} \frac{|\eta(s, x)-\eta(t, x)-\eta(s, y)+\eta(t, y)|}{|t-s|^{\tau}|x-y|^{\lambda}} \\ 
&\quad +\sup _{\substack{0 \leq s<t \leq T \\ x \in \mathbb{R}^{n}}} \frac{|\eta(s, x)-\eta(t, x)|}{|t-s|^{\tau}}+\sup _{\substack{0 \leq t \leq T \\ x, y \in \mathbb{R}^{n} ; x \neq y}} \frac{|\eta(t, y)-\eta(t, x)|}{|x-y|^{\lambda}}.
\end{aligned}
\end{equation*}
Here, $\tau$ and $\lambda$ are (weighted) H\"older  exponents in time and in space, respectively. We introduce the following spaces of functions
\begin{equation*}
\begin{split}
&\ \ \ \ \ \ \ \ \ \ \  C^{\tau,\lambda}([0,T]\times\mathbb{R}^{n};\mathbb{R}^{m}):=\left\{\eta: [0,T]\times \mathbb{R}^{n}\rightarrow \mathbb{R}^{m}\text{ such that }\|\eta\|_{\tau,\lambda}<\infty\right\},\\
&\ \ \ \ \ \ \ \ \ \ \  C^{\tau,\lambda;\beta}([0,T]\times\mathbb{R}^{n};\mathbb{R}^{m}):=\left\{\eta: [0,T]\times \mathbb{R}^{n}\rightarrow \mathbb{R}^{m}\text{ such that }\|\eta\|_{\tau,\lambda;\beta}<\infty\right\}.   
\end{split}    
\end{equation*}
Clearly, $C^{\tau,\lambda}([0,T]\times\mathbb{R}^{n};\mathbb{R}^{m})\subseteq C^{\tau,\lambda;\beta}([0,T]\times\mathbb{R}^{n};\mathbb{R}^{m}).$

\begin{remark}\label{rem:eta}
In particular, if we assume  $\sup_{x\in \R}|\eta(0,x)|< \infty$, then $ \|\eta\|_{\tau,\lambda}<\infty$ implies that $|\eta(t,x)|$ is uniformly bounded for all $(t,x)\in[0,T]\times \R^d$ (see \cite[Lemma~A.2]{BSDEYoung-I}). Letting $\tilde \eta(t,x):=\eta(t, x)-\eta(0, x)$,  we have $\tilde \eta(0,x)\equiv0$ and  $\tilde \eta(dt, x_t) =\eta(dt, x_t)$. Thus, we may assume throughout the paper that $\eta(0,x)\equiv 0$ for Equation~\eqref{e:ourBSDE}, and then the condition $\|\eta\|_{\tau, \lambda}<\infty$ yields the uniform boundedness of $|\eta(t,x)|$. We also remark that $\|\cdot\|_{\tau,\lambda;0}$ differs from $\|\cdot\|_{\tau,\lambda}$, with  $\|\cdot\|_{\tau,\lambda;0}\le \|\cdot\|_{\tau,\lambda}$.
\end{remark}

\begin{remark}\label{rem:fBs}
A typical example of $\eta$ one may keep in mind is a realization of a fractional Brownian sheet $\{B(t,x), (t,x)\in[0,T]\times \mathbb{R}^{n}\}$    with Hurst parameters $H_i\in (0,1),i=0,1, \dots,n$, which is a centered Gaussian random field  with covariance
\begin{align*}
&\mathbb{E}\left[B(t,x)B(s,y)\right] \\
&= \frac{1}{2^{n+1}}\left(|t|^{2H_0} + |s|^{2H_0} - |t-s|^{2H_0} \right)\prod_{i=1}^{n}\left(|x_{i}|^{2H_i}+|y_{i}|^{2H_i} - |x_i - y_i|^{2H_i}\right).
\end{align*}
More details can be found in, e.g., 
\cite[Lemma~A.4]{BSDEYoung-I}.
\end{remark}

\item{\it Seminorms for right continuous with left limits (RCLL) processes:}
Let $[s,t]\subset[0,T]$ be a fixed interval, and $T_{1},T_{0}$ be two stopping times that satisfy $s\le T_{1}\le T_{0}\le t$. Assume that $X:\Omega\times[s,t]\rightarrow \R^n$ is an RCLL adapted process. For $p\ge 1$ and $k\ge 1$, denote
\begin{equation*}
m^{(s,t)}_{p,k}(X;[T_{1},T_{0}]):=\left\{\esssup_{\omega\in\Omega, u\in[s,t]}\mathbb{E}_{u}\left[\|X\|^{k}_{p\text{-var};[T_{1}\vee (u\land T_{0}),T_{0}]}\right]\right\}^{\frac{1}{k}}.
\end{equation*}
In the following, we will omit the superscripts $s$ and $t$, if there is no ambiguity.\\

\item{\it Spaces of $(Y,Z)$:} Now, we introduce the norm that will be used for the solution $(Y,Z)$ of Eq.~\eqref{e:ourBSDE}. Let $[s,t]\subseteq [0,T]$ be a fixed interval. Denote 
\begin{equation*}
\begin{aligned}
\mathfrak{B}([s,t]):=\big\{(y,z) \text{ such that } &y:\Omega\times[s,t]\rightarrow\mathbb{R}^{n}\text{ is continuous and adapted, and }\\
&z:\Omega\times[s,t]\rightarrow\mathbb{R}^{m}\text{ is  progressively measurable}\big\},
\end{aligned}
\end{equation*}
where the dimensions $n$ and $m$ may vary depending on the context. For  $k\ge 1 $ and $(y,z)\in\mathfrak{B}([s,t])$, we define the norm 
\begin{equation*}
\|(y,z)\|_{p,k;[s,t]} := m_{p, k}(y;[s,t])+ \|z\|_{k\text{-BMO};[s,t]} + \|y_{t}\|_{L^{\infty}(\Omega)},
\end{equation*}
where
\begin{equation*}
\|z\|_{k\text{-BMO};[s,t]} := \left\{\esssup_{\omega\in\Omega, u\in[s,t]}\mathbb{E}_{u}\left[\left(\int_{u}^{t}\left|z_{r}\right|^{2} dr\right)^{\frac{k}{2}}\right]\right\}^{\frac{1}{k}}.
\end{equation*}
In Part I \cite{BSDEYoung-I}, the metric space for the solution of \eqref{e:ourBSDE} with a bounded driver is 
\begin{equation*}
\mathfrak B_{p,k}(s, t):=\Big\{ (y,z) \in \mathfrak{B}([s,t]):\  \|(y,z)\|_{p,k;[s,t]}<\infty \Big\},
\end{equation*}
equipped with the norm $\|(\cdot, \cdot)\|_{p,k;[s,t]}$. When the driver is unbounded, we use the following metric space for the solution of \eqref{e:ourBSDE} in this paper:
\begin{equation*}
\mathfrak{H}_{p,k}(s,t) := \left\{ (y,z)\in\mathfrak{B}([s,t]): \ \|(y,z)\|_{\mathfrak{H}_{p,k};[s,t]} < \infty \right\},
\end{equation*}
where 
\begin{equation}\label{e:for Hpk}
\|(y,z)\|_{\mathfrak{H}_{p,k};[s,t]} := \|y_{t}\|_{L^{k}(\Omega)} + \left\|\|y\|_{p\text{-var};[s,t]}\right\|_{L^{k}(\Omega)} + \Big\|\big|\int_{s}^{t}|z^{2}_{r}|dr\big|^{\frac{1}{2}}\Big\|_{L^{k}(\Omega)}.
\end{equation}

\begin{remark}\label{rem:boundedness of y}
The norm $\|(\cdot,\cdot)\|_{p,k;[s,t]}$ is rather stronger than the usual $L^2$-norm in the sense that  $\|(y,z)\|_{p,k;[s,t]}<\infty$ implies the boundedness of $y$, i.e., $\esssup\limits_{\omega\in\Omega, r\in[s,t]}|y_{r}|<\infty$. Indeed, 
\begin{equation}\label{eq:ess-mpk}
\esssup\limits_{\omega\in\Omega, r\in[s,t]}|y_{r}\mathbf{1}_{[T_1,T_0]}(r)| \le \|y_{T_0}\|_{L^{\infty}(\Omega)}+m_{p,1}(y;[T_1, T_0])\le \|(y,z)\|_{p,1;[T_1,T_0]}.
\end{equation}
\end{remark}

\item {\it Malliavin derivative:} For a Malliavin differentiable random variable $\xi\in\mathcal{F}_{T}$, we denote by  $D_{\cdot}(\xi)$ its Malliavin derivative (see, e.g., Nualart~\cite{nualart2006malliavin}). Denote 
\begin{align*}
\mathbb{D}^{1,2}:=&\bigg\{\xi\in\mathcal{F}_T: \xi \text{ is Malliavin differentiable}, \text{ and } \\
&\qquad \|\xi\|_{\mathbb{D}^{1,2}}:= \|\xi\|_{L^{2}(\Omega)} + \left\| \|D_{\cdot}(\xi)\|_{L^2([0,T])}\right\|_{L^{2}(\Omega)}<\infty\bigg\}.
\end{align*}	
\end{itemize}

\section{BSDEs with unbounded drivers}\label{sec:BSDE with unbounded X}

In this section, we investigate BSDE~\eqref{e:ourBSDE} with an unbounded driver $\eta \in C^{\tau,\lambda;\beta}([0,T]\times\mathbb{R}^{d};\mathbb{R}^M)$. The existence and uniqueness of the solution is established in Section~\ref{subsec:existence} and Section~\ref{subsec:uniqueness}, respectively. The linear case is treated separately in Section~\ref{subsec:linear case}, as it requires weaker assumptions and exhibits more refined properties due to its linear structure.

We assume $\eta(0,x)\equiv 0$ without loss of generality (see Remark~\eqref{rem:eta}). We make the following assumption to guarantee the existence of the nonlinear Young integral $\int \eta(dr, x_r)$ (see \cite[Section~2]{BSDEYoung-I}), where $x_{\cdot}\in C^{p\text{-var}}([0,T];\R^d)$ for some $p>2$.
\begin{assumptionp}{(H0)}\label{(H0)}
Let $p, \tau, \lambda$ be parameters satisfying  
\[p>2, \ \tau\in(1/2,1],\ \lambda\in(0,1], \text{ and }\tau+\frac{\lambda}{p}>1.\]
\end{assumptionp}
In many situations, one can choose $p>2$ for the path $x_t$ to be as close to $2$ as we want (e.g., when $x_t$ is a realization of a semi-martingale), and thus \ref{(H0)} can be replaced by the following weaker assumption:
\begin{assumptionp}{(H0')}\label{(H0')}
Let $\tau,\lambda$ be parameters satisfying  
\[\tau\in(1/2,1],\ \lambda\in(0,1], \text{ and }\tau+\frac{\lambda}{2}>1.\]
\end{assumptionp}


Now we impose conditions on $X, \xi, \eta$, and the coefficient functions below. 

\begin{assumptionp}{(A0)}\label{(A0)}
Let $x = (x^{1},...,x^{d})^{\top}\in\R^d$. For BSDE~\eqref{e:ourBSDE}, we assume  $X$ is given by 
\begin{equation}\label{e:X_t = x + ...}
X_{t} = x + \int_{0}^{t}\sigma(r,X_{r})dW_{r} + \int_{0}^{t} b(r,X_{r}) dr,\ t\in[0,T],
\end{equation}
where $\sigma(t,x):[0,T]\times\mathbb{R}^{d}\to\mathbb{R}^{d\times d}$ and $b(t,x):[0,T]\times\mathbb{R}^{d}\to\mathbb{R}^{d}$ are measurable functions that are globally Lipschitz in $x$ and satisfy the following condition:
\[\sup_{t\in[0,T],x\in \mathbb{R}^{d}}\left\{|\sigma(t,x)| \vee |b(t,x)|\right\}\le L,\]
for some positive constant $L$.
\end{assumptionp}

\begin{assumptionp}{(A1)}\label{(A1)} Let $\tau,\lambda,p$ be the parameters satisfying Assumption~\ref{(H0)}.

\begin{itemize}
\item[$(1)$] There exists $\varepsilon\in(0,1)$  such that 
\[ \beta\ge 0,\  \tau+\frac{1-\varepsilon}{p}>1, \text{ and }  \lambda+\beta<\frac{2 \varepsilon}{1+\varepsilon}\tau.\] 

\item[$(2)$]$\eta\in C^{\tau,\lambda;\beta}([0,T]\times\mathbb{R}^{d};\mathbb{R}^{M})$ and $g\in C^{2}_{b}(\mathbb{R}^{N};\mathbb{R}^{M})$ with  
\[\|g\|_{\infty;\mathbb{R}^{N}}\vee\|\nabla g\|_{\infty;\mathbb{R}^{N}}\vee \|\nabla^2 g\|_{\infty;\mathbb{R}^{N}} \le C_{1}.\] 

\item[$(3)$] $f(t,x,y,z):[0,T]\times\mathbb{R}^{d}\times\mathbb{R}^{N}\times\mathbb{R}^{N\times d}\rightarrow \mathbb{R}^{N}$ is a function satisfying  \begin{equation}\label{e:assump of f}
\left\{
\begin{aligned}
&|f(t,x,0,0)| \le C_{1}( 1 + |x|^{\frac{\lambda+\beta}{\varepsilon}\vee 1}),\\
&|f(t,x_1,y,z) - f(t,x_2,y,z)| \le C_1 (1 + |x_1|^{\mu'} + |x_2|^{\mu'}) |x_1 - x_2|^{\mu} ,\\
&|f(t,x,y_1,z_1) - f(t,x,y_2,z_2)| \le C_{\mathrm{Lip}} \left(|y_1 - y_2| + |z_1 - z_2|\right),\
\end{aligned}\right.
\end{equation}
for all $  (t,x_i,y_i,z_i)\in [0,T]\times\mathbb{R}^{d}\times\mathbb{R}^{N}\times\mathbb{R}^{N\times d}$, $i=1,2$, where $C_{1}$, $C_{\mathrm{Lip}}$, $\mu$, and $\mu'$ are positive constants with $\mu\in(0,1]$.
\end{itemize}
\end{assumptionp}		

\begin{remark}\label{rem:tau>3/4}
If we assume  \ref{(H0')} and condition (1) in \ref{(A1)} with $p>2$, we have
\begin{equation}\label{lambda+}
\lambda+\beta<2\tau-1 \text{ and } \tau>\frac34.
\end{equation} 

\begin{proof}[Proof of \eqref{lambda+}]
First, it is clear that $\tau+\frac{\lambda}{2}>1$ and $\tau + \frac{(1-\varepsilon)}{2}>1$. Next, since $\frac{(1+\varepsilon)(\lambda+\beta)}{\varepsilon \tau}$ is decreasing in $\varepsilon$ and, noting that $\varepsilon<2\tau-1$ by $\tau + \frac{1-\varepsilon}{2}>1$, we have $\frac{2(\lambda+\beta)}{2\tau-1}<\frac{(1+\varepsilon)(\lambda+\beta)}{\varepsilon \tau}<2$,
which yields that $\lambda+\beta<2\tau-1$. Also, noting that $2-2\tau<\lambda\le\lambda+\beta$, we have
$2-2\tau<2\tau-1$ and hence $\tau>\frac{3}{4}$.
\end{proof}

\end{remark}

\begin{remark}\label{rem:Hurst condition} In Remark \ref{rem:tau>3/4}, we have shown that given \ref{(H0')},   condition (1) in \ref{(A1)} implies \eqref{lambda+}. In this remark, we prove the converse, that is, given \ref{(H0')}, \eqref{lambda+} yields  condition (1) in \ref{(A1)} with $p>2$ to be determined.  This indicates that, if $p>2$ can be chosen arbitrarily close to $2$, then assuming \ref{(H0')}, \eqref{lambda+} is equivalent to condition (1) in \ref{(A1)} with $p=2$.

Now, assuming that  $\tau,\lambda\in(0,1]$ and $\beta\ge 0$ satisfy $\tau+\frac{\lambda}{2} > 1$ and $\lambda + \beta < 2\tau - 1$, we aim to  find $p>2$ and $\varepsilon\in(0,1)$ such that $(\tau,\lambda,\beta,\varepsilon,p)$ satisfies condition (1) in Assumption~\ref{(A1)}. Let $\delta\in(0, 2\tau-1)$ be a constant.  Thus, we can choose $p$ sufficiently close to $2$ such that  $\tau+\frac{1-\varepsilon}{p}>1$ with $\varepsilon=2\tau - 1-\delta$. Moreover,  the last condition of condition~$(1)$ in \ref{(A1)} holds if we choose $\delta$ sufficiently small, since $\lambda+\beta<2\tau-1$ and 
\begin{equation*}
\limsup_{\delta\downarrow 0}\frac{(1 + \varepsilon)(\lambda + \beta)}{\varepsilon\tau}= \limsup_{\delta\downarrow 0} \frac{(2\tau-\delta)(\lambda+\beta)}{(2\tau-1-\delta)\tau} <
 \lim_{\delta\downarrow 0} \frac{(2\tau-\delta)(2\tau-1)}{(2\tau-1-\delta)\tau} = 2.
\end{equation*}
\end{remark}

\begin{remark}\label{rem:tau lambda beta}
In the proof of Remark~\ref{rem:tau>3/4},  we have $\lambda+\beta<2\tau-1$, and hence $\frac{\lambda + \beta + 1}{\tau}<2.$
This fact, together with the condition $\frac{(1+\varepsilon)(\lambda+\beta)}{\varepsilon\tau}<2$ in Assumption~\ref{(A1)}, implies that 
\begin{equation*}
\frac{\lambda+\beta + (\frac{\lambda+\beta}{\varepsilon}\vee 1)}{\tau} < 2,
\end{equation*}
which will be used in  \nameref{proof:3} of Theorem~\ref{thm:existence of solution of unbounded BSDEs}. As a consequence, the term $\frac{\lambda+\beta}{\varepsilon}\vee 1$ also appears in Assumption~\ref{(A1)} for $f(t,x,0,0)$ and in the increment conditions for $|\Xi_t|$ in \ref{(T)} below.
\end{remark}

\begin{assumptionp}{(T)}\label{(T)}   
Let $X$ satisfy \ref{(A0)}, and let $(\tau, \lambda, \beta, \varepsilon,p)$ satisfy \ref{(A1)}. For the terminal condition $\xi$, we assume that there is a continuous adapted  process $\Xi:[0,T]\times\Omega\rightarrow\mathbb{R}^{N}$ with $ \Xi_{T}=\xi$ such that   for any stopping time $S\le T$, 
\begin{equation}\label{e:conditions of xi}
\left\{\begin{aligned}
&|\Xi_{t}|\le C_1 (1 + \|X\|^{\frac{\lambda+\beta}{\varepsilon}\vee 1}_{\infty;[0,t]}),\ \text{ for all }t\in[0,T]\text{ a.s.,}\\
&\mathbb{E}\left[\left\|\mathbb{E}_{\cdot}\left[\Xi_{T}\right]\right\|^{k}_{p\text{-var};[S,T]}\right]\le C_1 \left\{\mathbb{E}\left[|T-S|\right]\right\}^{l},\\
&\mathbb{E}\left[\left|\Xi_{T} - \Xi_{S}\right|^{k}\right]\le C_1 \left\{\mathbb{E}\left[|T-S|\right]\right\}^{l},
\end{aligned}\right.
\end{equation}
for some $k\in(1,2]$, $l>0$, and $C_1>0$  independent of $S$.
\end{assumptionp}	
\begin{remark}

When $N=1$, typical examples are given in Example~\ref{ex:assump (T)}. In particular, the terminal value $\xi$ is not necessarily bounded as assumed in \cite{DiehlFriz,DiehlZhang} as well as in \cite[Assumption (H1)]{BSDEYoung-I}.  
\end{remark}

For simplicity of notation, we denote $\Theta_{2} := (\tau,\lambda,\beta,p,k,l,\varepsilon,T, C_{1},C_{\text{Lip}},L,\|\eta\|_{\tau,\lambda;\beta}),$ where all parameters
appear in assumptions \ref{(H0)}, \ref{(A0)}, \ref{(A1)}, and \ref{(T)}. These assumptions are made throughout Sections~\ref{subsec:existence} and Section~\ref{subsec:uniqueness}, and without loss of generality, the proofs presented in this section  assume  $M=d=1$,  unless otherwise specified.

\subsection{The existence}\label{subsec:existence}

\begin{lemma}\label{lem:m_p,q of X}
Assume $X$ satisfies Assumption~\ref{(A0)}. Then for any $p>2$ and $q\ge 1,$ 
\begin{equation}\label{e:m_p,q X is bounded}
m_{p,q}(X;[0,T]) \lesssim_{p,q,L,T} 1.
\end{equation}
\end{lemma}

\begin{proof}
Note that $X_t$ has the representation $X_t = x + M_t + A_t$, where $M_t$
is the diffusion and $A_t$ is the drift. Then, since the quadratic variation process $\langle M\rangle$ is bounded by $L^2 T,$ \eqref{e:m_p,q X is bounded} follows directly from the BDG inequality for $p$-variation (see, e.g., \cite[Lemma~A.1]{BSDEYoung-I}).
\end{proof}

For $n\ge 1$, define the stopping time $T_n$ as follows (we stipulate that $\inf\{\varnothing\} := \infty$):
\begin{equation}\label{e:Tn}
T_{n} := T\land\inf\big\{t>0;\ |X_{t}| > n\big\}.
\end{equation}	

\begin{lemma}\label{lem:existence of X}
Assume that $X$ satisfies Assumption~\ref{(A0)}. Let $\{T_{n}\}_{n\ge 1}$ be  a sequence of stopping times given by \eqref{e:Tn}. Then for every $q>0$, there exists $C>0$ such that for all $n\ge |x|,$
\begin{equation}\label{e:moment-T-Tn}
\mathbb{E}\big[|T - T_{n}|^{q}\big] \le C \exp\left\{-(n-|x|)^{2}/C\right\}.
\end{equation}
Moreover, for any $h\in C^{\mathrm{Lip}}(\mathbb{R}^{d};\mathbb{R})$ and $q'>1$, there exists $C'>0$ such that for all $n\ge |x|,$
\begin{equation}\label{e:h(Tn) - h(T)}
\mathbb{E}\left[|h(X_{T}) - h(X_{T_{n}})|^{q'}\right]\le C' \exp\left\{-(n-|x|)^{2}/C'\right\}.
\end{equation}

\begin{proof}
 
By \cite[Corollary~4.1]{Siska2012}, with $(x_{t},\xi)$ there replaced by $(X_{t} - x,0)$, there exists a constant $C^{*}>0$ such that the following inequality holds for all $n\ge |x|,$
\begin{equation}\label{e:from siska}
\mathbb{P}\left\{T^{*}_{n - |x|} \le T\right\}\le 6 \exp\left\{-(n - |x|)^{2}/C^{*}\right\},
\end{equation}
where $T^{*}_{n - |x|} := \inf\{t > 0;|X_{t} - x|\ge n - |x|\}.$ Noting that $\mathbb{P}\{T_{n}<T\} = \mathbb{P}\{\sup_{t\in[0,T]} |X_{t}| > n\}$ and $\mathbb{P}\{T^{*}_{n - |x|} \le T\} = \mathbb{P}\{\sup_{t\in[0,T]} |X_{t}-x| \ge n - |x|\},$ by the triangular inequality $|X_{t} - x| + |x| \ge |X_{t}|$ we have $\mathbb{P}\{T_{n}<T\} \le \mathbb{P}\{T^{*}_{n - |x|} \le T\}$. Then, from \eqref{e:from siska} it follows that
\begin{equation*}
\mathbb{E}\big[|T-T_{n}|^{q}\big]\le T^{q}\mathbb{P}\left\{T_{n}<T\right\}\le T^{q}\mathbb{P}\left\{T^{*}_{n-|x|} \le T\right\} \le 6T^{q}\exp\left\{-(n - |x|)^2/C^{*}\right\},
\end{equation*}
and then \eqref{e:moment-T-Tn} holds by taking $C := (6T^{q})\vee C^{*}.$

In addition, denoting by $C'_{\text{Lip}}$ the 
Lipschitz constant of $h$,  we have 
\begin{equation*}
\begin{aligned}
\mathbb{E}\left[|h(X_{T}) - h(X_{T_{n}})|^{q'}\right]&\lesssim_{C'_{\text{Lip}},q'} \mathbb{E}\left[\|X\|_{p\text{-var};[T_n,T]}^{q'}\right]\\
&\lesssim_{\Theta_{2},q'} \mathbb{E}\bigg[\Big|\int_{T_n}^{T}|\sigma(r,X_r)|^2 dr\Big|^{\frac{q'}{2}}\bigg] + 2\mathbb{E}\bigg[\Big|\int_{T_n}^{T}b(r,X_{r})dr\Big|^{q'}\bigg]\\
&\lesssim_{\Theta_{2},q'} \mathbb{E}\left[|T_{n} - T|^{\frac{q'}{2}}\right] + \mathbb{E}\left[|T_n - T|^{q'}\right],
\end{aligned}
\end{equation*}
where the second inequality is due to the BDG inequality for $p$-variation. Then \eqref{e:h(Tn) - h(T)} holds by \eqref{e:moment-T-Tn}. 
\end{proof}
\end{lemma}

The following proposition shows that the limit of solutions on random intervals $[0,S_n]$ is a solution on $[0,T]$ when $S_{n}\uparrow T$ as $n\to \infty.$ 

\begin{proposition}\label{prop:limit of solutions is also a solution}
Assume \ref{(H0)} and \ref{(A0)}, and suppose that $(\beta,\varepsilon,\eta,g,f)$ satisfies the conditions in \ref{(A1)} except for $\frac{(1 + \varepsilon)(\lambda + \beta)}{\varepsilon\tau}<2$ in  (1). 
Let $\left\{S_{n}\right\}_{n=1}^{\infty}$ be a sequence of stopping times increasing to $T$ such that $X$ is bounded on  $[0,S_{n}]$ (i.e. $\esssup_{(t,\omega)\in[0,T]\times \Omega}|X_{t\land S_{n}}|<\infty$) for each $n\ge 1$. Assume that there is a sequence of processes $\{(Y^n,Z^n)\}_{n\ge 1}\subset\mathfrak{B}_{p,k'}(0,T)$  with some $k'>1$, such that $\sup_{n\ge 1}\|Y^{n}_{S_n}\|_{L^q(\Omega)}<\infty$ for some $q>\frac{1}{\varepsilon}$. Further, assume that for each $n\ge 1$ and $t\in[0,T]$, 
\begin{equation*}
\begin{cases}
\displaystyle Y^{n}_{t\land S_n} = Y^{n}_{S_n} + \int_{t\land S_{n}}^{S_{n}}f(r,X_r,Y^{n}_r,Z^{n}_r)dr + \sum_{i=1}^{M}\int_{t\land S_n}^{S_n}g_{i}(Y^{n}_{r})\eta_{i}(dr,X_{r}) - \int_{t\land S_{n}}^{S_n} Z^{n}_{r} dW_r,\\[5pt]
Y^{n}_{t} = Y^{n}_{S_{n}},\  Z^{n}_{t} = 0,\ \text{ for }\ t>S_n,
\end{cases}
\end{equation*}
and that there exists a continuous adapted  process $Y$ such that 
\begin{equation}\label{eq:conver.Y}
\lim_{n\rightarrow \infty}\mathbb{E}\left[\|Y^{n} - Y\|^{q}_{\infty;[0,S_{n}]}\right] = 0.
\end{equation} 
Then, there exists a progressively measurable process $Z$  such that $(Y,Z)\in\mathfrak{H}_{p,q}(0,T)$  and satisfies \eqref{e:ourBSDE} with $\xi = Y_{T}$.

If additionally $\lim_{n\rightarrow\infty}\mathbb{P}\left[S_{n} = T\right] = 1,$  then we have
\begin{equation}\label{e:con-pq'}
\lim_{n\rightarrow\infty}\left\|(Y^{n},Z^{n}) - (Y,Z)\right\|_{\mathfrak{H}_{p,q'};[0,T]} = 0,\text{ for every }q'\in(1,q).
\end{equation}
\end{proposition}

\begin{proof}
\textbf{Step 1. }
In this step, we will prove the existence of $Z.$ Denote 
\begin{equation}\label{e:finite H}
H := \sup\limits_{m\ge 1}\|(Y^{m},Z^{m})\|_{\mathfrak{H}_{p,q};[0,T]},
\end{equation}
which is finite by the estimate obtained in  \cite[Proposition~4.2]{BSDEYoung-I}, 
with $f(\omega,t,y,z)$ there replaced by $f(t,X_{t}(\omega),y,z)$. Denote $\delta Y^{n,m}_r := Y^{n}_r - Y^{m}_r$, $\delta Z^{n,m}_r := Z^{n}_r - Z^{m}_r$, $\delta f^{n,m}_r := f(r,X_r,Y^n_r,Z^n_r) - f(r,X_r,Y^m_r,Z^m_r)$, and $\delta g^{n,m}_r := g(Y^{n}_r) - g(Y^{m}_r)$. Split $[0,T]$ into $[t_{i},t_{i+1}],$ $i=1,2,...,v,$ evenly and denote $\theta := \frac{T}{v}.$ Let $\tau^{n}_{i} := t_{i}\land S_{n}.$ By the same calculation leading to 
\cite[(4.9)]{BSDEYoung-I}, for each $\kappa>1$ and $i=1,2,...,v,$
\begin{equation}\label{e:Zn+1<=Yn+1'}
\begin{aligned}
&\mathbb{E}\left[\left\|\delta Y^{n,m}\right\|^{\kappa}_{p\text{-var};[\tau^{n}_{i},\tau^{n}_{i+1}]}\right] + \mathbb{E}\bigg[\Big|\int_{\tau^{n}_{i}}^{\tau^{n}_{i+1}}|\delta Z^{n,m}_{r}|^{2}dr\Big|^{\frac{\kappa}{2}}\bigg]\\
&\lesssim_{\Theta_{2},\kappa}  \mathbb{E}\bigg[\Big\|\int_{\cdot}^{\tau^{n}_{i+1}}\delta g^{n,m}_{r}\eta(dr,X_r)\Big\|_{p\text{-var};[\tau^{n}_{i},\tau^{n}_{i+1}]}^{\kappa}\bigg]  + \mathbb{E}\bigg[\Big|\int_{\tau^{n}_{i}}^{\tau^{n}_{i+1}}|\delta f^{n,m}_{r}|dr\Big|^{\kappa}\bigg] + \mathbb{E}\left[|\delta Y^{n,m}_{\tau^{n}_{i+1}}|^{\kappa}\right]\\
&\lesssim_{\Theta_{2},\kappa} \mathbb{E}\bigg[\Big|\left(1 + \|X\|^{\lambda+\beta}_{\infty;[0,T]}+\|X\|^{\lambda+\beta}_{p\text{-var};[0,T]}\right)\left(1 + \|Y^{n}\|_{p\text{-var};[0,S_{n}]} + \|Y^{m}\|_{p\text{-var};[0,S_{n}]}\right)^{1-\varepsilon}\\
&\qquad\qquad  \times \left(\|\delta Y^{n,m}\|_{p\text{-var};[\tau^{n}_{i},\tau^{n}_{i+1}]} + \|\delta Y^{n,m}\|_{\infty;[0,S_{n}]} \right)^{1-\varepsilon} \|\delta Y^{n,m}\|_{\infty;[0,S_{n}]}^{\varepsilon}\Big|^{\kappa}\bigg] \\
&\qquad  + \theta^{\frac{\kappa}{2}}\mathbb{E}\bigg[\Big|\int_{\tau^{n}_{i}}^{\tau^{n}_{i+1}}|\delta Z^{n,m}_{r}|^{2}dr\Big|^{\frac{\kappa}{2}}\bigg] + \mathbb{E}\left[\|\delta Y^{n,m}\|^{\kappa}_{\infty;[0,S_{n}]}\right],
\end{aligned}
\end{equation}
where the second inequality follows from \cite[Lemma~2.1, Lemma~2.2]{BSDEYoung-I}, 
and the fact $|\delta f^{n,m}_r|\lesssim |\delta Z^{n,m}_{r}| + |\delta Y^{n,m}_{r}|.$ 
Since $q>\frac{1}{\varepsilon}$, we can always choose an $\kappa>1$ sufficiently close to $1$ such that $q>\frac{\kappa}{\varepsilon}.$ Thus, for such $\kappa$ we claim that the right-hand side of \eqref{e:Zn+1<=Yn+1'} is bounded by (up to a multiplicative constant)
\begin{equation}\label{e:bound ABCD}
\begin{aligned}
&\left\{\mathbb{E}\Big[\left(\|\delta Y^{n,m}\|_{p\text{-var};[\tau^{n}_{i},\tau^{n}_{i+1}]} + \|\delta Y^{n,m}\|_{\infty;[0,S_{n}]}\right)^{\kappa}\Big]\right\}^{1-\varepsilon}\bigg\{\mathbb{E}\left[\|\delta Y^{n,m}\|_{\infty;[0,S_{n}]}^{q}\right]\bigg\}^{\frac{\kappa\varepsilon}{q}} \\
&\quad  + \theta^{\frac{\kappa}{2}}\mathbb{E}\bigg[\Big|\int_{\tau^{n}_{i}}^{\tau^{n}_{i+1}}|\delta Z^{n,m}_{r}|^{2}dr\Big|^{\frac{\kappa}{2}}\bigg] + \left\|\|\delta Y^{n,m}\|_{\infty;[0,S_{n}]}\right\|^{\kappa}_{L^{\kappa}(\Omega)}.
\end{aligned}
\end{equation}
The above claim can be verified by H\"older's inequality as follows (noting $\frac{\varepsilon \kappa}{1-\varepsilon + \varepsilon^{2}} \cdot \frac{1-\varepsilon + \varepsilon^{2}}{\varepsilon}= \kappa <q$), 
\begin{equation*}
\begin{aligned}
&\mathbb{E}\bigg[\Big|(AB^{1-\varepsilon})(C + D)^{1-\varepsilon}\|\delta Y^{n,m}\|_{\infty;[0,S_{n}]}^{\varepsilon}\Big|^{\kappa}\bigg]\\
&\le \bigg\{\mathbb{E}\bigg[\left|(AB^{1-\varepsilon})^{\frac{1}{\varepsilon-\varepsilon^2}}\right|^{\kappa}\bigg]\bigg\}^{\varepsilon - \varepsilon^2} \bigg\{\mathbb{E}\Big[(C + D)^{\frac{(1-\varepsilon)\kappa}{1-\varepsilon + \varepsilon^2}}\|\delta Y^{n,m}\|_{\infty;[0,S_{n}]}^{\frac{\varepsilon \kappa}{1-\varepsilon + \varepsilon^2}}\Big]\bigg\}^{1-\varepsilon + \varepsilon^2}\\
&\lesssim_{\Theta_{2},x,H,\kappa,q} \bigg\{\mathbb{E}\Big[(C + D)^{\kappa}\Big]\bigg\}^{1-\varepsilon}\bigg\{\mathbb{E}\Big[\|\delta Y^{n,m}\|_{\infty;[0,S_{n}]}^{q}\Big]\bigg\}^{\frac{\kappa\varepsilon}{q}},
\end{aligned}
\end{equation*}
where $A := (1 + \|X\|^{\lambda + \beta}_{\infty;[0,T]} + \|X\|^{\lambda+\beta}_{p\text{-var};[0,T]}),$ $B := (1 + \|Y^{n}\|_{p\text{-var};[0,S_{n}]} + \|Y^{m}\|_{p\text{-var};[0,S_{n}]}),$ $C := \|\delta Y^{n,m}\|_{p\text{-var};[\tau^{n}_{i},\tau^{n}_{i+1}]},$ $D := \|\delta Y^{n,m}\|_{\infty;[0,S_{n}]},$ 
and the second inequality follows from $\E[|AB^{1-\varepsilon}|^{\frac{\kappa}{\varepsilon - \varepsilon^2}}]\le \{\E[|A|^{\frac{q\kappa}{(q\varepsilon-\kappa)(1-\varepsilon)}}]\}^{\frac{q\varepsilon - \kappa}{q\varepsilon}} \{\E[|B|^{q}]\}^{\frac{\kappa}{q\varepsilon}}$ (noting $\frac{\kappa}{\varepsilon}<q$) and Assumption~\ref{(A0)}. For the first term of \eqref{e:bound ABCD}, by Young's inequality, we have 
\begin{equation}\label{est:for5.8}
\begin{aligned}
&\left\{\mathbb{E}\left[\left(\|\delta Y^{n,m}\|_{p\text{-var};[\tau^{n}_{i},\tau^{n}_{i+1}]} + \|\delta Y^{n,m}\|_{\infty;[0,S_{n}]}\right)^{\kappa}\right]\right\}^{1-\varepsilon}\left\{\mathbb{E}\left[\|\delta Y^{n,m}\|_{\infty;[0,S_{n}]}^{q}\right]\right\}^{\frac{\kappa\varepsilon}{q}}\\
&\lesssim \theta^{\frac{\kappa}{1-\varepsilon}} \mathbb{E}\left[\left(\|\delta Y^{n,m}\|_{p\text{-var};[\tau^{n}_{i},\tau^{n}_{i+1}]} + \|\delta Y^{n,m}\|_{\infty;[0,S_{n}]}\right)^{\kappa}\right] + \theta^{-\frac{\kappa}{\varepsilon}} \left\|\|\delta Y^{n,m}\|_{\infty;[0,S_{n}]}\right\|^{\kappa}_{L^{q}(\Omega)}.
\end{aligned}
\end{equation}
Then, combining \eqref{e:Zn+1<=Yn+1'}, \eqref{e:bound ABCD} and \eqref{est:for5.8}, we obtain 
\begin{equation*}
\begin{aligned}
&\left\|\left\|\delta Y^{n,m}\right\|_{p\text{-var};[\tau^{n}_{i},\tau^{n}_{i+1}]}\right\|_{L^{\kappa}(\Omega)} + \Big\|\big|\int_{\tau^{n}_{i}}^{\tau^{n}_{i+1}}|\delta Z^{n,m}_{r}|^{2}dr\big|^{\frac{1}{2}}\Big\|_{L^{\kappa}(\Omega)}\\ 
&\lesssim_{\Theta_{2},x,H,\kappa,q} \theta^{\frac{1}{1-\varepsilon}}\left\|\left\|\delta Y^{n,m}\right\|_{p\text{-var};[\tau^{n}_{i},\tau^{n}_{i+1}]}\right\|_{L^{\kappa}(\Omega)} + \theta^{-\frac{1}{\varepsilon}}\left\|\|\delta Y^{n,m}\|_{\infty;[0,S_{n}]}\right\|_{L^{q}(\Omega)}\\
&\quad + \theta^{\frac{1}{2}}\Big\|\big|\int_{\tau^{n}_{i}}^{\tau^{n}_{i+1}}|\delta Z^{n,m}_{r}|^{2}dr\big|^{\frac{1}{2}}\Big\|_{L^{\kappa}(\Omega)} + (1 + \theta^{\frac{1}{1-\varepsilon}})\left\|\|\delta Y^{n,m}\|_{\infty;[0,S_n]}\right\|_{L^{\kappa}(\Omega)}.
\end{aligned}
\end{equation*}
Choosing $\theta>0$ sufficiently small (the choice of $\theta$ does not depend on $n$ and $m$), we have
\[\left\|\|\delta Y^{n,m}\|_{p\text{-var};[\tau^{n}_{i},\tau^{n}_{i+1}]}\right\|_{L^{\kappa}(\Omega)} + \Big\|\big|\int_{\tau^{n}_{i}}^{\tau^{n}_{i+1}}|\delta Z^{n,m}_{r}|^{2}dr\big|^{\frac{1}{2}}\Big\|_{L^{\kappa}(\Omega)} \lesssim_{\Theta_{2},x,H,\kappa,q} \left\|\|\delta Y^{n,m}\|_{\infty;[0,S_n]}\right\|_{L^{q}(\Omega)}.\]
Noting that the above inequality holds for all $i=1,2,...,\frac{T}{\theta}$, we get 
\begin{equation}\label{e:Z <= Y}
\left\|\|\delta Y^{n,m}\|_{p\text{-var};[0,S_n]}\right\|_{L^{\kappa}(\Omega)} + \Big\|\big|\int_{0}^{S_{n}}|\delta Z^{n,m}_{r}|^{2}dr\big|^{\frac{1}{2}}\Big\|_{L^{\kappa}(\Omega)} \lesssim_{\Theta_{2},x,H,\kappa,q} \left\|\|\delta Y^{n,m}\|_{\infty;[0,S_n]}\right\|_{L^{q}(\Omega)}.
\end{equation}

On the other hand, by \eqref{eq:conver.Y} and Fatou's Lemma, we see that $\left\|\|Y\|_{\infty;[0,T]}\right\|_{L^{q}(\Omega)}$ is finite. Therefore, 
as $n\rightarrow \infty$, 
\begin{equation*}
\left\|\|Y^{n} \mathbf{1}_{[0,S_{n}]} - Y\|_{\infty;[0,T]}\right\|_{L^{q}(\Omega)} \le \left\|\|Y^{n} - Y\|_{\infty;[0,S_n]}\right\|_{L^{q}(\Omega)} + \left\|\|Y\|_{\infty;[S_n,T]}\right\|_{L^{q}(\Omega)}\longrightarrow 0,
\end{equation*}
which implies that $\{ Y^{n} \mathbf{1}_{[0,S_{n}]} \}_n$ is a Cauchy sequence with respect to the norm $\left\|\|\cdot\|_{\infty;[0,T]}\right\|_{L^{q}(\Omega)}$. In addition, by \eqref{e:Z <= Y}, we have that for $n<m'<m$,
\begin{equation*}
\begin{aligned}
& \left\|\|\delta Y^{m',m}\|_{p\text{-var};[0,S_n]}\right\|_{L^{\kappa}(\Omega)}  + \Big\|\big|\int_{0}^{S_{n}}|\delta Z^{m',m}_{r}|^{2}dr\big|^{\frac{1}{2}}\Big\|_{L^{\kappa}(\Omega)} \\
&\lesssim_{\Theta_{2},x,H,\kappa,q} \left\|\|\delta Y^{m',m}\|_{\infty;[0,S_{m'}]}\right\|_{L^{q}(\Omega)}\\
&\le \left\|\| Y^{m'} \mathbf{1}_{[0,S_{m'}]} - Y^{m} \mathbf{1}_{[0,S_m]}\|_{\infty;[0,T]}\right\|_{L^{q}(\Omega)}.
\end{aligned}
\end{equation*}
Hence,  $\{Z^{m}_{t}\mathbf{1}_{[0,S_n]}(t)\}_{m\ge 1}$ is a Cauchy sequence for each $n$, and we denote by $Z^{(n)}$ the limit. Now let 
$$
Z_{t} := Z^{(1)}_t \mathbf{1}_{[0,S_{1}]}(t) + \sum_{n\ge 2}Z^{(n)}_{t} \mathbf{1}_{(S_{n-1},S_n]}(t).
$$
Then, for each $n\ge 1$,
\begin{equation}\label{e:m,n converge}
\lim_{m\rightarrow\infty}\Big\{\left\|\left\|Y - Y^{m}\right\|_{p\text{-var};[0,S_{n}]}\right\|_{L^{\kappa}(\Omega)} + \Big \|\big|\int_{0}^{S_{n}} |Z_{r} - Z^{m}_{r}|^{2}dr\big|^{\frac{1}{2}}\Big\|_{L^{\kappa}(\Omega)}\Big\}  = 0.
\end{equation}
Furthermore, by Fatou's Lemma and the finiteness of $H,$ we have $(Y,Z)\in \mathfrak{H}_{p,q}(0,T)$. \\

\textbf{Step 2.}
We will show that $(Y,Z)$ is a solution. Note that for any $m>n$,
\begin{equation}\label{eq:mn}
Y^{m}_{t\land S_{n}} = Y^{m}_{S_{n}} + \int_{t\land S_{n}}^{S_{n}}f(r,X_r,Y^{m}_r,Z^{m}_r)dr + \int_{t\land S_{n}}^{S_{n}}g(Y^{m}_{r})\eta(dr,X_{r}) - \int_{t\land S_{n}}^{S_{n}} Z^{m}_{r} dW_r,\ t\in[0,T]. 
\end{equation}
By \eqref{e:m,n converge}, \eqref{eq:conver.Y}, and \cite[Lemma~2.2]{BSDEYoung-I}, we have $\lim_{m}\|g(Y^{m}) - g(Y)\|_{p\text{-var};[0,S_{n}]}=0$ and $\lim_{m}\|g(Y^m)-g(Y)\|_{\infty;[0,S_n]} = 0$ in probability.
Hence, by the estimate for Young integral in \cite[Proposition~2.1]{BSDEYoung-I}, letting $m$ in \eqref{eq:mn} go to infinity, we obtain
\[Y_{t\land S_{n}} = Y_{S_{n}} + \int_{t\land S_{n}}^{S_{n}}f(r,X_r,Y_r,Z_r)dr + \int_{t\land S_{n}}^{S_{n}}g(Y_{r})\eta(dr,X_{r}) - \int_{t\land S_{n}}^{S_{n}} Z_{r} dW_r,\ t\in[0,T].\]
Furthermore, letting $n$ go to infinity we have
\[Y_{t} = Y_{T} + \int_{t}^{T}f(r,X_r,Y_r,Z_r)dr + \int_{t}^{T}g(Y_{r})\eta(dr,X_{r}) - \int_{t}^{T} Z_{r} dW_r,\ t\in[0,T].\]

\textbf{Step 3.}
Finally, we will prove \eqref{e:con-pq'} assuming $\lim_{n\to \infty}\mathbb{P}\left[S_{n} = T\right] = 1$. Note that $\big\{\|Y^{m}\|^{q'}_{p\text{-var};[0,T]}\big\}_{m\ge 1}$ is uniformly integrable, since $\sup_{m\ge 1}\|(Y^{m},Z^{m})\|_{\mathfrak{H}_{p,q};[0,T]}<\infty$. Then, by \eqref{e:m,n converge} we have
\[\lim_{m\rightarrow\infty}\mathbb{E}\left[\|Y^{m}-Y\|^{q'}_{p\text{-var};[0,S_{n}]}\right] = 0\text{, for each }n\ge 1.\]
Therefore,  we have
\begin{equation*}
\begin{aligned}
&\lim_{m\rightarrow\infty}\mathbb{E}\left[\left\|Y^{m} - Y\right\|^{q'}_{p\text{-var};[0,T]}\right]\\
&\le \lim_{n\rightarrow\infty}\lim_{m\rightarrow\infty}\mathbb{E}\left[\left\|Y^{m} - Y\right\|^{q'}_{p\text{-var};[0,S_{n}]}\right] + \lim_{n\rightarrow\infty}\sup_{m\ge 1}\E\left[\|Y^{m} - Y\|^{q'}_{p\text{-var};[0,T]}\mathbf{1}_{[S_{n}<T]}\right]=0,
\end{aligned}
\end{equation*}
where the second term on the right hand side converges  to 0 due to the uniform integrability of $\|Y^{m} - Y\|^{q'}_{p\text{-var};[0,T]}$ and the assumption $\lim
_{n\to\infty}\mathbb{P}\left[S_{n} < T\right] = 0$. Similarly, we can show $\lim_{m\to \infty}\E[|\int_{0}^{T}|Z^{m}_{r} - Z_{r}|^{2}dr|^{\frac{q'}{2}}] = 0$, and this completes the proof. 
\end{proof}

\begin{remark}\label{rem:why g is bounded}
In Proposition~\ref{prop:limit of solutions is also a solution}, the boundedness of $g(\cdot)$ is required, in order to ensure the uniform boundedness of $\|(Y^m,Z^m)\|_{\mathfrak{H}_{p,q};[0,T]}$ (see \eqref{e:finite H}). Indeed, to estimate the $L^q$-type norm as in \cite[Proposition~4.2]{BSDEYoung-I}, an interpolation argument is applied, which requires $g(\cdot)$ to be bounded. \end{remark}

\begin{lemma}\label{lem:estimate for Z}
Let $f:[0,T]\times\Omega\rightarrow \R^{N}$ be a progressively measurable process, and let $X:[0,T]\times\Omega\rightarrow \R^{d}$ and $g_{i}:[0,T]\times\Omega\rightarrow \R^{N}$, $i=1,2,...,M$, be continuous adapted  processes. Assume that \ref{(H0)} holds and $\eta\in C^{\tau,\lambda}_{\mathrm{loc}}([0,T]\times\R^d;\R^{M}).$ Suppose that $ \mathbb{P}(d \omega)$-a.s. $\|f\|_{L^{1}(0,T)}(\omega)<\infty,$ $\|X\|_{p\text{-}\mathrm{var};[0,T]}(\omega)<\infty,$ and $\|g_{i}\|_{p\text{-}\mathrm{var};[0,T]}(\omega)<\infty$ for $i=1,2,...,M$. Let $S\in[0,T]$ be a stopping time, and $\xi\in\mathcal{F}_{S}$ with $\|\xi\|_{L^{2}(\Omega)}<\infty$. Assume $(Y,Z)\in\mathfrak{H}_{p,2}(0,T)$ satisfies
\begin{equation*}
Y_{t\land S} = \xi + \int_{t\land S}^{S} f_{r} dr + \sum_{i=1}^{M}\int_{t\land S}^{S} g_{i}(r) \eta_{i}(dr,X_{r}) - \int_{t\land S}^{S} Z_{r} dW_{r},\ t\in[0,T].
\end{equation*} 
Then, for any $q>1$,  
\begin{equation*}
\begin{aligned}
\mathbb{E}_{t}\left[\Big\|\int_{\cdot}^{S}Z_r dW_{r}\Big\|^{q}_{p\text{-}\mathrm{var};[t\land S,S]}\right]&\lesssim_{p,q}  \mathbb{E}_{t}\left[\Big\|\int_{\cdot}^{S} f_r dr + \sum_{i=1}^{M}\int_{\cdot}^{S}g_{i}(r)\eta_{i}(dr,X_{r})\Big\|^{q}_{\infty;[t\land S,S]}\right]\\
&\quad + \mathbb{E}_{t}\left[\|\xi - \mathbb{E}_{\cdot}[\xi]\|^{q}_{\infty;[t\land S,S]}  \right].
\end{aligned}
\end{equation*}
\end{lemma}

\begin{proof} 
By \cite[Lemma~A.1]{BSDEYoung-I} with $M_{t}:=\int_{0}^{t} Z_{r} \mathbf{1}_{[0,S]}(r)dW_{r},$ 
\begin{equation}\label{e:|Z|_p-var <= |Z|_infty}
\mathbb{E}_{t}\left[\Big\|\int_{\cdot}^{S}Z_{r} dW_{r}\Big\|^{q}_{p\text{-var};[t\land S,S]}\right]\lesssim_{p,q} \mathbb{E}_{t}\left[\Big\|\int_{\cdot}^{S}Z_{r} dW_{r}\Big\|^{q}_{\infty;[t\land S,S]}\right].
\end{equation}
Note that for all $t\in[0,T],$ the following equations hold a.s.,
\begin{equation}\label{e:equality Z}
\int_{t\land S}^{S}Z_{r}dW_{r} = \xi - \mathbb{E}_{t\land S}[\xi] - (Y_{t\land S} - \mathbb{E}_{t\land S}[\xi]) + \int_{t\land S}^{S}f_{r}dr + \sum_{i=1}^{M}\int_{t\land S}^{S}g_{i}(r) \eta_{i}(dr,X_{r}),
\end{equation}
and
\begin{equation}\label{e:represent of Y}
Y_{t\land S} - \mathbb{E}_{t\land S}\left[\xi\right] = \mathbb{E}_{t\land S}\left[\int_{t\land S}^{S}f_{r}dr + \sum_{i=1}^{M}\int_{t\land S}^{S}g_{i}(r) \eta(dr,X_{r})\right].
\end{equation}
By \eqref{e:represent of Y} and Doob's maximal inequality, we get for $t\in[0,T]$,
\begin{equation*}
\begin{aligned}
\quad\mathbb{E}_{t}\left[\left\|Y_{\cdot} - \mathbb{E}_{\cdot}\left[\xi\right]\right\|^{q}_{\infty;[t\land S,S]}\right]
&\le \mathbb{E}_{t}\left[\sup_{s\in[t,T]}\left|\mathbb{E}_{s}\bigg[\Big\|\int_{\cdot}^{S}f_{r}dr + \sum_{i=1}^{M}\int_{\cdot}^{S}g_{i}(r)\eta_{i}(dr,X_{r}) \Big\|_{\infty;[t\land S,S]}\bigg]\right|^{q}\right]\\
&\lesssim_{q} \mathbb{E}_{t}\left[\Big\| \int_{\cdot}^{S}f_{r}dr + \sum_{i=1}^{M}\int_{\cdot}^{S}g_{i}(r)\eta_{i}(dr,X_{r}) \Big\|^{q}_{\infty;[t\land S,S]}\right].
\end{aligned}
\end{equation*}
This together with \eqref{e:|Z|_p-var <= |Z|_infty} and \eqref{e:equality Z} gives the desired estimate. 
\end{proof}

To apply the localization argument, we need a uniform estimate for the localized solutions, which is provided by the following lemma. For each $n\ge 1$, consider the following equation:
\begin{equation}\label{e:Yn}
\begin{cases}
dY^{n}_{t} = - f(t,X_{t},Y^{n}_{t},Z^{n}_{t})dt -  \sum_{i=1}^{M}g_{i}(Y^{n}_{t})\eta_{i}(dt,X_{t}) + Z^{n}_{t} dW_{t},\quad t\in[0,T_n],\\
Y^{n}_{T_{n}} = \Xi_{T_{n}},
\end{cases}
\end{equation}
where $T_n$ is defined in \eqref{e:Tn} and $\Xi_{\cdot}$ satisfies \ref{(T)}.

\begin{lemma}\label{lem:growth of (Y^n,Z^n)}
Assume \ref{(A0)}, \ref{(T)}, and \ref{(A1)}. Let $(Y^n,Z^n)\in \bigcap_{q>1}\mathfrak{B}_{p,q}(0,T)$ be the unique solution of Eq.~\eqref{e:Yn} (see \cite[Section~4]{BSDEYoung-I} for its well-posedness). Then for every $q>1,$  
\begin{equation}\label{e:Gronwall to m_p,q (Y)}
m_{p,q}(Y^{n};[0,T]) + \|Z^{n}\|_{q\text{-}\mathrm{BMO};[0,T]} \lesssim_{\Theta_{2},q} n^{\frac{\lambda+\beta}{\varepsilon}\vee 1}\ \text{ for all }\ n\ge |x|.
\end{equation}
\end{lemma}

\begin{proof}
For $r\in[0,T]$, denote $f^{n}_{r} := f(r,X_r,Y^{n}_r,Z^{n}_{r})$. By \eqref{e:assump of f} in \ref{(A1)}, we have
\begin{equation}\label{e:frn}
\begin{aligned}
&\bigg\{\mathbb{E}_{t}\Big[\Big|\int_{t\land T_{n}}^{T_{n}}|f^{n}_{r}|dr\Big|^{q}\Big]\bigg\}^{\frac{1}{q}}\\
&\le C_{1}|T-t|\left( 1 + n^{\frac{\lambda+\beta}{\varepsilon}\vee 1}\right) + C_{\text{Lip}}|T-t|\left\{\mathbb{E}_{t}\left[\|Y^{n}\|^{q}_{\infty;[t,T]}\right]\right\}^{\frac{1}{q}} + C_{\text{Lip}}\Big\{\mathbb{E}_{t}\Big[\Big|\int_{t}^{T}|Z^{n}_{r}|dr\Big|^{q}\Big]\Big\}^{\frac{1}{q}}\\
&\lesssim_{\Theta_{2}} n^{\frac{\lambda+\beta}{\varepsilon}\vee 1} + \|Y^{n}_{T_n}\|_{L^{\infty}(\Omega)} + |T-t| m_{p,q}(Y^{n};[t,T]) + |T-t|^{\frac{1}{2}}\left\|Z^{n}\right\|_{q\text{-BMO};[t,T]},
\end{aligned}
\end{equation}
where in the second step we use the fact 
$\|Y^{n}\|_{\infty;[r,T]}\le \|Y^{n}\|_{p\text{-var};[r,T]} + |Y_{T_n}^n|.$ For every $t\in[0,T]$ and $n\ge |x|$, by the estimate for nonlinear Young integral established in \cite[Lemma~2.1]{BSDEYoung-I}, we have 
\begin{equation}\label{e:gyrn}
\begin{aligned}
&\bigg\{\mathbb{E}_{t}\left[\Big\|\int_{\cdot}^{T_{n}}g(Y^{n}_{r})\eta(dr,X_{r})\Big\|^{q}_{p\text{-var};[t\land T_{n+1},T_{n+1}]}\right]\bigg\}^{\frac{1}{q}}\\
&\lesssim_{\Theta_2,q} T^{\tau}\|\eta\|_{\tau,\lambda;\beta}\Big\{\mathbb{E}_{t}\Big[ \left(1 + \|X\|^{q\beta}_{\infty;[0,T_{n}]}\right) \|g(Y^{n}_{\cdot})\|^{q}_{\infty;[0,T_{n}]}\|X\|^{q\lambda}_{p\text{-var};[t\land T_{n},T_{n}]}\\ 
&\quad+ \left(1 + \|X\|^{q(\lambda + \beta)}_{\infty;[0,T_{n}]}\right)\Big(\|g(Y^{n}_{\cdot})\|^{q}_{\infty;[0,T_{n}]} + \|g(Y^{n}_{\cdot})\|^{q\varepsilon}_{\infty;[0,T_{n}]}\|g(Y^{n}_{\cdot})\|^{q(1-\varepsilon)}_{p\text{-var};[t\land T_{n},T_{n}]}\Big) \Big]\Big\}^{\frac{1}{q}}\\
&\lesssim_{\Theta_2,q} \left(1+n^{\beta}\right)+\left(1+n^{\lambda + \beta}\right)\left(1 + m_{p,q}(Y^{n};[t,T])^{1-\varepsilon}\right)\\
&\lesssim n^{\lambda + \beta}\left(1 + m_{p,q}(Y^{n};[t,T])^{1-\varepsilon}\right),
\end{aligned}
\end{equation}
where in the second step, we use the fact  $\|X\|_{\infty;[0,T_n]}\le n,$ the finiteness of $m_{p,q}(X;[0,T])$ obtained in Lemma~\ref{lem:m_p,q of X}, and the finiteness and Lipschitzness of $g(\cdot).$
By Lemma~\ref{lem:estimate for Z} and the following equation
\begin{equation*}
Y^{n}_{t} = \Xi_{T_{n}} + \int_{0}^{t\land T_{n}}f^{n}_{r}dr + \int_{0}^{t\land T_{n}} g(Y^{n}_{r}) \eta(dr,X_{r}) - \int_{0}^{t\land T_{n}} Z^{n}_{r} dW_{r},
\end{equation*}
we get
\begin{equation*}
\begin{aligned}
&\left\{\mathbb{E}_{t}\left[\|Y^{n}\|^{q}_{p\text{-var};[t\land T_{n},T_{n}]}\right]\right\}^{\frac{1}{q}} + \bigg\{\mathbb{E}_{t}\left[\Big\|\int_{\cdot}^{T_{n}}Z^{n}_{r}dW_{r}\Big\|^{q}_{p\text{-var};[t\land T_{n},T_{n}]}\right]\bigg\}^{\frac{1}{q}}\\	
&\lesssim_{\Theta_{2},q} \|Y^{n}_{T_{n}}\|_{L^{\infty}(\Omega)} + \bigg\{\mathbb{E}_{t}\left[\Big\|\int_{\cdot}^{T_{n}}g(Y^{n}_{r})\eta(dr,X_{r})\Big\|^{q}_{p\text{-var};[t\land T_{n},T_{n}]}\right]\bigg\}^{\frac{1}{q}} + \bigg\{\mathbb{E}_{t}\left[\Big|\int_{t\land T_{n}}^{T_{n}}|f^{n}_{r}|dr\Big|^{q}\right]\bigg\}^{\frac{1}{q}}\\
& \lesssim_{\Theta_{2},q} \|\Xi_{T_{n}}\|_{L^{\infty}(\Omega)} + n^{\lambda + \beta}m_{p,q}(Y^{n};[t,T])^{1-\varepsilon} + n^{\frac{\lambda+\beta}{\varepsilon}\vee 1} + |T-t| m_{p,q}(Y^{n};[t,T])\\
&\quad + |T-t|^{\frac{1}{2}}\left\|Z^{n}\right\|_{q\text{-BMO};[t,T]},
\end{aligned}
\end{equation*} 
where the second inequality follows from \eqref{e:frn}
and \eqref{e:gyrn}. Then, there exists a constant $C>0$ such that
\begin{equation}\label{e:to use Young ineq}
\begin{aligned}
&m_{p,q}(Y^{n};[t,T]) + \left\|Z^{n}\right\|_{q\text{-BMO};[t,T]}\\
&\le C\Big(\|\Xi_{T_{n}}\|_{L^{\infty}(\Omega)} + n^{\lambda + \beta}m_{p,q}(Y^{n};[t,T])^{1-\varepsilon} + n^{\frac{\lambda+\beta}{\varepsilon}\vee 1} + |T-t| m_{p,q}(Y^{n};[t,T])\\
&\qquad + |T-t|^{\frac{1}{2}}\left\|Z^{n}\right\|_{q\text{-BMO};[t,T]}\Big).
\end{aligned}
\end{equation}
Note that by Young's inequality,
\begin{equation*}
n^{\lambda + \beta}m_{p,q}(Y^{n};[t,T])^{1-\varepsilon}\le (3C)^{\frac{1 - \varepsilon}{\varepsilon}}n^{\frac{\lambda+\beta}{\varepsilon}} + \frac{1}{3C} m_{p,q}(Y^n;[t,T]).
\end{equation*}
Thus, we can choose $\theta:= T-t$ sufficiently small so that $C(\theta^{\frac{1}{2}}\vee \theta) \le \frac{1}{3}, and $  
then \eqref{e:to use Young ineq} yields
\begin{equation}\label{e:iteration 0}
m_{p,q}(Y^{n};[t,T]) +
\|Z^{n}\|_{q\text{-BMO};[t,T]}\\	
\lesssim_{\Theta_{2},q} \|\Xi_{T_{n}}\|_{L^{\infty}(\Omega)} + n^{\frac{\lambda+\beta}{\varepsilon}\vee 1}.
\end{equation}
Similarly, setting $T^{n}_{i} := \left(T - (i-1)\theta\right)\vee 0,$ we have for $i = 1,2,...,\lfloor \frac{T}{\theta}\rfloor + 1,$
\begin{equation}\label{e:iteration 1}
m_{p,q}(Y^{n};[T^{n}_{i+1},T^{n}_{i}]) + \|Z^{n}\|_{q\text{-BMO};[T^{n}_{i+1},T^{n}_{i}]}	
\lesssim_{\Theta_{2},q}\|Y^{n}_{T_{n}\land T^{n}_{i}}\|_{L^{\infty}(\Omega)} + n^{\frac{\lambda+\beta}{\varepsilon}\vee 1}.
\end{equation}
From \eqref{eq:ess-mpk} it follows that 
\begin{equation}\label{e:iteration 2}
\|Y^{n}_{T_{n}\land T^{n}_{i}}\|_{L^{\infty}(\Omega)}\le \sum_{2\le j\le i}m_{p,q}(Y^{n};[T^{n}_{j},T^{n}_{j-1}]) + \|\Xi_{T_{n}}\|_{L^{\infty}(\Omega)}.
\end{equation}
By iterating \eqref{e:iteration 1} and \eqref{e:iteration 2}, and by the estimate~\eqref{e:iteration 0}, we have
\begin{equation*}
\begin{aligned}
m_{p,q}(Y^{n};[T^{n}_{i+1},T^{n}_{i}]) + \|Z^{n}\|_{q\text{-BMO};[T^{n}_{i+1},T^{n}_{i}]} &\lesssim_{\Theta_{2},q} m_{p,q}(Y^{n};[T^{n}_{2},T]) + \|\Xi_{T_{n}}\|_{L^{\infty}(\Omega)} + n^{\frac{\lambda+\beta}{\varepsilon}\vee 1}\\
&\lesssim_{\Theta_{2},q} \|\Xi_{T_{n}}\|_{L^{\infty}(\Omega)} + n^{\frac{\lambda+\beta}{\varepsilon}\vee 1}.
\end{aligned}
\end{equation*}
In addition, by the bound $\|\Xi_{T_{n}}\|_{L^{\infty}(\Omega)}\le C_1 n^{\frac{\lambda+\beta}{\varepsilon}\vee 1}$ which follows from Assumption~\ref{(T)}, we have
\begin{equation*}
m_{p,q}(Y^{n};[T^{n}_{i+1},T^{n}_{i}]) + \|Z^{n}\|_{q\text{-BMO};[T^{n}_{i+1},T^{n}_{i}]} \lesssim_{\Theta_2,q} n^{\frac{\lambda+\beta}{\varepsilon}\vee 1}.
\end{equation*}
Summing up the above inequalities for $i=1,2,...,\lfloor\frac{T}{\theta}\rfloor+1$, we obtain the desired \eqref{e:Gronwall to m_p,q (Y)}.
\end{proof}

The existence  of the solution to Eq.~\eqref{e:ourBSDE} is obtained in  Theorem~\ref{thm:existence of solution of unbounded BSDEs} below.

\begin{theorem}\label{thm:existence of solution of unbounded BSDEs}
Assume \ref{(A0)}, \ref{(T)}, and \ref{(A1)}. Then, Eq.~\eqref{e:ourBSDE} admits a solution $(Y,Z)\in\bigcap_{q>1}\mathfrak{H}_{p,q}(0,T)$. In addition, for $n\ge 1$, let $T_n$ be the stopping time defined in \eqref{e:Tn} and $(Y^n,Z^n)\in\bigcap_{q>1}\mathfrak{B}_{p,q}(0,T)$ be  the unique solution of Eq.~\eqref{e:Yn} (see \cite[Section~4]{BSDEYoung-I} for its well-posedness). Then we have for each $ q>1 $,
\begin{equation}\label{e:con-0}
\lim_{n\to \infty}	\left\|(Y^{n},Z^{n}) - (Y,Z)\right\|_{\mathfrak{H}_{p,q};[0,T]}=0. 
\end{equation}
\end{theorem}

\begin{proof}[Proof of Theorem~\ref{thm:existence of solution of unbounded BSDEs}]\makeatletter\def\@currentlabelname{the simplified version of the proof}\makeatother\label{proof:3}
We provide a detailed proof for the case $N=1$ and $f \equiv 0$ here. The reader is referred to Appendix~\ref{append:C} for a complete proof which basically follows the strategy presented here   while is  more involved.

Note that \cite[Assumption (H2)]{BSDEYoung-I}  with $S_0=T_n$ holds for each $n\in\mathbb N$. By the well-posedness of BSDEs with stopping terminal times established in \cite[Proposition~4.1]{BSDEYoung-I}, we have $\|(Y^{n},Z^{n})\|_{p,q;[0,T]}<\infty$ for every $n\ge 1$ and $q>1$. For $n\ge 1$
denote 
\begin{equation}\label{e:def of alpha}
\delta Y^{n}_t := Y^{n+1}_{t} - Y^{n}_{t}
,\ \delta g^{n}_{t} := g(Y^{n+1}_{t}) - g(Y^{n}_{t}),\text{ and }
\alpha^{n}_{t}:= \frac{\delta g^{n}_{t}}{\delta Y^n_{t}}\mathbf{1}_{\{\delta Y^{n}_{t}\neq 0\}}.
\end{equation}
It follows that		
\begin{equation}\label{e:deltaY}
\begin{cases}
d \delta Y^{n}_{t} = -  \alpha^{n}_{t}\delta Y^{n}_{t}\eta(dt,X_{t}) + \delta Z^{n}_{t} dW_{t},\ t\in[0,T_{n}],\\
\delta Y^{n}_{T_{n}} = Y^{n+1}_{T_{n}} - \Xi_{T_n}.
\end{cases}
\end{equation}	
Our goal is to show that $\{Y^n=Y^1+\sum_{i=1}^{n-1}\delta Y^i, n\ge 2\}$ is a Cauchy sequence in a proper Banach space, and for this purpose, we shall estimate $\delta Y^n$. This will be done in two steps.

\textbf{Step 1.} 
Let $A^{t;n}_s$ be the unique solution of the following Young differential equation 
\begin{equation}\label{e:def of A^t;n}
A^{t;n}_{s} = 1 + \int_{t}^{s} \alpha^{n}_r A^{t;n}_{r} \eta(dr,X_r),\quad s\in[t,T].
\end{equation} 
In this step, we will estimate the conditional $\frac{k}{k-1}$-moments of $A^{t\land T_{n};n}_{T_n}$
(recall that $k\in(1,2]$ is a given constant in Assumption~\ref{(T)}). This estimate will be used in the next step to provide an estimate for $\delta Y^{n}_{t},$ in view of the representation $\delta Y^{n}_{t} = \mathbb{E}_{t\land T_n}[A^{t\land T_{n};n}_{T_{n}}\cdot\delta Y^{n}_{T_{n}}].$

Clearly, we have $\esssup\limits_{t\in[0,T],\omega\in\Omega}|\alpha^{n}_{t}|\lesssim_{\Theta_2} 1$. In addition, since $\nabla g$ is bounded, we have (following the calculus as in \cite[(3.41)]{BSDEYoung-I})
\begin{equation*}
\|\alpha^{n}\|_{p\text{-var};[t,T]}\lesssim_{\Theta_2} \left(\|Y^{n}\|_{p\text{-var};[t,T]} + \|Y^{n+1}\|_{p\text{-var};[t,T]}\right).
\end{equation*}
Then, by the above inequality and \eqref{e:Gronwall to m_p,q (Y)}, we have
\begin{equation}\label{e:m_p2 alpha''}
m_{p,2}(\alpha^{n};[0,T])  \lesssim_{\Theta_2} n^{\frac{\lambda + \beta}{\varepsilon}\vee 1}\text{ for all }n\ge |x|. 
\end{equation}
Denote 
\begin{equation*}
\Gamma^{n}_{t} := \mathbb{E}_{t\land T_{n}}\left[|A^{t\land T_{n};n}_{T_n}|^{\frac{k}{k-1}}\right].
\end{equation*}
Since $A^{t;n}_{s} = \exp\{\int_{t}^{s}\alpha^{n}_{r}\eta(dr,X_{r})\},$ we have 
\begin{equation}\label{e:Gamma vs. tilde Gamma}
\Gamma^{n}_t = \mathbb{E}_{t}\left[\exp\Big\{\int_{t\land T_n}^{T_n}\frac{k}{k-1}\alpha^{n}_{r}\eta(dr,X_r)\Big\}\right]. 
\end{equation}
By the well-posedness for linear BSDEs with bounded drivers \cite[Corollary~3.3]{BSDEYoung-I}, there exists a process $\gamma^{n}$ such that $(\Gamma^n,\gamma^n)$ satisfies the following linear BSDE
\footnote{
When $N>1,$ the representation~\eqref{e:Gamma vs. tilde Gamma} fails in general. However, following the same line of proof as in \cite[Proposition~3.5]{BSDEYoung-I}, we can show that $\Gamma^{n}_{t}$ is bounded by the solution of some linear BSDE with terminal time $T_{n}$. Consequently, the estimate for $\Gamma^{n} $ will be the same as for $N=1$. See Remark~\ref{rem:difficulties in existence} and \nameref{proof:3'} of Theorem~\ref{thm:existence of solution of unbounded BSDEs} for further discussion.},
\begin{equation}\label{e:Gamma = 1 + int...}
\begin{cases}
d\Gamma^{n}_{t} = - \frac{k}{k-1}\alpha^{n}_{t} \Gamma^{n}_{t} \eta(dt,X_{t}) + \gamma^{n}_{t} dW_{t},\ t\in[0,T_{n}],\\
\Gamma^{n}_{T_n}= 1.
\end{cases}
\end{equation}
Moreover, by \cite[Proposition~3.3 and Remark~4.1]{BSDEYoung-I}, we have 
$\|(\Gamma^{n},\gamma^{n})\|_{p,2;[0,T]}<\infty$, which implies that $\esssup\limits_{t\in[0,T],\omega\in\Omega}\left|\Gamma^{n}_{t}\right|<\infty.$ Then apply the estimate for nonlinear Young integrals \cite[Proposition~2.1]{BSDEYoung-I} to get that for $t\in[0,T]$,
\begin{equation}\label{e:unbounded BSDEs-4-1}
\begin{aligned}
&\mathbb{E}_{t}\Big[\Big\|\int_{\cdot}^{T_{n}}\alpha^{n}_{r}\Gamma^{n}_{r}\eta(dr,X_{r})\Big\|^{2}_{p\text{-var};[t\land T_{n},T_{n}]}\Big]\\
&\lesssim_{\Theta_2} |T-t|^{2\tau}\mathbb{E}_{t}\Big[\Big(\|\alpha^{n}\Gamma^{n}\|^{2}_{\infty;[t\land T_{n},T_{n}]} + \|\alpha^{n}\Gamma^{n}\|^{2}_{p\text{-var};[t\land T_{n},T_{n}]} \Big)\left(1 + \|X\|^{2(\lambda + \beta)}_{\infty;[0,T_{n}]}\right)\\
&\qquad + \|\alpha^{n}\Gamma^{n}\|^{2}_{\infty;[t\land T_{n},T_{n}]}\|X\|^{2\lambda}_{p\text{-var};[t\land T_{n},T_{n}]}\left(1 + \|X\|^{2\beta}_{\infty;[0,T_{n}]}\right)\Big].
\end{aligned}
\end{equation}
We estimate each term on the right-hand side of the above inequality. In view of the fact $\Gamma_{T_n}^n=1,$ 
\begin{equation}\label{eq-5.1}
\mathbb{E}_{t}\left[\|\alpha^{n}\Gamma^{n}\|^{2}_{\infty;[t\land T_{n},T_{n}]}\right]\lesssim_{\Theta_{2}} \esssup_{s\in[t,T],\omega\in\Omega}\left|\Gamma^{n}_{s}\right|^{2}\lesssim 1+m_{p,2}(\Gamma^{n};[t,T])^{2}.
\end{equation}
Furthermore, by the above inequality, we have 
\begin{equation}\label{eq-5.2}
\begin{split}
&\mathbb{E}_{t}\left[\|\alpha^{n}\Gamma^{n}\|^{2}_{p\text{-var};[t\land T_{n},T_{n}]}\right]\\
&\lesssim\mathbb{E}_{t}\left[\|\Gamma^{n}\|^{2}_{p\text{-var};[t\land T_{n},T_{n}]}\right]\esssup_{s\in[t,T],\omega\in\Omega}\left|\alpha^{n}_{s}\right|^{2} + \mathbb{E}_{t}\left[\|\alpha^{n}\|^{2}_{p\text{-var};[t\land T_{n},T_{n}]}\right]\esssup_{s\in[t,T],\omega\in\Omega}\left|\Gamma^{n}_{s}\right|^{2}\\
&\lesssim_{\Theta_{2}} m_{p,2}(\Gamma^{n};[t,T])^{2} + m_{p,2}(\alpha^{n};[t,T])^{2}\left(1+m_{p,2}(\Gamma^{n};[t,T])^{2}\right).
\end{split}
\end{equation}
In addition, by Lemma~\ref{lem:m_p,q of X}, we have 
\begin{equation}\label{eq-5.3}
\mathbb{E}_{t}\left[\|\alpha^{n}\Gamma^{n}\|^{2}_{\infty;[t\land T_{n},T_{n}]}\|X\|^{2\lambda}_{p\text{-var};[t\land T_{n},T_{n}]}\right] \lesssim_{\Theta_{2}} \esssup_{s\in[t,T],\omega\in\Omega}\left|\Gamma^{n}_{s}\right|^{2}\cdot m_{p,2}(X;[0,T])^{2\lambda} \lesssim_{\Theta_{2}} 1 + m_{p,2}(\Gamma^{n};[t,T])^{2}.
\end{equation}
In view of \eqref{e:unbounded BSDEs-4-1}, \eqref{eq-5.1}, \eqref{eq-5.2}, \eqref{eq-5.3}, and the fact that $\|X\|_{\infty;[0,T_n]}\le n$ for $n\ge |x|$, we obtain 
\begin{equation}\label{e:unbounded BSDEs-4}
\begin{split}
&\mathbb{E}_{t}\Big[\Big\|\int_{\cdot}^{T_{n}}\alpha^{n}_{r}\Gamma^{n}_{r}\eta(dr,X_{r})\Big\|^{2}_{p\text{-var};[t\land T_{n},T_{n}]}\Big]\\
&\lesssim_{\Theta_2} |T-t|^{2\tau}n^{2(\lambda + \beta)}\left(1 + m_{p,2}(\alpha^{n};[0,T])^{2}\right)\left(1+ m_{p,2}(\Gamma^{n};[t,T])^{2}\right).    
\end{split}
\end{equation}
On the other hand, in view of \eqref{e:Gamma = 1 + int...} and Lemma~\ref{lem:estimate for Z},  
\begin{equation*}
\begin{aligned}
&\mathbb{E}_{t}\left[\Big\|\int_{\cdot}^{T_{n}}\gamma^{n}_{r}dW_{r}\Big\|^{2}_{p\text{-var};[t\land T_{n},T_{n}]}\right]\lesssim_{\Theta_{2}} \mathbb{E}_{t}\left[\Big\|\int_{\cdot}^{T_{n}}\alpha^{n}_{r}\Gamma^{n}_{r}\eta(dr,X_{r})\Big\|^{2}_{p\text{-var};[t\land T_{n},T_{n}]}\right],
\end{aligned}
\end{equation*}
which, by \eqref{e:unbounded BSDEs-4}, implies
\begin{equation}\label{e:unbounded BSDEs-5}
\begin{aligned}		
&\mathbb{E}_{t}\left[\Big\|\int_{\cdot}^{T_{n}}\gamma^{n}_{r}dW_{r}\Big\|^{2}_{p\text{-var};[t\land T_{n},T_{n}]}\right]\\
&\lesssim_{\Theta_{2}} |T-t|^{2\tau}n^{2(\lambda + \beta)}\left(1 + m_{p,2}(\alpha^{n};[0,T])^{2} \right)\left( 1+ m_{p,2}(\Gamma^{n};[t,T])^{2}\right).
\end{aligned}
\end{equation}
Combining \eqref{e:Gamma = 1 + int...}, \eqref{e:unbounded BSDEs-4}, \eqref{e:unbounded BSDEs-5}, and noting $m_{p,2}(\alpha^{n};[0,T])\lesssim_{\Theta_{2}} n^{\frac{\lambda + \beta}{\varepsilon}\vee 1}$, there exists $C>0$ such that for every $t\in[0,T]$ and $n\ge |x|$, 	
\begin{equation}\label{e:unbounded BSDEs-6}
m_{p,2}(\Gamma^n;[t,T])\le C |T-t|^{\tau}\left(n^{\frac{(1+\varepsilon)(\lambda + \beta)}{\varepsilon}\vee (\lambda + \beta + 1)} + n^{\frac{(1+\varepsilon)(\lambda + \beta)}{\varepsilon}\vee (\lambda + \beta + 1)}m_{p,2}(\Gamma^n;[t,T])\right).
\end{equation}
Thus, for any $t\in[0,T]$ satisfying $|T-t| \le \delta$ with 
\begin{equation}\label{e:delta value}
\delta := (2C)^{-\frac{1}{\tau}}  n^{-(\frac{(1+\varepsilon)(\lambda + \beta)}{\tau\varepsilon}\vee\frac{\lambda+\beta+1}{\tau})},
\end{equation}
the inequality \eqref{e:unbounded BSDEs-6} yields that $m_{p,2}(\Gamma^{n};[t,T])\le 1$, which then, according to \eqref{eq:ess-mpk}, implies the following local estimate
\begin{equation}\label{e:Gamma^n <= 2}
\esssup_{s\in[T-\delta,T],\omega\in\Omega}|\Gamma^{n}_{s}|\le 2.
\end{equation} 
To obtain a global estimate, for $ 0\le j\le \left\lfloor T/\delta\right\rfloor$, denote by $T^{j}_{n}:=T_{n}\land (T-j\delta),$ and $(\Gamma^{n;j}_t,\gamma^{n;j}_t)$ the unique solution of
\begin{equation*}
\begin{cases}
d\Gamma^{n;j}_{s} = - \alpha^{n}_{s} \Gamma^{n;j}_{s} \eta(ds,X_{s}) + \gamma^{n;j}_{s} dW_{s},\ s\in[0,T^{j}_{n}],\\
\Gamma^{n;j}_{T^j_n}= 1.
\end{cases}
\end{equation*}
Similar to \eqref{e:Gamma^n <= 2}, we have (let $T-(j+1)\delta = 0$ if $(j+1)\delta > T$)
\begin{equation}\label{e:Gamma<=2'}
\esssup_{s\in[T-(j+1)\delta,T-j\delta],\omega\in\Omega}\left|\Gamma^{n;j}_s\right|\le 2.
\end{equation}
Denoting $\widetilde{\Gamma}^{t;n}_{s}:=|A^{t;n}_{s}|^{\frac{k}{k-1}},$ we have $\Gamma^{n;j}_{t} = \mathbb{E}_{t}\left[\widetilde{\Gamma}^{t\land T^{j}_{n};n}_{T^{j}_{n}}\right]$ and $\widetilde{\Gamma}^{t;n}_{s} = \exp\left\{\int_{t}^{s}  \frac{k}{k-1}\alpha^{n}_{r}\eta(dr,X_r)\right\}.$
By \eqref{e:Gamma<=2'} and the equality $\widetilde{\Gamma}^{T^{j+1}_{n};n}_{T^{j}_{n}}\widetilde{\Gamma}^{T^{j+2}_{n};n}_{T^{j+1}_{n}} = \widetilde{\Gamma}^{T^{j+2}_{n};n}_{T^{j}_{n}},$ we have
\[\left|\Gamma^{n;j}_{T^{j+2}_{n}}\right| = \left|\mathbb{E}_{T^{j+2}_{n}}\left[\widetilde{\Gamma}^{T^{j+2}_{n};n}_{T^{j}_{n}}\right]\right| = \mathbb{E}_{T^{j+2}_{n}}\left[\widetilde{\Gamma}^{T^{j+2;n}_{n}}_{T^{j+1}_{n}}\mathbb{E}_{T^{j+1}_{n}}\left[\widetilde{\Gamma}^{T^{j+1}_{n};n}_{T^{j}_{n}}\right]\right]\le 4,\]
and then \begin{equation}\label{e:e:existence-7'}
\esssup_{s\in[0,T],\omega\in\Omega}|\Gamma^{n}_{s}|\le \prod_{j=0}^{\left\lfloor T/\delta\right\rfloor}\sup_{t\in[T-(j+1)\delta,T-j\delta]}\left\|\mathbb{E}_{t\land T^{j}_{n}}\left[ \widetilde{\Gamma}^{t\land T^{j}_{n};n}_{T^{j}_{n}}\right]\right\|_{L^{\infty}(\Omega)}\le 2^{\frac{T}{\delta} + 1}.
\end{equation}
Thus, in view of \eqref{e:delta value} and \eqref{e:e:existence-7'}, there exists a constant $C^*>0$ such that
\begin{equation}\label{e:exis unbounded BSDE-7}
\esssup_{s\in[0,T],\omega\in\Omega}|\Gamma^{n}_{s}|\le \exp\left\{C^* n^{\frac{(1+\varepsilon)(\lambda + \beta)}{\tau\varepsilon}\vee\frac{\lambda+\beta+1}{\tau}}\right\},\text{ for all }n\ge |x|.
\end{equation} 

\textbf{Step 2.} We are ready to estimate $\delta Y^{n}=Y^{n+1}-Y^{n}$ which satisfies the linear BSDE~\eqref{e:deltaY}. By the product rule shown in \cite[Corollary~2.1]{BSDEYoung-I}, we have
\begin{equation*}
\delta Y^{n}_{t\land T_{n}} = A^{t\land T_{n};n}_{T_{n}}\left(Y^{n+1}_{T_{n}} - \Xi_{T_{n}}\right) - \int_{t\land T_{n}}^{T_{n}}A^{t\land T_{n};n}_{r}\delta Z^{n}_{r} dW_{r},\  t\in[0,T].
\end{equation*}
Taking conditional expectation $\mathbb{E}_{t\land T_{n}}[\cdot]$ on both sides of the above equation leads to 
\begin{equation}\label{e:delta Y = A (Y - Xi)}
\|\delta Y^{n}_{\cdot}\|_{\infty;[0,T_{n}]} = \left\|\mathbb{E}_{\cdot\land T_{n}}\left[A^{\cdot\land T_{n};n}_{T_{n}}\left(Y^{n+1}_{T_{n}} - \Xi_{T_{n}}\right)\right]\right\|_{\infty;[0,T_{n}]}.
\end{equation}
From the regularity results obtained in \cite[Proposition~4.3]{BSDEYoung-I}, it follows that
\begin{equation}\label{e:Y-h}
\begin{aligned}
&\|Y^{n+1}_{T_{n}} - \Xi_{T_{n}}\|_{L^{k}(\Omega)} \\
&\le \|Y^{n+1}_{T_{n}} - \Xi_{T_{n+1}}\|_{L^{k}(\Omega)} + \|\Xi_{T_{n+1}} - \Xi_{T_{n}}\|_{L^{k}(\Omega)}\\
&\lesssim_{\Theta_{2},x}  \left\||T_{n+1} - T_{n}|^{\tau}\right\|_{L^{4k}(\Omega)} + \|\|\mathbb{E}_{\cdot}\left[\Xi_{T_{n+1}}\right]\|_{p\text{-var};[T_{n},T_{n+1}]}\|_{L^{k}(\Omega)} + \|\Xi_{T_{n+1}} - \Xi_{T_{n}}\|_{L^{k}(\Omega)}.
\end{aligned}
\end{equation}
In addition, by the triangular inequality, the BDG inequality for $p$-variation, and condition~\eqref{e:conditions of xi}, we have 
\begin{equation}\label{e:SPDE-estimate-Stoping 3}
\begin{aligned}
&\mathbb{E}\left[\|\mathbb{E}_{\cdot}\left[\Xi_{T_{n+1}}\right]\|^{k}_{p\text{-var};[T_{n},T_{n+1}]}\right] \\
&\lesssim_{\Theta_{2}} \mathbb{E}\left[\|\mathbb{E}_{\cdot}\left[\Xi_{T_{n+1}} - \Xi_{T}\right]\|^{k}_{p\text{-var};[T_{n},T_{n+1}]}\right] + \mathbb{E}\left[\|\mathbb{E}_{\cdot}\left[\Xi_{T}\right]\|^{k}_{p\text{-var};[T_{n},T_{n+1}]}\right]\\
&\lesssim_{\Theta_{2}} \mathbb{E}\left[|\Xi_{T_{n+1}} - \Xi_{T}|^{k}\right] + \mathbb{E}\left[\|\mathbb{E}_{\cdot}\left[\Xi_{T}\right]\|^{k}_{p\text{-var};[T_{n},T_{n+1}]}\right]\\
&\lesssim_{\Theta_{2}} \left\{\mathbb{E}\left[|T - T_{n}|\right]\right\}^{l}.
\end{aligned}
\end{equation}
Thus, combining \eqref{e:Y-h}, \eqref{e:SPDE-estimate-Stoping 3}, and Lemma~\ref{lem:existence of X}, there exists a constant $C'$ such that
\begin{equation}\label{e:<= exp -1/C n^2}
\|Y^{n+1}_{T_{n}} - \Xi_{T_{n}}\|_{L^{k}(\Omega)} \lesssim_{\Theta_{2},x} \exp\left\{-\frac{1}{C'}n^{2}\right\}.
\end{equation}

Recall $\Gamma_t^n = \mathbb{E}_{t\land T_{n}}\left[|A^{t\land T_{n};n}_{T_n}|^{\frac{k}{k-1}}\right]$.
Applying H\"older's inequality to \eqref{e:delta Y = A (Y - Xi)}, by \eqref{e:exis unbounded BSDE-7} and \eqref{e:<= exp -1/C n^2}, we have 
\begin{equation}\label{e:dY<=h(Y)Gamma}
\begin{aligned}
\mathbb{E}\left[\|\delta Y^{n}_{\cdot}\|_{\infty;[0,T_{n}]}\right] &\le \mathbb{E}\left[\Big\{\sup_{t\in[0,T]}\left|\mathbb{E}_{t\land T_{n}}\left[\left|Y^{n+1}_{T_{n}} - \Xi_{T_{n}}\right|^{k}\right]\right|\Big\}^{\frac{1}{k}}\Big\{\sup_{t\in[0,T]}\left|\Gamma^{n}_{t}\right|\Big\}^{\frac{k-1}{k}}\right]\\
&\lesssim_{\Theta_{2}} \|Y^{n+1}_{T_{n}} - \Xi_{T_{n}}\|_{L^{k}(\Omega)}\Big\{\esssup_{t\in[0,T],\omega\in\Omega}|\Gamma^{n}_{t}|\Big\}^{\frac{k-1}{k}}\\
&\lesssim_{\Theta_{2},x} \exp\left\{-\frac{1}{C'}n^{2} + \frac{(k-1)C^*}{k}n^{\frac{(1+\varepsilon)(\lambda + \beta)}{\tau\varepsilon}\vee\frac{\lambda+\beta+1}{\tau}}\right\}.
\end{aligned}
\end{equation}
By Assumption~\ref{(A1)} and Remark~\ref{rem:tau lambda beta}, $\frac{(1+\varepsilon)(\lambda + \beta)}{\tau\varepsilon}\vee\frac{\lambda+\beta+1}{\tau}<2$. This fact, combined with \eqref{e:dY<=h(Y)Gamma}, implies that 
\begin{equation*}
\lim_{n\to\infty}\sum_{m\ge n}\mathbb{E}\left[\|\delta Y^{m}\|_{\infty;[0,T_{n}]}\right] = 0.
\end{equation*}
Hence, for each fixed $n\ge 1,$ we have
\begin{equation}\label{e:cauchy Y^m}
\E\Big[\lim_{n'\to \infty}\sum_{m\ge n'}\|\delta Y^{m}\|_{\infty;[0,T_{n}]}\Big] = 0.
\end{equation}
Let $Y_{t} := \limsup_{m\to\infty} Y^{m}_{t}$. Then $Y$ is adapted, and by \eqref{e:cauchy Y^m}, $\lim_{m\to\infty}\|Y^{m} - Y\|_{\infty;[0,T_n]} = 0$ a.s. for each  $n\ge 1$. Since $\lim_{n\to\infty}\mathbb{P}\{T_{n} = T\} = 1$ due to the fact that $\left[T_{n} < T\right] = [\sup_{t\in[0,T]}|X_{t}|> n ]$, it follows that $\lim_{m\to\infty}\|Y^{m} - Y\|_{\infty;[0,T]} = 0$ a.s. Thus, the adapted process $Y$ is a.s. continuous on $[0,T].$ Furthermore, by Fatou's Lemma, we have that
\begin{equation}\label{e:Y^n - Y ->0 in L1}
\lim_{n\to \infty}\mathbb{E}\left[\|Y^{n} - Y\|_{\infty;[0,T_{n}]}\right]\le \lim_{n\to\infty}\sum_{m\ge n}\mathbb{E}\left[\|\delta Y^{m}\|_{\infty;[0,T_{n}]}\right] = 0.
\end{equation}
According to \cite[Proposition~4.2]{BSDEYoung-I}, we have for every $q>1$,
\begin{equation*}
\sup_{n}\|(Y^{n},Z^{n})\|_{\mathfrak{H}_{p,q + \varepsilon};[0,T]}<\infty.
\end{equation*}
Then, for all $q>1$, $\E[\|Y^n\|^{q}_{\infty;[0,T_n]}]$ is uniformly bounded in $n$, and hence  $\mathbb{E}\big[\|Y\|^{q}_{\infty;[0,T]}\big]<\infty$  by Fatou's Lemma. Since $\|Y^n - Y\|^{q}_{\infty;[0,T_n]}$ is uniformly integrable in $n$, from \eqref{e:Y^n - Y ->0 in L1} it follows that, for all $q>1$,
\begin{equation*}
\lim_{n\rightarrow \infty}\mathbb{E}\left[\|Y^{n} - Y\|^{q}_{\infty;[0,T_{n}]}\right] = 0.
\end{equation*}
Thus, by Proposition~\ref{prop:limit of solutions is also a solution} (with $S_n = T_n$), there exists a process $Z$ such that $(Y,Z)$ is a solution of \eqref{e:ourBSDE} with $(Y,Z)\in\mathfrak{H}_{p,q}(0,T).$  Furthermore, noting that $\lim\limits_{n\rightarrow\infty}\mathbb{P}\left[T_{n} = T\right] = 1$, by Proposition~\ref{prop:limit of solutions is also a solution} again, we have 
\begin{equation}\label{e:(Y^n,Z^n) -> (Y,Z)}
\lim_{n\rightarrow\infty}\left\|(Y^{n},Z^{n}) - (Y,Z)\right\|_{\mathfrak{H}_{p,q'};[0,T]} = 0,\text{ for every }q'\in \left(1,q\right).
\end{equation}
This yields the desired \eqref{e:con-0}, since $q$ can be arbitrarily large. The  proof is concluded. 
\end{proof}

\begin{remark}\label{rem:difficulties in existence}
In the above proof, when proving the key estimate~\eqref{e:dY<=h(Y)Gamma}, one critical step is to estimate $\Gamma^{n}_{t}$, i.e., to obtain \eqref{e:exis unbounded BSDE-7}. This is achieved, in Step 1 of the proof, by analyzing BSDE~\eqref{e:Gamma = 1 + int...} in stead of  estimating directly  the explicit expression~\eqref{e:Gamma vs. tilde Gamma} for the following reasons:

\begin{enumerate}
  
\item When $N>1$, in general $\Gamma^{n}_{t}$ does not admit an explicit expression (see Remark~\ref{rem:N=1 vs. N>1}),  while the estimates \eqref{e:unbounded BSDEs-4}--\eqref{e:exis unbounded BSDE-7} of BSDE~\eqref{e:Gamma = 1 + int...} can be extended to the multi-dimensional case in a straightforward way (see \eqref{e:unbounded BSDEs-4'}--\eqref{e:exis unbounded BSDE-7'} in \nameref{proof:3'}).

\medskip

\item Even for the case $N=1$, the expression~\eqref{e:Gamma vs. tilde Gamma} 
\begin{equation*}
\Gamma_t^n=\mathbb{E}_{t}\left[\exp\left\{\int_{t\land T_n}^{T_{n}}\frac{k}{k-1}\alpha^{n}_{r}\eta(dr,X_{r})\right\}\right] = \sum_{j=0}^{\infty}\frac{1}{j!}\mathbb{E}_{t}\left[\left|\int_{t\land T_n}^{T_{n}}\frac{k}{k-1}\alpha^{n}_{r}\eta(dr,X_{r})\right|^{j}\right]
\end{equation*}
does not make the analysis easier due to the nonlinear Young integral. More precisely, letting $\Lambda_{t}:= \int_{t\land T_n}^{T_{n}}\frac{k}{k-1}\alpha^{n}_{r}\eta(dr,X_{r})$ and following the argument used in \cite[p.~12]{DiehlZhang}, although we still have
\begin{equation*}
\left|\Lambda_{t}\right|^{j+1} = (j+1)\int_{t\land T_n}^{T_{n}}\frac{k}{k-1} \left|\Lambda_{r}\right|^{j}\alpha^{n}_{r}\eta(dr,X_{r}),
\end{equation*}
the extra term $X_{r}$ makes the growth of $j\mapsto\mathbb{E}_{t}[\left|\Lambda_{t}\right|^{j+1}]$ much more involved (see \cite[(2.1)]{BSDEYoung-I} with $(y_r,x_r)$ replaced by $(|\Lambda_{r}|^{j}\alpha^n_r, X_r)$). 
\end{enumerate}
\end{remark}

\begin{example}\label{ex:assump (T)}
We provide some examples of $\Xi=\{\Xi_t, t\in[0,T]\}$ which satisfies Assumption~\ref{(T)} with $k=2$. Let $(\tau,\lambda, \beta,\varepsilon,p)$ be parameters  given in \ref{(A1)} and  suppose that  $X$ satisfies Assumption~\ref{(A0)}. 
Assume  $\Xi_{T} \in \mathbb{D}^{1,2}$ and 
\begin{equation}\label{e:condition for C-O formula}
\esssup_{t\in[0,T]}\mathbb{E}\left[|D_t(\Xi_T)|^{q}\right]<\infty, \text{ for some } q>2, 
\end{equation}
where  $D$ is the Malliavin derivative. By Clark-Ocone's formula (see Clark \cite{clark1970representation} or Nualart \cite[Proposition~1.3.14]{nualart2006malliavin}), the BDG inequality (see \cite[Corollary~A.1]{BSDEYoung-I}), and H\"older's inequality,  we get
\begin{equation*}
\begin{aligned}
\mathbb{E}\left[\left\|\mathbb{E}_{\cdot}\left[ \Xi_T \right]\right\|^{2}_{p\text{-}\mathrm{var};[S,T]}\right]&\lesssim_{p} \mathbb{E}\Big[\int_{S}^{T}\left|\mathbb{E}_{t}\left[D_t(\Xi_T)\right]\right|^{2}dt\Big]\\
&\le \esssup_{t\in[0,T]}\left\{\mathbb{E}\left[|D_t(\Xi_T)|^{q}\right]\right\}^{\frac{2}{q}} \cdot T^{\frac{2}{q}} \cdot \left\{\mathbb{E}\left[|T-S|\right]\right\}^{\frac{q-2}{q}} .
\end{aligned}
\end{equation*} 
Assume further that for all $t\in[0,T]$
\begin{equation}\label{e:E. xi'}
|\Xi_{t}|\lesssim \left( 1 + \|X\|_{\infty;[0,t]}\right),\ \ \mathbb{E}\left[|\Xi_{S}-\Xi_{T}|^{2}\right]\lesssim \mathbb{E}\left[|T-S|\right],
\end{equation}
for all stopping time $S\le T$. Then, $\Xi$ satisfies Assumption~\ref{(T)} with $l = \frac{q-2}{q}$. Below we provide three examples of  $\Xi$ satisfying \eqref{e:condition for C-O formula} and \eqref{e:E. xi'}. 
\begin{enumerate}[label=(\alph*)]
\item 
Let $\Xi_{t} := h(X_{t_1\land t},X_{t_2\land t},...,X_{t_m\land t}),$ where $h\in C^{\mathrm{Lip}}(\R^{d\times m};\R).$ Now we briefly justify the two conditions \eqref{e:condition for C-O formula} and \eqref{e:E. xi'}. First, we refer to \cite[Theorem~2.2.1]{nualart2006malliavin} for the following fact,
\begin{equation}\label{e:Malliavin deri of X}
\sup_{t\in[0,T]}\mathbb{E}\Big[\sup_{s\in[0,T]}|D_t(X_s)|^{q'}\Big]\lesssim 1,\ \text{for every}\ q'\ge 1.
\end{equation}
Noting this inequality and the Lipschitzness of $h$, by \cite[Proposition~1.2.4]{nualart2006malliavin} we have $
|D_{t}(\Xi_{T})|\lesssim  \sum_{j=1}^{m}|D_{t}(X_{t_{j}})|$,
and hence \eqref{e:condition for C-O formula} follows from \eqref{e:Malliavin deri of X} for all $q>2$. Second, the representation~\eqref{e:X_t = x + ...} and the boundedness of $(\sigma,b)$ imply \eqref{e:E. xi'}.\\
\item Let $\Xi_{t} := h(W_{\cdot\land t})$ where $h\in C^{\mathrm{Lip}}(C([0,T];\R^{d});\R)$ with $C([0,T];\R^{d})$ equipped with the uniform metric. According to Cheridito-Nam~\cite[Proposition~3.2]{Cheridito2014}, the Malliavin derivative of $\Xi_{T} = h(W_{\cdot})$ is bounded, and therefore \eqref{e:condition for C-O formula} holds for every $q>2.$ In addition,  \eqref{e:E. xi'} is satisfied, since by the BDG inequality for $p$-variation (see, e.g., \cite[Corollary~A.1]{BSDEYoung-I}) we have
\[\E[|\Xi_{S} - \Xi_{T}|^2] \lesssim \E[\|W_{\cdot\vee S} - W_{S}\|^2_{\infty;[0,T]}] \le \E[\|W_{\cdot}\|^2_{p\text{-}\mathrm{var};[S,T]}]
\lesssim_{p} \E[|T-S|].\]

\item Let $\Xi_t := \sup_{s\le t}X^{1}_{s}$, where $X^{1}$ is the first component of $X$. Clearly we have $|\Xi_{t}|\le \|X\|_{\infty;[0,t]}$ and
\[\mathbb{E}\left[\left|\Xi_{S} - \Xi_{T}\right|^{2}\right]\le\mathbb{E}\left[\left\|X\right\|^{2}_{p\text{-}\mathrm{var};[S,T]}\right]\lesssim_{L,p,T} \mathbb{E}\left[|T-S|\right],\]
where the last inequality is from the BDG inequality for $p$-variation.  The condition \eqref{e:E. xi'} is verified. For the condition \eqref{e:condition for C-O formula}, it can be proved by adapting the proof of \cite[Proposition~2.1.10]{nualart2006malliavin}.


\end{enumerate}

\end{example}


Proposition~\ref{prop:independent of choice of stopping times} below indicates that the solution of \eqref{e:ourBSDE} is independent of the  approximating stopping times $\{T_n\}_n$ given in  \eqref{e:Tn}, i.e., one can replace $\{T_n\}_n$ by any other sequence of stopping times $\{S_n\}_n$ increasing to $T$ such that $X$ is bounded on $[0,S_n]$. This convergence result will play a key role in the study of Young PDEs in Section \ref{subsec:Nonlinear-FK}.

Let $S$ be a stopping time such that $S\le T$ and $X$ is bounded on $[0,S]$. By \cite[Proposition~4.1]{BSDEYoung-I}, there exists  $(Y^{S},Z^{S})\in\mathfrak{B}_{p,2}(0,T)$ uniquely solving BSDE
\begin{equation}\label{e:YS}
Y_{t\wedge S} = \Xi_S + \int_{t\land S}^{S} f(r,X_r,Y_r,Z_r) dr + \sum_{i=1}^{M}\int_{t\wedge S}^{S} g_{i}(Y_{r})\eta_{i}(dr,X_r) -\int_{t\wedge S}^S Z_rdW_r ,\ t\in[0,T].
\end{equation}

\begin{proposition}\label{prop:independent of choice of stopping times}
Assume \ref{(A0)}, \ref{(T)}, and \ref{(A1)} are satisfied. We also assume $|x|\le C_{X}$ for some constant $C_{X}>0$ where we recall $X_0=x$ by \eqref{e:X_t = x + ...}.  Let $\left\{S_{j}\right\}_{j\ge 1}$ be a sequence of stopping times increasing to $T$ such that $X$ is bounded on each $[0,S_{j}]$. Let $(Y^{S_{j}}, Z^{S_{j}})$ be the unique solution to \eqref{e:YS} with $S=S_j$, and let $(Y,Z)$ be the solution to \eqref{e:ourBSDE} by Theorem~\ref{thm:existence of solution of unbounded BSDEs}. Then, for every $t\in[0,T]$, $\{Y^{S_{j}}_{S_{j}\land t}\}_{j\ge 1}$ is a Cauchy sequence in $L^{1}(\Omega)$, converging to $Y_t$. 

If we further assume  $\mathbb{P}\{S_{j}\neq T\} \le p_{j}$ for some sequence $\{p_{j}\}_{j\ge 1}$ with $p_{j}\downarrow 0$,  then for every $\delta>0$, there exists $J_{\delta}\in \mathbb N$ such that for all $j\ge J_{\delta},$
\begin{equation}\label{e:local-uniform-conv}
\|Y^{S_{j}}_{S_{j}\land t} - Y_{t}\|_{L^{1}(\Omega)}\le \delta,
\end{equation}
where $J_{\delta}$ only depends on $\delta$, $C_{X}$, $\{p_{j}\}_{j\ge 1}$, and $\Theta_{2}.$
\end{proposition}

\begin{proof}
We shall prove for the case $k=2,$ $f\equiv 0,$ and $N=1$. The general case can be proved in a similar way as in \nameref{proof:3'} of Theorem~\ref{thm:existence of solution of unbounded BSDEs}. 

\textbf{Step 1.} We first prove  \eqref{e:local-uniform-conv}, assuming $|x|\le C_{X}$ and   $\mathbb{P}\{S_{j}\neq T\} \le p_j$ with $p_j \downarrow 0$. Let $T_{n}$ be defined as in \eqref{e:Tn}, and $Y^n$ be the unique solution of \eqref{e:Yn}. 
Let $(Y^{n,j},Z^{n,j}):=(Y^{\tau_{n,j}},Z^{\tau_{n,j}})\in \mathfrak{B}_{p,2}(0,T)$ be the unique solution of \eqref{e:YS} with $S = \tau_{n,j}:= T_n\wedge S_j$ for $n,j\ge 1$. Hence, we have
\begin{equation*}
Y_{t\wedge \tau_{n,j}}^{n,j} = \Xi_{\tau_{n,j}} + \int_{t\wedge \tau_{n,j}}^{\tau_{n,j}}g(Y^{n,j}_{r})\eta(dr,X_{r}) -\int_{t\wedge \tau_{n,j}}^{\tau_{n,j}}Z^{n,j}_r dW_{r},\ t\in[0,T]. 
\end{equation*}
The difference between $(Y^{n}, Z^{n})$ and $(Y^{n,j}, Z^{n,j})$ satisfies the following linear BSDE:
\begin{equation*}
\delta Y^{n,j}_{t\wedge \tau_{n,j}} = Y^{n}_{\tau_{n,j}} - \Xi_{\tau_{n,j}} + \int_{t\wedge \tau_{n,j}}^{\tau_{n,j}} \alpha^{n,j}_{r}\delta Y^{n,j}_{r}\eta(dr,X_{r}) -\int_{t\wedge \tau_{n,j}}^{\tau_{n,j}} \delta Z^{n,j}_{r} dW_{r},\ t\in[0,T],
\end{equation*}	
where $\delta Y^{n,j}_{t} := Y^{n}_{t} - Y^{n,j}_{t}$, $\delta Z^{n}_{t} := Z^n_{t} - Z^{n,j}_{t}$, and  $\alpha^{n,j}_{t} := \frac{g(Y^{n}_{t}) - g(Y^{n,j}_{t})}{\delta Y^{n,j}_{t}}\mathbf{1}_{\{\delta Y^{n,j}_{t} = 0\}}(t)$, and thus we have the Feynman-Kac formula:
\begin{equation}\label{e:deltaYni}
\delta Y^{n,j}_{t\wedge\tau_{n,j}}=\mathbb{E}_{t\land \tau_{n,j}}\bigg[ \left(Y^{n}_{\tau_{n,j}} - \Xi_{\tau_{n,j}}\right) \exp\Big\{\int_{t\land \tau_{n,j}}^{\tau_{n,j}}\alpha^{n,j}_{r}\eta(dr,X_{r})\Big\}\bigg],\ t\in[0,T]. 
\end{equation}

We shall estimate $\|Y_{S_j\wedge t}^{S_j}-Y_t\|_{L^1(\Omega)}$ in three steps. 

\

{\bf (i)} In the first step, we estimate $\|\delta Y^{n,j}_{t\wedge\tau_{n,j}}\|_{L^1(\Omega)} = \|Y^{n}_{t\wedge\tau_{n,j}} - Y^{n,j}_{t\wedge\tau_{n,j}}\|_{L^1(\Omega)}$. 
In view of \eqref{e:deltaYni}, we first estimate $\|Y^{n}_{\tau_{n,j}} - \Xi_{\tau_{n,j}}\|_{L^{2}(\Omega)}$. By \cite[Proposition~4.3]{BSDEYoung-I} with $q=2$ and $(S_0,S_{1})$ replaced by $(T_n,\tau_{n,j})$, we have for $|x|\le C_{X}$, 
\begin{equation}\label{e:Y^n - Xi on tau_n,i}
\begin{aligned}
\|\Xi_{T_{n}} - Y^{n}_{\tau_{n,j}}\|_{L^{2}(\Omega)}&\lesssim_{\Theta_{2},C_{X}} \left\||T_{n} - \tau_{n,j}|^{\tau}\right\|_{L^{8}(\Omega)} + \left\|\|\mathbb{E}_{\cdot}[\Xi_{T_{n}}]\|_{p\text{-var};[\tau_{n,j},T_{n}]}\right\|_{L^{2}(\Omega)}\\
&\le\left\||T - S_{j}|^{\tau}\right\|_{L^{8}(\Omega)} + \left\|\|\mathbb{E}_{\cdot}[\Xi_{T_{n}}]\|_{p\text{-var};[\tau_{n,j},T_{n}]}\right\|_{L^{2}(\Omega)}.
\end{aligned}
\end{equation}
To estimate the second term on the right-hand side,  
note that  Assumption~\ref{(T)} yields $$\left\|\Xi_{T_n} - \Xi_{T}\right\|_{L^2(\Omega)}\lesssim_{\Theta_{2}} \|T - T_n\|^{\frac{l}{2}}_{L^{1}(\Omega)}$$ and 
\[\left\|\|\mathbb{E}_{\cdot}[\Xi_{T}]\|_{p\text{-var};[\tau_{n,j},T_n]}\right\|_{L^{2}(\Omega)}\le \left\|\|\mathbb{E}_{\cdot}[\Xi_{T}]\|_{p\text{-var};[S_j,T]}\right\|_{L^{2}(\Omega)} \lesssim_{\Theta_{2}} \|T - S_j\|^{\frac{l}{2}}_{L^{1}(\Omega)},\]
and thus by  the same argument leading to \eqref{e:SPDE-estimate-Stoping 3}, we get
\begin{equation}\label{e:E[Xi]}
\begin{aligned}
\left\|\|\mathbb{E}_{\cdot}[\Xi_{T_n}]\|_{p\text{-var};[\tau_{n,j},T_{n}]}\right\|_{L^{2}(\Omega)}&\lesssim_{\Theta_{2}} \left\|\Xi_{T_n} - \Xi_{T}\right\|_{L^{2}(\Omega)} +  \left\| 
\|\mathbb{E}_{\cdot}[\Xi_{T}]\|_{p\text{-var};[\tau_{n,j},T_n]} \right\|_{L^{2}(\Omega)}\\
&\lesssim_{\Theta_{2}}  \|T - T_n\|^{\frac{l}{2}}_{L^{1}(\Omega)} + \|T - S_j\|^{\frac{l}{2}}_{L^{1}(\Omega)}.
\end{aligned}
\end{equation}
Using Assumption~\ref{(T)} again, we get 
\begin{equation}\label{e:tau-T}
\|\Xi_{\tau_{n,j}} - \Xi_{T_{n}}\|_{L^{2}(\Omega)}\lesssim_{\Theta_{2}} \|T-T_n\|_{L^{1}(\Omega)}^{\frac{l}{2}} + \|T-S_{j}\|_{L^{1}(\Omega)}^{\frac{l}{2}}.
\end{equation}
Thus, by \eqref{e:Y^n - Xi on tau_n,i}, \eqref{e:E[Xi]} and, \eqref{e:tau-T}, it follows from the triangular inequality and Lemma~\ref{lem:existence of X} that	for some $C>0$,	
\begin{equation}\label{e:h-Y}
\|Y^{n}_{\tau_{n,j}} - \Xi_{\tau_{n,j}}\|_{L^{2}(\Omega)} \lesssim_{\Theta_{2},C_{X}} \exp\{-Cn^2\} + \|T - S_j\|^{\frac{l}{2}}_{L^{1}(\Omega)} +  \left\||T - S_{j}|^{\tau}\right\|_{L^{8}(\Omega)}.
\end{equation}

\

Denote
\begin{equation*}
\Gamma^{n,j}_{t} := \mathbb{E}_{t}\bigg[\exp\Big\{\int_{t\land \tau_{n,j}}^{\tau_{n,j}}2\alpha^{n,j}_{r}\eta(dr,X_{r})\Big\}\bigg],\ t\in[0,T]. 
\end{equation*}
In view of the fact that $\tau_{n,j}\le T_n,$ by the  procedure leading to \eqref{e:exis unbounded BSDE-7}, there exists $C^{\diamond}>0$ such that for all $n\ge |x|,$
\begin{equation}\label{e:Gamma-ni}
\esssup\limits_{\omega\in\Omega, t\in[0,T]} |\Gamma^{n,j}_{t\wedge \tau_{n,j}}| \le \exp\left\{C^{\diamond}n^{\frac{(1+\varepsilon)(\lambda+\beta)}{\tau\varepsilon}\vee \frac{\lambda+\beta+1}{\tau}}\right\},\  \text{ for all } j\ge 1.
\end{equation}
By \eqref{e:deltaYni} we have for $t\in[0,T]$, 
\begin{equation}\label{e:deltaY-L2}
\|\delta Y^{n,j}_{t\wedge\tau_{n,j}}\|_{L^1(\Omega)}=\|Y^n_{t\wedge\tau_{n,j}}-Y^{n,j}_{t\wedge\tau_{n,j}}\|_{L^1(\Omega)}\le \esssup\limits_{\omega\in\Omega, t\in[0,T]} |\Gamma^{n,j}_{t\wedge \tau_{n,j}}|^{\frac{1}{2}}  \|Y^{n}_{\tau_{n,j}} - \Xi_{\tau_{n,j}}\|_{L^2(\Omega)}.
\end{equation}
Denote $C_n := \exp\left\{\frac{C^{\diamond}}{2}n^{\frac{(1+\varepsilon)(\lambda+\beta)}{\tau\varepsilon}\vee \frac{\lambda+\beta+1}{\tau}}\right\}$. While by Remark~\ref{rem:tau lambda beta}, $\frac{(1+\varepsilon)(\lambda+\beta)}{\tau\varepsilon}\vee \frac{\lambda+\beta+1}{\tau} < 2.$ Thus, by \eqref{e:h-Y}, \eqref{e:Gamma-ni}, and \eqref{e:deltaY-L2}, there exists $C^{\star}>0$ such that
\begin{equation}\label{e:ineq-0}
\|\delta Y^{n,j}_{t\wedge\tau_{n,j}}\|_{L^{1}(\Omega)}\lesssim_{\Theta_{2},C_{X}} \exp\left\{-C^{\star} n^{2}\right\} + C_n\left(\|T - S_j\|^{\frac{l}{2}}_{L^{1}(\Omega)} +  \left\||T - S_{j}|^{\tau}\right\|_{L^{8}(\Omega)}\right).
\end{equation}

\

{\bf (ii)} In this step, we estimate $\|Y_{t\wedge\tau_{n,j}}^n-Y_t\|_{L^1(\Omega)}$ and then $\|Y_{t\wedge\tau_{n,j}}^{n,j}-Y_t\|_{L^1(\Omega)}$.   Noting that \cite[Proposition~4.2]{BSDEYoung-I} yields the uniform boundedness of $\mathbb{E}\left[\|Y^{n}\|^{2}_{p\text{-var};[0,T_{n}]}\right]$, we get 
\begin{equation}\label{e:ineq-1}
\|Y^{n}_{t\land \tau_{n,j}} - Y^{n}_{t\land T_{n}}\|_{L^{1}(\Omega)} \le \mathbb{P}\left\{T_{n} \neq \tau_{n,j}\right\}^{\frac{1}{2}}\sup_{n}\left\|\|Y^{n}\|_{p\text{-var};[0,T_{n}]}\right\|_{L^{2}(\Omega)}\lesssim_{\Theta_{2},C_{X}} \mathbb{P}\left\{T \neq S_{j}\right\}^{\frac{1}{2}}.
\end{equation}
Moreover, by the estimate for the $\mathfrak{H}_{p,q}$ norm of $(Y,Z)$ established in \cite[Corollary~4.1]{BSDEYoung-I}, we have $\left\|\|Y\|_{p\text{-var};[0,T]}\right\|_{L^{2}(\Omega)}\lesssim_{\Theta_{2},C_{X}} 1$ (note that by Theorem~\ref{thm:existence of solution of unbounded BSDEs} $(Y,Z)\in\mathfrak{H}_{p,q}(0,T)$ for every $q>1$). By the proof of Lemma~\ref{lem:existence of X}, there exists $\hat{C}>0$ such that $$\mathbb{P}\{T_n < T\}^{\frac{1}{2}}\lesssim_{\Theta_{2},C_{X}}\exp\{-\hat{C}n^{2}\}.$$ Thus, by \eqref{e:dY<=h(Y)Gamma}, the following inequality holds for some constant $\check{C}>0,$ 
\begin{equation}\label{e:ineq-2}
\begin{aligned}
\|Y^{n}_{t\land T_{n}} - Y_{t}\|_{L^{1}(\Omega)} &\le \|Y^{n}_{t\land T_{n}} - Y_{t\land T_{n}}\|_{L^{1}(\Omega)} + \mathbb{P}\left\{T \neq T_{n}\right\}^{\frac{1}{2}}\left\|\|Y\|_{p\text{-var};[0,T]}\right\|_{L^{2}(\Omega)}\\
&\lesssim_{\Theta_{2},C_{X}} \exp\left\{-\check{C} n^{2}\right\} + \exp\left\{-\hat{C}n^{2}\right\}.
\end{aligned}
\end{equation}
 
Combining \eqref{e:ineq-0}, \eqref{e:ineq-1}, and \eqref{e:ineq-2} with the assumption $\mathbb{P}\{S_{j}\neq T\}\le p_{j}$  with $p_{j}\downarrow 0$, we get that for every $\delta > 0,$ there exists an $n_0\ge 1$ such that for every $n\ge n_0$, there exists an $I_{n}\ge 1$ depending on $C_n$ and $\{p_{j}\}_{j\ge 1}$ such that 
\begin{equation}\label{e:Y^n,i - Y}
\|Y^{n,j}_{t\land \tau_{n,j}} - Y_{t}\|_{L^{1}(\Omega)}<\frac{\delta}{2},\text{ for all }j\ge I_{n}.
\end{equation}

\

{\bf (iii)} Finally, we prove the desired estimation \eqref{e:local-uniform-conv}. Fix $j\ge 1$. Since $X$ is bounded on $[0,S_j]$,  we can find an $n\ge 1$ such that $\tau_{n,j} = S_{j}$, and hence $Y^{n,j} = Y^{S_{j}}$ when $n$ is sufficiently large. In view of \eqref{e:Y^n,i - Y}, to calculate the difference between $Y^{n,j}$ and $Y^{S_{j}}$, it suffices to establish a uniform (in $j$) estimate for \[\widetilde{\delta Y}^{n,j} := Y^{n+1,j} - Y^{n,j}.\] 
Following the same argument leading to \eqref{e:deltaY-L2}, there exist constants $\tilde{C},C'>0$ such that 
\begin{equation}\label{e:estimate for tilde delta Y}
\|\widetilde{\delta Y}^{n,j}_{t\land \tau_{n,j}}\|_{L^{2}(\Omega)}\lesssim_{\Theta_{2},C_{X}} \exp\left\{\tilde{C}n^{\frac{(1+\varepsilon)(\lambda + \beta)}{\tau\varepsilon}\vee \frac{\lambda+\beta+1}{\tau}}\right\}\exp\left\{- C' n^{2}\right\}.
\end{equation}
Notice that the right-hand side of \eqref{e:estimate for tilde delta Y} is independent of $j\ge 1$. Since $\frac{(1+\varepsilon)(\lambda + \beta)}{\tau\varepsilon}\vee \frac{\lambda+\beta + 1}{\tau} < 2$, for every $\delta>0$ there exists an integer $n_{1}\ge 1$ such that 
\begin{equation}\label{e:sum tilde delta Y}
\sum_{n\ge n_{1}}\|\widetilde{\delta Y}^{n,j}_{t\land \tau_{n,j}}\|_{L^{2}(\Omega)}\le \frac{\delta}{4}.
\end{equation}
Also, by the moments estimate \cite[Proposition~4.3]{BSDEYoung-I} and Lemma~\ref{lem:existence of X}, for each  $m\ge 1$, we have for $0<C^*<C'$
\begin{equation*}
\sum_{n\ge m}\|Y^{n+1,j}_{t\land \tau_{n+1, j}} - Y^{n+1,j}_{t\land \tau_{n,j}}\|_{L^{1}(\Omega)}\lesssim_{C_{X}} \sum_{n\ge m}\exp\left\{ - C^* n^{2}\right\}.
\end{equation*}
Then, for each $\delta>0$, there exists $n_{2}\ge 1$ such that
\begin{equation}\label{e:sum Y^{n+1,i} - Y^{n+1,i}}
\sum_{n\ge n_{2}}\|Y^{n+1,j}_{t\land  \tau_{n+1, j}} - Y^{n+1,j}_{t\land \tau_{n,j}}\|_{L^{1}(\Omega)}\le \frac{\delta}{4}.
\end{equation}
Combining \eqref{e:sum tilde delta Y} and \eqref{e:sum Y^{n+1,i} - Y^{n+1,i}}, we have
\[\sum_{n\ge n_{1}\vee n_{2}}\|Y^{n+1,j}_{t\land  \tau_{n+1, j}} - Y^{n,j}_{t\land \tau_{n,j}}\|_{L^{1}(\Omega)}\le \frac{\delta}{2},\]
and hence by Fatou's Lemma and the fact that $\lim_{n}Y^{n,j} = Y^{S_{j}}$, we have  for all $j\ge 1, n\ge n_{1}\vee n_{2}$,
\begin{equation}\label{e:Y^{n,i} - Y^{S_{i}}}
\|Y^{n,j}_{t\land \tau_{n,j}} - Y^{S_{j}}_{t\land S_{j}}\|_{L^{1}(\Omega)}\le \frac{\delta}{2}.
\end{equation}
Combining \eqref{e:Y^n,i - Y} and \eqref{e:Y^{n,i} - Y^{S_{i}}}, for every $\delta>0$, letting $J_\delta := I_{n_{0} \vee n_1\vee n_2}$ we have for every $j\ge J_\delta$,
\[\|Y^{S_{j}}_{t\land S_{j}} - Y_{t}\|_{L^{1}(\Omega)}\le \delta.\]

\textbf{Step 2.} Now we consider the case without assuming $\mathbb{P}\{S_{j}\neq T\} \le p_j$ with $p_j \downarrow 0$. That is, we assume $\lim_{j\rightarrow \infty}\mathbb{P}\{ S_{j} = T\}<1.$
The estimates \eqref{e:ineq-0}, \eqref{e:ineq-2}, and \eqref{e:Y^{n,i} - Y^{S_{i}}} still hold. Hence, for every $\delta>0,$ there exists $n'\ge 1$ such that
\begin{equation}\label{e:lim of three terms}
\limsup_{j\rightarrow\infty}\|Y^{n'}_{t\land \tau_{n',j}} - Y^{n',j}_{t\land \tau_{n',j}}\|_{L^{1}(\Omega)} + \|Y_{t} - Y^{n'}_{t\land T_{n'}}\|_{L^{1}(\Omega)} + \limsup_{j\rightarrow\infty}\|Y^{n',j}_{t\land \tau_{n',j}} - Y^{S_{j}}_{t\land S_j}\|_{L^{1}(\Omega)}\le \delta.
\end{equation}
In addition, since $\|\|Y^{n'}\|_{p\text{-var};[0,T]}\|_{L^{2}(\Omega)}<\infty$, by the dominated convergence theorem, we have 
\begin{equation}\label{e:lim of one term}
\lim_{j\rightarrow \infty}\|Y^{n'}_{t\land T_{n'}} - Y^{n'}_{t\land \tau_{n',j}}\|_{L^{1}(\Omega)}\le \lim_{j\rightarrow \infty}\|Y^{n'}_{t\land \tau_{n',j}} - Y^{n'}_{t\land T_{n'}}\|_{L^{2}(\Omega)} = 0,
\end{equation}
Then, by \eqref{e:lim of three terms} and \eqref{e:lim of one term}, we have $\limsup_{j\rightarrow\infty}\|Y_{t} - Y^{S_{j}}_{t\land S_{j}}\|_{L^{1}(\Omega)}<\delta$ for every $\delta>0,$ which implies the $L^1(\Omega)$-convergence of $Y^{S_j}_{S_{j}\land t}\rightarrow Y_{t}$. This completes the proof.
\end{proof}

In Theorem~\ref{thm:existence of solution of unbounded BSDEs}, we get the existence of the solution $(Y,Z)\in\mathfrak{H}_{p,q}(0,T)$ to \eqref{e:ourBSDE}. Let $S$ be a stopping time such that $X$ is bounded on $[0,S]$. The following result shows that the solution restricted on the random interval $[0,S]$ belongs to $\mathfrak B_{p,q}(0,T)$ if and only if  $Y$ is bounded on $[0,S]$. 

\begin{proposition}\label{prop:coincidence of two BSDEs}
Assume \ref{(A0)}, \ref{(T)}, and \ref{(A1)} are satisfied. Suppose $(Y,Z)\in\mathfrak{H}_{p,q}(0,T)$ is a solution to \eqref{e:ourBSDE}, with $q>1$. Let $S$ be a stopping time not greater than $T$ such that $X$ is bounded on $[0,S]$. Then, $(\bar{Y}_t,\bar{Z}_t):=(Y_{t\land S}, Z_{t}\mathbf{1}_{[0,S]}(t))\in \bigcap_{q'\ge1}\mathfrak B_{p,q'}(0,T)$ if  $\bar{Y}$ is bounded.  Moreover, 
$(\bar{Y},\bar{Z})\in\mathfrak B_{p,1}(0,T)$  only if  $\bar{Y}$ is bounded. 
\end{proposition}

\begin{proof}
Assume $\bar{Y}$ is a bounded process. For $j\ge 1,$ denote 
\[S_{j}:=\inf\Big\{t>0;\|Y\|_{p\text{-var};[0,t]} + \int_{0}^{t}|Z_{r}|^{2}dr = j\Big\}\land S.\] 
Then, we have $S_{j}\uparrow S$ as $j\rightarrow\infty$ and $(\bar{Y}^{j},\bar{Z}^{j})\in \mathfrak{B}_{p,q}(0,T)$, where 
$(\bar{Y}^{j}_t,\bar{Z}^{j}_t):=(Y_{t\land S_j}, Z_{t}\mathbf{1}_{[0,S_{j}]}(t)),\ t\in[0,T].$
Moreover, by \cite[Proposition~4.1]{BSDEYoung-I}, $(\bar{Y}^{j},\bar{Z}^{j})$ is the unique solution in $\bigcap_{q'\ge 1}\mathfrak B_{p,q'}(0,T)$ to 
\begin{equation*}
\begin{cases}
d\bar{Y}^{j}_t = -f(t,X_t,\bar{Y}^{j}_t,\bar{Z}^{j}_t)dt - \sum_{i=1}^{M}g_{i}(\bar{Y}^{j}_{t})\eta_{i}(dt,X_{t}) + \bar{Z}^{j}_{t}dW_{t}, ~ t\in[0,S_j],\\[5pt]
\bar{Y}^{j}_{S_j}=Y_{S_j},
\end{cases}
\end{equation*}
noting that \cite[Assumption~(H2)]{BSDEYoung-I} is satisfied for every $k=q'\ge 1$, since $m_{p,q'}(X;[0,T])<\infty$ for every $q'\ge 1$  by Lemma~\ref{lem:m_p,q of X}. Note that $X_{t\land S_{j}}$ and $Y_{S_j}$ are bounded  uniformly in $j\ge 1$ since $S_j\le S$. By the same proof of Lemma~\ref{lem:growth of (Y^n,Z^n)}, we can show
$\sup_{j\ge 1}\|(\bar{Y}^{j},\bar{Z}^{j})\|_{p,q';[0,T]}<\infty$  for all $q'\ge 1$, which together with  Fatou's lemma implies $\|(\bar{Y},\bar{Z})\|_{p,q';[0,T]}<\infty$ for all $q'\ge 1$.

On the other hand, 
assuming $(\bar{Y},\bar{Z})\in \mathfrak B_{p,1}(0,T)$, it follows directly from \eqref{eq:ess-mpk}  that $\bar Y$ is bounded. 
\end{proof}

\subsection{The uniqueness}\label{subsec:uniqueness}

Let $X^{t,x}$ solve the SDE
\begin{equation}\label{e:forward system}
X^{t,x}_{s} = x + \int_{t}^{s}b(X^{t,x}_{r})dr + \int_{t}^{s}\sigma(X^{t,x}_{r})dW_{r},\ (s,x)\in[t,T]\times\mathbb{R}^{d},\end{equation} 
where we stress that $b(\cdot)$ and $\sigma(\cdot)$ are independent of the time variable, and consider the following Markovian type BSDE:
\begin{equation}\label{e:backward system}
 Y_{s}  = h(X_T^{t,x})+\int_{s}^{T}f(r,X^{t,x}_{r}, Y_{r},  Z_{r} )dr + \sum_{i=1}^{M}\int_{s}^{T}g_{i}(  Y_{r} )\eta_{i}(dr,X^{t,x}_{r}) -\int_{s}^{T}  Z_{r}  dW_{r}, \  s\in[t,T].
\end{equation}

The uniqueness of the solution to Eq.~\eqref{e:backward system} is stated below:

\begin{theorem}\label{thm:uniqueness of the unbounded BSDE} Consider the forward-backward SDE system \eqref{e:forward system}-\eqref{e:backward system} and assume  \ref{(A0)} and \ref{(A1)}. Suppose that  the function $h(x)$ in \eqref{e:backward system} belongs to $C^{\mathrm{Lip}}(\mathbb{R}^{d};\mathbb{R}^{N})$. Assume further that for every $n\ge 1$, the set of regular points
\begin{equation}\label{e:regular set}
\Lambda_n:=\left\{x\in\mathbb{R}^{d};|x| = n\text{ and }\mathbb{P}\left\{T^{0,x}_n>0\right\}=0\right\}\text{ is closed,}
\end{equation}
where $T^{t,x}_n:=T\land\inf\left\{s>t;|X^{t,x}_{s}| > n\right\}$. Then, the following two statements hold.
\begin{itemize}
\item[(i)] Assume  $(\widetilde{Y}^{t,x},\widetilde{Z}^{t,x})$ and $(\widehat{Y}^{t,x},\widehat{Z}^{t,x})$  are two solutions in $\mathfrak{H}_{p,q}(t,T)$ of Eq.~\eqref{e:backward system} with $q>1$, satisfying 
\begin{equation}\label{e:assumption of widetilde Y}
|\widetilde{Y}^{t,x}_{s}|\vee|\widehat{Y}^{t,x}_{s}|\le C_{t,x}\left(1+|X^{t,x}_{s}|^{\frac{\lambda+\beta}{\varepsilon}\vee 1}\right) \text{ a.s.,  for all } s\in[t,T],
\end{equation}
for some constant $C_{t,x}>0$. Then $(\widetilde{Y}^{t,x},\widetilde{Z}^{t,x}) = (\widehat{Y}^{t,x},\widehat{Z}^{t,x})$.\\

\item[(ii)] Let $(Y^{t,x},Z^{t,x})\in\bigcap_{v>1}\mathfrak{H}_{p,v}(t,T)$ be the solution  of Eq.~\eqref{e:backward system} obtained in Theorem~\ref{thm:existence of solution of unbounded BSDEs}. Then there exists a constant $C>0$ independent of $(t,x)$ such that
\begin{equation*}
|Y^{t,x}_{s}|\le C\left(1+|X^{t,x}_{s}|^{\frac{\lambda+\beta}{\varepsilon}\vee 1}\right) \text{ a.s., for all }\, s\in[t,T].
\end{equation*}
\end{itemize}
\end{theorem}

\begin{proof}[Proof of Theorem~\ref{thm:uniqueness of the unbounded BSDE}]\makeatletter\def\@currentlabelname{the proof}\makeatother\label{proof:4}
For simplicity, we assume that $N=1$ and $f=0$. The proof can be adapted for the  situation $N>1$ and $f\neq 0$  by using the same argument as in \nameref{proof:3'} of Theorem~\ref{thm:existence of solution of unbounded BSDEs}. Let $(Y^{t,x},Z^{t,x})$ be the solution of Eq.~\eqref{e:backward system}  obtained in Theorem~\ref{thm:existence of solution of unbounded BSDEs}. 

First, we prove part (i), the uniqueness of the solution. Let $(\widetilde{Y}^{t,x},\widetilde{Z}^{t,x})\in\mathfrak{H}_{p,q}(t,T)$ be a solution to Eq.~\eqref{e:backward system} for some $q>1.$ Additionally, we assume $\widetilde{Y}^{t,x}$ satisfies \eqref{e:assumption of widetilde Y} and hence $\widetilde{Y}^{t,x}$ is bounded on $[0,T^{t,x}_{n}]$. The uniqueness will be proved by showing that for every fixed $(t,x)$, $Y^{n;t,x} \rightarrow \widetilde{Y}^{t,x}$ as $n\rightarrow \infty$.

By Proposition~\ref{prop:coincidence of two BSDEs} we have
$\big(\widetilde{Y}^{t,x}_{\cdot\land T^{t,x}_{n}}, \widetilde{Z}^{t,x}_{\cdot}\mathbf{1}_{[t,T^{t,x}_{n}]}(\cdot)\big)\in\mathfrak{B}_{p,q}(t,T).$ Moreover, by \cite[Proposition~4.1]{BSDEYoung-I}, $(\widetilde{Y}^{t,x}_{s\land T^{t,x}_{n}}, \widetilde{Z}^{t,x}_{s}\mathbf{1}_{[t,T^{t,x}_{n}]}(s))$ is the unique solution in $\mathfrak{B}_{p,q}(t,T)$ to 
\begin{equation*}
\begin{cases}
dY_s = - f(s,X^{t,x}_s,Y_s,Z_s)ds -g(Y_{s})\eta(ds,X^{t,x}_{s}) + Z_{s}dW_{s}, ~ s\in[t,T^{t,x}_n],\\[3pt]
Y_{T^{t,x}_n}=\widetilde{Y}^{t,x}_{T^{t,x}_n}.
\end{cases}
\end{equation*}
Furthermore, \eqref{e:assumption of widetilde Y} and Lemma \ref{lem:m_p,q of X} imply that
\[\sup_{n\ge 1}\left\|\widetilde{Y}^{t,x}_{T^{t,x}_n}\right\|_{L^{q'}(\Omega)}<\infty,\text{ for every }q'>1.\]
Combining the above estimate with \cite[Proposition~4.2]{BSDEYoung-I}, we obtain the following estimate,
\[\sup_{n\ge 1}\left\|\left(\widetilde{Y}^{t,x}_{\cdot\land T^{t,x}_{n}},\widetilde{Z}^{t,x}_{\cdot}\mathbf{1}_{[t,T^{t,x}_{n}]}(\cdot)\right)\right\|_{\mathfrak{H}_{p,q'};[t,T]}<\infty\ \text{for every }q' > 1,\]
which indicates  that the norms of the truncated processes $(\widetilde{Y}^{t,x}_{\cdot\land T^{t,x}_{n}},\widetilde{Z}^{t,x}_{\cdot}\mathbf{1}_{[t,T^{t,x}_{n}]}(\cdot))$ are uniformly bounded in $n.$ By Fatou's Lemma, we get,
\begin{equation*}
(\widetilde{Y}^{t,x},\widetilde{Z}^{t,x})\in \mathfrak{H}_{p,q'}(t,T),\text{ for every }q'>1,
\end{equation*}
which allows us to apply the regularity result in \cite[Corollary~4.2]{BSDEYoung-I} to the get the following estimate: 
\begin{equation}\label{e:apply Cor 4.2}
\begin{aligned}
&\E\left[|\widetilde{Y}^{t,x}_{T} - \widetilde{Y}^{t,x}_{T^{t,x}_{n}}|^{2}\right]\\
&\lesssim_{\Theta_{2}} \left\{\E\left[|T-T^{t,x}_{n}|^{8\tau}\right]\right\}^{\frac{1}{4}}\left(1 + |x|^{\frac{2(\lambda+\beta)}{\varepsilon}\vee 2} + \|h(X^{t,x}_{T})\|^{2}_{L^{4}(\Omega)}\right) + \E\left[\|\E_{\cdot}[h(X^{t,x}_{T})]\|^{2}_{p\text{-var};[T^{t,x}_{n},T]}\right]\\
&\lesssim_{\Theta_2,x} \left\{\E\left[|T-T^{t,x}_{n}|\right]\right\}^{\frac{1}{4}} + \E\left[\|\E_{\cdot}[h(X^{t,x}_{T})])\|^{2}_{p\text{-var};[T^{t,x}_{n},T]}\right].
\end{aligned}
\end{equation}
In addition, by Example~\ref{ex:assump (T)} (A0) with $\Xi_{s} = h(X^{t,x}_{s\land T})$ and $q=4$, we also have 
\begin{equation}\label{e:apply Example 5.1}
\E\left[\|\E_{\cdot}[h(X^{t,x}_{T})]\|^{2}_{p\text{-var};[T^{t,x}_{n},T]}\right]\lesssim_{\Theta_{2}} \left\{\mathbb{E}\left[|T - T^{t,x}_{n}|\right]\right\}^{\frac{1}{2}}.
\end{equation}
Combining \eqref{e:apply Cor 4.2} and \eqref{e:apply Example 5.1}, we get
\begin{equation}\label{e:Y - Y}
\|h(X^{t,x}_{T}) - \widetilde{Y}^{t,x}_{T^{t,x}_{n}}\|_{L^{2}(\Omega)} = \|\widetilde{Y}^{t,x}_{T} - \widetilde{Y}^{t,x}_{T^{t,x}_{n}}\|_{L^{2}(\Omega)}\lesssim_{\Theta_{2},x}  \left\{\E[|T-T_n^{t,x}|] \right\}^{\frac{1}{8}} + \left\{\E[|T-T_n^{t,x}|] \right\}^{\frac{1}{4}}.
\end{equation}
Furthermore, by Example~\ref{ex:assump (T)} (A0) again, we have 
\begin{equation}\label{e:h - h}
\left\|h(X^{t,x}_{T}) - h(X^{t,x}_{T^{t,x}_{n}})\right\|_{L^{2}(\Omega)}\lesssim_{\Theta_{2}} \left\{\E[|T-T^{t,x}_{n}|]\right\}^{\frac{1}{2}}.
\end{equation}
Moreover, by \eqref{e:moment-T-Tn} with $q=2$ there, there exists a constant $C'>0$ such that for all $n>|x|$,
\begin{equation}\label{e:T - T}
\E\left[|T - T^{t,x}_{n}|\right]\lesssim_{\Theta_{2}} \exp\Big\{-\frac{(n-|x|)^2}{C'}\Big\}\le \exp\Big\{\frac{|x|^2}{C'} - \frac{n^2}{2C'}\Big\} \lesssim_{x} \exp\Big\{-\frac{n^2}{2C'}\Big\}.
\end{equation}
Consequently, by \eqref{e:Y - Y}, \eqref{e:h - h}, and \eqref{e:T - T}, there exists a constant $C>0$ such that
\begin{equation}\label{e:uniqueness tilde C}
\begin{aligned}
\|Y^{n;t,x}_{T^{t,x}_{n}} - \widetilde{Y}^{t,x}_{T^{t,x}_{n}} \|_{L^{2}(\Omega)}&\le \|h(X^{t,x}_{T}) - h(X^{t,x}_{T^{t,x}_n})\|_{L^{2}(\Omega)}+ \|h(X^{t,x}_{T}) - \widetilde{Y}^{t,x}_{T^{t,x}_{n}} \|_{L^{2}(\Omega)}\\
&\lesssim_{\Theta_{2},x} \exp\left\{-Cn^{2}\right\}.
\end{aligned}
\end{equation}

On the other hand, for $n>|x|$, by the definition of $T^{t,x}_{n}$, we have $\big\|X^{t,x}_{\cdot\land T^{t,x}_{n}}\big\|_{\infty;[t,T]}\le n$. This together with condition~\eqref{e:assumption of widetilde Y} yields the following estimate for the terminal value:
\[\Big\|\widetilde{Y}^{t,x}_{T^{t,x}_{n}}\Big\|_{L^{\infty}(\Omega)} \lesssim_{\Theta_{2},t,x} \Big\|1 + \big|X^{t,x}_{T^{t,x}_{n}}\big|^{\frac{\lambda+\beta}{\varepsilon}\vee 1}\Big\|_{L^{\infty}(\Omega)}\lesssim n^{\frac{\lambda+\beta}{\varepsilon}\vee 1}\text{ for all }n> |x|.\]  
This estimate allows us to repeat the proof of  Lemma~\ref{lem:growth of (Y^n,Z^n)} with the terms $T_n,X_{\cdot},\Xi_{T_n},Y^n_{\cdot},Z^n_{\cdot}$ replaced respectively by $T^{t,x}_{n},X^{t,x}_{\cdot}$, $\widetilde{Y}^{t,x}_{T^{t,x}_{n}}$, $\widetilde{Y}^{t,x}_{\cdot\land T^{t,x}_{n}}$, $\widetilde{Z}^{t,x}_{\cdot}\mathbf{1}_{[t,T^{t,x}_{n}]}(\cdot)$, and we get
\begin{equation}\label{e:mp2 tilde Y}
m_{p,2}(\widetilde{Y}^{t,x}_{\cdot \land T^{t,x}_{n}};[t,T])\lesssim_{\Theta_{2},t,x} n^{\frac{\lambda+\beta}{\varepsilon}\vee 1}\text{ for all }n> |x|.
\end{equation}
For $s\in[t,T]$, denote
\[\widetilde{\alpha}^{n}_{s}:=\frac{g(Y^{n;t,x}_{s\land T^{t,x}_{n}}) - g(\widetilde{Y}^{t,x}_{s\land T^{t,x}_{n}})}{Y^{n;t,x}_{s\land T^{t,x}_{n}} - \widetilde{Y}^{t,x}_{s\land T^{t,x}_{n}}} 
\mathbf{1}_{\left\{Y^{n;t,x}_{s\land T^{t,x}_{n}} - \widetilde{Y}^{t,x}_{s\land T^{t,x}_{n}}\neq 0\right\}}\  \text{ and }\widetilde{\Gamma}^{n}_{s} := \E_{s}\left[\exp\left\{\int_{s\land T^{t,x}_{n}}^{T^{t,x}_{n}} 2\widetilde{\alpha}^{n}_{r}\eta(dr,X^{t,x}_{r})\right\}\right].\]
Similar to \eqref{e:m_p2 alpha''}, by \eqref{e:mp2 tilde Y} we have
\begin{equation*}
m_{p,2}(\widetilde{\alpha}^{n};[t,T])\lesssim_{\Theta_{2},t,x}n^{\frac{\lambda+\beta}{\varepsilon}\vee 1} \text{ for all }n> |x|.
\end{equation*}
Therefore, by the calculation leading to \eqref{e:exis unbounded BSDE-7}, there exists a constant $C^*>0$ such that
\begin{equation}\label{e:Gamma in unique}
\esssup_{s\in[t,T],\omega\in\Omega}|\widetilde{\Gamma}^n_{s}|\le \exp\left\{C^{*}n^{\frac{(1+\varepsilon)(\lambda+\beta)}{\tau\varepsilon}\vee \frac{\lambda+\beta+1}{\tau}}\right\},\text{ for all }n>|x|.
\end{equation}

Note that for $s\in[t,T]$,
\begin{equation*}
\begin{aligned}
Y^{n;t,x}_{s} - \widetilde{Y}^{t,x}_{s} = Y^{n;t,x}_{T^{t,x}_{n}} - \widetilde{Y}^{t,x}_{T^{t,x}_{n}} + \int_{s\land T^{t,x}_{n}}^{T^{t,x}_{n}}\widetilde{\alpha}^{n}_{r}\left(Y^{n;t,x}_{r} - \widetilde{Y}^{t,x}_{r}\right) \eta(dr,X^{t,x}_{r})  - \int_{s\land T^{t,x}_{n}}^{T^{t,x}_{n}}\left[Z^{n;t,x}_{r} - \widetilde{Z}^{t,x}_{r}\right]dW_{r}.
\end{aligned}
\end{equation*}
Then, similar to \eqref{e:delta Y = A (Y - Xi)}, by the product rule proved in \cite[Corollary~2.1]{BSDEYoung-I} we have 
\begin{equation*}
\|Y^{n;t,x}_{\cdot} - \widetilde{Y}^{t,x}_{\cdot}\|_{\infty;[t,T^{t,x}_{n}]} = \bigg\|\mathbb{E}_{\cdot\land T^{t,x}_{n}}\bigg[\exp\Big\{\int_{\cdot\land T^{t,x}_{n}}^{T^{t,x}_{n}}\widetilde{\alpha}^{n}_{r}\eta(dr,X^{t,x}_{r})\Big\}\left( Y^{n;t,x}_{T^{t,x}_{n}} - \widetilde{Y}^{t,x}_{T^{t,x}_{n}}\right)\bigg]\bigg\|_{\infty;[t,T^{t,x}_{n}]}.
\end{equation*}
Hence, by \eqref{e:uniqueness tilde C} and \eqref{e:Gamma in unique}, there exists a constant $\widetilde{C}>0$ such that
\begin{equation*}
\begin{aligned}
\mathbb{E}\left[\| Y^{n;t,x}_{\cdot} - \widetilde{Y}^{t,x}_{\cdot}\|_{\infty;[t,T^{t,x}_{n}]}\right]&\le \mathbb{E}\left[\Big\{\sup_{s\in[t,T]}\Big|\widetilde{\Gamma}^{n}_{s}\Big|\Big\}^{\frac{1}{2}}\bigg\{\sup_{s\in[t,T]}\bigg|\mathbb{E}_{s\land T^{t,x}_{n}}\Big[\left|Y^{n;t,x}_{T^{t,x}_{n}} - \widetilde{Y}^{t,x}_{T^{t,x}_{n}}\right|^{2}\Big]\bigg|\bigg\}^{\frac{1}{2}}\right]\\
&\lesssim_{\Theta_{2},t,x} \exp\left\{ \widetilde{C}n^{\frac{(1+\varepsilon)(\lambda + \beta)}{\tau\varepsilon}\vee\frac{\lambda+\beta+1}{\tau}} -\frac{1}{\widetilde{C}}n^{2}\right\}.
\end{aligned}
\end{equation*}

Thus, by Remark~\ref{rem:tau lambda beta}, the fact that $\frac{(1+\varepsilon)(\lambda + \beta)}{\tau\varepsilon}\vee\frac{\lambda+\beta+1}{\tau} < 2$ implies that
\[\lim_{n\rightarrow\infty}\mathbb{E}\left[\|Y^{n;t,x}_{\cdot} - \widetilde{Y}^{t,x}_{\cdot}\|_{\infty;[t,T^{t,x}_{n}]}\right] = 0.\]
On the other hand, by Theorem~\ref{thm:existence of solution of unbounded BSDEs}, it follows that
\[\lim_{n\rightarrow\infty}\mathbb{E}\left[\|Y^{n;t,x}_{\cdot} - Y^{t,x}_{\cdot}\|_{\infty;[t,T^{t,x}_{n}]}\right] = 0,\]
which gives $\widetilde{Y}^{t,x}=Y^{t,x}$. 
Finally, $\widetilde{Z}^{t,x} = Z^{t,x}$ follows from the following equality: 
\begin{equation*}
\int_{t}^{s}\left(\widetilde{Z}^{t,x}_{r} - Z^{t,x}_{r}\right)dW_{r} = 0,\ s\in[t,T].
\end{equation*}

\

Now we prove part (ii), for which it suffices to construct a function $u:[0,T]\times\R^d\rightarrow \R$ such that  
\begin{equation}\label{e:con-u}
u(s,X^{t,x}_{s}) = Y^{t,x}_{s} \text{ a.s., for } s\in[t,T]~ \text{ and } \sup_{t\in[0,T]}|u(t,x)| \lesssim 1 + |x|^{\frac{\lambda+\beta}{\varepsilon}\vee 1}. 
\end{equation}

For $(t,x)\in[0,T]\times \R^d,$ denote $T^{t,x}_{n}:=\inf\{s>t;|X^{t,x}_{s}|> n\}\land T$ for $n\ge 1$. Then, by \cite[Proposition~4.1]{BSDEYoung-I}, the equation
\begin{equation*}
\begin{cases}
dY_{s} = - f(s,X^{t,x}_s,Y_s,Z_{s})ds - g(Y_{s})\eta(ds,X^{t,x}_{s}) + Z_{s}dW_{s},\quad s\in[t,T^{t,x}_{n}],\\
Y_{T^{t,x}_{n}} = h(X^{t,x}_{T^{t,x}_{n}}),
\end{cases}
\end{equation*}
has  a unique solution in $\mathfrak{B}_{p,2}(t,T)$ which is now denoted by $(Y^{n;t,x},Z^{n;t,x})$. Let $F_{n}(x)$ be a smooth cutoff function such that $F_{n}(x)=1$ for $|x|\le n$ and $F_{n}(x)=0$ for $|x|>n+1$. Denote \[\eta^{\delta}(t,x):=\int_{\mathbb{R}}\rho^{\delta}(t-s)\eta(s,x)ds,\ t\in[0,T],x\in\mathbb{R}^{d},\]
where $\rho\in C^{\infty}_{c}(\mathbb{R})$ is a mollifier and $\rho^{\delta}(t): = \rho(t/\delta)/\delta$. By the proof of \cite[Lemma~A.2]{BSDEYoung-I}, we have $\lim_{\delta\downarrow 0}\|F_{n}(x)\eta(t,x) - F_{n}(x)\eta^{\delta}(t,x)\|_{\tau',\lambda} = 0$ for every $\tau'\in(0,\tau).$ Choose $(\tau',\lambda)$ such that Assumption~\ref{(A1)} holds with $(\tau,\lambda)$ replaced by $(\tau',\lambda)$. By \cite[Proposition~4.1]{BSDEYoung-I} again, the following BSDE admits a unique solution in $\mathfrak{B}_{p,2}(t,T)$: 
\begin{equation}\label{e:BSDE-Y^n,delta}
\begin{cases}
dY_{s} = - f(s,X^{t,x}_s,Y_s,Z_{s})ds - g(Y_{s})\partial_{s}\eta^{\delta}(s,X^{t,x}_{s})ds + Z_{s}dW_{s},\quad s\in[t,T^{t,x}_{n}],\\
Y_{T^{t,x}_{n}} = h(X^{t,x}_{T^{t,x}_{n}}),
\end{cases}
\end{equation} 
which is denoted by $(Y^{n,\delta;t,x},Z^{n,\delta;t,x})$.  
Clearly, 
$Y^{n,\delta;t,x}_{t}$ and  $Y^{n;t,x}_{t}$ are both constants a.s., and we denote them by two deterministic functions 
$$ u^{n}(t,x) := Y^{n;t,x}_{t} \text{ and } u^{n,\delta}(t,x) := Y^{n,\delta;t,x}_{t}, \text{ for } t\in[0,T], |x|\le n.$$

Next, we will give the representation of $Y^{n,\delta;t,x}$ by $u^{n,\delta}.$ For any $\zeta\in\mathcal{F}_{t}$, denote by $\mathscr{Y}^{n,\delta;t,\zeta}$ the unique solution to Eq.~\eqref{e:BSDE-Y^n,delta} with $X^{t,x}$ replaced by $X^{t,\zeta}$, where $X^{t,\zeta}$ is the unique solution of \eqref{e:forward system} with $x$ replaced by $\zeta$. By Lemma~\ref{lem:flow property}, there exists a progressively measurable modification of $Y^{n,\delta;t,x}_{s}(\omega)$ (which we also denote by $Y^{n,\delta;t,x}_{s}(\omega)$, for simplicity) such that 
\[\mathscr{Y}^{n,\delta;s,X^{t,x}_{ s \wedge  T^{t,x}_n  } }_s = Y^{n,\delta;s,X^{t,x}_{s\land T^{t,x}_n}}_s \text{ a.s. for } s\in[t,T].\] 
The uniqueness of the solution to \eqref{e:BSDE-Y^n,delta}  implies 
\[\mathbf{1}_{[s<T^{t,x}_{n}]}\mathscr{Y}^{n,\delta;t,x}_{s\land T^{t,x}_{n}} = \mathbf{1}_{[s<T^{t,x}_{n}]}\mathscr{Y}^{n,\delta;s,X^{t,x}_{s\land T^{t,x}_{n}}}_{s} \text{ a.s. for } s\in[t,T].\] 
Hence, we get
\begin{equation}\label{e:YnYn}
Y^{n,\delta;t,x}_s =\left\{
\begin{aligned}
&Y^{n,\delta;s,X^{t,x}_{s}}_{s}, & s<T^{t,x}_{n};\\
&h(X^{t,x}_{T^{t,x}_n}), & s\ge  T^{t,x}_{n}
\end{aligned}\right.\text{\ \ \ a.s. for } s\in[t,T],|x|\le n.
\end{equation} 
By Lemma~\ref{lem:flow property}, $u^{n,\delta}(s,X^{t,x}_{s\land T^{t,x}_{n}}) = Y^{n,\delta;s,X^{t,x}_{s\land T^{t,x}_{n}}}_{s}$ a.s., and thus \eqref{e:YnYn} yields
\begin{equation}\label{e:unYn}
Y^{n,\delta;t,x}_{s} = \left\{
\begin{aligned}
&u^{n,\delta}(s,X^{t,x}_{s}),\ & s<T^{t,x}_{n};\\
&h(X^{t,x}_{T^{t,x}_n}),\ & s\ge  T^{t,x}_{n}
\end{aligned}\right. \text{\ \ \ a.s. for } s\in[t,T],|x|\le n.
\end{equation}

Noting $\lim_{\delta\downarrow 0}\|F_{n}(x)\eta(t,x) - F_{n}(x)\eta^{\delta}(t,x)\|_{\tau',\lambda} = 0,$ by the stability result proved in \cite[Corollary~3.1]{BSDEYoung-I}, we have $u^{n,\delta}(t,x)\rightarrow u^{n}(t,x)$ and $Y^{n,\delta;t,x}_{s}\rightarrow Y^{n;t,x}_{s}$ for $s\in[t,T]$, as $\delta\rightarrow 0$. This convergence and \eqref{e:unYn} yield that
\begin{equation}\label{e:unYn'}
Y^{n;t,x}_{s} = \left\{
\begin{aligned}
&u^{n}(s,X^{t,x}_{s}),\ &s<T^{t,x}_{n};\\
&h(X^{t,x}_{T^{t,x}_n}),\ &s\ge  T^{t,x}_{n}
\end{aligned}\right. \text{\ \ \ a.s. for every } s\in[t,T],|x|\le n.
\end{equation}
On the other hand, by Theorem~\ref{thm:existence of solution of unbounded BSDEs}, $Y^{n;t,x}_{s}\rightarrow Y^{t,x}_{s}$ as $n\rightarrow\infty.$ Hence, $Y^{t,x}_{t}=\lim_n Y^{n;t,x}_t$ is a constant a.s. and we denote it by a deterministic function function $u(t,x) := Y^{t,x}_{t} =\lim u^{n}(t,x)$ for $t\in[0,T], x\in\R^d$. As $n\rightarrow \infty,$ noting that $\lim_{n}T^{t,x}_{n} = T,$ by \eqref{e:unYn'} and the fact that $\lim_{n}Y^{n;t,x}_{T} = \lim_{n}h(X^{t,x}_{T^{t,x}_{n}}) = h(X^{t,x}_{T}) = u(T,X^{t,x}_{T})$ a.s., we get 
\begin{equation}\label{e:cond1-u}
Y^{t,x}_{s} = u(s,X^{t,x}_{s})\text{ a.s. for } s\in[t,T],x\in\mathbb{R}^{d}.
\end{equation}

Noting that \[\mathbb{E}\left[\|Y^{t,x}\|^{2}_{p\text{-var};[t,T]}\right]\lesssim_{\Theta_{2}} 1 + |x|^{\frac{2(\lambda+\beta)}{\varepsilon}\vee 2} + \mathbb{E}\left[\|X^{t,x}\|^{2}_{p\text{-var};[t,T]}\right]\lesssim_{\Theta_{2}} 1 + |x|^{\frac{2(\lambda+\beta)}{\varepsilon}\vee 2},\]
where the first inequality follows from \cite[Corollary~4.1]{BSDEYoung-I} and 
the second one holds since $|h(X^{t,x}_{T})|\lesssim_{\Theta_{2}} 1 + |x| + \|X^{t,x}\|_{p\text{-var};[t,T]}$,  we have\begin{equation}\label{e:cond2-u}
|u(t,x)| = \|Y^{t,x}_{t}\|_{L^{2}(\Omega)}\le \left\|\|Y^{t,x}\|_{p\text{-var};[t,T]}\right\|_{L^{2}(\Omega)} + \|h(X^{t,x}_{T})\|_{L^{2}(\Omega)} \lesssim_{\Theta_{2}} 1 + |x|^{\frac{\lambda+\beta}{\varepsilon}\vee 1}.
\end{equation}

Combining \eqref{e:cond1-u} and \eqref{e:cond2-u}, we get \eqref{e:con-u} and the proof of part (ii) is concluded.
\end{proof}

\begin{remark}\label{rem:exit time} Condition~\eqref{e:regular set} ensures the continuity of $x\mapsto
T^{t,x}_{n}$, (see, e.g., Lemma~\ref{lem:T t,zeta}). Moreover, this continuity will be used to prove the Markov property for the forward-backward SDE (see Lemma~\ref{lem:flow property}), which will play a critical role in the proof of Theorem~\ref{thm:uniqueness of the unbounded BSDE}.

A sufficient condition for \eqref{e:regular set}, e.g., is as follows. Denote $D := \{x\in\R^d ;|x| < n\}$. Assume the strict uniform ellipticity for $\sigma(x)\sigma^{\top}(x)$ (equivalently, $\sigma^{\top}(x)\sigma(x)$) on $\bar D$, i.e., there exists $\rho>0$ such that $
a^{\top}\sigma(x)\sigma^{\top}(x)a\ge \rho|a|^{2}$ for all $x\in\bar{D}, a\in\mathbb{R}^{d}.$
Then by Lions-Menaldi \cite[Lemma~2.4]{lions1982optimal} (see also Buckdahn-Nie \cite[Lemma~3.1]{buckdahn2016generalized}), $\Lambda_{n}=D$ and hence is closed. 

\end{remark}

By the proof of part (ii) in Theorem~\ref{thm:uniqueness of the unbounded BSDE}, we have the following corollary.
 
\begin{corollary}\label{cor:u(t,x)}
Under the assumptions of Theorem~\ref{thm:uniqueness of the unbounded BSDE}, define $u(t,x):=Y^{t,x}_{t}$ for $(t,x)\in[0,T]\times\mathbb{R}^{d}$. Then, we have $Y^{t,x}_{s} = u(s,X^{t,x}_{s})$ a.s. for every  $(s,x)\in[t,T]\times\mathbb{R}^{d},$  and 
\[|u(t,x)|\le C\left(1+|x|^{\frac{\lambda+\beta}{\varepsilon}\vee 1}\right), \text{ for all }\, t\in[0,T],x\in\mathbb{R}^{d},\]
where $C>0$ is a constant independent of $(t,x)$. 
\end{corollary}

\subsection{Linear BSDEs}\label{subsec:linear case}

In view of the importance of the linear case, for instance, in Pontryagin's stochastic maximum principle (see the first motivation in Part I \cite[Section~1.1]{BSDEYoung-I}), we consider the following linear BSDE,
\begin{equation}\label{e:BSDE-linear}
Y_{t} = \xi + \sum_{i=1}^{M}\int_{t}^{T} \alpha^{i}_{r}Y_{r}\eta_{i}(dr,X_{r}) + \int_{t}^{T} \left( Z_{r}G_{r} + f_{r}\right) dr - \int_{t}^{T}Z_{r}dW_{r},\ t\in[0,T],
\end{equation}
where $\xi\in\mathcal{F}_{T}$, $X_t$ and 
$\alpha_t$ are continuous adapted  processes, and $f$ and $G$ are progressively measurable functions. We make the following (weaker) assumption for this linear case.

\begin{assumptionp}{(A2)}\label{(A2)} 
Let $\tau,\lambda,p$ be the parameters  satisfying Assumption~\ref{(H0)}.
\begin{itemize}
\item[$(1)$] There exist $\rho,\chi,q\in(0,\infty)$ such that
\begin{equation}\label{e:X-exp}
\mathbb{E}\left[\exp\left\{\rho\left(\|X\|^{\chi}_{p\text{-var};[0,T]} + |X_0|^{\chi}\right)\right\}\right]<\infty;\ \mathbb{E}\left[\exp\left\{\|\alpha\|^{q}_{p\text{-var};[0,T]} + |\alpha_0|^{q}\right\}\right]<\infty.
\end{equation}

\item[$(2)$] There exists $\beta\ge 0$ such that $\eta\in C^{\tau,\lambda;\beta}([0,T]\times\mathbb{R}^{d};\mathbb{R}^{M})$. 

\item[$(3)$] $(\tau,\lambda,\beta,q)$ satisfies (recall that $N$ is the dimension of  $\xi$ and $Y_t$)
\begin{equation*}
\left\{\begin{aligned}
\lambda+\beta<\chi\text{ and }q>\frac{\chi}{\chi-(\lambda+\beta)},\quad \text{if } N=1;\\
\frac{\lambda+\beta}{\tau}<\chi\text{ and }q>\frac{\chi}{\chi\tau-(\lambda+\beta)},\quad\text{if } N>1.
\end{aligned}\right.
\end{equation*}

\item[$(4)$] Denote \begin{equation}\label{e:exp martingale}
M_{t} := \exp\left\{\int_{0}^{t}(G_{r})^{\top} dW_{r} - \frac{1}{2}\int_{0}^{t} |G_{r}|^{2} dr\right\}.
\end{equation} Assume that there exists $k>1$ such that $(\xi,f)$ satisfies
$$\E\left[\left|\xi\right|^{k} M_{T}\right]\text{ and }\E\left[\left|\int_{0}^{T}|f_{r}|dr\right|^{k} M_{T}\right]<\infty.$$

\item[$(5)$]  
Assume $G$ satisfies 
\begin{equation}\label{e:G-exp}
\E\left[\exp\left\{\gamma\int_{0}^{T} |G_{r}|^{2} dr \right\}\right] < \infty, 
\end{equation}
where $\gamma$ is a positive constant such that
\begin{equation}\label{e:for the gamma}
\gamma > \frac{1}{2}\left[\left(\frac{\sqrt{v}}{\sqrt{k} 
- \sqrt{v}}\right)^{2} \vee \left(\frac{\sqrt{v}}{\sqrt{v} - 1}\right)^{2}\right] ~\text{ for some }
 v\in(1,k).\end{equation}

\end{itemize}
\end{assumptionp}

\begin{remark}\label{rem:kappa = 2}
As an example, when $X=W$ is a Brownian motion, we can choose $\chi=2$ in Condition \eqref{e:X-exp}, and then $\lambda+\beta<2$ (resp. $\frac{\lambda+\beta}{\tau}<2$) is required when $N=1$ (resp. when $N>1$) in order to  satisfy condition (3) in  Assumption~\ref{(A2)}.
\end{remark}

In \cite[Section~3.2]{BSDEYoung-I}, we studied the well-posedness of the linear BSDE~\eqref{e:BSDE-linear} assuming that $\eta$, $ \alpha$, and $\xi$ are bounded. In this section, we will remove this  boundedness restriction. To this end, we shall derive an explicit formula for the solution, which is known as the Feynman-Kac formula for the associated linear Young PDE. For simplicity of notation, we assume $M=1$ without loss of generality throughout this section.

We first consider the case with a null drift, i.e., $(f,G)=0$ in \eqref{e:BSDE-linear}. To get an explicit formula for the solution, it suffices to prove that the process
\begin{equation*}
Y_t := \E_t \left[(\Gamma_T^t)^{\top}\xi\right], t\in[0,T],
\end{equation*}
is a unique solution to \eqref{e:BSDE-linear}, where $(\Gamma_{s}^{t}, s\ge t)$ uniquely solves the following Young differential equation,
\begin{equation}\label{e:gamma'}
\begin{cases}
d\Gamma_s^t = (\alpha_{s})^{\top}\Gamma_s^t \eta(ds, X_s),\ s>t,\\
\Gamma_t^t=I_{N},
\end{cases}
\end{equation}
with $I_{N}$ denoting an $N\times N$ unit  matrix. We briefly explain the idea here and the detailed argument will be provided in \nameref{proof:5} of Theorem~\ref{thm:linear BSDE} afterwards. 

\

We introduce the following property of $\Gamma$ which plays a key role in the study of linear BSDEs.
\begin{lemma}\label{lem:For Gamma}
Assume \ref{(A2)} holds. Then for each $t\in[0,T]$, there exists a unique solution  $\Gamma^{t}_{\cdot}\in C^{\tau\text{-}\mathrm{var}}([t,T])$  to \eqref{e:gamma'}. Furthermore, we have $\Gamma^{t}_{T} = \Gamma^{s}_{T}\Gamma^{t}_{s}$ a.s.  for $s\in[t,T]$.
\end{lemma}

\begin{proof}
 Eq.~\eqref{e:gamma'} can be written as, for $\nu\in[t,T],$
\begin{equation}\label{e:YDE}
\Gamma^{t}_{\nu} = I_{N} + \int_{t}^{\nu} (\alpha_{r})^{\top} \Gamma^{t}_{r} dM_{r},
\end{equation}
where $M_{t} := \int_{0}^{t}\eta(dr,X_{r})$, and the integral is Young integral. Then, by \cite[Proposition~7]{lejay2010controlled}, almost surely, \eqref{e:YDE} admits a unique solution $\Gamma^t_.$ in $C^{\tau\text{-var}}([t,T])$. 
Fix $s\in[t,T]$  and define
\begin{equation*}
\hat{\Gamma}_{\nu}:=\left\{
\begin{aligned}
&\Gamma^{t}_{\nu},\ \nu\in[t,s],\\
&\Gamma^{s}_{\nu}\Gamma^{t}_{s},\ \nu\in[s,T].
\end{aligned}
\right.
\end{equation*}
Clearly, $\hat \Gamma_\nu$ satisfies \eqref{e:YDE} for  ${\nu}\in[t,T]$. The uniqueness of the solution
 implies $\hat{\Gamma}_{\nu} = \Gamma^{t}_{\nu}$ a.s. for $\nu\in[t,T]$ and the desired result follows when $\nu = T$.
\end{proof}

The uniqueness of the solution to \eqref{e:BSDE-linear} is a consequence of  It\^o's formula. Indeed, given a solution  $(Y',Z')$  of \eqref{e:BSDE-linear}, applying  the product rule (\cite[Corollary~2.1]{BSDEYoung-I}) to $ (\Gamma_s^t)^{\top} Y'_s$ on  $[t,T]$, we have $d((\Gamma_s^t)^{\top} Y'_s)= (\Gamma_s^t)^{\top} Z_s dW_s$, and hence 
\[(\Gamma_T^t)^{\top} \xi = Y'_t + \int_t^T (\Gamma_s^t)^{\top} Z_sdW_s,\]
which implies that $Y'_t = \E_t \left[(\Gamma_T^t)^{\top}\xi\right].$

To prove the existence,  we construct a solution to Eq.~\eqref{e:BSDE-linear}.  
For any fixed $t\in[0,T]$, by the product rule, clearly the inverse of $\Gamma^t_r$ for any $r\in [t,T]$, denoted by $(\Gamma_r^t)^{-1}$, exists. Moreover, the flow $(\Gamma_\cdot^t)^{-1}$ satisfies
\begin{equation*}
(\Gamma_r^t)^{-1}=I_N -\int_t^r (\Gamma_s^t)^{-1}(\alpha_s)^\top \eta(ds, X_s), ~ r\in[t,T].
\end{equation*}
In particular, when $r=T$ we have
\begin{equation*}
(\Gamma_T^t)^{-1}=I_N -\int_t^T (\Gamma_s^t)^{-1}(\alpha_s)^\top \eta(ds, X_s).
\end{equation*}
Denote $\tilde \Gamma^t_s:=((\Gamma^t_s)^{\top})^{-1}$. Then
\begin{equation}\label{e:tilde Gamma, Gamma^-1}
\tilde \Gamma^t_T = I_{N} - \int_t^T \alpha_{s} \tilde \Gamma^t_s  \eta(ds, X_s). 
\end{equation}
Multiplying \eqref{e:tilde Gamma, Gamma^-1} by $(\Gamma_T^t)^{\top} \xi$, we obtain (noting that $(\Gamma^{t}_{T})^{\top} = (\Gamma^{t}_s)^{\top} (\Gamma^{s}_{T})^{\top}$ by Lemma~\ref{lem:For Gamma})
\[\xi = (\Gamma_T^t)^{\top} \xi - \int_t^T  \alpha_{s}(\Gamma^s_T)^{\top} \xi \eta(ds, X_s). \]
Taking the conditional expectation of the above equality, we have
\begin{equation}\label{e:product}
\E_t\left[(\Gamma^t_T)^{\top}\xi\right] = \E_t[\xi] + \E_t\left[\int_t^T  \alpha_{s}(\Gamma^s_T)^{\top} \xi \eta(ds, X_s)\right]. 
\end{equation}
Thus, assuming the following tower rule,
\begin{equation}\label{e:tower}
\E_t\left[\int_t^T  \alpha_{s} (\Gamma^s_T)^{\top} \xi \eta(ds, X_s)\right] = \mathbb{E}_{t}\left[\int_{t}^{T}\alpha_{s}\mathbb{E}_{s}\left[(\Gamma^s_T)^{\top} \xi\right]\eta(ds,X_{s})\right],
\end{equation}
the equality \eqref{e:product} implies 
\begin{equation}\label{e:BSDE-alternative}
\E_t\left[(\Gamma^t_T)^{\top} \xi\right] = \E_t\left[\xi\right] + \mathbb{E}_{t}\left[\int_{t}^{T}\alpha_{s}\mathbb{E}_{s}\left[(\Gamma^s_T)^{\top} \xi\right]\eta(ds,X_{s})\right],
\end{equation}
and hence $Y_t=\E_t[(\Gamma^t_T)^{\top}\xi]$ is a solution to \eqref{e:BSDE-linear} by the martingale representation theorem.

As indicated above, one key step for the well-posedness is to validate the tower rule~\eqref{e:tower}.  Therefore, it is sufficient to verify the conditions of Lemma~\ref{lem:tower-law 2}, which requires an estimate on $\Gamma^{t}_{T}$.

\begin{lemma}\label{lem:estim of Gamma}
Assume \ref{(A2)} holds and let $(\Gamma_s^t, s\ge t)$ be given by \eqref{e:gamma'} with $t\in[0,T]$. Then we have for all  $\kappa \ge 0$,
\begin{equation}\label{e:forward Gamma}
\mathbb{E}\left[\|\Gamma^{0}_{\cdot}\|^{\kappa}_{\frac{1}{\tau}\text{-}\mathrm{var};[0,T]}\right]\vee \mathbb{E}\left[\|(\Gamma^{0}_{\cdot})^{-1}\|^{\kappa}_{\frac{1}{\tau}\text{-}\mathrm{var};[0,T]}\right] <\infty,
\end{equation}
and
\begin{equation}\label{e:backward Gamma}
\mathbb{E}\left[\|\Gamma^{\cdot}_{T}\|^{\kappa}_{\frac{1}{\tau}\text{-}\mathrm{var};[0,T]}\right]\vee \mathbb{E}\left[\|(\Gamma^{\cdot}_{T})^{-1}\|^{\kappa}_{\frac{1}{\tau}\text{-}\mathrm{var};[0,T]}\right] <\infty.
\end{equation}
\end{lemma}
\begin{proof}
We will only prove \eqref{e:backward Gamma}, and \eqref{e:forward Gamma} can be shown in a similar way.

{\bf Case $N>1$.}  We first prove $\mathbb{E}\left[\|(\Gamma^{\cdot}_{T})^{-1}\|^{\kappa}_{\frac{1}{\tau}\text{-var};[0,T]}\right] <\infty$. Clearly  \eqref{e:gamma'} is equivalent to
\begin{equation}\label{e:Gamma_T}
\Gamma^{t}_{T} = I_{N} + \int_{t}^{T}(\alpha_{r})^{\top}\Gamma^{t}_{r}\eta(dr,X_r), ~t\in[0,T].
\end{equation}
Noting $\Gamma^{r}_{T}\Gamma^{t}_{r} = \Gamma^{t}_{T}$ by Lemma~\ref{lem:For Gamma}, we get the following equation of $(\Gamma_T^\cdot)^{-1}$:
\begin{equation}\label{e:Gamma-1}
I_{N} = (\Gamma^{t}_{T})^{-1} + \int_{t}^{T} (\alpha_{r})^{\top}(\Gamma^{r}_{T})^{-1}\eta(dr,X_r).
\end{equation}
Since $(\Gamma^{r}_{T})^{-1} = \Gamma^{0}_r (\Gamma^{0}_{T})^{-1}$ and $\Gamma^{0}_{\cdot}\in C^{\frac{1}{\tau}\text{-var}}([0,T];\mathbb{R}^{N\times N})$ a.s., we have $(\Gamma^{\cdot}_{T})^{-1}\in C^{\frac{1}{\tau}\text{-var}}([0,T];\mathbb{R}^{N\times N})$ a.s. Applying the estimate for nonlinear Young integrals (\cite[Proposition~2.1]{BSDEYoung-I}) to $\int_{\cdot}^{T}(\alpha_r)^{\top}(\Gamma^{r}_{T})^{-1}\eta(dr,X_r)$ on the right-hand side of \eqref{e:Gamma-1}, we can get
\begin{equation}\label{e:eq-Gamma-1}
\begin{aligned}
\|(\Gamma^{\cdot}_{T})^{-1}\|_{\frac{1}{\tau}\text{-var};[t,T]}\le C&|T-t|^{\tau}\|\eta\|_{\tau,\lambda;\beta}\left(1 + \|X\|_{p\text{-var};[0,T]} + |X_{0}|\right)^{\lambda+\beta}\\
&\times \left(\|\alpha\|_{p\text{-var};[0,T]} + |\alpha_{0}|\right)\left(1 + \|(\Gamma^{\cdot}_{T})^{-1}\|_{\frac{1}{\tau}\text{-var};[t,T]}\right),
\end{aligned}
\end{equation}
where $C$ is a positive constant depending on the parameters in \ref{(A2)}.  

Firstly we estimate $\|(\Gamma^{\cdot}_{T})^{-1}\|_{\frac{1}{\tau}\text{-var}}$ on small intervals. 
Denote the length of small intervals by
\begin{equation}\label{e:theta}
\theta:=\left(2C\|\eta\|_{\tau,\lambda;\beta}\left(1 + \|X\|_{p\text{-var};[0,T]} + |X_{0}|\right)^{\lambda+\beta}\left(\|\alpha\|_{p\text{-var};[0,T]} + |\alpha_0|\right)\right)^{-\frac{1}{\tau}}\wedge T.
\end{equation}
Then, we get by \eqref{e:eq-Gamma-1},  
\begin{equation*}
\|(\Gamma^{\cdot}_{T})^{-1}\|_{\frac{1}{\tau}\text{-var};[t,T]} \le 1\ \text{  and  }\ \|(\Gamma^{\cdot}_{T})^{-1}\|_{\infty;[t,T]} \le \bar C:= 1 + \sqrt{N}, ~ \text{ for } t\in [T-\theta, T],
\end{equation*}
where the second inequality is due to the fact  $\|(\Gamma^{\cdot}_{T})^{-1}\|_{\infty;[t,T]}\le \|(\Gamma^{\cdot}_{T})^{-1}\|_{\frac{1}{\tau}\text{-var};[t,T]} + |I_{N}|$. Similarly, we have, for every $t\in [T-(n+1)\theta, T-n\theta]$ (we stipulate $T-n\theta=0$ if $T-n\theta<0$), 
\begin{equation}\label{e:Gamma^-1}
\|(\Gamma^{\cdot}_{T-n\theta})^{-1}\|_{\frac{1}{\tau}\text{-var};[t,T-n\theta]} \le 1\ \text{  and  }\ \|(\Gamma^{\cdot}_{T-n\theta})^{-1}\|_{\infty;[t,T-n\theta]} \le  \bar C.
\end{equation}

We are ready to estimate $\|(\Gamma^{\cdot}_{T})^{-1}\|_{\frac{1}{\tau}\text{-var};[0,T]}$. 
By \eqref{e:Gamma^-1} and the fact that $(\Gamma^{a}_{c})^{-1} = (\Gamma^{a}_{b})^{-1}(\Gamma^{b}_{c})^{-1}$ for $a\le b\le c$, we have
\begin{equation}\label{e:Gamma^-1'}
\|(\Gamma^{\cdot}_{T})^{-1}\|_{\infty;[T-n\theta,T]} \le {\bar C}^{n}.
\end{equation}
Thus, 
\begin{equation*}
\begin{aligned}
&\|(\Gamma^{\cdot}_{T})^{-1}\|_{\frac{1}{\tau}\text{-var};[T-(n+1)\theta,T]} \\
&\le \|(\Gamma^{\cdot}_{T-n\theta })^{-1}(\Gamma^{T-n\theta}_{T})^{-1}\|_{\frac{1}{\tau}\text{-var};[T-(n+1)\theta,T-n\theta]} + \|(\Gamma^{\cdot}_{T})^{-1}\|_{\frac{1}{\tau}\text{-var};[T-n\theta,T]}\\
&\le  | (\Gamma^{T-n\theta}_{T})^{-1}|\cdot \|(\Gamma^{\cdot}_{T-n\theta })^{-1}\|_{\frac{1}{\tau}\text{-var};[T-(n+1)\theta,T-n\theta]}+\|(\Gamma^{\cdot}_{T})^{-1}\|_{\frac{1}{\tau}\text{-var};[T-n\theta,T]}\\
&\le \bar{C}^{n} + \|(\Gamma^{\cdot}_{T})^{-1}\|_{\frac{1}{\tau}\text{-var};[T-n\theta,T]},
\end{aligned}
\end{equation*}
where the last step follows from \eqref{e:Gamma^-1} and \eqref{e:Gamma^-1'}. 
Multiplying the above inequality  by $(\bar{C} - 1)$  recursively yields 
\[(\bar{C} - 1)\|(\Gamma^{\cdot}_{T})^{-1}\|_{\frac{1}{\tau}\text{-var};[T-(n+1)\theta,T]} - \bar{C}^{n + 1}\le (\bar{C} - 1)\|(\Gamma^{\cdot}_{T})^{-1}\|_{\frac{1}{\tau}\text{-var};[T-n\theta,T]} - \bar{C}^{n}\le \cdots \le -1,\]
and then 
\[\|(\Gamma^{\cdot}_{T})^{-1}\|_{\frac{1}{\tau}\text{-var};[0,T]} \le (\bar{C}^{\frac{T}{\theta}+1} - 1)/(\bar{C} - 1).\]
Recalling that $\theta$ is given in \eqref{e:theta},  we can find a constant $C'>0$ such that
\begin{equation}\label{e:Gamma 1/tau}
\|(\Gamma^{\cdot}_{T})^{-1}\|_{\frac{1}{\tau}\text{-var};[0,T]} \le C'\exp\Big\{C' \Big(1 + \|X\|^{\frac{\lambda+\beta}{\tau}}_{p\text{-var};[0,T]} + |X_{0}|^{\frac{\lambda+\beta}{\tau}}\Big)\Big(\|\alpha\|^{\frac{1}{\tau}}_{p\text{-var};[0,T]} + |\alpha_0|^{\frac{1}{\tau}}\Big)\Big\},
\end{equation}
and thus by Young's inequality we have for some constant $C''>0$,
\begin{equation*}
\begin{aligned}
&\mathbb{E}\left[\|(\Gamma^{\cdot}_{T})^{-1}\|^{\kappa}_{\frac{1}{\tau}\text{-var};[0,T]}\right]\\
&\qquad \le C'' \mathbb{E}\left[\exp\Big\{C''\Big(\|X\|^{\frac{q(\lambda+\beta)}{q\tau - 1}}_{p\text{-var};[0,T]} + |X_{0}|^{\frac{q(\lambda+\beta)}{q\tau - 1}}\Big) + \frac{1}{2}\left(\|\alpha\|^{q}_{p\text{-var};[0,T]} + |\alpha_0|^{q}\right)\Big\}\right],
\end{aligned}
\end{equation*}
which is finite due to the condition \eqref{e:X-exp} and the fact $\frac{q(\lambda+\beta)}{q\tau - 1} < \chi$.

Secondly, we show that $\mathbb{E}\left[\|\Gamma^{\cdot}_{T}\|^{\kappa}_{\frac{1}{\tau}\text{-var};[0,T]}\right]<\infty$. For any fixed $t\in[0,T]$, by \eqref{e:tilde Gamma, Gamma^-1}, we have
$$(\Gamma_T^t)^{-1}=I_N -\int_t^T (\Gamma_r^t)^{-1}(\alpha_r)^\top \eta(dr, X_r), ~ t\in[0,T].
$$
Left-multiplying both sides by $\Gamma_T^t$ and noting that $\Gamma_T^t=\Gamma_T^r\Gamma_r^t$ for $r\in[t,T]$, we get the equation which $\Gamma_T^\cdot$ satisfies
$$\Gamma_T^t=I_N+\int_t^T \Gamma_T^r (\alpha_r)^\top \eta(dr,X_r), ~ t\in[0,T].$$
Then, we can prove $\mathbb{E}\left[\|\Gamma^{\cdot}_{T}\|^{\kappa}_{\frac{1}{\tau}\text{-var};[0,T]}\right]<\infty$ in the same way as for $(\Gamma^{\cdot}_{T})^{-1}$.

{\bf Case $N=1$.} Note that $\Gamma^{t}_{T}$ and $(\Gamma^{t}_{T})^{-1}$ have explicit expressions respectively
\begin{equation*}
\Gamma^{t}_{T} = \exp\Big\{\int_{t}^{T}\alpha_{r}\eta(dr,X_r)\Big\}\ \text{ and }\  (\Gamma^{t}_{T})^{-1} = \exp\Big\{\int_{t}^{T}-\alpha_{r}\eta(dr,X_r)\Big\},
\end{equation*}
which enables us to relax the condition for $N=1$ (see condition (3) in Assumption \ref{(A2)}). 

By the mean value theorem, for any partition $\pi$ on $[0,T]$ we have
\begin{equation*}
\begin{aligned}					
&\sum_{[t_{i},t_{i+1}]\in\pi}\big|\Gamma_T^{t_i}-\Gamma_T^{t_{i+1}}\big|^{\frac{1}{\tau}}\le \sup_{s\in[0,T]}\exp\left\{\frac{1}{\tau}\int_{s}^{T}\alpha_{r}\eta(dr,X_{r})\right\} \sum_{[t_{i},t_{i+1}]\in\pi} \left|\int_{t_{i}}^{t_{i+1}}\alpha_{r}\eta(dr,X_{r})\right|^{\frac{1}{\tau}},
\end{aligned}
\end{equation*}
and hence
\begin{equation}\label{e:Gamma-kappa}
\begin{aligned}
&\mathbb{E}\left[\left\|\Gamma_T^{\cdot}\right\|^{\kappa}_{\frac{1}{\tau}\text{-var};[0,T]}\right] \\
& \le \mathbb{E}\left[\sup_{s\in[0,T]}\exp\Big\{\kappa\int_{s}^{T}\alpha_{r}\eta(dr,X_{r})\Big\} \Big\| \int_{\cdot}^{T}\alpha_{r}\eta(dr,X_{r})\Big\|^{\kappa}_{\frac{1}{\tau}\text{-var};[0,T]}\right]\\
& \le \left\{\mathbb{E}\bigg[\sup_{s\in[0,T]}\exp\Big\{2\kappa\int_{s}^{T}\alpha_{r}\eta(dr,X_{r})\Big\}\bigg]\right\}^{\frac{1}{2}} \cdot \Bigg\{\mathbb{E}\left[\Big\| \int_{\cdot}^{T}\alpha_{r}\eta(dr,X_{r})\Big\|^{2\kappa}_{\frac{1}{\tau}\text{-var};[0,T]}\right]\Bigg\}^{\frac{1}{2}}.
\end{aligned}
\end{equation}
By \cite[Proposition~2.1]{BSDEYoung-I}, there exists $C'>0$ such that
\[\Big\|\int_{\cdot}^{T}\alpha_{r}\eta(dr,X_r)\Big\|_{\frac{1}{\tau}\text{-var};[0,T]}\le C'\left(1  + \|X\|^{\lambda+\beta}_{p\text{-var};[0,T]} + |X_0|^{\lambda+\beta}\right)\left(\|\alpha\|_{p\text{-var};[0,T]} + |\alpha_0|\right).\]
Plugging this into \eqref{e:Gamma-kappa}, we have that
\begin{equation*}
\begin{aligned}
&\mathbb{E}\left[\left\|\Gamma_T^{\cdot}\right\|^{\kappa}_{\frac{1}{\tau}\text{-var};[0,T]}\right]\\ &\le \left\{\mathbb{E}\left[\exp\left\{2\kappa C' \left(1 + \|X\|^{\lambda+\beta}_{p\text{-var};[0,T]} + |X_{0}|^{\lambda+\beta}\right)\left(\|\alpha\|_{p\text{-var};[0,T]} + |\alpha_0|\right)\right\}\right]\right\}^{\frac{1}{2}}\\
&\quad \times (C')^{\kappa}\left\{\mathbb{E}\left[\left(1 + \|X\|^{\lambda+\beta}_{p\text{-var};[0,T]}+|X_0|^{\lambda+\beta}\right)^{2\kappa}\left(\|\alpha\|_{p\text{-var};[0,T]}+|\alpha_0|\right)^{2\kappa}\right]\right\}^{\frac{1}{2}}\\
&\lesssim_{\kappa,C'} \left\{\mathbb{E}\left[\exp\Big\{C''\left( \|X\|^{\frac{q(\lambda+\beta)}{q-1}}_{p\text{-var};[0,T]} + |X_0|^{\frac{q(\lambda+\beta)}{q-1}}\right)\Big\}\right]\right\}^{\frac{1}{4}}\left\{\mathbb{E}\left[\exp\left\{\|\alpha\|^{q}_{p\text{-var};[0,T]} + |\alpha_0|^{q}\right\}\right]\right\}^{\frac{1}{4}}\\
&\quad \times \left\{\mathbb{E}\left[\left(1 + \|X\|^{\lambda+\beta}_{p\text{-var};[0,T]} + |X_{0}|^{\lambda+\beta}\right)^{2\kappa}\left(\|\alpha\|_{p\text{-var};[0,T]} + |\alpha_0|\right)^{2\kappa}\right]\right\}^{\frac{1}{2}}
\end{aligned}
\end{equation*}
is finite, where the second inequality follows from Young's inequality and  H\"older's inequality, and the last one is due to Assumption~\ref{(A2)}. 

The finiteness of $\mathbb{E}\left[\|(\Gamma^{\cdot}_{T})^{-1}\|^{\kappa}_{\frac{1}{\tau}\text{-var};[0,T]}\right]$ can be proven in the same way. 
\end{proof}

Now we are ready to establish the well-posedness of the linear BSDE \eqref{e:BSDE-linear}. 

\begin{theorem}\label{thm:linear BSDE}
Assume \ref{(A2)} holds. Then, there exists a unique solution $(Y,Z)$ to BSDE~\eqref{e:BSDE-linear} in $\mathfrak{H}_{p,v}(0,T)$ where $v$ is a parameter appearing in \eqref{e:for the gamma}.
Moreover, 
\begin{equation}\label{e:in Theorem 5}
Y_{t} = \mathbb{E}_t\left[\left((\Gamma^{t}_{T})^{\top}\xi + \int_{t}^{T}\Gamma^{t}_{s} f_{s} ds \right) M_{T}\right]M^{-1}_{t}\ \text{a.s.} ~ \text{ for }~ t\in[0,T],
\end{equation}
where $M_{t}$ is defined in \eqref{e:exp martingale},
and $\Gamma^{t}_\cdot$ is the  unique solution to the following Young differential equation,
\begin{equation}\label{eq:Gamma}
\Gamma^{t}_s = I_{N} + \sum_{i=1}^{M}\int_{t}^{s}(\alpha^{i}_r)^{\top}\Gamma^{t}_{r}\eta_{i}(dr,X_r),\quad s\in[t,T],
\end{equation}
with $I_{N}$ denoting the $N\times N$ identity matrix.			
\end{theorem}

\begin{proof}[Proof of Theorem~\ref{thm:linear BSDE}]\makeatletter\def\@currentlabelname{the proof}\makeatother\label{proof:5}
We split our proof into four steps. In step 1, we  establish some moment estimations. In step 2, we  prove the well-posedness 
for the case $f,G\equiv 0$.  In step 3, we consider the case $G\equiv 0$,  and  in the step 4, we complete our proof for the general case.

\textbf{Step 1.} Recall that $\Gamma^{t}_{\cdot}$ solves Eq.~\eqref{e:gamma'}, and $k>1$ is a constant such that $\|\xi\|_{L^{k}(\Omega)}<\infty$. In this step, we shall prove the following estimates: for each $ k'\in(1,k)$,
\begin{equation}\label{e:exp-int3'}
\mathbb{E}\left[\left\|\mathbb{E}_{\cdot}\left[(\Gamma_T^0)^{\top}\xi\right]
\right\|^{k'}_{p\text{-var};[0,T]}\right]<\infty,
\end{equation}
\begin{equation}\label{e:exp-int4}
\mathbb{E}\left[\left\|(\Gamma_T^\cdot)^{\top}\xi \right\|_{p\text{-var};[0,T]}^{k'}\right]<\infty,
\end{equation}
and 
\begin{equation}\label{e:exp-int4'}
\mathbb{E}\left[\left\|\mathbb{E}_{\cdot}\left[(\Gamma_T^\cdot)^{\top}\xi\right]\right\|_{p\text{-var};[0,T]}^{k'}\right]<\infty.
\end{equation}

By the BDG inequality for $p$-variation and Doob’s maximal inequality, we have, for $1<k'<k$,
\begin{equation}\label{e:BDG for Gamma xi}
\E\left[\left\|\mathbb{E}_{\cdot}\left[(\Gamma_T^0)^{\top}\xi\right]\right\|^{k'}_{p\text{-var};[0,T]}\right]\lesssim_{p,k'}\E\left[\left|(\Gamma_T^0)^{\top}\xi\right|^{k'}\right].
\end{equation}
Then, by \eqref{e:BDG for Gamma xi}, H\"older's inequality, and Lemma~\ref{lem:estim of Gamma}, we have
\begin{equation*}
\begin{aligned}
&\mathbb{E}\left[\left\|\mathbb{E}_{\cdot}\left[(\Gamma_T^0)^{\top}\xi\right]\right\|^{k'}_{p\text{-var};[0,T]}\right]  \lesssim_{p,k'} \left\{\mathbb{E}\left[\left|\xi\right|^{k}\right]\right\}^{\frac{k'}{k}} \Big\{\mathbb{E}\Big[\left|(\Gamma_T^0)^{\top}\right|^{\frac{kk'}{k-k'}}\Big]\Big\}^{\frac{k-k'}{k}}< \infty,
\end{aligned}
\end{equation*}
which proves \eqref{e:exp-int3'}.

Noting $\frac{1}{\tau} < 2 < p$, by Lemma~\ref{lem:estim of Gamma}, all moments of $\|(\Gamma^{\cdot}_{T})^{\top}\|_{p\text{-var};[0,T]}$ are finite. In addition, since $\|\xi\|_{L^{k}(\Omega)}<\infty$ and $k>k'$, \eqref{e:exp-int4} follows directly from H\"older's inequality.

Regarding \eqref{e:exp-int4'}, noting that $\mathbb{E}_{t}\left[(\Gamma_T^t)^{\top}\xi\right] = ((\Gamma^{0}_{t})^{-1})^{\top}\mathbb{E}_{t}\left[(\Gamma_T^0)^{\top}\xi\right]$, we have, for $1<k'<k$, 
\begin{equation*}
\begin{aligned}
\mathbb{E}\left[\left\|\mathbb{E}_{\cdot}\left[(\Gamma_T^\cdot)^{\top} \xi\right]\right\|^{k'}_{p\text{-var};[0,T]}\right] 
& \lesssim_{k'} \mathbb{E}\left[\left\|((\Gamma^{0}_{\cdot})^{-1})^{\top}\right\|^{k'}_{p\text{-var};[0,T]}\left\|\mathbb{E}_{\cdot}\left[(\Gamma_T^0)^{\top}\xi\right]\right\|^{k'}_{\infty;[0,T]}\right]\\
& \quad +  \mathbb{E}\left[\left\|((\Gamma^{0}_{\cdot})^{-1})^{\top}\right\|^{k'}_{\infty;[0,T]}\left\|\mathbb{E}_{\cdot}\left[(\Gamma_T^0)^{\top}\xi\right]\right\|^{k'}_{p\text{-var};[0,T]}\right],
\end{aligned}
\end{equation*}
the right-hand side of which is finite due to  H\"older's inequality, \eqref{e:exp-int3'}, and Lemma~\ref{lem:estim of Gamma}.

\textbf{Step 2.} 
In this step, we will prove the existence and uniqueness of the solution to \eqref{e:BSDE-linear} when $f,G\equiv 0$. By \eqref{e:exp-int4}, for $1<k'<k$ we have (noting $\|B\|_{\infty;[0,T]} \le \|B\|_{p\text{-var};[0,T]} + |B_{T}|$ for $B\in C^{p\text{-var}}$),
\begin{equation}\label{e:to use tower-law1}
\mathbb{E}\left[\left\|(\Gamma_T^\cdot)^{\top}\xi\right\|_{\infty;[0,T]}^{k'}\right] \vee \mathbb{E}\left[\left\|(\Gamma_T^\cdot)^{\top}\xi\right\|_{p\text{-var};[0,T]}^{k'}\right]<\infty;
\end{equation}
and similarly by \eqref{e:exp-int4'},
\begin{equation}\label{e:to use tower-law2}
\mathbb{E}\left[\left\|\mathbb{E}_{\cdot}\left[(\Gamma_T^\cdot)^{\top}\xi\right]\right\|_{\infty;[0,T]}^{k'}\right] \vee \mathbb{E}\left[\left\|\mathbb{E}_{\cdot}\left[(\Gamma_T^\cdot)^{\top}\xi\right]\right\|_{p\text{-var};[0,T]}^{k'}\right]<\infty.
\end{equation}
In view of \eqref{e:X-exp}, \eqref{e:to use tower-law1}, and \eqref{e:to use tower-law2}, by Lemma~\ref{lem:tower-law 2} we get
\begin{equation}\label{e:fubini}
\mathbb{E}_{t}\left[\int_{t}^{T} \alpha_r(\Gamma_T^r)^{\top}\xi\eta(dr,X_{r})\right] = \mathbb{E}_{t}\left[\int_{t}^{T} \alpha_r\mathbb{E}_{r}\left[(\Gamma_T^r)^{\top} \xi\right]\eta(dr,X_{r})\right],
\end{equation}
and thus \eqref{e:BSDE-alternative} holds, i.e.,
\begin{equation}\label{e:BSDE-alternative'}
\E_t\left[(\Gamma^t_T)^{\top}\xi\right] = \E_t[\xi]+\mathbb{E}_{t}\left[\int_{t}^{T}\alpha_{r}\mathbb{E}_{r}\left[(\Gamma^r_T)^{\top}\xi\right]\eta(dr,X_{r})\right].
\end{equation}
By \cite[Proposition~2.1]{BSDEYoung-I}, we have, for $1 < k^{*} < k'< k,$
\begin{equation*}
\begin{aligned}
&\E\left[\Big|\int_0^T \alpha_t \E_t\left[(\Gamma_T^t)^{\top}\xi\right]\eta(dt,X_t)\Big|^{k^*}\right]\\
&\lesssim_{\tau,\lambda,\beta,p,k^*}\|\eta\|_{\tau,\lambda;\beta}^{k^*}T^{k^*\tau}\E\bigg[\left(1 + \left\|X\right\|^{k^*(\lambda+\beta)}_{p\text{-var};[0,T]} + |X_0|^{k^*(\lambda+\beta)} \right)\left(\|\alpha\|^{k^{*}}_{p\text{-var};[0,T]} + |\alpha_0|^{k^{*}}\right)\\
&\qquad\qquad\qquad\qquad\qquad\qquad\qquad\qquad\qquad\qquad\quad  \times \Big(\left\|\E_{\cdot}\left[(\Gamma_T^\cdot)^{\top}\xi\right]\right\|^{k^*}_{p\text{-var};[0,T]} + |\xi|^{k^*}\Big)\bigg]
\end{aligned}
\end{equation*}
is finite, which follows from H\"older's inequality, \eqref{e:exp-int4'}, \eqref{e:X-exp} and the fact $\|xy\|_{p\text{-var}}\le \|x\|_{p\text{-var}}\|y\|_{\infty} + \|y\|_{p\text{-var}}\|x\|_{\infty}$.
Thus, 
\begin{equation*}
\Big\|\xi+\int_0^T \alpha_{t}\E_t\left[ (\Gamma_T^t)^{\top}\xi\right]\eta(dt, X_t)\Big\|_{L^{k^*}(\Omega)}<\infty.
\end{equation*} 
Then, by the martingale representation theorem (as same as the proof in \cite[Lemma~3.1]{BSDEYoung-I}), there exists a unique progressively measurable process $Z$ such that 
\begin{equation}\label{e:Z k*}
\mathbb{E}\bigg[\Big|\int_{0}^{T}|Z_{r}|^{2}dr\Big|^{\frac{k^{*}}{2}}\bigg]<\infty,
\end{equation}
and for every $t\in[0,T]$,
\begin{equation}\label{e:MRT}
\xi+\int_t^T \alpha_{s}\E_s\left[(\Gamma_T^s)^{\top}\xi\right]\eta(ds, X_s) - \mathbb{E}_{t}\left[\xi+\int_t^T \alpha_{s} \E_s\left[(\Gamma_T^s)^{\top}\xi\right]\eta(ds, X_s)\right]=\int_t^T Z_s dW_s.
\end{equation}
By \eqref{e:BSDE-alternative'} and \eqref{e:MRT}, we obtain that $Y_t := \E_t[(\Gamma_T^t)^{\top}\xi]$ together with $Z_t$ satisfies BSDE~\eqref{e:BSDE-linear} with $f,G\equiv 0$. In addition, by \eqref{e:exp-int4'} and \eqref{e:Z k*}, we have $(Y,Z)\in\mathfrak{H}_{p,k^{*}}(0,T)$.

The uniqueness of the solution can be shown by It\^o's formula.  Indeed, if $(Y,Z)\in\mathfrak{H}_{p,k^{*}}(0,T)$ is a solution of Eq.~\eqref{e:BSDE-linear}, applying the product rule (\cite[Corollary~2.1]{BSDEYoung-I}) to $(\Gamma_s^t)^{\top}Y_s $ on $[t,T]$ and then taking conditional expectation, we have 
\[Y_{t} = \mathbb{E}_{t}\left[(\Gamma_{T}^t)^{\top}\xi\right] + \mathbb{E}_t\left[\int_{t}^{T}(\Gamma^{t}_{r})^{\top} Z_r  dW_r\right].\]
While by Lemma~\ref{lem:estim of Gamma} (noting $\Gamma^{t}_{r} = (\Gamma^{r}_{T})^{-1}\Gamma^{t}_{T}$) and $\mathbb{E}\big[\big|\int_{0}^{T}|Z_r|^{2}dr\big|^{\frac{k^*}{2}}\big]<\infty$, we have
\begin{equation*}
\mathbb{E}\left[\Big|\int_{t}^{T}|(\Gamma^{t}_{r})^{\top}Z_r |^2 dr\Big|^{\frac{1}{2}}\right]<\infty.
\end{equation*}
Thus, we have $\mathbb{E}_t\big[\int_{t}^{T}(\Gamma^{t}_{r})^{\top} Z_r  dW_r\big]=0$, and then $Y_{t} = \mathbb{E}_{t}[(\Gamma_{T}^t)^{\top}\xi]$.

\textbf{ Step 3.}
In this step, we will prove the well-posedness assuming $G\equiv 0$.

Recall that $\tilde{\Gamma}^{t}_{T} = ((\Gamma^{t}_{T})^{\top})^{-1}$ satisfies \eqref{e:tilde Gamma, Gamma^-1}.
Multiplying \eqref{e:tilde Gamma, Gamma^-1} by $(\Gamma^{t}_{T})^{\top}(\xi + \int_{t}^{T}\tilde{\Gamma}^{s}_{T}f_{s}ds),$ and using the equality $(\Gamma^{t}_{T})^{\top}\tilde{\Gamma}^{s}_{T} = (\Gamma^{t}_{s})^{\top}$ and the equality $\int_{t}^{T}\alpha_{s}\tilde{\Gamma}^{t}_{s}\eta(ds,X_{s}) = I_{N} - \tilde{\Gamma}^{t}_{T}$, we have
\begin{equation}\label{e:xi + int_t^T tilde Gamma...}
\begin{aligned}
\xi + \int_{t}^{T}\tilde{\Gamma}^{s}_{T}f_{s}ds &= (\Gamma^{t}_{T})^{\top}\xi + \int_{t}^{T}(\Gamma^{t}_{s})^{\top}f_{s} ds - \int_{t}^{T}\alpha_{s}(\Gamma^{s}_{T})^{\top}\xi \eta(ds,X_{s}) \\
&\quad + \left(I_{N} - (\Gamma^{t}_{T})^{\top}\right)\int_{t}^{T}\tilde{\Gamma}^{s}_{T}f_{s}ds.
\end{aligned}
\end{equation}
In addition, for $s\in[t,T],$ since $(\Gamma^{s}_{T})^{\top} =  \tilde{\Gamma}^{t}_{s}(\Gamma^{t}_{T})^{\top}$ we have $d\left(I_{N} - (\Gamma^{s}_{T})^{\top}\right) = - d\tilde{\Gamma}^{t}_{s} (\Gamma^{t}_{T})^{\top}.$ Then in view of \eqref{e:tilde Gamma, Gamma^-1} and integration by parts formula, we have
\begin{equation*}
\begin{aligned}
\left(I_{N} - (\Gamma^{t}_{T})^{\top}\right)\int_{t}^{T}\tilde{\Gamma}^{s}_{T} f_{s} ds &= - \int_{t}^{T}\alpha_{s}\tilde{\Gamma}^{t}_{s}(\Gamma^{t}_{T})^{\top}\Big(\int_{s}^{T}\tilde{\Gamma}^{r}_{T}f_{r}dr\Big)\eta(ds,X_{s}) + \int_{t}^{T}\tilde{\Gamma}^{s}_{T}f_{s}ds - \int_{t}^{T}f_{s}ds\\
& = - \int_{t}^{T}\alpha_{s}\Big(\int_{s}^{T}(\Gamma^{s}_{r})^{\top}f_{r}dr\Big)\eta(ds,X_{s}) + \int_{t}^{T}\tilde{\Gamma}^{s}_{T}f_{s}ds - \int_{t}^{T}f_{s}ds.
\end{aligned}
\end{equation*}
By the above equality and \eqref{e:xi + int_t^T tilde Gamma...}, we have
\begin{equation}\label{e:Gamma xi} 
(\Gamma^{t}_{T})^{\top}\xi + \int_{t}^{T}(\Gamma^{t}_{s})^{\top} f_{s}ds = \xi + \int_{t}^{T}\alpha_{s}\Big( (\Gamma^{s}_{T})^{\top}\xi + \int_{s}^{T}(\Gamma^{s}_{r})^{\top}f_{r}dr\Big)\eta(ds,X_{s}) + \int_{t}^{T}f_{s}ds.
\end{equation}
In view of \eqref{e:Gamma xi}, if we have the following tower rule:
\begin{equation}\label{e:tower rule for f}
\mathbb{E}_{t}\left[\int_{t}^{T}\alpha_{s} \Big(\int_{s}^{T}(\Gamma^{s}_{r})^{\top}f_{r}dr\Big)\eta(ds,X_{s})\right] = \mathbb{E}_{t}\left[\int_{t}^{T}\alpha_{s} \mathbb{E}_{s}\Big[\int_{s}^{T}(\Gamma^{s}_{r})^{\top}f_{r}dr\Big]\eta(ds,X_{s})\right],
\end{equation}
then $Y_{t} := \E_{t}\left[ (\Gamma^{t}_{T})^{\top}\xi + \int_{t}^{T}(\Gamma^{t}_{r})^{\top}f_{r}dr\right]$
satisfies
\begin{equation*}
Y_{t} = \mathbb{E}_{t}\left[\xi + \int_{t}^{T}\alpha_{r}Y_{r} \eta(dr,X_{r}) + \int_{t}^{T}f_{r}dr\right],\quad t\in[0,T],
\end{equation*}
and by the same procedure as in step 2, for each $k^*\in(1,k)$, 
there exists a progressively measurable process $Z$ such that $(Y,Z)\in\mathfrak{H}_{p,k^*}(0,T)$, and this pair uniquely solves Eq.~\eqref{e:BSDE-linear} with $G\equiv 0$.

To complete the proof, it suffices to prove the tower rule \eqref{e:tower rule for f}. In order to apply Lemma~\ref{lem:tower-law 2}, we will establish estimates  for  the moments of the uniform norm (resp. $p$-variation norm) of $s\mapsto\int_{s}^{T}(\Gamma^{s}_{r})^{\top}f_{r}dr$ (resp. $s\mapsto\E_{s}[\int_{s}^{T}(\Gamma^{s}_{r})^{\top}f_{r}dr]$). 
Note that 
\begin{equation}\label{e:L^k norm for f}
\int_{s}^{T}(\Gamma^{s}_{r})^{\top} f_{r} dr = \tilde{\Gamma}^{0}_{s}\int_{s}^{T}(\Gamma^{0}_{r})^{\top} f_{r} dr, \quad s\in[0,T].
\end{equation}
We have 
\begin{equation*}
\begin{aligned}
\Big\|\int_{\cdot}^{T}(\Gamma^{\cdot}_{r})^{\top} f_{r} dr\Big\|_{p\text{-var};[0,T]} &\le \|\tilde{\Gamma}^{0}_{\cdot}\|_{p\text{-var};[0,T]} \int_{0}^{T}\left|(\Gamma^{0}_{r})^{\top} f_{r}\right| dr + \|\tilde{\Gamma}^{0}_{\cdot}\|_{\infty;[0,T]} \int_{0}^{T}\left|(\Gamma^{0}_{r})^{\top} f_{r}\right| dr\\
&\le\left(\|\tilde{\Gamma}^{0}_{\cdot}\|_{p\text{-var};[0,T]} + \|\tilde{\Gamma}^{0}_{\cdot}\|_{\infty;[0,T]}\right) \left\|\Gamma^{0}_{\cdot}\right\|_{\infty;[0,T]}\int_{0}^{T}|f_{r}|dr\\
&\lesssim \left(\|\tilde{\Gamma}^{0}_{\cdot}\|_{p\text{-var};[0,T]} + 1\right) \left(\|\Gamma^{0}_{\cdot}\|_{p\text{-var};[0,T]} + 1\right)\int_{0}^{T}|f_{r}|dr;
\end{aligned}
\end{equation*}
and hence by Lemma~\ref{lem:estim of Gamma} and the fact $\E[\big|\int_{0}^{T}|f_{r}|dr\big|^{k}]<\infty$, we have for $1<k''<k$,
\begin{equation}\label{e:for the drift1}
\mathbb{E}\Big[ \Big\|\int_{\cdot}^{T}(\Gamma^{\cdot}_{r})^{\top} f_{r} dr\Big\|^{k''}_{\infty;[0,T]} \Big]\le \mathbb{E}\Big[ \Big\|\int_{\cdot}^{T}(\Gamma^{\cdot}_{r})^{\top} f_{r} dr\Big\|^{k''}_{p\text{-var};[0,T]} \Big]<\infty.
\end{equation}
In addition, noting 
\begin{equation*}
\mathbb{E}_{s}\Big[\int_{s}^{T}(\Gamma^{s}_{r})^{\top}f_{r}dr\Big] = \tilde{\Gamma}^{0}_{s}\left(  \mathbb{E}_{s}\Big[ \int_{0}^{T} (\Gamma^{0}_{r})^{\top}f_{r} dr - \int_{0}^{s}(\Gamma^{0}_{r})^{\top}f_{r} dr \Big]\right), 
\end{equation*}
we have
\begin{equation}\label{e:decomposition of E_{s}[int_s^T]}
\begin{aligned}
&\Big\|\mathbb{E}_{\cdot}\Big[\int_{\cdot}^{T}(\Gamma^{\cdot}_{r})^{\top}f_{r}dr\Big]\Big\|_{p\text{-var};[0,T]} \\
&\le \|\tilde{\Gamma}^{0}_{\cdot}\|_{p\text{-var};[0,T]} \Big\|\E_{\cdot}\Big[\int_{0}^{T}(\Gamma^{0}_{r})^{\top} f_{r}dr\Big]\Big\|_{\infty;[0,T]} + \|\tilde{\Gamma}^{0}_{\cdot}\|_{\infty;[0,T]} \Big\|\E_{\cdot}\Big[\int_{0}^{T}(\Gamma^{0}_{r})^{\top} f_{r}dr\Big]\Big\|_{p\text{-var};[0,T]}\\
& \quad + \|\tilde{\Gamma}^{0}_{\cdot}\|_{p\text{-var};[0,T]} \int_{0}^{T}\left|(\Gamma^{0}_{r})^{\top} f_{r}\right|dr + \|\tilde{\Gamma}^{0}_{\cdot}\|_{\infty;[0,T]} \int_{0}^{T}\left|(\Gamma^{0}_{r})^{\top} f_{r}\right|dr .
\end{aligned}
\end{equation}
Thus, by \eqref{e:decomposition of E_{s}[int_s^T]}, H\"older's inequality, and the BDG inequality for $p$-variation, we have for $1 < k'' < k'< k$,
\begin{align}\nonumber
& \mathbb{E}\left[\Big\|\mathbb{E}_{\cdot}\Big[\int_{\cdot}^{T}(\Gamma^{\cdot}_{r})^{\top}f_{r}dr\Big]\Big\|^{k''}_{p\text{-var};[0,T]}\right]\\ \nonumber
&\lesssim_{k''} \mathbb{E}\left[\|\tilde{\Gamma}^{0}_{\cdot}\|^{k''}_{p\text{-var};[0,T]} \Big\|\E_{\cdot}\Big[\int_{0}^{T}(\Gamma^{0}_{r})^{\top} f_{r}dr\Big]\Big\|^{k''}_{\infty;[0,T]}\right] + \mathbb{E}\left[\|\tilde{\Gamma}^{0}_{\cdot}\|^{k''}_{\infty;[0,T]} \Big\|\E_{\cdot}\Big[\int_{0}^{T}(\Gamma^{0}_{r})^{\top} f_{r}dr\Big]\Big\|^{k''}_{p\text{-var};[0,T]}\right]\\ \label{e:Gamma f <= Gamma f}
&\quad + \mathbb{E}\left[\left(\|\tilde{\Gamma}^{0}_{\cdot}\|^{k''}_{p\text{-var};[0,T]} + \|\tilde{\Gamma}^{0}_{\cdot}\|^{k''}_{\infty;[0,T]}\right) \Big|\int_{0}^{T}(\Gamma^{0}_{r})^{\top} f_{r}dr\Big|^{k''}\right]\\ \nonumber
&\lesssim_{p,k'} \bigg(\left\{\mathbb{E}\Big[\|\tilde{\Gamma}^{0}_{\cdot}\|^{\frac{k''k'}{k' - k''}}_{p\text{-var};[0,T]}\Big]\right\}^{\frac{k' - k''}{k'}} + \left\{\mathbb{E}\Big[\|\tilde{\Gamma}^{0}_{\cdot}\|^{\frac{k'' k'}{k' - k''}}_{\infty;[0,T]}\Big]\right\}^{\frac{k' - k''}{k'}}\bigg) \left\{\mathbb{E}\Big[ \Big|\int_{0}^{T}(\Gamma^{0}_{r})^{\top} f_{r}dr\Big|^{k'}\Big]\right\}^{\frac{k''}{k'}}.
\end{align}
We claim that the right-hand side of the above inequality is finite. Indeed, by Lemma~\ref{lem:estim of Gamma}, we have 
\begin{equation}\label{e:tilde Gamma < infty}
\E\left[\|\tilde{\Gamma}^{0}_{\cdot}\|^{\frac{k''k'}{k' - k''}}_{\infty;[0,T]}\right] \lesssim 1 +  \E\left[\|\tilde{\Gamma}^{0}_{\cdot}\|^{\frac{k''k'}{k' - k''}}_{p\text{-var};[0,T]}\right]<\infty,
\end{equation}
and 
\begin{equation}\label{e:Gamma f < infty}
\begin{aligned}
\mathbb{E}\left[ \Big|\int_{0}^{T}(\Gamma^{0}_{r})^{\top} f_{r}dr\Big|^{k'}\right]&\le \E\left[ \left\|\Gamma^{0}_{\cdot}\right\|^{k'}_{\infty;[0,T]} \Big|\int_{0}^{T}|f_{r}| dr \Big|^{k'} \right]\\
&\le \left\{\E\Big[ \left\|\Gamma^{0}_{\cdot}\right\|^{\frac{kk'}{k-k'}}_{\infty;[0,T]}\Big]\right\}^{\frac{k-k'}{k}} \left\{\E\Big[\Big|\int_{0}^{T}\left|f_{r}\right| dr \Big|^{k} \Big]\right\}^{\frac{k'}{k}} < \infty.
\end{aligned}
\end{equation}
Combining \eqref{e:Gamma f <= Gamma f}, \eqref{e:tilde Gamma < infty}, and \eqref{e:Gamma f < infty}, we get 
\begin{equation}\label{e:for the drift2}
\mathbb{E}\left[\Big\|\mathbb{E}_{\cdot}\Big[\int_{\cdot}^{T}(\Gamma^{\cdot}_{r})^{\top}f_{r}dr\Big]\Big\|^{k''}_{\infty;[0,T]}\right]\le \mathbb{E}\left[\Big\|\mathbb{E}_{\cdot}\Big[\int_{\cdot}^{T}(\Gamma^{\cdot}_{r})^{\top}f_{r}dr\Big]\Big\|^{k''}_{p\text{-var};[0,T]}\right]<\infty.
\end{equation}

Now, the desired tower rule \eqref{e:tower rule for f} follows from  Lemma~\ref{lem:tower-law 2} in view of \eqref{e:X-exp}, \eqref{e:for the drift1}, and \eqref{e:for the drift2}.

\

\textbf{Step 4.}
In this last step, we will prove the well-posedness for  \eqref{e:BSDE-linear}.

Denote $\tilde{W}_{t} := W_t - \int_{0}^{t} G_{r} dr$. Then, by the Girsanov theorem (see, e.g., \cite[Theorem~3.5.1]{karatzas1991brownian}), $\tilde{W}$ is a $d$-dimensional Brownian motion on $(\Omega,(\mathcal{F}_{t})_{t\in[0,T]},\tilde{\mathbb{P}})$, where $\tilde{\mathbb{P}}$ is defined by
\begin{equation*}
\tilde{\mathbb{P}}(A) := \E\left[\mathbf{1}_{A}M_{T}\right],\quad A\in\mathcal{F}_{T},
\end{equation*}
with $M$ being defined in \eqref{e:exp martingale}. Hence, Eq.~\eqref{e:BSDE-linear} can be rewritten as 
\begin{equation}\label{e:new equation}
Y_{t} = \xi + \sum_{i=1}^{M}\int_{t}^{T}\alpha^{i}_{r}Y_{r}\eta(dr,X_{r}) + \int_{t}^{T}f_{r}dr - \int_{t}^{T}Z_{r}d\tilde{W}_{r},\quad t\in[0,T].
\end{equation}

Let   $\tilde{\mathbb{E}}$ be the expectation under probability $\tilde{\mathbb{P}}$ and   denote $\tilde{\mathbb{E}}_{t}[\cdot]:=\tilde{\E}[\cdot|\mathcal F_t]$. Denote, for $k^{*}>1$,
\begin{equation*}
\tilde{\mathfrak{H}}_{p,k^{*}}(0,T) := \left\{ (y,z)\in\mathfrak{B}([0,T]): \ \|(y,z)\|_{\tilde{\mathfrak{H}}_{p,k^{*}};[0,T]} < \infty \right\},
\end{equation*}
where $\|(y,z)\|_{\tilde{\mathfrak{H}}_{p,k^{*}};[0,T]}$ is defined by \eqref{e:for Hpk} with $\E[\cdot]$ replaced by $\tilde{\E}[\cdot]$. Then, noting that $\tilde{\E}\left[|\xi|^{k}\right]<\infty$ and $\tilde{\E}\left[\left|\int_{0}^{T}|f_{r}|dr\right|^{k}\right]<\infty$, by Step 3, we see that 
\begin{equation}\label{e:Y_s = tilde E}
Y_{t} := \tilde{\E}_{t}\left[(\Gamma^{t}_{T})^{\top}\xi + \int_{t}^{T}(\Gamma^{t}_{s})^{\top} f_{s} ds\right],\quad t\in[0,T],
\end{equation}
together with some process $Z$ 
uniquely solves \eqref{e:new equation} in the space $\bigcup_{1\le k^*<k}\tilde{\mathfrak{H}}_{p,k^*}(0,T)$ with 
$(Y,Z)\in \bigcap_{1\le k^*<k}\tilde{\mathfrak{H}}_{p,k^*}(0,T)$, and hence $(Y,Z)$ satisfies Eq.~\eqref{e:BSDE-linear}. Then, the Feynman-Kac representation \eqref{e:in Theorem 5} follows from applying Bayes's rule (see, e.g., \cite[Lemma~3.5.3]{karatzas1991brownian}) to  \eqref{e:Y_s = tilde E}.

To complete the proof, we still need to show that $(Y,Z)$ is the unique solution  in the space $\mathfrak{H}_{p,v}(0,T)$, where $v\in(1,k)$  is the parameter in the condition \eqref{e:for the gamma}.

We first prove that $(Y,Z)\in \mathfrak{H}_{p,v}(0,T)$, for which it suffices to verify that $\tilde{\mathfrak{H}}_{p,k'}(0,T) \subset \mathfrak{H}_{p,v}(0,T)$ for $k'\in (1,k)$ sufficiently close to $k$. 
By \eqref{e:for the gamma}, there exists $k'\in(v,k)$ such that 
\begin{equation}\label{e:ine for gamma}
\gamma > \frac{1}{2} \Big(\frac{\sqrt{v}}{\sqrt{k'} 
- \sqrt{v}}\Big)^{2}.
\end{equation}
For any random variable $\zeta\in\mathcal{F}_{T}$ with $\tilde{\E}[|\zeta|^{k'}]<\infty$, by H\"older's inequality we have
\begin{align}\nonumber
&\E\left[|\zeta|^{v}\right] \\ \nonumber
&= \E\bigg[|\zeta|^{v} \exp\Big\{\frac{v}{k'}\int_{0}^{T}(G_{r})^{\top}dW_{r} - \frac{v}{2k'} \int_{0}^{T} |G_{r}|^{2} dr - \frac{v}{k'}\int_{0}^{T} (G_{r})^{\top}dW_{r} + \frac{v}{2k'}\int_{0}^{T} |G_{r}|^{2} dr\Big\}\bigg]\\ \label{e:for the zeta^v}
&\le \left\{\tilde{\mathbb{E}}\left[|\zeta|^{k'}\right]\right\}^{\frac{v}{k'}} \left\{\E\left[ \exp\Big\{ - \frac{v}{k' - v} \int_{0}^{T}(G_{r})^{\top}dW_{r} + \frac{v}{2(k' - v)}\int_{0}^{T}|G_{r}|^{2}dr\Big\}\right]\right\}^{\frac{k' - v}{k'}}.
\end{align}
Applying H\"older's inequality to the second term on the right-hand side of \eqref{e:for the zeta^v}, we have for  $\theta > 1$,
\begin{equation}\label{e:for the Holder}
\begin{aligned}
&\left\{\E\left[ \exp\Big\{ - \frac{v}{k' - v}\int_{0}^{T}(G_{r})^{\top}dW_{r} + \frac{v}{2(k' - v)}\int_{0}^{T}|G_{r}|^{2}dr\Big\}\right]\right\}^{\frac{k' - v}{k'}}\\
&\le \left\{\E\left[\exp\Big\{ -\frac{\theta v}{k' - v}\int_{0}^{T}(G_{r})^{\top}dW_{r} - \frac{\theta^2 {v}^{2}}{2(k' - v)^{2}}\int_{0}^{T}|G_{r}|^{2}dr\Big\}\right]\right\}^{\frac{k' - v}{\theta k'}}\\
&\quad \cdot \left\{E\left[\exp\Big\{\Big(\frac{\theta^{2}{v}^{2}}{2(k' - v)^{2}(\theta-1) } + \frac{\theta v}{2(k'-v)(\theta-1)}\Big)\int_{0}^{T}|G_{r}|^{2}dr\Big\}\right]\right\}^{\frac{(\theta -1)(k' - v)}{\theta k'}}\\
&\le \left\{\E\left[\exp\Big\{\Big(\frac{\theta^{2}{v}^{2}}{2(k' - v)^{2}(\theta-1) } + \frac{\theta v}{2(k'-v)(\theta-1)}\Big)\int_{0}^{T}|G_{r}|^{2}dr\Big\}\right]\right\}^{\frac{(\theta-1)(k' - v)}{\theta k'}}.
\end{aligned}
\end{equation}
Letting $\theta = 1 + \sqrt{\frac{k'}{v}}$ in \eqref{e:for the Holder}, by \eqref{e:ine for gamma} and \eqref{e:G-exp} we have that 
\begin{align}\label{e:exp gamma}
&\E\left[ \exp\Big\{ - \frac{v}{k' - v}\int_{0}^{T}(G_{r})^{\top}dW_{r} + \frac{v}{2(k' - v)}\int_{0}^{T}|G_{r}|^{2}dr\Big\}\right]\\ \nonumber
&\le \left\{\E\bigg[\exp\Big\{\frac{1}{2}\left(\frac{\sqrt{v}}{\sqrt{k'} - \sqrt{v}}\right)^2 \int_{0}^{T}|G_{r}|^2 dr \Big\}\bigg]\right\}^ \frac{\sqrt{k'}}{\sqrt v+\sqrt{k'}}\le\left\{ \E\left[\exp\Big\{\gamma \int_{0}^{T}|G_{r}|^2 dr \Big\}\right]\right\}^\frac{\sqrt{k'}}{\sqrt v+\sqrt{k'}}
\end{align}
is finite. Then, combining \eqref{e:for the zeta^v} and \eqref{e:exp gamma}, we have $\E[|\zeta|^{v}]<\infty$,
which implies $\tilde{\mathfrak{H}}_{p,k'}(0,T)\subset \mathfrak{H}_{p,v}(0,T)$.

Now we  prove the uniqueness of the solution $(Y,Z)$ in the space $\mathfrak{H}_{p,v}(0,T)$. Let $(Y',Z')$ be another solution to Eq.~\eqref{e:BSDE-linear} in $\mathfrak{H}_{p,v}(0,T)$. We will show that $(Y',Z')\in\tilde{\mathfrak{H}}_{p,v'}(0,T)$ for some $v'\in (1,v)$, which implies that $(Y',Z')$ is also a solution to Eq.~\eqref{e:new equation} in $\tilde{\mathfrak{H}}_{p,v'}(0,T)$. Then the uniqueness of the solution to Eq.~\eqref{e:BSDE-linear} follows from the uniqueness of the solution to Eq.~\eqref{e:new equation}.

To show $(Y',Z')\in\tilde{\mathfrak{H}}_{p,v'}(0,T)$ for some $v'\in(1,v)$, it suffices to show
$\mathfrak{H}_{p,v}(0,T) \subset \tilde{\mathfrak{H}}_{p,v'}(0,T)$ for some $v'\in(1,v)$.  
By \eqref{e:for the gamma}, there exists $v' > 1$ such that
\begin{equation}\label{e:ine for gamma2}
\gamma > \frac{1}{2}\Big(\frac{\sqrt{v}}{\sqrt{v} - \sqrt{v'}}\Big)^{2}.
\end{equation}
Then, for  $\zeta\in\mathcal{F}_{T}$ with $\E[|\zeta|^{v}]<\infty$, by H\"older's inequality we have
\begin{equation}\label{e:for the zeta^v'}
\begin{aligned}
\tilde{\E}\left[|\zeta|^{v'}\right] &= \E\left[|\zeta|^{v'} \exp\Big\{\int_{0}^{T}(G_{r})^{\top}dW_{r} -  \frac{1}{2}\int_{0}^{T}  |G_{r}|^{2} dr\Big\}\right]\\
&\le \left\{\E\left[|\zeta|^{v}\right]\right\}^{\frac{v'}{v}} \left\{\E\left[ \exp\Big\{ \frac{v}{v - v'}\int_{0}^{T}(G_{r})^{\top}dW_{r} - \frac{v}{2(v - v')}\int_{0}^{T}|G_{r}|^{2}dr\Big\}\right]\right\}^{\frac{v - v'}{v}}.
\end{aligned}
\end{equation}
In addition, by H\"older's inequality, for $\theta' := 1 + \sqrt{\frac{v'}{v}}$,
\begin{equation}\label{e:for the Holder2}
\begin{aligned}
& \left\{\E\left[ \exp\Big\{ \frac{v}{v - v'}\int_{0}^{T}(G_{r})^{\top}dW_{r} - \frac{v}{2(v - v')}\int_{0}^{T}|G_{r}|^{2}dr\Big\}\right]\right\}^{\frac{v - v'}{v}}\\
&\le  \left\{\E\bigg[ \exp\Big\{ \frac{\theta' v}{v - v'}\int_{0}^{T}(G_{r})^{\top}dW_{r} - \frac{{\theta'}^2 v^2}{2(v - v')^2}\int_{0}^{T}|G_{r}|^{2}dr\Big\}\bigg]\right\}^{\frac{v - v'}{\theta' v}}\\
&\quad \cdot \left\{\E\bigg[ \exp\Big\{\Big(\frac{{\theta'}^{2} v^{2}}{2(\theta' - 1)(v - v')^2} - \frac{\theta' v}{2(\theta' - 1)(v - v')}\Big)\int_{0}^{T}|G_{r}|^{2}dr\Big\}\bigg]\right\}^{\frac{(\theta' - 1)(v - v')}{\theta' v}}\\
&\le \left\{\E\bigg[ \exp\Big\{ \frac{1}{2}\Big(\frac{\sqrt{v}}{\sqrt{v} - \sqrt{v'}}\Big)^{2} \int_{0}^{T}|G_{r}|^{2}dr\Big\}\bigg]\right\}^{\frac{(\theta' - 1)(v - v')}{\theta' v}}.
\end{aligned}
\end{equation}
Combining \eqref{e:G-exp}, \eqref{e:ine for gamma2}, \eqref{e:for the zeta^v'}, and \eqref{e:for the Holder2}, we have $\tilde{\E}[|\zeta|^{v'}]<\infty$,
and hence $\mathfrak{H}_{p,v}(0,T)\subset\tilde{\mathfrak{H}}_{p,v'}(0,T)$. 
\end{proof}

\begin{remark}\label{rem:N=1 vs. N>1}
In $(3)$ of Assumption~\ref{(A2)}, a stronger condition is imposed for $N>1$ compared to $N=1$. This is because the solution $\Gamma^{t}_{\cdot}$ of  Eq.~\eqref{e:gamma'} admits an explicit formula in the one dimension case.
More precisely, denoting $dx^{i}_{r}:=\eta_{i}(dr,X_r)$, we have 
$
\Gamma^{t}_{s} = I_{N} + \sum_{i=1}^{M}\int_{t}^{s}(\alpha^{i}_{r})^{\top}\Gamma^{t}_{r} dx^{i}_{r} \quad \text{for } s\in[t,T].
$
When $N=1$, It\^o's formula established in \cite[Proposition~2.2]{BSDEYoung-I} yields $\Gamma^{t}_{s} = \exp\left\{\sum_{i=1}^{M} \int_{t}^{s}\alpha^{i}_{r}dx^{i}_{r}\right\}$. Thus, the estimate for $\Gamma^{t}_{s}$ can be derived from the estimate for $\int_{t}^{s}\alpha^{i}_{r}dx^{i}_{r},$ which is bounded by  $C \|\alpha\|_{p\text{-}\mathrm{var}}\|x\|_{\frac{1}{\tau}\text{-}\mathrm{var}}$ for a generic constant $C$, in contrast to the bound in the multi-dimensional case $\log |\Gamma^{t}_{s}|\lesssim (\|\alpha\|_{p\text{-}\mathrm{var}}\|x\|_{\frac{1}{\tau}\text{-}\mathrm{var}})^{\frac{1}{\tau}}$ (see \eqref{e:Gamma 1/tau}).

\end{remark}

\begin{remark}\label{rem:H1-H2} 
A typical example of $X$ satisfying \eqref{e:X-exp} is the Brownian motion, in which case $\chi=2$. In addition, a function $\eta$ with $\|\eta\|_{\tau,\lambda;\beta}<\infty$ can be taken as  a realization of a fractional Brownian sheet (see \cite[Lemma~A.4]{BSDEYoung-I}).   

\end{remark}

\begin{remark}[Comparison with the result in \cite{HuLiMi}]\label{rem:comparison-HLM}
When $\eta$ is a realization of  $(1+d)$-dimensional fractional Brownian sheet $B(t,x)$ with Hurst parameters $(H_0,H,...,H)$, our requirement of the well-posedness for BSDE~\eqref{e:BSDE-linear} is weaker than that in \cite{HuLiMi}. More precisely, the result in \cite{HuLiMi} only applies to the case $N=1$ and $d H<1$ (see the condition (2.4) therein), while we allow $N>1$, and moreover, $dH<2$ is sufficient for our requirement $\lambda+\beta<2$ (for the case $X = W$, $N=1$, and $\chi = 2$) in \ref{(A2)} (see \cite[Lemma~A.4]{BSDEYoung-I}). 
This is due to the fact that we only consider a realization of B(t,x) without taking the randomization, which ignores the integrability with respect to the fractional Brownian sheet. 

\end{remark}

Now we provide a probabilistic representation for the solution of  the following Young PDE:
\begin{equation}\label{e:linear Young PDE Cauchy}
\partial_t u+L u+u \partial_t \eta=0, \quad u(T, x)=u_T(x),
\end{equation}
where $Lv(t,x) := \frac{1}{2} \text{tr}[\sigma\sigma^{\top}(x) \nabla^{2}v(t,x)] + b^{\top}(x)\nabla v(t,x)$.
Here, we assume strict uniform ellipticity of $\sigma$, i.e., there exists a positive constant $\nu$ such that
\begin{equation}\label{e:uniformly elliptic}
(\sigma\sigma^{\top}(x))_{i,j}a_i a_j \ge \nu |a|^2 , \ \text{ for all }a\in\mathbb{R}^d \text{ and } x\in\mathbb{R}^d.
\end{equation}

\begin{proposition}\label{prop:Feynman-Kac formula and linear BSDE}
Assume $N=1$ and $\chi = 2$. Suppose that $X^{t,x}$ is given by \eqref{e:forward system}, with $b\in C^{1}_{b}$, $\sigma\in C^{3}_{b}$, and $\sigma\sigma^{\top}$ satisfying the strict uniform ellipticity condition \eqref{e:uniformly elliptic}. Assume $u_{T}(\cdot)\in C^{\mathrm{Lip}}(\R^d;\R)$. Assume further that $(\tau,\lambda,\beta)\in(\frac{1}{2},1]\times (0,1] \times [0,\infty)$ satisfies $\beta + \lambda < 2$ and $\lambda + 4\tau > 4.$ Let $(Y^{t,x},Z^{t,x})$ be the unique solution to BSDE~\eqref{e:BSDE-linear} with $\alpha \equiv 1$, $f = G = 0$, $\xi = u_{T}(X^{t,x}_{T})$, and $X = X^{t,x}$.
Define the function $u$ by
\begin{equation}\label{e:u(t,x)}
u(t,x):= Y^{t,x}_{t}.
\end{equation}
Then, for any $\tau' \in (0,\frac{\lambda}{2} + \tau - 1)$, $u(\cdot,x)\in C^{\tau'\text{-}\mathrm{H\ddot ol}}([0,T];\R)$. Moreover, $u$
is a weak solution to \eqref{e:linear Young PDE Cauchy} in the following sense:
\begin{equation}\label{e:weak-sol-linear}
\begin{aligned}
\int_{\R^d}u(t,x)\varphi(x) dx &= \int_{\R^d}u_{T}(x)\varphi(x) dx + \int_{\R^d} \int_{t}^{T} u(s,x)L^{*}\varphi(x) ds dx\\
&\quad + \int_{\R^d}\int_{t}^{T} u(s,x) \varphi(x)\eta(ds,x) dx,\quad t\in[0,T],
\end{aligned}
\end{equation}
where $\varphi$ is any smooth function with compact support, $L^{*}$ is the adjoint  operator  of $L$, and the inner integral in the last term is a Young integral. Further more, we have
\begin{equation}\label{e:u(s,X_s) = Y_s}
u(s,X^{t,x}_{s}) = Y^{t,x}_{s},\quad (s,x)\in[t,T]\times\R^d.
\end{equation}
\end{proposition}

\begin{proof} 
Note that \ref{(A2)} holds for every $p\in(2,\frac{\lambda}{1-\tau}),$  $\rho>0$, $q>\frac{2}{2-(\lambda+\beta)}$, and $k>1$. By Theorem~\ref{thm:linear BSDE}, we have the following representation for $Y^{t,x}$:
\begin{equation}\label{e:Ystx}
Y^{t,x}_{s} = \mathbb{E}_{s}\left[u_{T}(X^{t,x}_{T})\exp\Big\{\int_{s}^{T}\eta(dr,X^{t,x}_{r})\Big\}\right],~ t\le s\le T.
\end{equation}
and hence $u(t,x) = \mathbb{E}\left[u_{T}(X^{t,x}_{T})\exp\left\{\int_{t}^{T}\eta(dr,X^{t,x}_{r})\right\}\right]$.
Then, for any $\tau'\in(0,\frac{\lambda}{2} + \tau -1)$, by \cite[Theorem~4.15]{HuLe}, we have $u(\cdot,x)\in C^{\tau'\text{-H\"ol}}(0,T)$, and that $u(t,x)$ is a weak solution to Eq.~\eqref{e:linear Young PDE Cauchy}.

By the Markov property of $X^{t,x}$ and the fact $X_r^{t,x}=X_r^{s, X_s^{t,x}}$ for $t\le s\le r\le T$, we have
\begin{equation*}
\begin{aligned}
\mathbb{E}\left[ u_{T}(X^{s,x}_{T})\exp\Big\{\int_s^T \eta(dr,X^{s,x}_{r})\Big\} \right]\Bigg|_{x = X^{t,x}_{s}} &= \mathbb{E}_{s}\left[ u_{T}(X^{s,X^{t,x}_{s}}_{T})\exp\Big\{\int_{s}^{T}\eta(dr,X^{s,X^{t,x}_{s}}_{r})\Big\} \right] \\
&= \mathbb{E}_{s}\left[ u_{T}(X^{t,x}_{T})\exp\Big\{\int_{s}^{T}\eta(dr,X^{t,x}_{r})\Big\}\right].
\end{aligned}
\end{equation*}
Combining \eqref{e:u(t,x)} and \eqref{e:Ystx}, we have
\begin{equation*}
\begin{aligned}
u(s,X^{t,x}_{s}) = \mathbb{E}_{s}\left[ u_{T}(X^{t,x}_{T})\exp\Big\{\int_{s}^{T}\eta(dr,X^{t,x}_{r})\Big\}\right] = Y^{t,x}_{s}.
\end{aligned}
\end{equation*}
This proves \eqref{e:u(s,X_s) = Y_s}.
\end{proof}

\begin{remark}\label{rem:compare to HuLe}
In Proposition~\ref{prop:Feynman-Kac formula and linear BSDE}, we assume $\lambda+\beta<2$ and $ \tau+\frac\lambda 4 > 1$ to validate the Feynman-Kac formula. In contrast, for the well-posedness of BSDE~\eqref{e:BSDE-linear} (see Theorem~\ref{thm:linear BSDE}), we only require $\lambda+\beta<2$ and $\tau + \frac{\lambda}{2} > 1$ (see Remark~\ref{rem:kappa = 2}). 
This is because the verification of weak solutions requires the temporal H\"older continuity of $u(t,x)=Y_t^{t,x}$ in order to make sense of the Young integral in \eqref{e:weak-sol-linear}, and this in turn requires more regularity of $\eta$ (see \cite[Proposition~2.11]{HuLe}).
\end{remark}



\section{Applications}\label{sec:applications}

\subsection{Nonlinear Feynman-Kac formula}\label{subsec:Nonlinear-FK} 
In this section, we study the following Young PDE on $[0,T]\times \mathbb{R}^{d}$: 
\begin{equation}\label{e:SPDE-app}
\begin{cases}
 \displaystyle -\partial_{t} u= \mathcal{L} u (t,x) + g_{i}(u)\partial_{t}\eta_{i}(t,x), \quad (t,x)\in(0,T)\times \R^d, \\
 \displaystyle u(T,x)=h(x), \quad  x\in \R^d,
\end{cases}
\end{equation}
where $\mathcal{L} u := \frac{1}{2}\text{tr}\left[\sigma\sigma^{\top} \nabla^{2}u\right] + b^{\top}(x)\nabla u +  f(t,x,u,\sigma^{\top}\nabla u)$.
Here, $g\in C^{2}_{b} (\mathbb{R};\mathbb{R}^{M}), \eta\in C^{\tau,\lambda;\beta}([0,T]\times \mathbb{R}^{d};\mathbb{R}^{M}),\sigma\in C(\R^d;\R^{d\times d})$, $b\in C(\R^d;\R^d)$, $f\in C([0,T]\times \R^d\times \R \times \R^d;\R)$, and $h\in C(\R^d;\R)$ 
via our nonlinear Young BSDE.  We also assume strict uniform ellipticity \eqref{e:uniformly elliptic} of $\sigma$.

To address the nonlinear Young driver, we shall study the Young PDE \eqref{e:SPDE-app} via approximation. We first consider a Cauchy-Dirichlet problem on a bounded domain with the noise $\eta$ smooth in time. Let $\{\eta^m\}_{n\in \mathbb N}$ with $\eta^m\in C^{\tau,\lambda;\beta}([0,T]\times\R^d;\R^M)$ be a sequence approximation of $\eta$ such that  $\eta^m$  is smooth in time and $\partial_{t}\eta^{m}(t,x)$ is continuous in $(t,x)$. Let  $\{D_{n}\}_{n\in\mathbb N}$ be a sequence of bounded open subsets of $\R^d$ with  smooth boundaries $\partial D_n$.  We refer to \cite[Appendix~C.1]{Evans} for the definition of smooth boundary (or $C^\infty$-boundary). For some fixed $n,m\in\mathbb N$,  consider the following PDE on $[0,T]\times \bar{D}_{n}$, with $\bar{D}_n$ being the closure of $D_{n}$,
\begin{equation}\label{e:SPDE-app0}
\begin{cases}
\displaystyle - \partial_{t} u(t,x) = \mathcal{L} u(t,x) +  g_{i}(u)\partial_{t}\eta^{m}_{i}(t,x), \ t\in(0,T),\quad x\in D_{n},\\
u(t,x)=h(x), \quad (t,x)\in \{T\}\times \bar{D}_{n}\cup (0,T)\times \partial D_{n}.
\end{cases}
\end{equation}

\begin{definition}\label{def:vis}
We call $u\in C([0,T]\times \bar{D}_n;\R)$ a viscosity solution of \eqref{e:SPDE-app0} if $u$ satisfies
\begin{equation*}
u(t,x) = h(x)\ \text{ for }(t,x)\in \{T\}\times \bar{D}_n\text{ or } (0,T)\times \partial D_n,
\end{equation*} 
and the following conditions hold for every $(t,x)\in (0,T)\times D_{n}:$
\begin{itemize}
\item[(i)]  
$
a + F(t,x,u(t,x),q,\Lambda) \ge 0\text { if }(a,q,\Lambda)\in \mathscr{P}^{2,+}u(t,x);
$
\item[(ii)]  
$
a + F(t,x,u(t,x),q,\Lambda) \le 0\text { if }(a,q,\Lambda)\in \mathscr{P}^{2,-}u(t,x).
$
\end{itemize}
Here, $F:[0,T]\times\R^{d}\times\R\times\R^d\times \mathcal{S}(d)\rightarrow \R$ is defined by ($\mathcal{S}(d)$ is the set of symmetric $d\times d$ matrices) 
\begin{equation}\label{e:F}
F(t,x,y,q,\Lambda) := \frac{1}{2}\text{tr}\left[\sigma(x)\sigma^{\top}(x) \Lambda\right] + b^{\top}(x)q + f(t,x,y,\sigma^{\top}(x)q) + \sum_{i=1}^{M}g_{i}(y)\partial_{t}\eta^{m}_{i}(t,x),
\end{equation}
and  $\mathscr{P}^{2,+}u(t,x)$ (resp. $\mathscr{P}^{2,-}u(t,x)$) is the set of $(a,q,\Lambda)\in \R\times \R^{d}\times \mathcal{S}(d)$ such that the following inequality holds as $(s,x')\rightarrow (t,x):$
\begin{equation*}
u(s,x') - u(t,x)\le (\text{resp. }\ge)\  a(s - t) + (x' - x)^{\top}q + \frac{1}{2} (x' - x)^{\top}\Lambda (x' - x) + o(|s - t| + |x' - x|^{2}).
\end{equation*}
\end{definition}

\begin{proposition}\label{prop:vis-sol}
Assuming the same conditions as in Theorem~\ref{thm:uniqueness of the unbounded BSDE} together with \eqref{e:uniformly elliptic}, Eq.~\eqref{e:SPDE-app0} has a unique viscosity solution.
\end{proposition}
\begin{proof}
The uniqueness can be obtained by a proper adaption of   the comparison theorem \cite[Theorem~8.2]{crandall1992user}. First, note that the function $F$ defined by Eq. \eqref{e:F} actually does \emph{not} fully meet the assumption of \cite[Theorem~8.2]{crandall1992user} (see Eq. (8.4) therein), since $F$ is not ``proper" noting that  both $f(t,x,y,z)$ and $g(y)\partial_{t}\eta^{m}(t,x)$ may not be monotonic  in $y$. However, since $D_n$ is bounded and so is $\{y=u(t,x), x\in D_n\}$ due to the continuity of $u$,  we have that $F$ is uniformly Lipschitz in $y$, and then we can follow the argument in \cite[Section~5.3]{BuckdahnZhang} to prove the comparison result for our equation. Also note that the condition (BC) in Eq. (8.4) in \cite{crandall1992user} can be removed if the viscosity subsolution and supersolution are assumed to be continuous (see also the comparison theorem \cite[Theorem~7.9]{crandall1992user} for elliptic equations), which is our situation.

Now, we prove the existence. For $n,m\in\mathbb{N}$, let $(Y^{n,m;t,x}_{s},Z^{n,m;t,x}_{s})$ be the unique solution  to the following BSDEs (see \cite{PardouxPeng1990} or \cite{zhang2017backward} for the well-posedness of BSDEs),
\begin{equation}\label{e:Y^n,m}
\begin{cases}
dY_{s} = - f(s,X^{t,x}_{s},Y_{s},Z_{s})ds -  g_{i}(Y_{s}) \eta^{m}_{i}(ds,X^{t,x}_{s}) - Z_{s}dW_{s},\quad s\in[0,T^{t,x}_{D_n}],\\
Y_{T^{t,x}_{D_n}} = h(X^{t,x}_{T^{t,x}_{D_{n}}}),
\end{cases}
\end{equation}
where $X^{t,x}$ is given by \eqref{e:forward system}, and 
\begin{equation}\label{e:T-ntx}
T^{t,x}_{D_n}:= T \land \inf\{s > t;\ X^{t,x}_{s}\notin\bar{D}_{n}\}.
\end{equation}

Then, the function 
\begin{equation}\label{e:unm}
u^{n,m}(t,x) := Y^{n,m;t,x}_{t}
\end{equation}
is continuous in $(t,x)\in[0,T]\times \bar{D}_n$ by Lemma~\ref{lem:u(t,x) is continuous}, 
and it is a viscosity solution of \eqref{e:SPDE-app0} by \cite[Theorem~4.6]{pardoux1998backward}, noting that  the set of regular points $\Lambda_n$ defined in \eqref{e:regular set} coincides with the closed set $\partial D_n$ by the condition \eqref{e:uniformly elliptic} and Remark~\ref{rem:exit time}.
Therefore,  $u = u^{n,m}$ is a viscosity solution of \eqref{e:SPDE-app0}. 
\end{proof}

Now we are ready to define the solution of the Young PDE~\eqref{e:SPDE-app}.
\begin{definition}\label{def:rough PDEs}
Assume $\eta\in C^{\tau,\lambda;\beta}([0,T]\times \R^d ;\R^{M})$. We call $u\in C([0,T]\times \mathbb{R}^{d})$ a solution of Eq.~\eqref{e:SPDE-app} if there exists a sequence of functions $\{\eta^{m}\}_{m\in\mathbb{N}}\subset C^{\tau,\lambda;\beta}([0,T]\times \R^d ;\R^{M})$ and a sequence of bounded open domains $\{D_{n}\}_{n\in\mathbb{N}}\subset\R^d$ such that 
\begin{itemize}
\item[(i)] $\eta^m$ is smooth in time, $\partial_{t}\eta^{m}(t,x)$ is continuous in $(t,x)$, and 
\[\eta^{m}\rightarrow\eta\text{ in }C^{\tau,\lambda;\beta}([0,T]\times \R^d ;\R^{M})\text{ as }m\rightarrow\infty;\]
\item[(ii)] the boundary $\partial D_n$ of $D_n$ is smooth, and $D_{n}\uparrow \mathbb{R}^d$ as $n\rightarrow \infty$;\\
\item[(iii)] for each $n$ and $m$, Eq.~\eqref{e:SPDE-app0} admits a unique viscosity solution $u^{n,m},$ and for every $ (t,x)\in [0,T]\times \mathbb{R}^{d}$ and every $ \varepsilon>0$, there exist $n_\varepsilon,m_\varepsilon\in\mathbb N$ such that $|u^{n,m}(t,x) - u(t,x)|\le \varepsilon$ for all $n\ge n_\varepsilon$ and $m\ge m_\varepsilon$.
\end{itemize}
We call $u$ a unique solution of Eq.~\eqref{e:SPDE-app} if $u$ is independent of the choice of $\left\{D_{n}\right\}_n$ and $\left\{\eta^{m}\right\}_m$.
\end{definition}

\begin{remark}\label{rem:reference of varying domains}
We explain Definition~\ref{def:rough PDEs} for Young PDEs, which involves two approximations: the approximation of the noise $\eta^m\to \eta$ and the approximation of the domain $D_n\to\R^d$. 

The  first approximation heuristically says that the solution of the PDE with an irregular driver $\eta(dt,x)$ can be defined as the limit of the solutions to the PDEs with  smooth drivers $\partial_{t}\eta^{m}(t,x)dt$. This idea was already used, for instance, to define the solution to BSDEs with rough drivers in \cite{DiehlFriz}. 

The second approximation reflects the intuition that 
the solution of the PDE on the whole space $\R^d$ can be obtained by taking a limit of a sequence of solutions to the PDEs on bounded domains. This idea has appeared, for instance, in \cite{candil2022localization} where    the convergence of solutions on bounded domains to the solution on the whole space $\R^d$ was investigated for  heat equation driven by Gaussian noise which is white in time and correlated in space. 
\end{remark}

We introduce some notations related to the approximations in Definition \eqref{def:rough PDEs}. Assume the same conditions as  in Theorem~\ref{thm:uniqueness of the unbounded BSDE}, and  let $D_n$ and $\eta^m$ be as in Definition \ref{def:rough PDEs}.  Denote  $$u^{\infty,m}(t,x):=Y^{\infty,m;t,x}_{t}\text{ and  } u^{n,\infty}(t,x):=Y^{n,\infty;t,x}_{t},$$ where, $(Y^{\infty,m;t,x},Z^{\infty,m;t,x})\in\bigcap_{q>1}\mathfrak{H}_{p,q}(0,T)$ is the unique solution to the following BSDE satisfying the condition \eqref{e:assumption of widetilde Y} (see Theorem~\ref{thm:existence of solution of unbounded BSDEs}  and Theorem~\ref{thm:uniqueness of the unbounded BSDE}),
\begin{equation}\label{e:Y^infty,m}
Y_{s} = h(X^{t,x}_{T}) + \int_{s}^{T}f(r,X^{t,x}_{r},Y_{r},Z_{r})dr + \sum_{i=1}^{M}\int_{s}^{T}g_{i}(Y_{r}) \eta^{m}_{i}(dr,X^{t,x}_{r}) - \int_{s}^{T}Z_{r}dW_{r},\quad s\in[t,T],
\end{equation}
and 
$(Y^{n,\infty;t,x},Z^{n,\infty;t,x})\in\bigcap_{q>1}\mathfrak{B}_{p,q}(0,T)$ is the unique solution to BSDE~\eqref{e:Y^n,m} with $\eta^m$ replaced by $\eta$ (the well-posedness can be seen in \cite[Proposition~4.1]{BSDEYoung-I}).
In addition, denote by $(Y^{\infty,\infty}, Z^{\infty,\infty})$ the unique solution to BSDE~\eqref{e:Y^n,m} where $\eta^m$ and $T^{t,x}{D_n}$ are replaced by $\eta$ and $T$, respectively, and which belongs to $\bigcap_{q>1}\mathfrak{H}_{p,q}(0,T)$ and satisfies \eqref{e:assumption of widetilde Y}. Denote  $$u^{\infty,\infty}(t,x):=Y^{\infty,\infty;t,x}_t.$$
We remark that similar to $Y^{n,m;t,x}_{t}$, $Y^{n,\infty;t,x}_{t}$ is a.s. a constant, and by Corollary~\ref{cor:u(t,x)}, $Y^{\infty,m;t,x}_{t}$ and $Y^{\infty,\infty;t,x}_t$ are a.s. constants. Hence, functions $u^{n,\infty},$ $u^{\infty,m},$ and $u^{\infty,\infty}$ are well defined.

We also introduce the following space for the Young driver $\eta$: for $(\tau,\lambda,\beta)\in(0,1]\times (0,1]\times [0,\infty)$,
\begin{equation*}
\begin{aligned}
&C^{\tau,\lambda;\beta}_{0}([0,T]\times\R^d;\R^{M}) \\
&:= \Big\{\eta\in C^{\tau,\lambda;\beta}([0,T]\times\R^d;\R^{M}): \text{ there exists }
\{\eta^{m}\}_{m\ge 1} \subset C^{\tau,\lambda;\beta} \text{ such that }\\
&\quad\qquad  \eta^{m}\text{ is smooth in time,  }\partial_{t}\eta^{m}(t,x)\text{ is continuous, and  }\lim_{m}\|\eta^{m} - \eta\|_{\tau,\lambda;\beta} = 0\Big\}.
\end{aligned}
\end{equation*}
Clearly, $C^{\tau,\lambda;\beta}_{0}([0,T]\times\R^d;\R^{M})$ is a subset of $C^{\tau,\lambda;\beta}([0,T]\times\R^d;\R^{M})$. Moreover, for every $\hat{\tau}\in(0,\tau)$, we have $C^{\tau,\lambda;\beta}([0,T]\times\R^{d};\R^{M})\subset C^{\hat\tau,\lambda;\beta}_{0}([0,T]\times\R^d;\R^{M})$ by approximation result proved in \cite[Lemma~A.3]{BSDEYoung-I}.

\

Now, we are ready to prove the existence and uniqueness of the solution to the Young PDE~\eqref{e:SPDE-app}. 

\begin{proposition}\label{prop:rough PDE}
Assume the same conditions as  in Theorem~\ref{thm:uniqueness of the unbounded BSDE}. In addition, assume that $f$ is jointly continuous in $(t,x,y,z)$, $\sigma$ satisfies condition \eqref{e:uniformly elliptic}, and $\eta\in C^{\tau,\lambda;\beta}_{0}([0,T]\times\R^d;\R^{M})$. Then, $u := u^{\infty,\infty}$ is the unique solution of the Young PDE~\eqref{e:SPDE-app} in the sense of Definition~\ref{def:rough PDEs}. Furthermore, the solution map $S: C^{\tau,\lambda;\beta}_{0}
\to C([0,T]\times \R^d)$ defined by $S (\eta) = u$ is continuous, where the topology of  $C^{\tau,\lambda;\beta}_{0}$ is induced by the norm $\|\cdot\|_{\tau,\lambda;\beta}$ and $C([0,T]\times \R^d)$ is equipped with  the locally uniform convergence topology.
\end{proposition}

\begin{proof}
Recall that, from the proof Proposition \ref{prop:vis-sol},  $u^{n,m}(t,x) := Y^{n,m;t,x}_{t}$ given by Eq. \eqref{e:unm}  is a unique viscosity solution of Eq.~\eqref{e:SPDE-app0}.
In order to prove $u^{\infty, \infty}$ solves \eqref{e:SPDE-app}  in the sense of  Definition~\ref{def:rough PDEs}, it suffices to show that $u^{n,m}(t,x)\to u^{\infty,\infty}(t,x)$ as $n,m\rightarrow \infty$.

We aim to establish an estimate on $|u^{n,m}(t,x) - u^{\infty,m}(t,x)|$ which is uniform in $m\in\mathbb N$ and $(t,x)\in[0,T]\times K$, where $K$ is a compact subset of $\R^d$. To this end, we shall apply Proposition~\ref{prop:independent of choice of stopping times}, and  we need to  verify that for all $m\ge 1$ and $(t,x)\in [0,T]\times K$, the family of equations ~\eqref{e:Y^infty,m} share a common parameter ${\Theta_2}$, where ${\Theta_2}$ is defined in the beginning of Section~\ref{sec:BSDE with unbounded X}. Firstly, the terminal time $T$ and the function $(f,g)$ in BSDE~\eqref{e:Y^n,m} do not depend on $m$ or $(t,x)$. Secondly, since $
\sup_{m\ge 1}\|\eta^m\|_{\tau,\lambda;\beta}<\infty$, 
BSDEs~\eqref{e:Y^n,m} share the same parameter $(\tau,\lambda,\beta,p,k,\varepsilon)$ in Assumption~\ref{(A1)}. Finally,  according to Example~\ref{ex:assump (T)}, $\Xi_{\cdot} := h(X^{t,x}_{\cdot})$ satisfies Assumption~\ref{(T)} with $k=2$ and the same constant $C_{1}$ as in \eqref{e:conditions of xi} for all $(t,x)\in [0,T]\times K$. 

To apply Proposition~\ref{prop:independent of choice of stopping times}, we still need to show that
there exists a sequence of numbers $\{p_{n}\}_{n\ge 1}$ with $p_{n} \downarrow 0$ such that
\[\mathbb{P}\left\{T^{t,x}_{D_{n}} \neq T\right\}\le p_{n},\text { for all }n\ge 1 \text{ and } (t,x)\in[0,T]\times K,\]  where $T_{D_n}^{t,x}$ is defined in \eqref{e:T-ntx}. Indeed, by the proof of Lemma~\ref{lem:existence of X}, 
there exists  positive constants $C$ and $C_K=C(K)$ such that, for $n$ sufficiently large satisfying $K\subset \left\{y\in \R^d;\  |y|\le \inf_{z\in \partial D_{n}}|z|\right\}$,
\[\sup_{(t,x)\in [0,T]\times K}\mathbb{P}\{T^{t,x}_{D_n} \neq T\}\le  \sup_{x\in K}\mathbb{P}\{T^{0,x}_{D_n} \neq T\} \le  C_K  \exp\{-C \inf_{y\in \partial D_n}|y|^2\} =: p_{n},\]
where the second inequality follows from the same argument used in the proof of Lemma~\ref{lem:existence of X} which leads to \eqref{e:moment-T-Tn}.

Now we can apply Proposition~\ref{prop:independent of choice of stopping times} to \eqref{e:Y^infty,m} and get that, for every $\delta>0$, there exists $n'\ge 1$, independent of $m\ge 1$ and $(t,x)\in[0,T]\times K$, such that for all $n\ge n'$,
\begin{equation*}
\left|Y^{n,m;t,x}_{t} - Y^{\infty,m;t,x}_{t}\right| + \left|Y^{n,\infty;t,x}_{t} - Y^{\infty,\infty;t,x}_{t}\right|\le \frac{\delta}{4}.
\end{equation*}
Hence, for every $\delta>0$, there exists an $n'\ge 1$ such that, for all $n\ge n'$ and $m\ge 1$,

\begin{equation}\label{e:u^n m - u^infty m}
|u^{n,m}(t,x) - u^{n',m}(t,x)| + |u^{n',\infty}(t,x) - u^{\infty,\infty}(t,x)|\le \frac{\delta}{2},\ \forall (t,x)\in[0,T]\times K.
\end{equation}

By the stability of solutions to BSDEs with different drivers, which is established in \cite[Corollary~3.1]{BSDEYoung-I}, for every fixed $n\ge 1$, we have $\|Y^{n,m;t,x}_{t} -  Y^{n,\infty;t,x}_{t}\|_{L^{\infty}(\Omega)}\rightarrow 0$ uniformly in $(t,x)\in[0,T]\times K$, as $m\rightarrow \infty$. 
Hence, we have
$u^{n,m}(t,x)\rightarrow u^{n,\infty}(t,x)$ uniformly in $(t,x)\in[0,T]\times K$ as $m\rightarrow \infty$. Then,  there exists $m'(n')\ge 1$ such that 
$|u^{n',m}(t,x) - u^{n',\infty}(t,x)|\le \frac{\delta}{2}$ for all $m\ge m'(n')$.
Hence, by \eqref{e:u^n m - u^infty m}, we have for all $n\ge n'$ and $m\ge m'(n')$, 
\begin{equation*}
\begin{aligned}
|u^{n,m}(t,x)-u^{\infty,\infty}(t,x)|&\le|u^{n,m}(t,x) - u^{n',m}(t,x)| + |u^{n',m}(t,x) - u^{n',\infty}(t,x)| \\
&\quad + |u^{n',\infty}(t,x) - u^{\infty,\infty}(t,x)|\le \delta.
\end{aligned}
\end{equation*}
This implies the  uniform convergence $u^{n,m}\to u^{\infty, \infty}(t,x)$ on $[0,T]\times K$, as $n,m\to\infty$. Noting the continuity of $u^{n,m},$ we also have $u^{\infty,\infty}\in C([0,T]\times\mathbb{R}^{d})$. Furthermore, since $\{\eta^m\}_{m}$ and $\{D_n\}_{n}$ are arbitrarily chosen, $u^{\infty,\infty}$ is the unique solution in the sense of Definition~\ref{def:rough PDEs}.

Finally, we prove the continuity of the solution map. Let $\{\eta^{(j)}\}_{j\ge 1}$ be a sequence of functions converging to $\eta$ in $C^{\tau,\lambda;\beta}_{0}([0,T]\times;\R^M)$, and 
denote by $u^{\infty,(j)}$ the solution to Eq.~\eqref{e:SPDE-app} with $\eta$ replaced by $\eta^{(j)}$. Similarly, let  $(Y^{n,(j);t,x}, Z^{n,(j);t,x})$ be  the unique solution to Eq.~\eqref{e:Y^n,m} with $\eta^m$ replaced by $\eta^{(j)}$, and denote $u^{n,(j)}(t,x) := Y^{n,(j);t,x}_{t}$ for $n\ge 1$. By the same argument leading to the inequality \eqref{e:u^n m - u^infty m}, for every $\delta>0$, there exists an $n'\ge 1$ independent of $j$ such that, for all $j\ge 1$, 
\begin{equation*}
|u^{n',(j)}(t,x) - u^{\infty,(j)}(t,x)| + |u^{n',\infty}(t,x) - u^{\infty,\infty}(t,x)|\le \frac{\delta}{2},\ \forall (t,x)\in[0,T]\times K.
\end{equation*}
In addition, by \cite[Corollary~3.1]{BSDEYoung-I} again, there exists $j'(n')$ such that for all $j\ge j'(n')$,
\begin{equation*}
|u^{n',(j)}(t,x) - u^{n',\infty}(t,x)|\le \frac{\delta}{2},\ \forall (t,x)\in[0,T]\times K.
\end{equation*}
Hence, for all $j\ge j'(n')$, we have
\begin{equation*}
|u^{\infty,(j)}(t,x) - u^{\infty,\infty}(t,x)|\le \delta,\ \forall (t,x)\in[0,T]\times K,
\end{equation*}
and this yields the continuity of  the solution map.
\end{proof}

The following PDE driven by fractional Brownian sheet (which is thus an SPDE)
\begin{equation}\label{e:SPDE-app2}
\begin{cases}
\displaystyle- \partial_{t} u(t,x) = \mathcal{L}u(t,x) + g(u)\partial_{t}B(t,x), (t,x)\in(0,T)\times \R^d,\\[10pt]
\displaystyle u(T,x)=h(x),\ x\in \R^d,
\end{cases}
\end{equation}
can be solved pathwisely as a Young PDE \eqref{e:SPDE-app}, provided that the sample paths of $B(t,x)$ lives in some $C^{\tau,\lambda;\beta}([0,T]\times\mathbb{R}^d)$. Indeed, by \cite[Lemma~A.4]{BSDEYoung-I}, the  paths of $B(t,x)$ belong to $C^{\tau,\lambda;\beta}([0,T]\times\mathbb{R}^d)$ a.s.,  where $(\tau,\lambda,\beta)$ essentially satisfies
$\tau<H_0,$ $\lambda<\min_{1\le j\le d}H_j,$ and $\beta> -\lambda + \sum_{j=1}^{d}H_j$. If we assume further $H_j\equiv H$ for $j=1,2,\dots,d$ for simplicity, then,  to ensure the existence of $\varepsilon$ and $p$ such that $(\tau,\lambda,\beta,\varepsilon,p)$ satisfies Assumption~\ref{(A1)}, it suffices to assume
\begin{equation}\label{e:condition of H}
H_{0} + \frac{H}{2} >1 \ \text{ and } d H < 2 H_0 - 1,
\end{equation}
as explained in Remark~\ref{rem:Hurst condition}. The shaded regions of $(H,H_0)$ satisfying \eqref{e:condition of H} are demonstrated in Figure~\ref{fig1} and Figure~\ref{fig2} for $1$-dimensional and $d$-dimensional respectively.

\begin{figure}[htbp]
\centering
\begin{minipage}{0.49\linewidth}
\caption{$1$-dim region of $(H,H_0)$}\label{fig1}
\centering
\includegraphics[width=7.5cm,height=7cm]{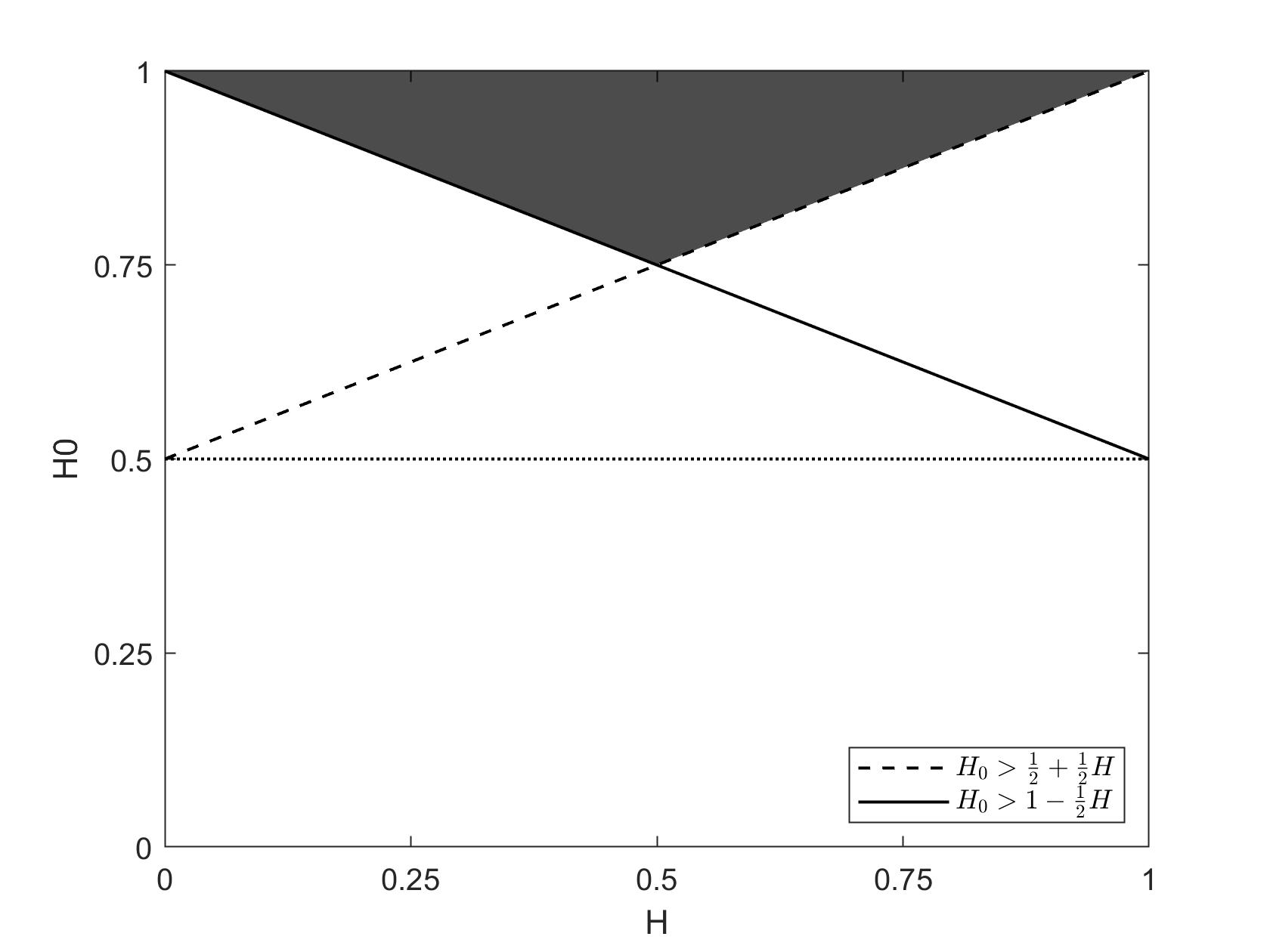}
\end{minipage}
\begin{minipage}{0.49\linewidth}
\caption{$d$-dim region of $(H,H_0)$}\label{fig2}
\centering  
\includegraphics[width=7.5cm,height=7cm]{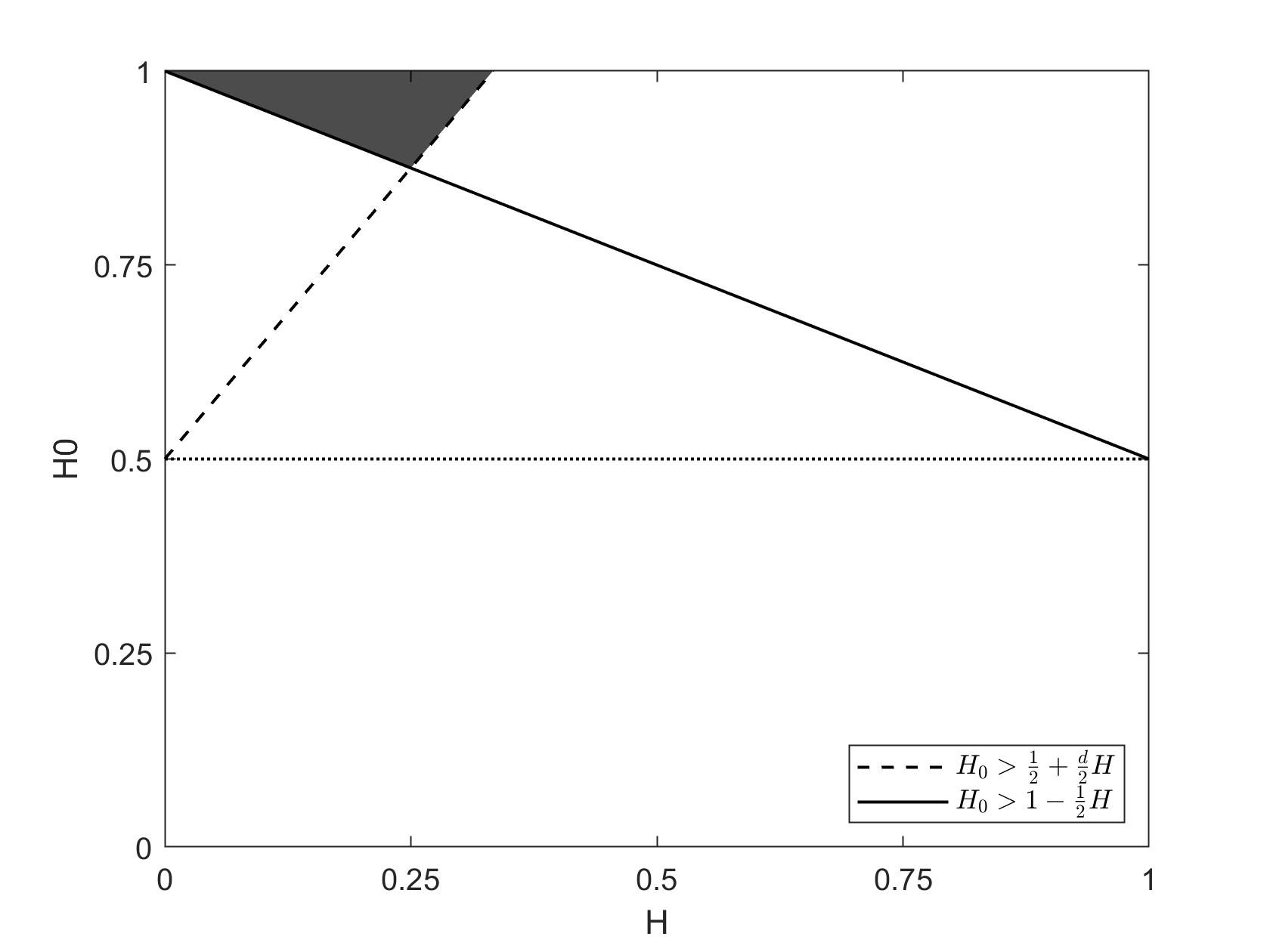}
\end{minipage}
\end{figure}

\begin{remark}
In the literature of SPDE, if an SPDE is solved pathwisely via deterministic PDEs, it is usually called  a \emph{Stratonovich} equation. In contrast, in the analysis of the so-called \emph{It\^o-Skorohod} equation, probabilistic instruments such as It\^o calculus and Malliavin calculus play a critical role. For related results, we refer to \cite{HuNualartSong-2011, hu2012feynman}, where the Feynman-Kac formulae for linear heat equations driven by fractional Brownian sheets were investigated. 
\end{remark}

\subsection{Localization error estimate for non-Lipschitz PDEs}\label{subsec:BSDEs with non-Lip}
In this section, we study the following BSDE with a non-Lipschitz coefficient $F:[0,T]\times\R^d\times \R^N \times\R^{N\times d}\rightarrow \R^N$: 
\begin{equation}\label{e:BSDE-example 1'}
Y_{t} = h(X_{T}) + \int_{t}^{T} F(r,X_{r},Y_{r},Z_{r}) dr - \int_{t}^{T}Z_{r}dW_r,\quad t\in[0,T].
\end{equation} 

\begin{remark}\label{rem:local-error}
Consider $F(t,x,y,z) = v(x)g(y)$, where $g\in C^{2}_{b}(\mathbb{R}^N;\R^N)$ and $v:\R^d\rightarrow \R$ satisfies 
\[|v(x_1)-v(x_2)|\lesssim (1+|x_1|+|x_2|)^{\beta}|x_1 - x_2|^{\lambda},\ \text{for all } x_1,x_2\in\mathbb{R}^{d},\]
for some $(\lambda,\beta)\in(0,1]\times[0,\infty)$ with $\lambda+\beta<1$.
The above equation can be viewed as a BSDE driven by a nonlinear Young integral. More precisely, suppose $X$ satisfies \ref{(A0)} and $h\in C^{\text{Lip}}(\mathbb{R}^d;\mathbb{R}^N)$.
Then, the well-posedness of Eq.~\eqref{e:BSDE-example 1'} follows directly from Theorem~\ref{thm:existence of solution of unbounded BSDEs} and Theorem~\ref{thm:uniqueness of the unbounded BSDE}, noting that $\eta(t,x):=v(x)t$ belongs to $C^{1,\lambda;\beta}([0,T]\times\mathbb{R}^{d})$. 
\end{remark}

In the following, we will apply a new localization method to solve stochastic Lipschitz BSDEs. This method will be applied to estimate localization errors for  PDEs with non-Lipschitz coefficients (see Proposition \ref{prop:localization error}).

\begin{proposition}\label{prop:non-Lip BSDE}
Assume $X$ satisfies Assumption~\ref{(A0)}. Let $0\le\theta_{1}<1$, $0\le \theta_{2}<2$, $\theta_{3}\ge 0$, $\mu\in(0,1]$,  and $\mu'\ge 0$. Assume $F = f_{0} + F_0$ and for all $ x,x_1,x_2\in\mathbb{R}^{d}, z,z_1,z_2\in\mathbb{R}^{N\times d}, y,y_1,y_2\in\mathbb{R}^{N},$ and $ t\in[0,T]$, the functions $h$, $f_0$, and $F_0$ satisfy
\begin{equation}\label{e:condi of F}
\begin{cases}
h\in C^{\mathrm{Lip}}(\mathbb{R}^{d};\mathbb{R}^N);\ |f_{0}(t,x,0,0)|\le C\left(1 + |x|^{\theta_{3}}\right);\\
|f_{0}(t,x,y_1,z_1) - f_{0}(t,x,y_2,z_2)|\le C\left(|y_1 - y_2| + |z_1 - z_2|\right);\\
|F_0(t,x,y,z)|\le C\left(1 + |x|^{\theta_{3}}\right);\\
|F_0(t,x,y_1,z) - F_0(t,x,y_2,z)|\le C(1+|x|^{\theta_{2}})|y_1 - y_2|;\\
|F_0(t,x,y,z_1) - F_0(t,x,y,z_2)|\le C(1+|x|^{\theta_{1}})|z_1 - z_2|;\\
|F(t,x_1,y,z) - F(t,x_2,y,z)|\le C |x_1 - x_2|^{\mu}(1 + |x_1|^{\mu'} + |x_{2}|^{\mu'});
\end{cases}
\end{equation}
where $C$ is a positive constant. Then, BSDE~\eqref{e:BSDE-example 1'} has a unique solution $(Y,Z)\in\mathfrak{H}^{2}(0,T)$,
where
\begin{equation}\label{e:def of H^q}
\mathfrak{H}^{q}(0,T):=\left\{(Y,Z)\in\mathfrak{B}([0,T]);\ \mathbb{E}\left[\|Y\|^{q}_{\infty;[0,T]}\right] + \mathbb{E}\Big[\Big|\int_{0}^{T}|Z_{r}|^{2}dr\Big|^{\frac{q}{2}}\Big]<\infty\right\}
\end{equation}
for $q\ge 1$.
\end{proposition}

\begin{proof} The proof follows the approach used in the proof of Theorem~\ref{thm:existence of solution of unbounded BSDEs}. We only consider the case when $N=1$ (the multi-dimensional case can be proved in the same manner as \nameref{proof:3'} of Theorem~\ref{thm:existence of solution of unbounded BSDEs}). We also assume $d=1$ without loss of generality.

Let $T_{n}$ be defined in \eqref{e:Tn}. Since $F(t,X_t,y,z)\mathbf{1}_{[0,T_n]}(t)$ is uniformly Lipschitz in $(y,z)$, the following BSDEs have a unique solution in $\mathfrak{H}^{2}(0,T)$, for $n\ge1$,
\begin{equation}\label{e:dYn}
\begin{cases}
dY^n_t = - F(t,X_t,Y^n_t,Z^n_t)dt + Z^n_t dW_t,\ t\in[0,T_n],\\[3pt]
 Y^n_{T_n} = h(X_{T_{n}}).
 \end{cases}
\end{equation}
For $t\in[0,T]$, denote
\begin{equation*}
\left\{\begin{aligned}
a^{n}_t&:= \frac{F(t,X_t,Y^{n+1}_t,Z^{n+1}_{t}) - F(t,X_t,Y^{n}_t,Z^{n+1}_{t})}{Y^{n+1}_t - Y^{n}_t}\mathbf{1}_{\{Y^{n+1}_{t}\neq Y^{n}_{t}\}}(t),\\ b^{n}_{t}&:= \frac{F(t,X_t,Y^{n}_t,Z^{n+1}_{t}) - F(t,X_t,Y^{n}_t,Z^{n}_{t})}{Z^{n+1}_t - Z^{n}_t}\mathbf{1}_{\{Z^{n+1}_{t}\neq Z^{n}_{t}\}}(t),
\end{aligned}\right.
\end{equation*}
and 
\[\Gamma^{n}_{t} := \mathbb{E}_{t}\left[\exp\left\{\int_{t\land T_n}^{T_{n}}2\left(a^{n}_{r} - \frac{1}{2}|b^{n}_{r}|^{2}\right)dr + \int_{t\land T_n}^{T_{n}} 2b^{n}_{r} dW_{r}\right\}\right].\]
By \eqref{e:condi of F}, it follows that $|b^{n}_{t}|^{2} + |a^{n}_{t}| \lesssim 1 + |X_t|^{(2\theta_1)\vee\theta_2}$. This estimate, combined with the boundedness of $|X_{\cdot\land T_n}|\le n$ when $n\ge |x|$, yields that there exists a constant $C^{\diamond}>0$ such that
\begin{equation}\label{e:esti of Gamma n}
\esssup_{t\in[0,T],\omega\in\Omega} |\Gamma^{n}_{t}|\lesssim \exp\left\{C^{\diamond}n^{(2\theta_{1})\vee \theta_{2}}\right\},\text{ for all }n\ge |x|.
\end{equation}

On the other hand, by the triangular inequality we have
\[\|Y^{n}_{T_n} - Y^{n+1}_{T_n}\|_{L^{2}(\Omega)} \le \|h(X_{T_{n}}) - h(X_{T_{n+1}})\|_{L^{2}(\Omega)} + \|Y^{n+1}_{T_n} - h(X_{T_{n+1}})\|_{L^{2}(\Omega)}.\]
For the first term of the right-hand side, we have $\|h(X_{T_{n}}) - h(X_{T_{n+1}})\|_{L^{2}(\Omega)} \lesssim \|X_{T_{n}} - X_{T_{n+1}}\|_{L^{2}(\Omega)}\lesssim \left\|T_{n+1} - T_{n}\right\|^{1/2}_{L^{1}(\Omega)}.$ For the second one, by \cite[Proposition~4.3]{BSDEYoung-I} with $S_0 = T_{n+1},$ $S_{1} = T_{n},$ $(Y,Z) = (Y^{n+1},Z^{n+1})$,  $\xi = h(X_{T_{n+1}}),$ $\eta(\cdot,\cdot) \equiv 0,$ $g(\cdot)\equiv 0,$ $p=3,$ $\tau = 1,$ $\lambda = \frac{1}{2},$ $\beta = 0,$ $\varepsilon = \frac{1}{2(\theta_3 \vee 1)},$ $f(t,y,z) = f_0(t,X_{t},y,z) + F_{0}(t,X_{t},Y^{n+1}_{t},Z^{n+1}_{t})$, and $q=2$, and by the inequality $\|h(X_{T_{n+1}})\|^2_{L^{4}(\Omega)}\lesssim 1 + |x|^2$ for all $n\ge |x|,$ we have
\begin{equation*}
\left\|Y^{n+1}_{T_n} - h(X_{T_{n+1}})\right\|^2_{L^{2}(\Omega)}\lesssim (1 + |x|^{2\theta_3 \vee 2})\left\{\mathbb{E}\left[|T_{n+1} - T_{n}|^{8}\right]\right\}^{\frac{1}{4}} + \mathbb{E}\left[\|\mathbb{E}_{\cdot}[h(X_{T_{n+1}})]\|^{2}_{3\text{-var};[T_{n},T_{n+1}]}\right].
\end{equation*}
In addition, by the BDG inequality for $p$-variation, Doob's maximal inequality, and the same calculation as in Example~\ref{ex:assump (T)}, we have
\begin{equation*}
\begin{aligned}
\mathbb{E}\left[\left\|\mathbb{E}_{\cdot}\left[h(X_{T_{n+1}})\right]\right\|^{2}_{3\text{-var};[T_n,T_{n+1}]}\right]
&\lesssim \mathbb{E}\left[\left|h(X_{T_{n+1}}) - h(X_{T})\right|^2\right] + \mathbb{E}\left[\left\|\mathbb{E}_{\cdot}\left[h(X_{T})\right]\right\|^{2}_{3\text{-var};[T_n,T]}\right]\\
&\lesssim \{\mathbb{E}\left[|T - T_{n+1}|\right]\}^{\frac{1}{2}}.
\end{aligned}
\end{equation*}
In addition, by Lemma~\ref{lem:existence of X}, there exists a constant $C^{\star}$ such that $\left\|T-T_{n}\right\|_{L^{1}(\Omega)}\lesssim \exp\left\{-C^{\star} (n-|x|)^{2}\right\}.$ Hence, there exists a constant $C'>0$ such that 
\begin{equation}\label{e:Y^n - Y^n+1 < -C' n^2}
\|Y^{n}_{T_n} - Y^{n+1}_{T_n}\|_{L^{2}(\Omega)}\lesssim (1 + |x|^{\theta_{3}\vee 1})\exp\left\{-C' (n - |x|)^{2}\right\},\text{ for all }n\ge |x|.
\end{equation}

By \eqref{e:esti of Gamma n}, \eqref{e:Y^n - Y^n+1 < -C' n^2}, and the same procedure leading to \eqref{e:dY<=h(Y)Gamma}, we have for all  $n\ge |x|,$
\begin{equation}\label{e:rate convergence}
\begin{aligned}
\mathbb{E}\left[\|Y^{n} - Y^{n+1}\|_{\infty;[0,T_n]}\right] &\lesssim \|Y^{n}_{T_n} - Y^{n+1}_{T_{n}}\|_{L^{2}(\Omega)}\bigg\{\esssup_{t\in[0,T],\omega\in\Omega}|\Gamma^{n}_{t}|\bigg\}^{\frac{1}{2}}\\
&\lesssim (1 + |x|^{\theta_{3}\vee 1}) \exp\left\{C^{\diamond}n^{(2\theta_{1}) \vee \theta_{2}} - \frac{C'}{2}n^{2} + 2C'|x|^2\right\}.
\end{aligned}
\end{equation}
In view of the  \eqref{e:rate convergence} and the fact that $(2\theta_1)\vee \theta_2<2$,  there exists a continuous  adapted  process $Y$ on $[0,T]$ such that
\begin{equation}\label{e:lim_n E[Y^n - Y]}
\lim_{n\rightarrow \infty}\mathbb{E}\left[\|Y^{n} - Y\|_{\infty;[0,T_{n}]}\right] = 0.
\end{equation}
Thus, by the same argument leading to \eqref{e:(Y^n,Z^n) -> (Y,Z)}, there exists a process $Z$ such that $(Y,Z)$ solves Eq.~\eqref{e:BSDE-example 1'} and $(Y,Z)\in \mathfrak{H}^{q}(0,T)$ for every $q>1$.

The uniqueness can be proved in a similar way to  the existence. Assume the pair  $(\tilde{Y},\tilde{Z})\in\mathfrak{H}^{2}(0,T)$ is a solution of \eqref{e:BSDE-example 1'}. Let
\begin{equation*}
\left\{\begin{aligned}
\tilde{a}^{n}_t&:= \frac{F(t,X_t,\tilde{Y}_t,\tilde{Z}_{t}) - F(t,X_t,Y^{n}_t,\tilde{Z}_{t})}{\tilde{Y}_t - Y^{n}_t} \mathbf{1}_{\{\tilde{Y}_{t}\neq Y^{n}_{t}\}}(t),\\ \tilde{b}^{n}_{t}&:= \frac{F(t,X_t,Y^{n}_t,\tilde{Z}_{t}) - F(t,X_t,Y^{n}_t,Z^{n}_{t})}{\tilde{Z}_t - Z^{n}_t} \mathbf{1}_{\{\tilde{Z}_{t}\neq Z^{n}_{t}\}}(t),
\end{aligned}\right.
\end{equation*}
and 
$\tilde{\Gamma}^{n}_{t} := \mathbb{E}_{t}\left[\exp\left\{\int_{t\land T_n}^{T_{n}}2\left(\tilde{a}^{n}_{r} - \frac{1}{2}|\tilde{b}^{n}_{r}|^{2}\right)dr + \int_{t\land T_n}^{T_{n}} 2\tilde{b}^{n}_{r} dW_{r}\right\}\right].$ From \eqref{e:condi of F} it follows that
\begin{equation*}
\log\Big(\esssup_{s\in[0,T],\omega\in\Omega} |\tilde{\Gamma}^{n}_{s}|\Big)\lesssim 1 + n^{(2\theta_{1})\vee \theta_{2}}.
\end{equation*}
Note that $(\tilde{Y},\tilde{Z})$ solves the following BSDE, 
\[\tilde{Y}_{t} = h(X_{T}) + \int_{t}^{T}\tilde{F}(r,\tilde{Y}_{r},\tilde{Z}_{r})dr - \int_{t}^{T}\tilde{Z}_{r}dW_{r},\quad t\in[0,T],\]
where $\tilde{F}(t,y,z) := f_0(t,X_{t},y,z) + F_{0}(t,X_{t},\tilde{Y}_{t},\tilde{Z}_{t})$ is uniformly Lipschitz in $(y,z)$. Also note that $\|\int_{0}^{T}|\tilde{F}(t,0,0)|^2dt\|_{L^{q/2}(\Omega)}<\infty$ for every $q>1.$ Then, by El Karoui et al.~\cite[Theorem~5.1]{ElKaroui1997}, $(\tilde{Y},\tilde{Z})\in\mathfrak{H}^{q}(0,T)$ for every $q>1.$ Thus, since $\tilde{Y}$ is a semimartingale, by the BDG inequality for $p$-variation, $(\tilde{Y},\tilde{Z})\in \mathfrak{H}_{3,q}(0,T)$ for every $q>1.$ Therefore, by \cite[Corollary~4.2]{BSDEYoung-I} with $S = T_n$, $(Y,Z) = (\tilde{Y},\tilde{Z}),$ $\eta(\cdot,\cdot) \equiv 0,$ $g(\cdot)\equiv 0,$ $p=3,$ $\tau = 1,$ $\lambda = \frac{1}{2},$ $\beta = 0,$ $\varepsilon = \frac{1}{2(\theta_3 \vee 1)},$ $f(t,y,z) = \tilde{F}(t,y,z),$ and $q=2$, we have
\begin{equation*}
\|\tilde{Y}_{T_n} - h(X_{T})\|^{2}_{L^{2}(\Omega)}\le \mathbb{E}\left[\|\tilde{Y}\|^{2}_{\infty;[T_n,T]}\right]\lesssim_{x} \left\{\mathbb{E}\left[|T - T_{n}|^{8}\right]\right\}^{\frac{1}{4}} + \mathbb{E}\left[\|\mathbb{E}_{\cdot}[h(X_{T})]\|^2_{3\text{-var};[T_{n},T]}\right].
\end{equation*}
Then, by repeating the procedure that leads to \eqref{e:lim_n E[Y^n - Y]}, we have $\lim_{n\rightarrow\infty}\mathbb{E}\left[\|Y^{n} - \tilde{Y}\|_{\infty;[0,T_n]}\right] = 0$.
Recall that $(Y,Z)$ is the solution satisfying \eqref{e:lim_n E[Y^n - Y]}. Then we have $\tilde{Y} = Y$. Furthermore, noting 
\begin{equation*}
\int_{0}^{t}\left(F(r,X_r,\tilde{Y}_{r},\tilde{Z}_{r})- F(r,X_r,Y_{r},Z_{r})\right)dr = \int_{0}^{t}\left(\tilde{Z}_{r} - Z_{r}\right)dW_{r},
\end{equation*}
it follows that $\tilde{Z} = Z$. 
\end{proof}

We now study the localization error for the following PDE,
\begin{equation}\label{e:PDE u}
\begin{cases}
-\partial_{t} u(t,x) = Lu(t,x) + F(t,x,u(t,x),(\sigma^{\top}\nabla u)(t,x)),\  t\in(0,T)\times\R^d,\\[3pt]
u(T,x)= h(x),
\end{cases}
\end{equation}
where $Lu(t,x):= \frac{1}{2}\text{tr}\left\{\sigma(x)\sigma^{\top}(x)\nabla^{2}u(t,x)\right\} + b^{\top}(x)\nabla u(t,x)$.
According to Li et al.~\cite[Theorem~3.17]{li2025random}, $u(t,x):=Y^{t,x}_{t}$, where $Y^{t,x}$ is the unique solution of Eq.~\eqref{e:BSDE-example 1'} on $[t,T]$ with $X$ replaced by $X^{t,x}$ ($X^{t,x}$ is given by \eqref{e:forward system}), is a viscosity solution to~\eqref{e:PDE u}.

Let
$D_{n} := \left\{x\in\R^d;|x|\le n\right\}$ and assume strict uniform ellipticity \eqref{e:uniformly elliptic} for $\sigma$. Note that $D_n$ is compact. By Proposition~\ref{prop:vis-sol}, the following Dirichlet problem admits a unique viscosity solution in $C([0,T]\times \bar{D}_{n}),$
\begin{equation}\label{e:PDE: u^n}
\begin{cases}
- \partial_{t} u^n(t,x) = Lu^n(t,x) + F(t,x,u^n(t,x),(\sigma^{\top}\nabla u^n)(t,x)), \ (t,x)\in (0,T)\times D_n,\\[3pt]
u^n(t,x) = h(x), \ (t,x)\in \{T\}\times \bar{D}_{n}\cup (0,T)\times \partial D_{n},
\end{cases}
\end{equation}
and $u^{n}(t,x) = Y^{n;t,x}_{t},$ where $(Y^{n;t,x},Z^{n;t,x})$ is the unique solution of Eq.~\eqref{e:dYn} on $[t,T]$ with $X$ replaced by $X^{t,x},$ and $T_n$ replaced by $T^{t,x}_{n}:=T\land \inf\left\{r>t;X^{t,x}_{r}\notin \bar{D}_{n}\right\}$. Then $u^n$ converges to $u$ as $n\to \infty$, and the convergence rate is provided in the proposition below.

\begin{proposition}\label{prop:localization error}
Assume that $(F,h)$ satisfies the assumptions of Proposition~\ref{prop:non-Lip BSDE}, $F$ is jointly continuous in $(t,x,y,z)$, $(b , \sigma)$ is independent of $t$ and satisfies Assumption~\ref{(A0)}. Assume further that \eqref{e:uniformly elliptic} holds. Let $u(t,x):=Y^{t,x}_t$ (resp. $u^{n}$) be the viscosity solution of Eq.~\eqref{e:PDE u} (resp. Eq.~\eqref{e:PDE: u^n}) mentioned above. Then there exist constants $C^{\star},C'>0$ such that for all $n\ge |x|,$
\begin{equation*}
|u^{n}(t,x) - u(t,x)|\le C^{\star}\exp\left\{-n^{2}/C^{\star} + C^{\star}|x|^2\right\}, \text{ and }\ 
 |u(t,x)|\le C'(1+|x|^{\theta_{3}\vee 1}).
\end{equation*}
\end{proposition}

\begin{proof}
By \eqref{e:rate convergence}, there exist constants $C_1,C_2>0$ such that for all $n\ge |x|,$
\begin{equation*}
|Y^{n;t,x}_{t} - Y^{t,x}_{t}|\le \sum_{m =n}^{\infty}|Y^{m;t,x}_{t} - Y^{m+1;t,x}_{t}| \lesssim (1+|x|^{\theta_3  \vee 1}) \sum_{m=n}^{\infty}\exp\left\{-C_1 m^{2} + C_2 |x|^2\right\}.
\end{equation*}
Thus, $|u^{n}(t,x) - u(t,x)| = |Y^{n;t,x}_{t} - Y^{t,x}_{t}| \lesssim (1+|x|^{\theta_3  \vee 1})\exp\left\{-C_{3} n^{2} + C_{2}|x|^2\right\}$ for some $C_{3}>0$. This proves the first estimate.  

By \cite[Corollary~4.1]{BSDEYoung-I} with $X = X^{t,x},$ $(Y,Z) = (Y^{t,x},Z^{t,x}),$ $\eta(\cdot,\cdot) \equiv 0,$ $g(\cdot)\equiv 0,$ $p=3,$ $\tau = 1,$ $\lambda = \frac{1}{2},$ $\beta = 0,$ $\varepsilon = \frac{1}{2(\theta_3 \vee 1)},$ $f(s,y,z) = f_{0}(s,X^{t,x}_{s},y,z) + F_{0}(s,X^{t,x}_{s},Y^{t,x}_{s},Z^{t,x}_{s}),$ and $q=2$,  it follows that $
|u(t,x)| = |Y^{t,x}_{t}|\lesssim 1 + |x|^{\theta_3 \vee 1}$, which proves the second estimate. \hfill 
\end{proof}

\appendix

\section{Some facts on processes with space-time regularity}\label{miscellaneous results}

The tower rule presented the following two lemmas are used in \nameref{proof:5} of Theorem~\ref{thm:linear BSDE}.

\begin{lemma}\label{lem:tower-law}
Suppose Assumption~\ref{(H0)} holds. Assume $\eta\in C^{\tau,\lambda}([0,T]\times\mathbb{R}^{d})$. Let $A:\Omega\times[t,T]\rightarrow\mathbb{R}$ be a continuous process with finite p-variation a.s., such that  $\mathbb{E}\left[|A_{r}|\right]<\infty$ for every $r\in[t,T]$.
Assume that the process $(\omega,r)\mapsto\mathbb{E}_{r}\left[A_{r}\right](\omega)$ has a continuous modification that has finite p-variation a.s.  Let $X:\Omega\times[t,T]\rightarrow\mathbb{R}^{d}$ and $B:\Omega\times[t,T]\rightarrow\mathbb{R}$ be two continuous adapted  processes with finite p-variation a.s. Assume 
\begin{equation}\label{e:con1}
\mathbb{E}\left[\|AB\|_{\infty;[t,T]}\left(1+\|X\|^{\lambda}_{p\text{-}\mathrm{var};[t,T]}\right) + \|AB\|_{p\text{-}\mathrm{var};[t,T]}\right]<\infty.
\end{equation}
In addition, let $C_r=\E_r[A_r]B_r$ and assume further
\begin{equation}\label{e:con2}
\mathbb{E}\left[\|C\|_{\infty;[t,T]}\left(1+\|X\|^{\lambda}_{p\text{-}\mathrm{var};[t,T]}\right) + \|C\|_{p\text{-}\mathrm{var};[t,T]}\right]<\infty.
\end{equation}
Then, we have 
\begin{equation}\label{e:before tower-law}
\mathbb{E}\left[\Big\|\int_{\cdot}^{T}A_{r}B_{r}\eta(dr,X_{r})\Big\|_{p\text{-}\mathrm{var};[t,T]}\right]<\infty.
\end{equation}
Moreover, the following tower rule holds,
\begin{equation}\label{e:tower-law}
\mathbb{E}_{t}\left[\int_{t}^{T}A_{r}B_{r}\eta(dr,X_{r})\right] = \mathbb{E}_{t}\left[\int_{t}^{T}\mathbb{E}_{r}\left[A_{r}\right]B_{r}\eta(dr,X_{r})\right].
\end{equation}
\end{lemma}
\begin{proof}
For $[u,v]\subset [t,T],$ by \cite[Proposition~2.1]{BSDEYoung-I} with $y_{r} = A_{r} B_{r}$ and $x_{r} = X_{r},$  we have 
\begin{equation*}
\begin{aligned}
&\left|\int^{v}_{u} A_{r} B_{r} \eta(dr,X_{r}) - A_{u} B_{u}\big(\eta(v,X_{u}) - \eta(u,X_{u})\big)\right|\\
&\lesssim_{\tau,\lambda,p} \|\eta\|_{\tau,\lambda}|u-v|^{\tau} \left(\|AB\|_{\infty;[u,v]}\|X\|^{\lambda}_{p\text{-var};[u,v]} + \|AB\|_{p\text{-var};[u,v]}\right).
\end{aligned}
\end{equation*}
Then, for a partition $\pi$ on $[t,T]$, by  the super-additivity of the right-hand side, we have,  
\begin{equation*}
\begin{aligned}
&\bigg|\int_{t}^{T}A_{r}B_{r}\eta(dr,X_{r})-\sum_{[t_{i},t_{i+1}]\in\pi}\big(A_{t_{i}}B_{t_{i}}\eta(t_{i+1},X_{t_i}) - A_{t_{i}}B_{t_{i}}\eta(t_{i},X_{t_i})\big)\bigg|\\
&\lesssim_{\tau,\lambda,p} \|\eta\|_{\tau,\lambda}|T-t|^{\tau}\left(\|AB\|_{\infty;[t,T]}\|X\|^{\lambda}_{p\text{-var};[t,T]}+\|AB\|_{p\text{-var};[t,T]}\right),
\end{aligned}
\end{equation*} the right-hand side of which is integrable by \eqref{e:con1}. 		Moreover, \cite[Proposition~2.1]{BSDEYoung-I} and \eqref{e:con1} yields 
\[\mathbb{E}\left[\left\|\int_{\cdot}^{T}A_{r}B_{r}\eta(dr,X_{r})\right\|_{p\text{-var};[t,T]}\right]<\infty.\] 
Hence, the set of random variables $\left\{\sum_{[t_{i},t_{i+1}]\in\pi}\big(A_{t_{i}}B_{t_{i}}\eta(t_{i+1},X_{t_i}) - A_{t_{i}}B_{t_{i}}\eta(t_{i},X_{t_i})\big)\right\}_{\pi\in \Pi}$, where $\Pi$ is the collection of all partitions on $[t, T]$,  is uniformly integrable. Therefore, 
\begin{equation*}
\begin{aligned}
\mathbb{E}_{t}\left[\int_{t}^{T}A_{r}B_{r}\eta(dr,X_{r})\right] &= \mathbb{E}_{t}\bigg[\lim_{|\pi|\downarrow 0}\sum_{[t_{i},t_{i+1}]\in\pi}A_{t_{i}}B_{t_{i}}\big(\eta(t_{i+1},X_{t_i}) - \eta(t_{i},X_{t_i})\big)\bigg]\\
&=\lim_{|\pi|\downarrow 0}\mathbb{E}_{t}\bigg[\sum_{[t_{i},t_{i+1}]\in\pi}A_{t_{i}}B_{t_{i}}\big(\eta(t_{i+1},X_{t_i}) - \eta(t_{i},X_{t_i})\big)\bigg].
\end{aligned}
\end{equation*}

Similarly, by replacing $A_r$ with $\E_r[A_r]$ and using \eqref{e:con2},  we can obtain that
\begin{equation*}
\mathbb{E}_{t}\left[\int_{t}^{T}\mathbb{E}_{r}\left[A_{r}\right]B_{r}\eta(dr,X_{r})\right] = \lim_{|\pi|\downarrow 0}\mathbb{E}_t\bigg[\sum_{[t_{i},t_{i+1}]\in\pi}\mathbb{E}_{t_{i}}\left[A_{t_{i}}\right]B_{t_{i}}\big(\eta(t_{i+1},X_{t_i}) - \eta(t_{i},X_{t_i})\big)\bigg].
\end{equation*}
The desired equation \eqref{e:tower-law} follows from the last two equations and the tower rule.		
\end{proof}

If we replace the condition $\|\eta\|_{\tau,\lambda}<\infty$ in Lemma~\ref{lem:tower-law} by $\|\eta\|_{\tau,\lambda;\beta}<\infty$, we have the following  result, of which the proof  is similar to that of Lemma~\ref{lem:tower-law} and hence is omitted.
\begin{lemma}\label{lem:tower-law 2}
Assume \ref{(H0)} holds and $\eta\in C^{\tau,\lambda;\beta}([0,T]\times\mathbb{R}^{d})$ for some $\beta\ge 0$. Let $A:\Omega\times[t,T]\rightarrow\mathbb{R}$ be a continuous process with finite p-variation a.s. such that   $\mathbb{E}\left[|A_{r}|\right]<\infty$ for every $ r\in[t,T]$.
Assume that the process $(\omega,r)\mapsto\mathbb{E}_{r}\left[A_{r}\right](\omega)$ has a continuous modification that has finite p-variation a.s.  Let $X:\Omega\times[t,T]\rightarrow\mathbb{R}^{d}$ and $B:\Omega\times[t,T]\rightarrow\mathbb{R}$ be two continuous adapted  processes with finite p-variation a.s. such that
$$\mathbb{E}\left[\left(1 + \|X\|_{\infty;[t,T]} + \|X\|_{p\text{-}\mathrm{var};[t,T]}\right)^{\lambda+\beta} \left(\|AB\|_{\infty;[t,T]}+\|AB\|_{p\text{-}\mathrm{var};[t,T]}\right)\right]<\infty.$$
In addition, let $C_r := \E_r[A_r]B_r$ and assume
$$\mathbb{E}\left[\left(1+\|X\|_{\infty;[t,T]} +\|X\|_{p\text{-}\mathrm{var};[t,T]}\right)^{\lambda+\beta}\left(\|C\|_{\infty;[t,T]}+ \|C\|_{p\text{-}\mathrm{var};[t,T]}\right)\right]<\infty.$$
Then \eqref{e:before tower-law} and \eqref{e:tower-law} hold.
\end{lemma}

\section{SDEs and BSDEs on bounded domains}\label{append:B}

Throughout this section, we assume that  the coefficient functions $b(\cdot)$ and $\sigma(\cdot)$ in Eq. \eqref{e:X_t = x + ...}  are independent of the time variable $t$ and satisfy Assumption~\ref{(A0)}.

Let $D$ be a bounded open subset of $\mathbb{R}^{d}$, with $\bar{D}$ being its closure and $\partial D$ its boundary. For a point $x\in \bar D$ and an $\mathcal F_t$-measurable random vector $\zeta\in \bar D$, denote by $X^{t,x}$ (resp. $X^{t,\zeta}$) the  unique solution to Eq.~\eqref{e:forward system} with the initial value $x$ (resp. $\zeta$).  Denote 
\begin{equation}\label{e:def of T^t,zeta}
T^{t,x}:=\inf\left\{s> t;\ X^{t,x}_s \notin \bar{D}\right\}\land T\ \text{ and }\ T^{t,\zeta}:=\inf\left\{s> t;\ X^{t,\zeta}_s \notin \bar{D}\right\}\land T.
\end{equation} 

\begin{lemma}\label{lem:T t,zeta}
Suppose  the set 
$\Lambda:=\left\{x\in\partial D;\mathbb{P}\left\{T^{0,x}>0\right\}=0\right\}$ of regular points is  closed. Let $\{\zeta_{m}\}_{m\ge1}$ be a sequence of $\mathcal F_t$-measurable random vectors taking values in $\bar D$ converging to $\zeta\in \bar D$ in probability as $m\to\infty$. Then,
$T^{t,\zeta_m}\rightarrow T^{t,\zeta}$ in probability as $m\to \infty$.
\end{lemma}
\begin{proof}
The proof is similar to that of \cite[Proposition~4.1]{pardoux1998backward}, hence is omitted.
\end{proof}

Consider the following BSDEs with random terminal times,
\begin{equation}\label{e:A12}
\left\{\begin{aligned}
dY_s^{t, x}=& - f\left(s, X_s^{t, x}, Y_s^{t, x}, Z_s^{t, x}\right) d s +  Z_s^{t, x} d W_s,\ s\in[t,T^{t,x}],\\
Y^{t,x}_{T^{t,x}} =& h\left(X_{T^{t,x}}^{t, x}\right) ; \\
d\mathscr{Y}_s^{t, \zeta}=& - f\left(s, X_s^{t, \zeta}, \mathscr{Y}_s^{t, \zeta}, \mathscr{Z}_s^{t, \zeta}\right) d s +  \mathscr{Z}_s^{t, \zeta} d W_s ,\ s\in[t,T^{t,\zeta}],\\
\mathscr{Y}_{T^{t,\zeta}}^{t, \zeta} = & h\left(X_{T^{t,\zeta}}^{t, \zeta}\right),
\end{aligned}\right.
\end{equation}
where $x\in \bar D$ is a point, $\zeta\in \bar D$ is an $\mathcal F_t$-measurable random vector, and  $T^{t,x}$ and $T^{t,\zeta}$ are defined in \eqref{e:def of T^t,zeta}. Assume $h\in C^{\text{Lip}}(\mathbb{R}^d;\mathbb{R}^{N})$ and  $f$ satisfies \eqref{e:assump of f} in \ref{(A1)} with $\lambda\in(0,1],$ $\beta\ge 0,\varepsilon\in(0,1)$. Then, by Theorem~4.2.1, Theorem~4.3.1, and Remark~4.3.2 in 
\cite{zhang2017backward},  each BSDE in Eq.~\eqref{e:A12} admits a unique solution in $\mathfrak{H}^{2}(t,T)$ (recall that $\mathfrak{H}^{2}(t,T)$ is defined in \eqref{e:def of H^q}). Below, we prove the Markov property for the process $Y^t$ in Eq.~\eqref{e:A12} (see also \cite[Theorem~5.1.3] {zhang2017backward} for the Markov property for BSDEs with deterministic terminal time). 

\begin{lemma}\label{lem:flow property}
Assume the set 
$\Lambda:=\left\{x\in\partial D;\ \mathbb{P}\left\{T^{0,x}>0\right\}=0\right\}$ of regular points is closed. 
Then, for every fixed $t\in[0,T]$ and $\mathcal{F}_t$-measurable random variable $\zeta\in\bar{D}$, we have for $s\in[t,T]$,
\begin{equation}\label{e:Y-markov}
\mathscr{Y}_s^{t, \zeta}(\omega) = Y_s^{t, \zeta(\omega)}(\omega) :=  Y_s^{t,x}(\omega)|_{x=\zeta(\omega)} \text { a.s.},
\end{equation}
where $(x,s,\omega)\mapsto Y^{t,x}_s(\omega)$ is a progressively measurable version of $Y^{t,x}_s$ for each $x\in\bar{D}$ (which is still denoted by $Y^{t,x}_s$). Consequently, $Y^{t,x}_\cdot$ is Markov in the sense that for any bounded Borel measurable function $\varphi$ and $t\le s\le r \le T$,  $
\E_{s}[\varphi(Y^{t,x}_{r})] = \E[\varphi(Y^{s,y}_{r})]\big|_{y=\bar{X}^{t,x}_{s}}$, where  $\bar{X}^{t,x}_s := X^{t,x}_{s\land T^{t,x}}.$ In particular, we have $Y_s^{t,x}=u(s,\bar{X}_s^{t,x})$ a.s., where $u(s,x) = Y_s^{s,x}$.
\end{lemma}

\begin{proof} 
Firstly, we assume that $\zeta(\omega) = \sum_{n=1}^{\infty} a_n \bm{1}_{A_{n}}(\omega)$ takes countably many values, where $a_{n} \in\bar{D}$ is a constant and $\{A_{n}\}_{n}\subset\mathcal{F}_{t}$ is a partition of $\Omega$. Then we have, noting $\sum_{n=1}^{\infty}\bm{1}_{A_{n}}(\omega)\equiv 1$,
\begin{equation*}
\begin{aligned}
Y^{t,\zeta(\omega)}_{s}(\omega) &= \sum_{n=1}^{\infty} \bm{1}_{A_{n}}(\omega)\cdot Y^{t,a_n}_{s}(\omega)\\
& = h(X^{t,\zeta(\omega)}_{T^{t,\zeta(\omega)}(\omega)}(\omega)) + \int_{s\land T^{t,\zeta(\omega)}(\omega)}^{T^{t,\zeta(\omega)}(\omega)}f(r,X^{t,\zeta(\omega)}_{r}(\omega),Y^{t,\zeta(\omega)}_{r}(\omega),\sum_{n=1}^{\infty}(\bm{1}_{A_{n}}Z^{t,a_n}_{r}(\omega)))dr \\
&\qquad \qquad \qquad \qquad \qquad \qquad \qquad \qquad \qquad \qquad - \int_{s\land T^{t,\zeta}}^{T^{t,\zeta}} \sum_{n=1}^{\infty}(\bm{1}_{A_{n}}Z^{t,a_n}_{r}) dW_{r} (\omega),
\end{aligned}
\end{equation*}
and 
\begin{equation*}
\int_{t}^{T} \mathbb{E}\bigg[ \Big| 
\sum_{n=1}^{\infty}(\bm{1}_{A_{n}}Z^{t,a_n}_{r}) \Big|^{2} \bigg] dr  \le \sup_{x\in\bar{D}} \mathbb{E}\left[\int_{t}^{T}\left| 
Z^{t,x}_{r} \right|^{2}dr\right]<\infty.
\end{equation*}
Therefore, $\big(Y^{t,\zeta}_{s},\sum_{n=1}^{\infty}(\bm{1}_{A_{n}} Z^{t,a_{n}}_{s})\big)$ solves the second equation in \eqref{e:A12}, and hence $\mathscr Y_s^{t, \zeta}= Y_s^{t,\zeta}$ a.s.

To deal with the general situation, we claim that
\begin{equation}\label{e:conti of initi value}
\left\|\left\|\mathscr{Y}^{t,\zeta^{m}} - \mathscr{Y}^{t,\zeta}\right\|_{\infty;[t,T]}\right\|_{L^{2}(\Omega)}\rightarrow 0, \text{ if }\left\|\zeta^{m} - \zeta\right\|_{L^{2}(\Omega)}\rightarrow 0\text{, as }m\rightarrow\infty,
\end{equation}
which will be proved afterwards. For an $\mathcal F_t$-measurable  random vector $\zeta\in \bar D$, let $\zeta^{m}\in\mathcal{F}_t$ take countably many values in $\bar D$ such that $\left\|\zeta^{m} - \zeta\right\|_{L^{2}(\Omega)}\rightarrow 0$ as $ m\rightarrow\infty.$
Then,  it follows from \eqref{e:conti of initi value} that
\begin{equation}\label{e:scr Y}
\lim_{m\to\infty}\left\|\mathscr{Y}^{t,\zeta^m}_s - \mathscr{Y}^{t,\zeta}_s\right\|_{L^{2}(\Omega)}=0.
\end{equation}
Moreover,
\begin{equation*}
\mathbb{E}\left[\big|Y^{t,\zeta^m}_{s} - Y^{t,\zeta}_s\big|^{2}\right]=\mathbb{E}\left[\mathbb{E}_{t}\left[\left|Y^{t,\zeta^m}_{s} - Y^{t,\zeta}_s\right|^2\right]\right]=\mathbb{E}\left[\left.\mathbb{E}\left[\left|Y^{t,x'}_{s} - Y^{t,x}_{s}\right|^2\right]\right|_{x'=\zeta^m,x=\zeta}\right],
\end{equation*}
where we use the fact that $\left|Y^{t,x'}_{s} - Y^{t,x}_{s}\right|^2\in\mathcal{F}^{t}_{T}$ is independent of $\mathcal{F}_t$. By  \eqref{e:conti of initi value} again, we have 
\[\mathbb{E}\left[\left|Y^{t,x'}_{s} - Y^{t,x}_{s}\right|^2\right]\rightarrow 0,\text{ whenever }x'\rightarrow x.\]
Noting that (see \cite[Theorem~4.2.1]{zhang2017backward}) $$  \sup_{x\in\bar{D}}\|Y^{t,x}_s\|_{L^{2}(\Omega)}\le \sup_{x\in\bar{D}}\|\|Y^{t,x}_{\cdot}\|_{\infty;[t,T]}\|_{L^{2}(\Omega)} < \infty,$$ we can apply dominated convergence theorem and get
\[\mathbb{E}\left[\left|Y^{t,\zeta^m}_{s} - Y^{t,\zeta}_s\right|^{2}\right]=\mathbb{E}\left[\left.\mathbb{E}\left[\left|Y^{t,x'}_{s} - Y^{t,x}_{s}\right|^2\right]\right|_{x'=\zeta^m,x=\zeta}\right]\rightarrow 0,\ \text{as }m\rightarrow \infty.\]
Noting $\mathscr{Y}^{t,\zeta_m}_{s} = Y^{t,\zeta_m}_{s}$ and combining the above convergence with \eqref{e:scr Y}, we get $\mathscr{Y}^{t,\zeta}_{s} = Y^{t,\zeta}_{s}$, which proves \eqref{e:Y-markov}. This implies the Markov property. Indeed, by the uniqueness of the solution, for $t\le s< r\le T$ we have
\begin{equation}\label{e:before the markov}
Y^{t,x}_{r} = \mathscr{Y}^{s,\bar{X}^{t,x}_{s}}_{r} = Y^{s,\bar{X}^{t,x}_{s}}_{r}.
\end{equation}
Noting that $Y_\cdot^{s,y}$ is independent of $\mathcal F_s$, we have for any bounded Borel measurable function $\varphi$, 
\begin{equation*}
\mathbb{E}_{s}\left[\varphi(Y^{t,x}_{r})\right] = \mathbb{E}\left[\varphi(Y^{s,y}_{r})\right]\big|_{y = \bar{X}^{t,x}_{s}}.
\end{equation*}

We have proven our results assuming \eqref{e:conti of initi value}. We now prove \eqref{e:conti of initi value}, to complete the proof. Recall that $\mathscr{Y}^{t,\zeta^m}$ solves Eq.~\eqref{e:A12} with $\zeta$ replaced by $\zeta^m$. Denote $(\delta\mathscr{Y}^{m}, \delta\mathscr{Z}^{m}) := (\mathscr{Y}^{t,\zeta^m} - \mathscr{Y}^{t,\zeta},\mathscr{Z}^{t,\zeta^m} - \mathscr{Z}^{t,\zeta}).$  Then, It\^o's formula yields, for $s\in[t,T]$,
\begin{equation}\label{e:delta scr Y}
\begin{aligned}
&\mathbb{E}\left[|\delta\mathscr{Y}^{m}_{s}|^{2}\right] + \mathbb{E}\left[\int_{s}^{T}|\delta\mathscr{Z}^{m}_{r}|^{2}dr\right]\\
&\le 2\int_{s}^{T}\mathbb{E}\left[|\delta \mathscr{Y}^{m}_{r}|\left|f(r,X^{t,\zeta}_r,\mathscr{Y}^{t,\zeta^m}_{r},\mathscr{Z}^{t,\zeta^m}_{r}) - f(r,X^{t,\zeta}_r,\mathscr{Y}^{t,\zeta}_{r},\mathscr{Z}^{t,\zeta}_{r})\right|\right]dr\\
&\quad + 2\int_{s}^{T}\mathbb{E}\left[|\delta \mathscr{Y}^{m}_{r}|\left|f(r,X^{t,\zeta^m}_r,\mathscr{Y}^{t,\zeta^m}_{r},\mathscr{Z}^{t,\zeta^m}_{r}) - f(r,X^{t,\zeta}_r,\mathscr{Y}^{t,\zeta^m}_{r},\mathscr{Z}^{t,\zeta^m}_{r})\right|\right]dr \\
&\quad + 2\mathbb{E}\left[\int_{T^{t,\zeta}\land T^{t,\zeta^m}}^{T^{t,\zeta}\vee T^{t,\zeta^m}} |\delta \mathscr{Y}^{m}_{r}|\left(\left|f(r,X^{t,\zeta^m}_r,\mathscr{Y}^{t,\zeta^m}_{r},\mathscr{Z}^{t,\zeta^m}_{r})\right| + \left|f(r,X^{t,\zeta}_r,\mathscr{Y}^{t,\zeta}_{r},\mathscr{Z}^{t,\zeta}_{r})\right|\right)dr\right] \\
&\quad + \mathbb{E}\left[\left|h(X^{t,\zeta^m}_{T^{t,\zeta^m}}) - h(X^{t,\zeta}_{T^{t,\zeta}})\right|^{2}\right].
\end{aligned}
\end{equation}

For the first term on the right-hand side of \eqref{e:delta scr Y}, by condition~\eqref{e:assump of f} and the Young's inequality ($2|a||b|\le a^{2} + b^{2}$), we get
\begin{equation}\label{e:continuity of u(t,x)''}
\begin{aligned}
&2\int_{s}^{T} \mathbb{E}\left[|\delta \mathscr{Y}^{m}_{r}|\left|f(r,X^{t,\zeta}_{r},\mathscr{Y}^{t,\zeta^m}_{r},\mathscr{Z}^{t,\zeta^{m}}_{r}) - f(r,X^{t,\zeta}_{r},\mathscr{Y}^{t,\zeta}_{r},\mathscr{Z}^{t,\zeta}_{r})\right|\right]dr\\
&\le 2C_{\text{Lip}} \mathbb{E}\left[\int_{s}^{T}|\delta\mathscr{Y}^{m}_{r}|^{2} + |\delta\mathscr{Y}^{m}_{r}|\cdot|\delta\mathscr{Z}^{m}_{r}| dr\right]\\
&\le 2C_{\text{Lip}}\mathbb{E}\left[\int_{s}^{T}|\delta \mathscr{Y}^{m}_{r}|^{2}dr\right] + 2|C_{\text{Lip}}|^{2}\mathbb{E}\left[\int_{s}^{T}|\delta \mathscr{Y}^{m}_{r}|^{2}dr\right] + \frac{1}{2}\mathbb{E}\left[\int_{s}^{T}|\delta \mathscr{Z}^{m}_{r}|^{2} dr\right].
\end{aligned}
\end{equation}

For the second term, in view of the local H\"older continuity of $x\mapsto f(t,x,y,z)$ assumed in \ref{(A1)}, by the inequality $\|\|X^{t,\zeta^m}_{\cdot} - X^{t,\zeta}_{\cdot}\|_{\infty;[t,T]}\|_{L^{2}(\Omega)}\lesssim\|\zeta^m - \zeta\|_{L^2(\Omega)}$ (see \cite[Theorem~3.2.4]{zhang2017backward}), 
it follows that
\begin{equation}\label{e:continuity of u(t,x)'}
\begin{aligned}
&\int_{s}^{T}\mathbb{E}\left[|\delta \mathscr{Y}^{m}_{r}|\left|f(r,X^{t,\zeta^m}_r,\mathscr{Y}^{t,\zeta^m}_{r},\mathscr{Z}^{t,\zeta^m}_{r}) - f(r,X^{t,\zeta}_r,\mathscr{Y}^{t,\zeta^m}_{r},\mathscr{Z}^{t,\zeta^m}_{r})\right|\right]dr\\
& \lesssim \int_{s}^{T} \mathbb{E}\left[|\delta \mathscr{Y}^{m}_{r}|^2\right] dr + \int_{s}^{T} \mathbb{E}\left[|X^{t,\zeta^m}_{r} - X^{t,\zeta}_{r}|^{\mu}(1 + |X^{t,\zeta^m}_{r}|^{\mu' + \frac{\mu}{2}} + |X^{t,\zeta}_{r}|^{\mu' + \frac{\mu}{2}})^2\right] dr \\
& \lesssim \int_{s}^{T} \mathbb{E}\left[|\delta \mathscr{Y}^{m}_{r}|^2\right] dr + \|\zeta^{m} - \zeta\|^{\mu}_{L^{2}(\Omega)}.
\end{aligned}
\end{equation}

For the third term, it follows that 
\begin{align}\label{e:continuity of u(t,x)}
&\mathbb{E}\left[\int_{T^{t,\zeta}\land T^{t,\zeta^m}}^{T^{t,\zeta}\vee T^{t,\zeta^m}} |\delta \mathscr{Y}^{m}_{r}|\left(\left|f(r,X^{t,\zeta^m}_r,\mathscr{Y}^{t,\zeta^m}_{r},\mathscr{Z}^{t,\zeta^m}_{r})\right| + \left|f(r,X^{t,\zeta}_r,\mathscr{Y}^{t,\zeta}_{r},\mathscr{Z}^{t,\zeta}_{r})\right|\right)dr\right]\nonumber\\
&\lesssim   \Big(1 + \|(\mathscr{Y}^{t,\zeta},\mathscr{Z}^{t,\zeta})\|^2_{\mathfrak{H}^{4}(t,T)} + \|(\mathscr{Y}^{t,\zeta^m},\mathscr{Z}^{t,\zeta^m})\|^2_{\mathfrak{H}^{4}(t,T)}\Big) \|T^{t,\zeta} - T^{t,\zeta^{m}}\|^{\frac{1}{2}}_{L^{2}(\Omega)}
\nonumber\\
&\lesssim \|T^{t,\zeta} - T^{t,\zeta^{m}}\|^{\frac{1}{2}}_{L^{2}(\Omega)},
\end{align}
where the last inequality follows from  
\cite[Proposition~4.2]{BSDEYoung-I}.

For the last term on the right-hand side of \eqref{e:delta scr Y}, by \cite[Theorems~5.2.1 and 5.2.2]{zhang2017backward} we have 
\begin{equation}\label{e:continuity of u(t,x)0}
\mathbb{E}\left[\left|h(X^{t,\zeta^m}_{T^{t,\zeta^m}}) - h(X^{t,\zeta}_{T^{t,\zeta}})\right|^{2}\right]\lesssim \|T^{t,\zeta} - T^{t,\zeta^{m}}\|_{L^{1}(\Omega)} + \|\zeta^m - \zeta\|^{2}_{L^{2}(\Omega)}.
\end{equation}

Combining the estimates \eqref{e:delta scr Y}-\eqref{e:continuity of u(t,x)0}, and applying  H\"older's inequality and Gronwall's inequality, we obtain, for $s\in[t,T]$,
\begin{equation}\label{e:est of delta mathscr Y}
\mathbb{E}\left[|\delta\mathscr{Y}^{m}_{s}|^{2}\right] + \mathbb{E}\left[\int_{s}^{T}|\delta\mathscr{Z}^{m}_{r}|^{2}dr\right]\lesssim \|\zeta^m - \zeta\|^{\mu}_{L^{2}(\Omega)} + \|T^{t,\zeta} - T^{t,\zeta^m}\|^{\frac{1}{2}}_{L^{2}(\Omega)}.
\end{equation}
Furthermore, by \eqref{e:A12} and the BDG inequality, we have
\begin{equation}\label{e:est of delta mathscr Y'}
\begin{aligned}
&\mathbb{E}\left[\|\delta\mathscr{Y}^{m}\|^{2}_{\infty;[t,T]}\right]\\
&\lesssim \mathbb{E}\left[\left|\int_{t}^{T}\left|f(r,X^{t,\zeta^m}_r,\mathscr{Y}^{t,\zeta^m}_{r},\mathscr{Z}^{t,\zeta^m}_{r}) - f(r,X^{t,\zeta}_r,\mathscr{Y}^{t,\zeta}_{r},\mathscr{Z}^{t,\zeta}_{r})\right|dr\right|^{2}\right] \\
&\quad + \mathbb{E}\left[\left|\int_{T^{t,\zeta}\land T^{t,\zeta^m}}^{T^{t,\zeta}\vee T^{t,\zeta^m}} \left(\left|f(r,X^{t,\zeta^m}_r,\mathscr{Y}^{t,\zeta^m}_{r},\mathscr{Z}^{t,\zeta^m}_{r})\right| + \left|f(r,X^{t,\zeta}_r,\mathscr{Y}^{t,\zeta}_{r},\mathscr{Z}^{t,\zeta}_{r})\right|\right)dr \right|^{2}\right] \\
&\quad + \mathbb{E}\left[\int_{t}^{T} |\delta \mathscr{Z}^{m}_{r}|^{2}dr\right] + \mathbb{E}\left[\left|h(X^{t,\zeta^m}_{T^{t,\zeta^m}}) - h(X^{t,\zeta}_{T^{t,\zeta}})\right|^{2}\right].
\end{aligned}
\end{equation}
Then by a similar analysis leading to \eqref{e:est of delta mathscr Y}, from \eqref{e:est of delta mathscr Y'}  we get 
\begin{equation}\label{e:est of delta mathscr Y''}
\begin{aligned}
&\mathbb{E}\left[\|\delta\mathscr{Y}^{m}\|^{2}_{\infty;[t,T]}\right]\\
&\lesssim \int_{t}^{T}\mathbb{E}\left[|\delta \mathscr{Y}^{m}_r|^2\right] dr + \mathbb{E}\left[\int_{t}^{T}|\delta \mathscr{Z}^{m}_{r}|^2 dr\right] + \|\zeta^{m} - \zeta\|^{\mu}_{L^{2}(\Omega)} + \|T^{t,\zeta} - T^{t,\zeta^m}\|^{\frac{1}{2}}_{L^{2}(\Omega)}.
\end{aligned}
\end{equation}
Then, the desired \eqref{e:conti of initi value} follows from \eqref{e:est of delta mathscr Y}, \eqref{e:est of delta mathscr Y''}, and Lemma~\ref{lem:T t,zeta}, and the proof is concluded.
\end{proof}

Since the coefficient functions $f$ and $h$ in  Eq.~\eqref{e:A12} are deterministic, $Y^{t,x}_{t}\in \mathcal F_t^t$ is a constant a.s. We denote $$u(t,x) := Y^{t,x}_{t}, ~(t,x)\in [0,T]\times\bar{D}.$$

\begin{lemma}\label{lem:u(t,x) is continuous}
Under the same conditions as 
in Lemma~\ref{lem:flow property}, the function $(t,x)\mapsto u(t,x)$ is continuous on $[0,T]\times \bar D$.
\end{lemma}

\begin{proof}
Let $T^{t,\zeta}$ be given in \eqref{e:def of T^t,zeta} and denote $S^{t,\zeta} := \inf\{r > t;\ X^{t,\zeta}_r \notin \bar{D}\}$.  Denote
$\bar{X}^{t,x}_{s} := X^{t,x}_{s\land T^{t,x}}.$ 
Since $(X^{t,x},X^{t,y})^{\top}$ solves a time-homogeneous SDE for which the pathwise uniqueness holds, by the Yamada-Watanabe theorem (see, e.g., \cite[Proposition~5.3.20]{karatzas1991brownian}), $(X^{t,x},X^{t,y})^{\top}$ has the same distribution as $(X^{0,x},X^{0,y})^{\top}$. Therefore,
for $x,y\in\bar{D}$, $ 0\le t< s < T$, and $\delta\in(s-t,T-t]$, we have that 
\begin{equation*}
\left(S^{s,y}\land T\right) - \left(S^{s,x}\land T\right) \overset{\text{d}}{=} \left(S^{0,y}\land (T-s)\right) - \left(S^{0,x}\land (T-s)\right),
\end{equation*}
where the symbol ``$\overset{\text{d}}{=}$'' denotes equality in  distribution.
Similarly, we also have
\begin{equation*}
\left(S^{s,\bar{X}^{t,x}_{s}}\land T\right) - \left(S^{s,x}\land T\right) \overset{\text{d}}{=} \left(S^{\delta,\bar{X}^{\delta-s+t,x}_{\delta}}\land (\delta + T-s)\right) - \left(S^{\delta,x}\land (\delta + T-s)\right).
\end{equation*} 
Noting that $T^{t,\zeta}=S^{t,\zeta}\wedge T$ and that for $s\in(t,t+\delta)\subset[0,T]$,
\begin{equation*}
\begin{cases}
\left|\big(a\land (T-s)\big) - \big(b\land(T-s)\big)\right|\le \left|\big(a\land T \big) - \big(b\land T\big)\right|,
\\
\left|\big(a\land (\delta + T-s)\big) - \big(b\land(\delta + T-s)\big)\right|\le \left|\big(a\land 2T\big) - \big(b\land 2T\big)\right|,
\end{cases} 
\end{equation*}
we get
\begin{equation}\label{e:T-T,=T-T0}
\mathbb{E}\left[\big|T^{s,y} - T^{s,x}\big|^2\right] \le \mathbb{E}\left[\big|T^{0,y} - T^{0,x}\big|^2\right]
\end{equation}
and
\begin{equation}\label{e:T-T,=T-T}
\mathbb{E}\left[\big|T^{s,\bar{X}^{t,x}_{s}} - T^{s,x}\big|^2\right] \le \mathbb{E}\left[\big|(S^{\delta,\bar{X}^{\delta-s+t,x}_{\delta}}\land 2T) - (S^{\delta,x}\land 2T)\big|^2\right].
\end{equation}
On the other hand, we claim that $\bar{X}^{\delta-s+t,x}_{\delta} \rightarrow x$  in probability as $s\downarrow t$. Indeed, by \cite[Theorem~5.2.2]{zhang2017backward},  
\begin{equation}\label{e:X^delta-s+t}
\begin{aligned}
\|\bar{X}^{\delta-s+t,x}_{\delta} - x\|^2_{L^{2}(\Omega)}\le \left\|\|X^{\delta-s+t,x}_{\cdot} - x\|_{\infty;[\delta-s+t,\delta]}\right\|^2_{L^{2}(\Omega)}\lesssim s - t,
\end{aligned}
\end{equation}
which implies our claim. Thus by Lemma~\ref{lem:T t,zeta}, we have 
\begin{equation}\label{e:2T0}
\mathbb{E}\left[\big|T^{0,y} - T^{0,x}\big|^2\right]\to 0,\ \text{ as }y\to x,
\end{equation}
\begin{equation}\label{e:2T}
\mathbb{E}\left[\big|(S^{\delta,\bar{X}^{\delta-s+t,x}_{\delta}}\land 2T) - (S^{\delta,x}\land 2T)\big|^2\right]\to 0,\ \text{ as } s\downarrow t.
\end{equation}
Then, \eqref{e:T-T,=T-T0} and \eqref{e:2T0} imply
\begin{equation}\label{e:after 2T0}
\sup_{s\in(t,t+\delta)}\mathbb{E}\left[\big|T^{s,y} - T^{s,x}\big|^2\right]\to 0,\ \text{ if }y\rightarrow x, 
\end{equation}
while \eqref{e:T-T,=T-T} and \eqref{e:2T} yield
\begin{equation}\label{e:after 2T}
\mathbb{E}\left[\big|T^{s,\bar{X}^{t,x}_{s}} - T^{s,x}\big|^2\right]\to 0,
\text{ as }s\downarrow t.
\end{equation}

Now, we are ready to prove the continuity of $u.$  By \eqref{e:before the markov}, we have $Y^{t,x}_{s} = Y^{s,\bar{X}^{t,x}_{s}}_{s}$. Then we have
\begin{equation*}
|u(t,x) - u(s,x)| \le \left|\mathbb{E}\left[Y^{t,x}_{t} - Y^{t,x}_{s}\right]\right| + \mathbb{E}\left[\big|Y^{s,\bar{X}^{t,x}_{s}}_{s} - Y^{s,x}_{s}\big|\right].
\end{equation*}
By the equality $$
\left|\mathbb{E}\left[Y^{t,x}_{t} - Y^{t,x}_{s}\right]\right| = \bigg|\mathbb{E}\Big[\int_{t}^{s\land T^{t,x}}f(r,X^{t,x}_{r},Y^{t,x}_{r},Z^{t,x}_{r})dr\Big]\bigg|,$$
we obtain that $\left|\mathbb{E}\left[Y^{t,x}_{t} - Y^{t,x}_{s}\right]\right|\lesssim|t-s|^{1/2}$.  Meanwhile, by \eqref{e:est of delta mathscr Y} and \cite[Theorem~5.2.2]{zhang2017backward}, 
\begin{equation*}
\begin{aligned}
\mathbb{E}\left[\big|Y^{s,\bar{X}^{t,x}_{s}}_{s} - Y^{s,x}_{s}\big|\right]&\lesssim \|\bar{X}^{t,x}_{s} - x\|^{\mu}_{L^{2}(\Omega)} + \|T^{s,\bar{X}^{t,x}_{s}} - T^{s,x}\|^{\frac{1}{4}}_{L^{2}(\Omega)}\\
&\le |t - s|^{\frac{\mu}{2}} + \|T^{s,\bar{X}^{t,x}_{s}} - T^{s,x}\|^{\frac{1}{4}}_{L^{2}(\Omega)}.
\end{aligned}
\end{equation*}
Thus, we have 
\begin{equation}\label{e:u(t,x)-u(s,x)}
|u(t,x)-u(s,x)|\lesssim |t - s|^{\frac{\mu}{2}} + \|T^{s,\bar{X}^{t,x}_{s}} - T^{s,x}\|^{\frac{1}{4}}_{L^{2}(\Omega)}.    
\end{equation}

In a similar way, we can also get  
\begin{equation}\label{e:u(s,x)-u(s,y)}
|u(s,x) - u(s,y)|\lesssim |x-y|^{\mu} + \|T^{s,x} - T^{s,y}\|^{\frac{1}{4}}_{L^{2}(\Omega)}.
\end{equation}
Noting \eqref{e:u(t,x)-u(s,x)} and \eqref{e:u(s,x)-u(s,y)}, by the triangular inequality we get
\begin{equation}\label{e:u(t,x)-u(s,y)}
|u(t,x) - u(s,y)|\lesssim |t-s|^{\frac{\mu}{2}} + |x-y|^{\mu} + \|T^{s,x} - T^{s,y}\|^{\frac{1}{4}}_{L^{2}(\Omega)} + \|T^{s,\bar{X}^{t,x}_{s}} - T^{s,x}\|^{\frac{1}{4}}_{L^{2}(\Omega)}.
\end{equation}
Combining \eqref{e:u(t,x)-u(s,y)} with   \eqref{e:after 2T0}, \eqref{e:after 2T} and \eqref{e:X^delta-s+t}, we get
$$\lim\limits_{s\downarrow t,y\rightarrow x}u(s,y) = u(t,x).$$

It remains to show the left continuity in time. In view of the triangular inequality 
\[|u(t,x) - u(s,y)|\le |u(t,x) - u(t,y)| + |u(t,y) - u(s,y)|,\]
and the above argument, we have
\begin{equation*}
|u(t,x) - u(s,y)|\lesssim |t-s|^{\frac{\mu}{2}} + |x-y|^{\mu} + \|T^{t,x} - T^{t,y}\|^{\frac{1}{4}}_{L^{2}(\Omega)} + \|T^{s,\bar{X}^{t,y}_{s}} - T^{s,y}\|^{\frac{1}{4}}_{L^{2}(\Omega)},
\end{equation*}
which yields $$\lim\limits_{t\uparrow s,x\rightarrow y}u(t,x) = u(s,y).$$
\end{proof}


\section{The proof of Theorem~\ref{thm:existence of solution of unbounded BSDEs} for the general case}\label{append:C}

The proof of Theorem~\ref{thm:existence of solution of unbounded BSDEs} for general cases when $N=1$ and $f\equiv 0$ is not trivial. Indeed, there are two technical problems. First, the condition~$f\neq 0$ may yield an additional drift term in Eq.~\eqref{e:deltaY}, which forces $A^{t;n}_s$ in \eqref{e:def of A^t;n} to solve a Young SDE rather than a Young ODE. Second, the dimension $N>1$ makes the arguments in \eqref{e:Gamma<=2'}--\eqref{e:e:existence-7'} fail, since $|A^{t;n}_s|^{k/(k-1)}$ cannot be represented by an exponential function in general. 

We first give two lemmas that will be useful for our proof. Denote 
\begin{equation*}
\mathrm{H}^{2} := \left\{A: A \text{ is a continuous, adapted process such that }\left\|\|A\|_{L^{2}([0,T])}\right\|_{L^{2}(\Omega)}<\infty\right\},
\end{equation*}
and 
\begin{equation*}
\begin{aligned}
\mathrm{H}^{2}_{\text{loc}}&:=\Big\{A:\text{there exists a sequence of increasing stopping times }\\
&\{S_{m}\}_{m\ge 1} \text{ such that }A_r \mathbf{1}_{[0,S_m]}(r)\in \mathrm{H}^2 \text{ and }\lim_{m}\mathbb{P}\{S_{m} = T\}=1\}\Big\}.
\end{aligned}
\end{equation*}

\begin{lemma}\label{lem:SDE}
Suppose Assumption~\ref{(A0)} holds, and $v\ge 1$ is an integer. Assume $\eta\in C^{\tau,\lambda;\beta}([0,T]\times \R^d; \R)$ for some $(\tau,\lambda,\beta)\in (0,1] \times (0,1] \times [0,\infty)$, where $\partial_{t}\eta_{i}(t,x)$ exists and is continuous in $(t,x)\in[0,T]\times \R^d$; and assume $\alpha^{i},\iota,c:[0,T]\times\Omega\rightarrow\R^{v\times v}$ are bounded adapted processes, $i=1,2,...,M$. Let $S$ be a stopping time with $S\le T$. Then the following stochastic differential equation (SDE) admits a unique solution $A^{S}\in\mathrm{H}^{2}_{\mathrm{loc}}$ 
\begin{equation}\label{e:SDE A}
A^{S}_{s} = I_{v} + \sum_{i=1}^{M}\int_{S}^{s\vee S} (\alpha^{i}_{r})^{\top}A^{S}_{r}\partial_r\eta_{i}(r,X_{r})dr + \int_{S}^{s\vee S} (\iota_{r})^{\top}A^{S}_{r} dr + \int_{S}^{s\vee S} (c_{r})^{\top} A^{S}_{r} dW_{r},\ s\in[0,T].
\end{equation}
In addition, for stopping time $S'$ with $S\le S'\le T,$ we have a.s., $A^{S'}_{t} A^{S}_{S'} = A^{S}_{t}$ holds for all $t\in[S',T]$.
\end{lemma}

\begin{proof}
Let $T_n$ be given by \eqref{e:Tn}. Since the processes $\alpha^{i}_t\partial_{t}\eta_{i}(t,X_{t})\mathbf{1}_{[S\land T_{n},T_n]}(t),$ $\iota_t \mathbf{1}_{[S\land T_{n},T_n]}(t),$ and $c_t \mathbf{1}_{[S\land T_{n},T_n]}(t)$ are all bounded, the parameters of Eq.~\eqref{e:SDE A} are Lipschitz on $[0,T_n]$. Therefore, the equation admits a unique solution in $\mathrm{H}^2$ on $[0,T_n]$ (see \cite[Theorem~5.2.1]{Oksendal}). Thus, by the equality~$\lim_{n}\mathbb{P}\left\{T_{n} = T\right\} = 1$ and the localization argument, the uniqueness of Eq.~\eqref{e:SDE A} in $\mathrm{H}^{2}_{\text{loc}}$ holds on the entire interval $[0,T]$.

The second assertion can be proved in the same way as for Lemme~\ref{lem:For Gamma}, and hence we omit it.
\end{proof}

For integers $v,j\ge 1,$ and matrix $\Lambda\in \R^{v^{2^{j}}\times v^{2^{j}}}$, let $I'(\Lambda) := \big((\Lambda e^{v;j}_{1})^{\top},...,\Lambda e^{v;j}_{v^{2^{j}}}\big)^{\top}\in \R^{v^{2^{j}}\times 1}$,
\begin{equation*}
\tilde{I}_{v^2\times 1}:= I'(I_{v}),\quad \tilde{I}_{v^{2^{j+1}}\times 1} := I'(\tilde{I}_{v^{2^{j}}\times 1} \tilde{I}_{v^{2^{j}}\times 1}^{\top}),
\end{equation*}
where $\{e^{v;j}_{1},...,e^{v;j}_{v^{2^{j}}}\}$ is the coordinate system of $\R^{v^{2^{j}}}$, and $I':\R^{v^{2^{j}}\times v^{2^{j}}}\rightarrow \R^{v^{2^{j+1}}\times 1}.$ 

Next, we will present a lemma to bound the higher-order moments of the solution to a matrix-valued SDE using the expectation of the solution to a vector-valued SDE.

\begin{lemma}\label{lem:(a,b,c)}
Under the same assumptions as in Lemma~\ref{lem:SDE}, assume further that $m_{p,2}(\alpha^{i};[0,T])<\infty,$ $i=1,2,...,M$. Denote by $A^{S}$ the unique solution of the SDE~\eqref{e:SDE A} with a stopping time $S$ satisfying $S\le T$. Then for any $j\ge 1$, there exist progressively measurable processes $\tilde{\alpha}^{i},$ $\tilde{\iota},$ $\tilde{c}:[0,T]\times\Omega\rightarrow \R^{v^{2^{j}}\times v^{2^{j}}},$ and a constant $C>0$ such that
\begin{equation}\label{e:(i) of alpha,b,c}
|\tilde{\alpha}^{i}_t|\le C |\alpha^{i}_t|,\ |\tilde{\iota}_t|\le C |\iota_t|,\ |\tilde{c}_t|\le C |c_t|;
\end{equation}
\begin{equation}\label{e:mpq tilde alpha}
m_{p,k}(\tilde{\alpha}^{i};[0,T])\le C\cdot m_{p,k}(\alpha^{i};[0,T]) ;
\end{equation}
and for any stopping time $S'$ satisfying $S\le S'\le T$, it follows that 
\begin{equation}\label{e:alpha,b,c}
\mathbb{E}_{S}\left[|A^{ S}_{S'}|^{2^{j}}\right] \le C \left|\mathbb{E}_{S}\left[\tilde{B}^{ S}_{S'}\right]\right|,
\end{equation}
where $\tilde{B}^{S}$ is the unique solution of the following $v^{2^{j}}$-dimensional vector-valued SDE
\begin{equation*}
\tilde{B}^{S}_{s} = \tilde{I}_{v^{2^{j}}\times 1} + \sum_{i=1}^{M}\int_{S}^{s\vee S}(\tilde{\alpha}^{i}_{r})^{\top}\tilde{B}^{S}_{r}\partial_{r}\eta_{i}(r,X_r)dr + \int_{S}^{s\vee S}(\tilde{\iota}_{r})^{\top}\tilde{B}^{S}_{r} dr + \int_{S}^{s\vee S} (\tilde{c}_{r})^{\top}\tilde{B}^{S}_{r}d W_{r}.
\end{equation*}
\end{lemma}

\begin{proof}
Assume $M=d=1$ without loss of generality. For $j = 1,$ set  $ C^{S}_{s}:=A^{S}_{s}(A^{S}_{s})^{\top}$; then it satisfies that 
\begin{equation*}
\begin{aligned}
C^{S}_{s} = I_{v} &+ \int_{S}^{s\vee S} \left( (\alpha_{r})^{\top} C^{S}_r + C^{S}_{r}\alpha_r \right)\partial_{r}\eta(r,X_r) + \int_{S}^{s\vee S} \Big( (\iota_{r})^{\top} C^{S}_r + C^{S}_{r}\iota_r \\
& + (c_{r})^{\top} C^{S}_{r}c_{r}\Big)dr + \int_{S}^{s\vee S}\left( (c_{r})^{\top} C^{S}_r + C^{S}_{r} c_r\right)d W_r,\ s\in[0,T]. 
\end{aligned}
\end{equation*}
Let $\tilde{C}^{S}_{s} := I'(C^{S}_{s})$. Similar to the proof of \cite[Proposition~3.5]{BSDEYoung-I}, it can be shown that there exist processes $\alpha^{\prime}_{r},\iota^{\prime}_{r},c^{\prime}_{r} \in \R^{v^2\times v^2}$ satisfying the following conditions.
\begin{itemize}
\item[(1)] 
Each entry of $\alpha^{\prime}_r$ (resp. $c^{\prime}_r$) is a linear combination of the elements of $\alpha_r$ (resp. $c_r$), and each entry of $\iota^{\prime}_{r}$ is a linear combination of the elements of $(\iota_{r},c_{r}).$

\item[(2)]
For $i,i'=1,2,...,v,$ the process $\alpha^{\prime}$ satisfies that for every $\Lambda\in\R^{v\times v}$,
\begin{equation*}
(\underbrace{\overbrace{ 0,\ 0,\ \cdots,\ 0}^{(i'-1)v},\ e^{\top}_{i},}_{i'v}\ 0,\ \cdots,\ 0
)_{1\times v^{2}}\cdot(\alpha^{\prime}_{r})^{\top} I'(\Lambda) = \left((\alpha_r)^{\top}\Lambda + \Lambda \alpha_r\right)_{ii'},
\end{equation*}
and $c^{\prime}$ satisfies the above equality with $(\alpha_{r},\alpha^{\prime}_{r})$ replaced by $(c_{r},c^{\prime}_{r}).$

\item[(3)] For $i,i'=1,2,...,v,$ the process $\iota^{\prime}$ satisfies that for every $\Lambda\in\R^{v\times v}$,
\begin{equation*}
(\underbrace{\overbrace{ 0,\ 0,\ \cdots,\ 0}^{(i'-1)v},\ e^{\top}_{i},}_{i'v}\ 0,\ \cdots,\ 0
)_{1\times v^{2}}\cdot (\iota^{\prime}_r)^{\top} I'(\Lambda) = \left( (\iota_r)^{\top}\Lambda 
 + \Lambda \iota_r + (c_{r})^{\top}\Lambda c_{r}\right)_{ii'}.
\end{equation*}
\end{itemize}
Then, $\tilde{C}^{S}_{s}$ uniquely solves the following $v^{2}$-dimensional SDE,
\begin{equation*}
\tilde{C}^{S}_{s} = \tilde{I}_{v^2\times 1} + \int_{S}^{s\vee S}(\alpha^{\prime}_{r})^{\top}\tilde{C}^{S}_{r}\eta(dr,X_r) + \int_{S}^{s\vee S}(\iota^{\prime}_{r})^{\top}\tilde{C}^{S}_{r} dr + \int_{S}^{s\vee S} (c^{\prime}_{r})^{\top}\tilde{C}^{S}_{r}d W_{r},\ s\in[0,T].
\end{equation*}
By condition~(1), both \eqref{e:(i) of alpha,b,c} and  \eqref{e:mpq tilde alpha} hold with $(\tilde{\alpha},\tilde{\iota},\tilde{c})$ replaced by $(\alpha^{\prime},\iota^{\prime},c^{\prime})$. Furthermore, \eqref{e:alpha,b,c} follows from the following fact, with $j=1$, and $\tilde{B}^{S}$ replaced by $\tilde{C}^{S}$,
\begin{equation*}
\mathbb{E}_{S}\left[|A^{S}_{S'}|^2\right] = \text{tr}\left\{\mathbb{E}_{S}\left[C^{S}_{S'}\right]\right\} \lesssim \left|\mathbb{E}_{S}\left[C^{S}_{S'}\right]\right|  = \big|\mathbb{E}_{S}\big[\tilde{C}^{S}_{S'}\big]\big|.
\end{equation*}

For $j = 2$, let $D^{S}_s := \tilde{C}^{S}_s(\tilde{C}^{S}_s)^{\top},$ then $D^{S}$ satisfies
\begin{equation*}
\begin{aligned}
D^{S}_{s} = \tilde{I}_{v^{2}\times 1}(\tilde{I}_{v^{2}\times 1})^{\top} &+ \int_{S}^{s\vee S} \left( (\alpha^{\prime}_{r})^{\top}D^{S}_r + D^{S}_{r}\alpha^{\prime}_r \right)\partial_{r}\eta(r,X_r)dr + \int_{S}^{s\vee S} \Big( (\iota^{\prime}_{r})^{\top}D^{S}_r + D^{S}_{r}\iota^{\prime}_r \\
& + (c^{\prime}_r)^{\top} D^{S}_{r}c^{\prime}_{r}\Big)dr + \int_{S}^{s\vee S}\left( (c^{\prime}_{r})^{\top} D^{S}_r + D^{S}_{r}c^{\prime}_r\right)d W_r,\ s\in[0,T]. 
\end{aligned}
\end{equation*}
Let $\tilde{D}^{S}_{s} := I'(D^{S}_{s})$. In view of the inequality $|A^{S}_{S'}|^{4}\lesssim |\tilde{C}^{S}_{S'}|^2$, by repeating the procedure for $j = 1$, there exists $(\alpha^{\star},\iota^{\star},c^{\star})$ so that \eqref{e:(i) of alpha,b,c}--\eqref{e:alpha,b,c} hold with $j = 2$ and $(\tilde{\alpha},\tilde{\iota},\tilde{c},\alpha,\iota,c,\tilde{B}^{S})$ replaced by $(\alpha^{\star},\iota^{\star},c^{\star},\alpha',\iota',c',\tilde{D}^{S})$. By iterating the same argument, the case for $j\ge 3$ can be proved.
\end{proof}

Now we give the full proof of Theorem~\ref{thm:existence of solution of unbounded BSDEs}. Let $\Theta_{2} := (\tau,\lambda,\beta,p,k,l,\varepsilon,T, C_{1},C_{\text{Lip}},L,\|\eta\|_{\tau,\lambda;\beta})$ be as defined in Section~\ref{subsec:existence}.

\begin{proof}[Proof of Theorem~\ref{thm:existence of solution of unbounded BSDEs}]
\makeatletter\def\@currentlabelname{the full proof}\makeatother\label{proof:3'}
Assume $M=d=1$ without loss of generality. First, we additionally assume that $\eta$ is smooth in time and that $\partial_{t}\eta(t,x)$ is continuous in $(t,x)$. The general cases will be proved at the end of the proof. By \cite[Proposition~4.1]{BSDEYoung-I}, for $n\ge 1$, the following BSDE admits a unique solution in $\bigcap_{q>1}\mathfrak{B}_{p,q}(0,T)$,
\begin{equation*}
\begin{cases}
d Y_{t} = - f^{n}_{t}dt - g(Y_{t})\partial_{t}\eta(t,X_{t})dt + Z_{t}dW_{t},\ t\in[0,T_n],\\
Y_{T_{n}} = \Xi_{T_{n}},
\end{cases}
\end{equation*}
where $f^{n}_{t}:=f(t,X_{t},Y_{t},Z_{t})$, and we denote the solution by $(Y^n,Z^n)$. For $i=1,2,...,N$, denote $Y^{n}_{t} = (Y^{n;1}_{t},Y^{n;2}_{t},...,Y^{n;N}_{t})^{\top},$ $Z^{n}_{t} = (Z^{n;1}_{t},Z^{n;2}_{t},...,Z^{n;N}_{t})^{\top},$ $\delta Y^{n}_{t} := Y^{n+1}_{t} - Y^{n}_{t},$ $\delta Z^{n}_{t} := Z^{n+1}_{t} - Z^{n}_{t}$, 
\begin{equation*}
\begin{aligned}
\delta g^{n;i}_{t} &:= g(\overbrace{Y^{n;1}_{t},...,Y^{n;i-1}_{t},}^{i-1}Y^{n+1;i}_{t},Y^{n+1;i+1}_{t},...,Y^{n+1;N}_{t}) \\
&\quad - g(\overbrace{Y^{n;1}_{t},...,Y^{n;i-1}_{t},Y^{n;i}_{t},}^{i}Y^{n+1;i+1}_{t},...,Y^{n+1;N}_{t}), 
\end{aligned}
\end{equation*}
and 
\begin{equation*}
\alpha^{n}_{t} := \begin{pmatrix}
\frac{\delta g^{n;1}_t}{\delta Y^{n;1}_{t}}\mathbf 1_{\{\delta Y^{n;1}_{t}\neq 0\}},\ \frac{\delta g^{n;2}_t}{\delta Y^{n;2}_{t}}\mathbf 1_{\{\delta Y^{n;2}_{t}\neq 0\}},...,\ \frac{\delta g^{n;N}_t}{\delta Y^{n;N}_{t}}\mathbf 1_{\{\delta Y^{n;N}_{t}\neq 0\}}
\end{pmatrix}\in\mathbb{R}^{N\times N}.
\end{equation*}
For $z\in\mathbb{R}^{N}$ and $i=1,2,...,N$, set
\begin{equation*}
\begin{aligned}
\delta f^{n;i}_{t}(z) &:= f(t,X_t,\overbrace{Y^{n;1}_{t},...,Y^{n;i-1}_{t},}^{i-1}Y^{n+1;i}_{t},Y^{n+1;i+1}_{t},...,Y^{n+1;N}_{t},z) \\
&\quad - f(t,X_t,\overbrace{Y^{n;1}_{t},...,Y^{n;i-1}_{t},Y^{n;i}_{t},}^{i}Y^{n+1;i+1}_{t},...,Y^{n+1;N}_{t},z), 
\end{aligned}
\end{equation*}
and 
\begin{equation*}
\begin{aligned}
\delta \tilde{f}^{n;i}_{t} &:= f^{n}(t,X_t,Y^{n}_t,\overbrace{Z^{n;1}_{t},...,Z^{n;i-1}_{t},}^{i-1}Z^{n+1;i}_{t},Z^{n+1;i+1}_{t},...,Z^{n+1;N}_{t}) \\
&\quad - f^{n}(t,X_t,Y^{n}_t,\overbrace{Z^{n;1}_{t},...,Z^{n;i-1}_{t},Z^{n;i}_{t},}^{i}Z^{n+1;i+1}_{t},...,Z^{n+1;N}_{t}). 
\end{aligned}
\end{equation*}
Let $(\iota^{n}_{r},c^{n}_{r})$ be defined as follows,
\begin{equation*}
\begin{cases}
\iota^{n}_{r} := \begin{pmatrix}
\frac{\delta f^{n;1}_{r}(Z^{n+1}_r)}{\delta Y^{n;1}_{r}}\mathbf 1_{\{\delta Y^{n;1}_{r}\neq 0\}},\ \frac{\delta f^{n;2}_{r}(Z^{n+1}_r)}{\delta Y^{n;2}_{r}}\mathbf 1_{\{\delta Y^{n;2}_{r}\neq 0\}},...,\ \frac{\delta f^{n;N}_{r}(Z^{n+1}_r)}{\delta Y^{n;N}_{r}}\mathbf 1_{\{\delta Y^{n;N}_{r}\neq 0\}}
\end{pmatrix}\in\mathbb{R}^{N\times N},\\
c^{n}_{r} := \begin{pmatrix}
\frac{\delta \tilde{f}^{n;1}_{r}}{\delta Z^{n;1}_{r}}\mathbf 1_{\{\delta Z^{n;1}_{r}\neq 0\}},\ \frac{\delta \tilde{f}^{n;2}_{r}}{\delta Z^{n;2}_{r}}\mathbf 1_{\{\delta Z^{n;2}_{r}\neq 0\}},...,\ \frac{\delta \tilde{f}^{n;N}_{r}}{\delta Z^{n;N}_{r}}\mathbf 1_{\{\delta Z^{n;N}_{r}\neq 0\}}
\end{pmatrix}\in\mathbb{R}^{N\times N}.
\end{cases}
\end{equation*}
It follows that  
\begin{equation}\label{e:delta Y'}
\delta Y^n_{t} = Y^{n+1}_{T_{n}} - \Xi_{T_n} + \int_{t\land T_{n}}^{T_{n}}\alpha^{n}_{r} \delta Y^{n}_{r}\eta(dr,X_{r}) + \int_{t\land T_{n}}^{T_{n}} \left(\iota^{n}_{r}\delta Y^{n}_{r} + c^{n}_{r} \delta Z^{n}_{r}\right)dr - \int_{t\land T_{n}}^{T_{n}} \delta Z^{n}_{r}dW_{r}.
\end{equation}

\textbf{Step 1.} For the stopping time $S\le T,$ denote by $A^{S;n}_s$ the unique solution of the following SDE, 
\begin{equation*}
A^{S;n}_{s} = I_{N} + \int_{S}^{s\vee S} (\alpha^{n}_r)^{\top} A^{S;n}_{r} \partial_{r}\eta(r,X_r)dr + \int_{S}^{s\vee S}(\iota^{n}_r)^{\top} A^{S;n}_{r}dr + \int_{S}^{s\vee S} (c^{n}_{r})^{\top} A^{S;n}_{r} dW_r ,\  s\in[0,T].
\end{equation*}
By Assumption~\ref{(A1)}, $\alpha^n,$ $\iota^n,$ and $c^n$ are uniformly bounded in $n$. In addition, by  the boundedness of the first derivative of $g$, the fact $m_{p,2}(Y^n;[0,T])<\infty$ implies that $m_{p,2}(\alpha^n;[0,T])<\infty$ (following the calculation as in \cite[(3.41)]{BSDEYoung-I}). Now set $\kappa := \lfloor \frac{\ln(k/(k-1))}{\ln 2} \rfloor  + 1$. By Lemma~\ref{lem:(a,b,c)} with $(v,j,\alpha,\iota,c,A^{S}) = (N,\kappa,\alpha^n,\iota^n,c^n,A^{S;n})$, there exists $(\tilde{\alpha}^n,\tilde{\iota}^n,\tilde{c}^{n})$ that satisfies \eqref{e:(i) of alpha,b,c}--\eqref{e:mpq tilde alpha} with $(\alpha,\iota,c)$ replaced by $(\alpha^n,\iota^n,c^n)$. Furthermore, denote by $\tilde{B}^{S;n}$ the unique solution of the following SDE,
\begin{equation*}
\tilde{B}^{S;n}_{t} = \tilde{I}_{N^{2^{\kappa}}\times 1} + \int_{S}^{t\vee S}(\tilde{\alpha}^{n}_{r})^{\top}\tilde{B}^{S;n}_{r}\partial_{r}\eta(r,X_r)dr + \int_{S}^{t\vee S}(\tilde{\iota}^{n}_{r})^{\top}\tilde{B}^{S;n}_{r} dr + \int_{S}^{t\vee S} (\tilde{c}^{n}_{r})^{\top}\tilde{B}^{S;n}_{r}d W_{r},
\end{equation*}
where $t\in[0,T].$ Then, by the fact that $2^{\kappa}\ge k/(k-1)$, and by \eqref{e:alpha,b,c} in Lemma~\ref{lem:(a,b,c)}, we have for every stopping time $S'$ that satisfies $S\le S'\le T,$ the following inequalities hold,
\begin{equation}\label{e:k/k-1 moment <= tilde B}
\left\{\mathbb{E}_{S}\left[|A^{S;n}_{S'}|^{\frac{k}{k-1}}\right]\right\}^{\frac{2^{\kappa}(k-1)}{k}} \le \mathbb{E}_{S}\left[|A^{S;n}_{S'}|^{2^{\kappa}}\right]\lesssim \left|\mathbb{E}_{S}\left[\tilde{B}^{S;n}_{S'}\right]\right|.
\end{equation}
Let $G^{S;n}_t$ denote the unique solution to the following SDE, which takes values in $\R^{N^{2^{\kappa}}\times N^{2^{\kappa}}}$.
\begin{equation*}
G^{S;n}_{t} = I_{N^{2^{\kappa}}} + \int_{S}^{t\vee S}(\tilde{\alpha}^{n}_{r})^{\top}G^{S;n}_{r}\partial_{r}\eta(r,X_r)dr + \int_{S}^{t\vee S}(\tilde{\iota}^{n}_{r})^{\top}G^{S;n}_{r}dr + \int_{S}^{t\vee S}(\tilde{c}^{n}_{r})^{\top}G^{S;n}_{r}dW_{r}.
\end{equation*}
Then, by the uniqueness shown in Lemma~\ref{lem:SDE},  $\tilde{B}^{S;n}_t = G^{S;n}_{t} \tilde{I}_{N^{2^{k}}\times 1}.$ Again, by Lemma~\ref{lem:(a,b,c)} with $(v,j,\alpha,\iota,c,A^{S}) = (N^{2^{\kappa}},1,\tilde{\alpha}^{n},\tilde{\iota}^{n},\tilde{c}^{n},G^{S;n})$, there exists a triplet $(\hat{\alpha}^{n},\hat{\iota}^{n},\hat{c}^{n})$ that satisfies \eqref{e:(i) of alpha,b,c}--\eqref{e:mpq tilde alpha} with $(\alpha,\iota,c)$ replaced by $(\tilde{\alpha}^n,\tilde{\iota}^n,\tilde{c}^n)$. Let $\tilde{H}^{S;n}$ be the unique solution of 
\begin{equation*}
\tilde{H}^{S;n}_{t} = \tilde{I}_{N^{2^{\kappa+1}}\times 1} + \int_{S}^{t\vee S}(\hat{\alpha}^{n}_{r})^{\top}\tilde{H}^{S;n}_{r}\partial_{r}\eta(r,X_r)dr + \int_{S}^{t\vee S}(\hat{\iota}^{n}_{r})^{\top}\tilde{H}^{S;n}_{r} dr + \int_{S}^{t\vee S} (\hat{c}^{n}_{r})^{\top}\tilde{H}^{S;n}_{r}d W_{r}.
\end{equation*}
It follows from \eqref{e:alpha,b,c} that, for every stopping time $S'$ such that $S\le S'\le T,$ 
\begin{equation}\label{e:tilde B <= tilde H}
\mathbb{E}_{S}\left[|G^{S;n}_{S'}|^{2}\right]\lesssim \left|\mathbb{E}_{S}\left[\tilde{H}^{S;n}_{S'}\right]\right|.
\end{equation}

Next, we will estimate $\E_{t\land S}[|G^{t\land S;n}_{S}|]$ for some specific stopping time $S$. By the well-posedness for linear BSDEs with bounded drivers
\cite[Corollary 3.3]{BSDEYoung-I}, there exists a unique solution $(\Psi^{n}_{t},\psi^{n}_t)$ of the following linear BSDE,
\begin{equation}\label{e:equation of Gamma, gamma}
\begin{cases}
d\Psi^{n}_t = -\hat{\alpha}^{n}_{t}\Psi^{n}_{t}\partial_{t}\eta(t,X_{t})dt - \left[\hat{\iota}^{n}_{t}\Psi^{n}_{t} + \hat{c}^{n}_{t}\psi^{n}_{t}\right]d t + \psi^{n}_{t}d W_{t},\ t\in[0,T_n],\\
\Psi^{n}_{T_n} = I_{N^{2^{\kappa +1}}}.
\end{cases}
\end{equation}
Note that for every $t\in[0,T]$,
\begin{equation*}
\mathbb{E}\bigg[\Big|\int_{0}^{T_n}|(\psi^{n}_{r})^{\top}(\hat{c}^{n}_{r})^{\top}\tilde{H}^{t\land T_n;n}_{r}|^{2}dr\Big|^{\frac{1}{2}}\bigg] \lesssim_{\Theta_{2}} \mathbb{E}\left[\|\tilde{H}^{t\land T_n;n}_{\cdot}\|_{\infty;[0,T_n]}^{2} + \int_{0}^{T_n}|\psi^{n}_{r}|^{2}dr\right]<\infty.
\end{equation*} 
Then applying It\^o's formula to $(\Psi^{n}_{\cdot})^{\top}\tilde{H}^{t\land T_{n};n}_{\cdot},$ and taking conditional expectation $\E_{t\land T_n}[\cdot]$, we have $(\Psi^{n}_{t})^{\top}\tilde{I}_{N^{2^{\kappa + 1}}\times 1} = (\Psi^{n}_{t\land T_{n}})^{\top} \tilde{I}_{N^{2^{\kappa + 1}}\times 1}  = \mathbb{E}_{t\land T_n}[\tilde{H}^{t\land T_{n};n}_{T_{n}}].$ To estimate $\Psi^{n},$ by the boundedness of $\hat{\iota}$ and \eqref{eq:ess-mpk}, we have
\begin{equation}\label{e:unbounded BSDEs-4'}
\begin{aligned}
\mathbb{E}_{t}\bigg[\Big\|\int_{\cdot}^{T_n}\hat{\iota}^{n}_r\Psi^{n}_r dr\Big\|^2_{p\text{-var};[t\land T_n,T_n]}\bigg]&\lesssim_{\Theta_2} \mathbb{E}_{t}\bigg[\Big|\int_{t\land T_{n}}^{T_{n}}|\Psi^{n}_{r}|dr\Big|^{2}\bigg]\\
&\lesssim |T-t|^{2}\left(1 + m_{p,2}(\Psi^n;[t,T])^2\right).
\end{aligned}
\end{equation}
In addition, it follows from the boundedness of $\hat{c}$ and the BDG inequality \cite[Lemma~A.1]{BSDEYoung-I} that 
\begin{equation}\label{e:unbounded BSDEs-5'}
\begin{aligned}
\mathbb{E}_{t}\bigg[\Big\|\int_{\cdot}^{T_n}\hat{c}^{n}_r \psi^{n}_r dr\Big\|^2_{p\text{-var};[t\land T_n,T_n]}\bigg]&\lesssim_{\Theta_{2}} \mathbb{E}_{t}\bigg[\Big|\int_{t\land T_{n}}^{T_{n}}|\psi^{n}_{r}|dr\Big|^{2}\bigg]\\
&\lesssim |T-t|\mathbb{E}_{t}\left[\int_{t\land T_{n}}^{T_{n}}|\psi^{n}_{r}|^{2}dr\right]\\
&\lesssim_{\Theta_{2}} |T-t| \mathbb{E}_{t}\bigg[\Big\|\int_{\cdot}^{T_n}\psi^{n}_{r}dW_r\Big\|^2_{p\text{-var};[t\land T_n,T_{n}]}\bigg].
\end{aligned}
\end{equation}
Furthermore, by Lemma~\ref{lem:estimate for Z} with 
$(q,f_t,g_t,Y_t,Z_t,S,\xi) = (2,\hat{\iota}^n_t \Psi^n_t + \hat{c}^n_{t}\psi^{n}_{t},\hat{\alpha}^{n}_t \Psi^n_t,\Psi^n_t,\psi^n_t,T_n,I_{N^{2^{\kappa+1}}}),$ 
\begin{equation}\label{e:estimate for psi^n}
\begin{aligned}
&\mathbb{E}_{t}\left[\Big\|\int_{\cdot}^{T_n}\psi^{n}_{r}dWr\Big\|^2_{p\text{-var};[t\land T_n,T_{n}]}\right]\\
&\lesssim_{\Theta_2} \mathbb{E}_{t}\left[\Big\|\int_{\cdot}^{T_n}\hat{c}^{n}_r \psi^{n}_r dr\Big\|^2_{p\text{-var};[t\land T_n,T_n]}\right] + \mathbb{E}_{t}\left[\Big\|\int_{\cdot}^{T_n}\hat{\iota}^{n}_r\Psi^{n}_r dr\Big\|^2_{p\text{-var};[t\land T_n,T_n]}\right]\\
&\quad + \mathbb{E}\left[\Big\|\int_{\cdot}^{T_{n}}\hat{\alpha}^{n}_{r}\Psi^{n}_{r}\eta(dr,X_{r})\Big\|^{2}_{p\text{-var};[t\land T_{n},T_{n}]}\right].
\end{aligned}
\end{equation}
Also, by the same calculation leading to \eqref{e:unbounded BSDEs-4}, with $(\alpha^n,\Gamma^n)$ replaced by $(\hat{\alpha}^n,\Psi^n),$ we have 
\begin{equation}\label{e:unbounded BSDEs-4+}
\begin{aligned}
&\mathbb{E}_{t}\left[\Big\|\int_{\cdot}^{T_{n}}\hat{\alpha}^n_{r}\Psi^{n}_{r}\eta(dr,X_{r})\Big\|^{2}_{p\text{-var};[t\land T_{n},T_{n}]}\right]\\
&\lesssim_{\Theta_{2}} |T-t|^{2\tau}n^{2(\lambda + \beta)}\left(1 + m_{p,2}(\hat{\alpha}^n;[0,T])^{2}\right)\left(1+ m_{p,2}(\Psi^{n};[t,T])^{2}\right).
\end{aligned}
\end{equation}
Moreover, by Lemma~\ref{lem:growth of (Y^n,Z^n)} we have for all $n\ge |x|,$ $m_{p,2}(Y^{n};[0,T])\lesssim_{\Theta_{2}} n^{\frac{\lambda + \beta}{\varepsilon}\vee 1}$. Then,
by the calculation leading to \eqref{e:m_p2 alpha''}, we have $m_{p,2}(\alpha^{n};[0,T])\lesssim_{\Theta_{2}} n^{\frac{\lambda + \beta}{\varepsilon}\vee 1}$ for all $n\ge |x|.$ Noting that
\begin{equation*}
m_{p,2}(\hat{\alpha}^{n};[0,T])\lesssim m_{p,2}(\tilde{\alpha}^{n};[0,T])\lesssim m_{p,2}(\alpha^{n};[0,T]),
\end{equation*}
we have $m_{p,2}(\hat{\alpha}^{n};[0,T])\lesssim_{\Theta_{2}} n^{\frac{\lambda + \beta}{\varepsilon}\vee 1}.$  Hence, by Eq.~\eqref{e:equation of Gamma, gamma}, and combining inequalities~\eqref{e:unbounded BSDEs-4'}, \eqref{e:unbounded BSDEs-5'}, \eqref{e:estimate for psi^n}, and \eqref{e:unbounded BSDEs-4+},  there exists a constant $C>0$ such that for all $n\ge |x|,$ 
\begin{equation*}
\begin{aligned}
&m_{p,2}(\Psi^n;[t,T]) + \|\psi^n\|_{2\text{-BMO};[t,T]}\\
&\le  C |T-t|^{\frac{1}{2}}\|\psi^n\|_{2\text{-BMO};[t,T]} + C|T-t|^{\tau}n^{\frac{(1+\varepsilon)(\lambda + \beta)}{\varepsilon}\vee (\lambda + \beta + 1)}\left(1 + m_{p,2}(\Psi^n;[t,T])\right).
\end{aligned}
\end{equation*}
Then we have $m_{p,2}(\Psi^{n};[T-\delta,T])\le 1,$ where 
\begin{equation}\label{e:delta value'}
\delta := C^{-2}\land \left((2C)^{-\frac{1}{\tau}}  n^{-(\frac{(1+\varepsilon)(\lambda + \beta)}{\tau\varepsilon}\vee\frac{\lambda+\beta+1}{\tau})}\right).
\end{equation} 
Therefore, by \eqref{eq:ess-mpk}, there exists a constant $C^{\dagger}>1$ such that $\esssup_{t\in[T-\delta,T],\omega\in\Omega}|\Psi^n_t|\le C^{\dagger}.$
For $0\le i\le \lfloor T/\delta\rfloor,$ denote $T^{i}_{n} := T_n\land (T-i\delta),$ and denote by $(\Psi^{n;i}_t,\psi^{n;i}_t)$ the unique solution of
\begin{equation*}
\begin{cases}
d\Psi^{n;i}_{t} = -\hat{\alpha}^{n}
_{t}\Psi^{n;i}_{t}\partial_{t}\eta(t,X_{t})dt - (\hat{\iota}^{n}_t \Psi^{n;i}_t + \hat{c}^{n}_t \psi^{n;i}_t) dt + \psi^{n;i}_{t}dW_{t},\ t\in[0,T^{i}_{n}],\\
\Psi^{n;i}_{T^{i}_{n}} = I_{N^{2^{\kappa+1}}}.
\end{cases}
\end{equation*}
In the following, we always assume $n\ge|x|$. By the procedure leading to $|\Psi^n_t|\le C^{\dagger},$ we have
\begin{equation}\label{e:Gamma<=2''}
\esssup_{t\in[T-(i+1)\delta,T - i\delta],\omega\in\Omega}|\Psi^{n;i}_t|\le C^{\dagger}.
\end{equation}
In addition, by It\^o's formula again, we have $(\Psi^{n;i}_{t})^{\top} \tilde{I}_{N^{2^{\kappa+1}}\times 1} = \mathbb{E}_{t\land T^{i}_{n}}[\tilde{H}^{t\land T^{i}_{n};n}_{T^{i}_{n}}].$ Then, by \eqref{e:tilde B <= tilde H} and \eqref{e:Gamma<=2''}, it follows that there exits a constant $C^{\diamond}>1$ such that for all $t\in [T-(i+1)\delta,T-i\delta],$
\begin{equation}\label{e:E[G] <= Gamma^n,i}
\mathbb{E}_{t\land T^{i}_{n}}\left[|G^{t\land T^{i}_{n};n}_{T^{i}_{n}}|\right]\le C^{\diamond} |\Psi^{n;i}_{t}|^{\frac{1}{2}}\le C^{\diamond} |C^{\dagger}|^{\frac{1}{2}}.
\end{equation}

Now we are ready to estimate $\E_{t\land T_n}[|A^{t\land T_n;n}_{T_{n}}|^{\frac{k}{k-1}}]$. Denote by $(\Gamma^{n;i},\gamma^{n;i})$ the unique solution of 
\begin{equation*}
\begin{cases}
d\Gamma^{n;i}_{t} = -\tilde{\alpha}^{n}_{t}\Gamma^{n;i}_{t}\partial_{t}\eta(t,X_{t})dt - (\tilde{\iota}^{n}_{t}\Gamma^{n;i}_{t} + \tilde{c}^{n}_{t}\gamma^{n;i}_{t})dt + \gamma^{n;i}_{t}dW_{t},\ t\in[0,T^{i}_{n}],\\
\Gamma^{n;i}_{T^{i}_{n}} = I_{N^{2^{\kappa}}}.
\end{cases}
\end{equation*}
By It\^o's formula, we have $(\Gamma^{n;i}_{t})^{\top} = \mathbb{E}_{t\land T^{i}_{n}}[G^{t\land T^{i}_{n};n}_{T^{i}_{n}}].$ 
Then, by \eqref{e:k/k-1 moment <= tilde B} and the fact that $\tilde{B}^{t\land T_n;n}_{T_{n}} = G^{t\land T_n;n}_{T_{n}}\tilde{I}_{N^{2^{\kappa}}\times 1}$, we obtain
\begin{equation}\label{e:k/k-1 moment <= Gamma}
\mathbb{E}_{t\land T_{n}}\left[|A^{t\land T_n;n}_{T_{n}}|^{\frac{k}{k-1}}\right]\lesssim |\Gamma^{n;0}_{t}|^{\frac{k}{2^{\kappa}(k-1)}}.
\end{equation}
For $i' \le i$, by Lemma~\ref{lem:SDE}, the equality $(G^{t\land T^{i}_{n};n}_{T^{i'}_{n}})^{\top} = (G^{t\land T^{i}_{n};n}_{ T^{i}_{n}})^{\top}(G^{T^{i}_{n};n}_{T^{i'}_{n}})^{\top}$ holds. Hence, by the fact that $(\Gamma^{n;i'}_{t})^{\top} = \mathbb{E}_{t\land T^{i}_{n}}[G^{t\land T^{i}_{n};n}_{T^{i'}_{n}}]$ for $t\le T-i\delta$ (noting $t\land T^{i}_{n} = t\land T^{i'}_{n}$), and by \eqref{e:E[G] <= Gamma^n,i}, we have 
\begin{equation*}
\begin{aligned}
\esssup_{t\in[T-(i+1)\delta,T-i\delta],\omega\in\Omega}\left|\Gamma^{n;i'}_{t}\right| &= \sup_{t\in[T-(i+1)\delta,T-i\delta]}\left\|\mathbb{E}_{t\land T^{i}_{n}}\left[(G^{t\land T^{i}_{n};n}_{T^{i'}_{n}})^{\top}\right]\right\|_{L^{\infty}(\Omega)} \\
&\le \sup_{t\in[T-(i+1)\delta,T-i\delta]}\left\|\mathbb{E}_{t\land T^{i}_{n}}\left[\left|(G^{t\land T^{i}_{n};n}_{ T^{i}_{n}})^{\top}\right|\cdot \left|\mathbb{E}_{T^{i}_{n}}\left[(G^{T^{i}_{n};n}_{T^{i'}_{n}})^{\top}\right]\right|\right]\right\|_{L^{\infty}(\Omega)}\\
&\le \sup_{t\in[T-(i+1)\delta,T-i\delta]}\left\|\mathbb{E}_{t\land T^{i}_{n}}\left[\left|(G^{t\land T^{i}_{n};n}_{ T^{i}_{n}})^{\top}\right|\right] \right\|_{L^{\infty}(\Omega)} \cdot \|\Gamma^{n;i'}_{T^{i}_{n}}\|_{L^{\infty}(\Omega)}\\
&\le C^{\diamond} |C^{\dagger}|^{\frac{1}{2}} \esssup_{t\in[T-i\delta,T-(i-1)\delta],\omega\in\Omega}|\Gamma^{n;i'}_{t}|.
\end{aligned}
\end{equation*}
Let $i'=0.$ Then, by the following inequality, 
\begin{equation*}
\esssup_{t\in[T-\delta,T],\omega\in\Omega}|\Gamma^{n;0}_{t}| = \sup_{t\in[T-\delta,T]}\left\|\mathbb{E}_{t\land T_{n}}\left[(G^{t\land T_{n};n}_{T_{n}})^{\top}\right]\right\|_{L^{\infty}(\Omega)} \le C^{\diamond}|C^{\dagger}|^{\frac{1}{2}},
\end{equation*}
it follows that for all $i=0,1,...,\lfloor T/\delta\rfloor$ and $t\in [T-(i+1)\delta,T-i\delta],$
\begin{equation}\label{e:e:existence-7''}
|\Gamma^{n;0}_{t}| \le C^{\diamond}|C^{\dagger}|^{\frac{1}{2}} \esssup_{t\in[T-i\delta,T-(i-1)\delta],\omega\in\Omega}|\Gamma^{n;0}_{t}| \le \cdots \le |C^{\diamond}|^{\frac{T}{\delta} + 1} |C^{\dagger}|^{\frac{T}{2\delta} + \frac{1}{2}}.
\end{equation}
Finally, by \eqref{e:delta value'} and \eqref{e:e:existence-7''}, there exists a constant $C^*>0$ such that for all $n\ge |x|,$
\begin{equation}\label{e:exis unbounded BSDE-7'}
\esssup_{t\in[0,T],\omega\in\Omega}|\Gamma^{n;0}_{t}|\le \exp\left\{C^* n^{\frac{(1+\varepsilon)(\lambda + \beta)}{\tau\varepsilon}\vee\frac{\lambda+\beta+1}{\tau}}\right\}.
\end{equation} 

\textbf{Step 2.} Since $(\delta Y^{n}, \delta Z^{n})$ solves Eq.~\eqref{e:delta Y'}, by It\^o's formula again, we have 
\begin{equation}\label{e:delta Y = A (Y - Xi)'}
\delta Y^{n}_{t} = \mathbb{E}_{t\land T_{n}} \left[(A^{t\land T_{n};n}_{T_{n}})^{\top}(Y^{n+1}_{T_{n}} - \Xi_{T_{n}})\right].
\end{equation}
On the other hand, by following the same procedure that led to \eqref{e:<= exp -1/C n^2}, we have
\[\|Y^{n+1}_{T_n} - \Xi_{T_n}\|_{L^{k}(\Omega)}\le \|\Xi_{T_n} - \Xi_{T_{n+1}}\|_{L^{k}(\Omega)} +  \|Y^{n+1}_{T_n} - \Xi_{T_{n+1}}\|_{L^{k}(\Omega)}\lesssim_{\Theta_{2},x} \exp\left\{-C^{\star}n^{2}\right\}, \]
where $C^{\star}>0$ is a constant. Combining \eqref{e:k/k-1 moment <= Gamma}, \eqref{e:exis unbounded BSDE-7'}, and the above estimate, and applying H\"older's inequality to \eqref{e:delta Y = A (Y - Xi)'}, it follows that
\begin{equation}\label{e:delta Y^n <= exp}
\begin{aligned}
&\mathbb{E}\left[\|\delta Y^{n}_{\cdot}\|_{\infty;[0,T_{n}]}\right]\\
&\le \mathbb{E}\left[\Big\{\sup_{t\in[0,T]}\left|\mathbb{E}_{t\land T_n}\left[|Y^{n+1}_{T_n} - \Xi_{T_{n}}|^{k}\right]\right|\Big\}^{\frac{1}{k}}\Big\{\sup_{t\in[0,T]}\mathbb{E}_{t\land T_n }\left[|A^{t\land T_{n};n}_{T_{n}}|^{\frac{k}{k-1}}\right]\Big\}^{\frac{k-1}{k}}\right]\\
&\lesssim_{\Theta_2} \|Y^{n+1}_{T_{n}} - \Xi_{T_{n}}\|_{L^{k}(\Omega)}\Big\{\esssup_{t\in[0,T],\omega\in\Omega}|\Gamma^{n;0}_{t}|\Big\}^{\frac{1}{2^{\kappa}}}\\
&\lesssim_{\Theta_{2},x} \exp\Big\{-C^{\star}n^{2} + \frac{C^{*}}{2^{\kappa}}n^{\frac{(1+\varepsilon)(\lambda + \beta)}{\tau\varepsilon}\vee\frac{\lambda+\beta+1}{\tau}}\Big\}.
\end{aligned}
\end{equation}
Then the existence follows from the same procedure as in \eqref{e:dY<=h(Y)Gamma}--\eqref{e:(Y^n,Z^n) -> (Y,Z)}.

By approximation results \cite[Corollary~3.1 and Lemma~A.3]{BSDEYoung-I}, the inequality
\begin{equation*}
\mathbb{E}\left[\|\delta Y^{n}_{\cdot}\|_{\infty;[0,T_{n}]}\right]\lesssim_{\Theta_{2},x} \exp\Big\{-C^{\star}n^{2} + \frac{C^{*}}{2^{\kappa}}n^{\frac{(1+\varepsilon)(\lambda + \beta)}{\tau\varepsilon}\vee\frac{\lambda+\beta+1}{\tau}}\Big\}
\end{equation*}
implied by \eqref{e:delta Y^n <= exp} still holds for general $\eta\in C^{\tau,\lambda;\beta}([0,T]\times \R^n;\R^M)$. Therefore, the existence of solutions also follows.
\end{proof}

{\bf Acknowledgment}  
J. Song is partially supported by National Natural Science Foundation of China (No. 12471142); J. Song and H. Zhang are supported by the Fundamental Research Funds for the Central Universities. H. Zhang is partially supported by NSF of China and Shandong (no. 12031009, ZR2023MA026); Young Research Project of Tai-Shan (no. tsqn202306054);
DFG CRC/TRR 388 ``Rough Analysis, Stochastic Dynamics and Related Fields'', Projects B04 and B05.

\bibliographystyle{plain}
\bibliography{Reference-BSDE}

\begin{thebibliography}{10}

\bibitem{Addona2022}
Davide Addona, Luca Lorenzi, and Gianmario Tessitore.
\newblock Regularity results for nonlinear {Y}oung equations and applications.
\newblock {\em J. Evol. Equ.}, 22(1):Paper No. 3, 34, 2022.

\bibitem{Barles1995}
Guy Barles, Christian Daher, and Marc Romano.
\newblock Convergence of numerical schemes for parabolic equations arising in
  finance theory.
\newblock {\em Math. Models Methods Appl. Sci.}, 5(1):125--143, 1995.

\bibitem{bender2000bsdes}
Christian Bender and Michael Kohlmann.
\newblock {BSDE}s with stochastic {L}ipschitz condition.
\newblock Technical report, University of Konstanz, Center of Finance and
  Econometrics (CoFE), 2000.

\bibitem{briand2008bsdes}
Philippe Briand and Fulvia Confortola.
\newblock {BSDE}s with stochastic {L}ipschitz condition and quadratic {PDE}s in
  {H}ilbert spaces.
\newblock {\em Stochastic Process. Appl.}, 118(5):818--838, 2008.

\bibitem{BuckdahnZhang}
Rainer Buckdahn, Christian Keller, Jin Ma, and Jianfeng Zhang.
\newblock Fully nonlinear stochastic and rough {PDE}s: classical and viscosity
  solutions.
\newblock {\em Probab. Uncertain. Quant. Risk}, 5:Paper No. 7, 59, 2020.

\bibitem{buckdahn2016generalized}
Rainer Buckdahn and Tianyang Nie.
\newblock Generalized {H}amilton-{J}acobi-{B}ellman equations with {D}irichlet
  boundary condition and stochastic exit time optimal control problem.
\newblock {\em SIAM J. Control Optim.}, 54(2):602--631, 2016.

\bibitem{Bugini2025a}
Fabio Bugini, Peter~K. Friz, and Wilhelm Stannat.
\newblock Parameter dependent rough {SDE}s with applications to rough {PDE}s.
\newblock {\em arXiv preprint arXiv:2409.11330}, 2024.

\bibitem{candil2022localization}
David Jean-Michel Candil.
\newblock {\em Localization Errors of the Stochastic Heat Equation}.
\newblock PhD thesis, EPFL, 2022.

\bibitem{Caruana2009}
Michael Caruana and Peter~K. Friz.
\newblock Partial differential equations driven by rough paths.
\newblock {\em J. Differential Equations}, 247(1):140--173, 2009.

\bibitem{caruana2011rough}
Michael Caruana, Peter~K. Friz, and Harald Oberhauser.
\newblock A (rough) pathwise approach to a class of non-linear stochastic
  partial differential equations.
\newblock {\em Ann. Inst. H. Poincar\'e{} C Anal. Non Lin\'eaire},
  28(1):27--46, 2011.

\bibitem{Cheridito2014}
Patrick Cheridito and Kihun Nam.
\newblock B{SDE}s with terminal conditions that have bounded {M}alliavin
  derivative.
\newblock {\em J. Funct. Anal.}, 266(3):1257--1285, 2014.

\bibitem{clark1970representation}
John M.~C. Clark.
\newblock The representation of functionals of {B}rownian motion by stochastic
  integrals.
\newblock {\em Ann. Math. Statist.}, 41:1282--1295, 1970.

\bibitem{Coghi-Nilssen21}
Michele Coghi and Torstein Nilssen.
\newblock Rough nonlocal diffusions.
\newblock {\em Stochastic Process. Appl.}, 141:1--56, 2021.

\bibitem{Cont2005}
Rama Cont and Ekaterina Voltchkova.
\newblock A finite difference scheme for option pricing in jump diffusion and
  exponential {L}\'evy models.
\newblock {\em SIAM J. Numer. Anal.}, 43(4):1596--1626, 2005.

\bibitem{crandall1992user}
Michael~G. Crandall, Hitoshi Ishii, and Pierre-Louis Lions.
\newblock User's guide to viscosity solutions of second order partial
  differential equations.
\newblock {\em Bull. Amer. Math. Soc. (N.S.)}, 27(1):1--67, 1992.

\bibitem{Deya2019}
Aur\'elien Deya, Massimiliano Gubinelli, Martina Hofmanov\'a, and Samy Tindel.
\newblock A priori estimates for rough {PDE}s with application to rough
  conservation laws.
\newblock {\em J. Funct. Anal.}, 276(12):3577--3645, 2019.

\bibitem{deya2012non}
Aur\'elien Deya, Massimiliano Gubinelli, and Samy Tindel.
\newblock Non-linear rough heat equations.
\newblock {\em Probab. Theory Related Fields}, 153(1-2):97--147, 2012.

\bibitem{DiehlFriz}
Joscha Diehl and Peter~K. Friz.
\newblock Backward stochastic differential equations with rough drivers.
\newblock {\em Ann. Probab.}, 40(4):1715--1758, 2012.

\bibitem{Diehl2017b}
Joscha Diehl, Peter~K. Friz, and Wilhelm Stannat.
\newblock Stochastic partial differential equations: a rough paths view on weak
  solutions via {F}eynman-{K}ac.
\newblock {\em Ann. Fac. Sci. Toulouse Math. (6)}, 26(4):911--947, 2017.

\bibitem{DiehlZhang}
Joscha Diehl and Jianfeng Zhang.
\newblock Backward stochastic differential equations with {Y}oung drift.
\newblock {\em Probab. Uncertain. Quant. Risk}, 2:Paper No. 5, 17, 2017.

\bibitem{ElKaroui1997}
Nicole El~Karoui, Shige Peng, and Marie-Claire Quenez.
\newblock Backward stochastic differential equations in finance.
\newblock {\em Math. Finance}, 7(1):1--71, 1997.

\bibitem{Evans}
Lawrence~C. Evans.
\newblock {\em Partial differential equations}, volume~19 of {\em Graduate
  Studies in Mathematics}.
\newblock American Mathematical Society, Providence, RI, second edition, 2010.

\bibitem{Friz2014a}
Peter~K. Friz and Harald Oberhauser.
\newblock Rough path stability of (semi-)linear {SPDE}s.
\newblock {\em Probab. Theory Related Fields}, 158(1-2):401--434, 2014.

\bibitem{FrizVictoir}
Peter~K. Friz and Nicolas~B. Victoir.
\newblock {\em Multidimensional stochastic processes as rough paths}, volume
  120 of {\em Cambridge Studies in Advanced Mathematics}.
\newblock Cambridge University Press, Cambridge, 2010.
\newblock Theory and applications.

\bibitem{Gerencser2017}
M\'at\'e{} Gerencs\'er and Istv\'an Gy\"ongy.
\newblock Localization errors in solving stochastic partial differential
  equations in the whole space.
\newblock {\em Math. Comp.}, 86(307):2373--2397, 2017.

\bibitem{gubinelli2006young}
Massimiliano Gubinelli, Antoine Lejay, and Samy Tindel.
\newblock Young integrals and {SPDE}s.
\newblock {\em Potential Anal.}, 25(4):307--326, 2006.

\bibitem{Gubinelli2010}
Massimiliano Gubinelli and Samy Tindel.
\newblock Rough evolution equations.
\newblock {\em Ann. Probab.}, 38(1):1--75, 2010.

\bibitem{Hesse2022}
Robert Hesse and Alexandra Neamţu.
\newblock Global solutions for semilinear rough partial differential equations.
\newblock {\em Stoch. Dyn.}, 22(2):Paper No. 2240011, 18, 2022.

\bibitem{Hilber2013}
Norbert Hilber, Oleg Reichmann, Christoph Schwab, and Christoph Winter.
\newblock {\em Computational methods for quantitative finance}.
\newblock Springer Finance. Springer, Heidelberg, 2013.
\newblock Finite element methods for derivative pricing.

\bibitem{Hocquet2024}
Antoine Hocquet and Alexandra Neamţu.
\newblock Quasilinear rough evolution equations.
\newblock {\em Ann. Appl. Probab.}, 34(5):4268--4309, 2024.

\bibitem{HuLe}
Yaozhong Hu and Khoa L\^e.
\newblock Nonlinear {Y}oung integrals and differential systems in {H}\"older
  media.
\newblock {\em Trans. Amer. Math. Soc.}, 369(3):1935--2002, 2017.

\bibitem{HuLiMi}
Yaozhong Hu, Juan Li, and Chao Mi.
\newblock B{SDE}s generated by fractional space-time noise and related {SPDE}s.
\newblock {\em Appl. Math. Comput.}, 450:Paper No. 127979, 30, 2023.

\bibitem{hu2012feynman}
Yaozhong Hu, Fei Lu, and David Nualart.
\newblock Feynman-{K}ac formula for the heat equation driven by fractional
  noise with {H}urst parameter {$H<1/2$}.
\newblock {\em Ann. Probab.}, 40(3):1041--1068, 2012.

\bibitem{HuNualartSong-2011}
Yaozhong Hu, David Nualart, and Jian Song.
\newblock Feynman-{K}ac formula for heat equation driven by fractional white
  noise.
\newblock {\em Ann. Probab.}, 39(1):291--326, 2011.

\bibitem{Kangro2000}
Raul Kangro and Roy Nicolaides.
\newblock Far field boundary conditions for {B}lack-{S}choles equations.
\newblock {\em SIAM J. Numer. Anal.}, 38(4):1357--1368, 2000.

\bibitem{karatzas1991brownian}
Ioannis Karatzas and Steven~E. Shreve.
\newblock {\em Brownian motion and stochastic calculus}, volume 113 of {\em
  Graduate Texts in Mathematics}.
\newblock Springer-Verlag, New York, second edition, 1991.

\bibitem{Lamberton2008}
Damien Lamberton and Bernard Lapeyre.
\newblock {\em Introduction to stochastic calculus applied to finance}.
\newblock Chapman \& Hall/CRC Financial Mathematics Series. Chapman \&
  Hall/CRC, Boca Raton, FL, second edition, 2008.

\bibitem{lejay2010controlled}
Antoine Lejay.
\newblock Controlled differential equations as {Y}oung integrals: a simple
  approach.
\newblock {\em J. Differential Equations}, 249(8):1777--1798, 2010.

\bibitem{li2025random}
Xinying Li, Yaqi Zhang, and Shengjun Fan.
\newblock Weighted solutions of random time horizon {BSDE}s with stochastic
  monotonicity and general growth generators and related {PDE}s.
\newblock {\em Stochastic Process. Appl.}, 190:Paper No. 104758, 2025.

\bibitem{Liang2024a}
Jiahao Liang and Shanjian Tang.
\newblock Mild solution of semilinear rough stochastic evolution equations.
\newblock {\em arXiv preprint arXiv:2401.16815}, 2024.

\bibitem{lions1982optimal}
Pierre-Louis Lions and Jos{\'e}-Luis Menaldi.
\newblock Optimal control of stochastic integrals and
  {H}amilton-{J}acobi-{B}ellman equations. {I}.
\newblock {\em SIAM J. Control Optim.}, 20(1):58--81, 1982.

\bibitem{Matache2004}
Ana-Maria Matache, Tobias von Petersdorff, and Christoph Schwab.
\newblock Fast deterministic pricing of options on {L}\'evy driven assets.
\newblock {\em M2AN Math. Model. Numer. Anal.}, 38(1):37--71, 2004.

\bibitem{nualart2006malliavin}
David Nualart.
\newblock {\em The {M}alliavin calculus and related topics}.
\newblock Probability and its Applications (New York). Springer-Verlag, Berlin,
  second edition, 2006.

\bibitem{Nualart-Xia20}
David Nualart and Panqiu Xia.
\newblock On nonlinear rough paths.
\newblock {\em ALEA Lat. Am. J. Probab. Math. Stat.}, 17(1):545--587, 2020.

\bibitem{Oksendal}
Bernt {\O}ksendal.
\newblock {\em Stochastic differential equations}.
\newblock Universitext. Springer-Verlag, Berlin, sixth edition, 2003.
\newblock An introduction with applications.

\bibitem{pardoux1998backward}
\'Etienne Pardoux.
\newblock Backward stochastic differential equations and viscosity solutions of
  systems of semilinear parabolic and elliptic {PDE}s of second order.
\newblock In {\em Stochastic analysis and related topics, {VI} ({G}eilo,
  1996)}, volume~42 of {\em Progr. Probab.}, pages 79--127. Birkh\"auser
  Boston, Boston, MA, 1998.

\bibitem{PardouxPeng1990}
\'Etienne Pardoux and Shige Peng.
\newblock Adapted solution of a backward stochastic differential equation.
\newblock {\em Systems Control Lett.}, 14(1):55--61, 1990.

\bibitem{BSDEYoung-I}
Jian Song, Huilin Zhang, and Kuan Zhang.
\newblock Backward stochastic differential equations with nonlinear {Y}oung
  drivers {I}.
\newblock {\em arXiv preprint arXiv:2504.18632}, 2025.

\bibitem{Siska2012}
David \v{S}i\v{s}ka.
\newblock Error estimates for approximations of {A}merican put option price.
\newblock {\em Comput. Methods Appl. Math.}, 12(1):108--120, 2012.

\bibitem{zhang2017backward}
Jianfeng Zhang.
\newblock {\em Backward stochastic differential equations}, volume~86 of {\em
  Probability Theory and Stochastic Modelling}.
\newblock Springer, New York, 2017.
\newblock From linear to fully nonlinear theory.

\end{thebibliography}

\end{document}